\documentclass[a4paper,reqno]{amsart}
\usepackage[utf8]{inputenc}
\usepackage{amsmath,amsfonts,amsthm,amssymb}
\usepackage{graphicx,enumerate}
\usepackage{mathtools,bm}
\usepackage{mathabx}
\usepackage{todonotes}
\usepackage{forest}
\pdfoutput=1
\makeatletter
\def\l@section{\@tocline{1}{12pt plus2pt}{0pt}{}{\bfseries}}
\def\l@subsection{\@tocline{2}{0pt}{2pc}{2pc}{}}
\makeatother
\makeatletter
\def\subsection{\@startsection{subsection}{2}{\z@}%
	{-3.25ex\@plus -1ex \@minus -.2ex}%
	{1.5ex \@plus .2ex}%
	{\normalfont\bfseries\boldmath}}
\def\subsubsection{\@startsection{subsubsection}{3}%
	\z@{.5\linespacing\@plus.7\linespacing}{-.5em}%
	{\normalfont\bfseries\boldmath}}
\renewcommand\paragraph{\@startsection{paragraph}{4}{\z@}%
	{3.25ex \@plus1ex \@minus.2ex}%
	{-1em}%
	{\normalfont\normalsize\bfseries}}
\makeatother

\numberwithin{equation}{section}
\newtheorem{theorem}{Theorem}[section]
\newtheorem*{theorem*}{Theorem}
\newtheorem{corollary}[theorem]{Corollary}
\newtheorem{observation}[theorem]{Observation}

\newtheorem{proposition}[theorem]{Proposition}

\theoremstyle{definition}
\newtheorem{remark}[theorem]{Remark}

\newcommand{\R}{\mathbb{R}}
\newcommand{\C}{\mathbb{C}}
\newcommand{\Z}{\mathbb{Z}}
\newcommand{\N}{\mathbb{N}}
\newcommand{\s}{\mathcal{S}}

\renewcommand{\P}{\mathbb{P}}

\newcommand{\e}{\varepsilon}
\newcommand{\jx}{j(x)}
\newcommand{\um}{\underline{m}}
\newcommand{\urho}{\underline{\rho}}
\newcommand{\uL}{\underline{\Lambda}}
\newcommand{\sdfrac}[2]{\mbox{\small$\displaystyle\frac{#1}{#2}$}}

\def\beq{\begin{equation}}
	\def\eeq{\end{equation}}

\def\beq{\begin{equation}}
	\def\eeq{\end{equation}}

\def\a{\alpha}
\def\r{\mathcal{R}}
\def\g{\gamma}
\def\b{\beta}
\def\F{\mathcal{F}}
\def\G{\mathcal{G}}
\def\H{\mathcal{H}}
\def\v{\varepsilon}
\def\l{\lambda}
\def\o{\omega}
\def\J{\mathcal{J}}
\def\A{\mathcal{A}}
\def\B{\mathcal{B}}
\def\beq{\begin{equation}}
	\def\eeq{\end{equation}}

\def\beq{\begin{equation}}
	\def\eeq{\end{equation}}

\def\M{\mathcal{M}}

\def\d{\delta}
\def\ep{\epsilon}
\def\L{\Lambda}
\def\O{\Omega}
\def\I{\mathcal{I}}

\DeclareMathOperator{\supp}{supp}
\DeclareMathOperator{\sgn}{sgn}

\allowdisplaybreaks
\title[Operators along hybrid curves]{\Large{Non-zero to zero curvature transition: Operators along hybrid curves with no quadratic (quasi-)resonances}}

\author[A. Gaitan]{Alejandra Gaitan}
\address[A. Gaitan]{Department of Mathematics, Purdue University, IN 47907, USA}
\email{alejandra.gaitanm@alumni.purdue.edu}
\author[V. Lie]{Victor Lie}
\address[V. Lie]{Department of Mathematics, Purdue University, IN 47907, USA \& Institute of Mathematics of the Romanian Academy, Bucharest, RO 70700, P.O. Box 1-764, Romania.}
\email{vlie@purdue.edu}

\thanks{\textit{Key words and phrases.} Time-frequency analysis, zero/non-zero curvature, hybrid operators, LGC method, sparse-uniform decomposition, time-frequency correlation set.}

\date{\today}
\begin{document}
\maketitle

\begin{abstract} Building on \cite{Lie19}, this paper develops a unifying study on the boundedness properties of several representative classes of \emph{hybrid} operators, \textit{i.e.} operators that enjoy both zero and non-zero curvature features. Specifically, via the LGC-method, we provide suitable $L^p$ bounds for three classes of operators: (1) Carleson-type operators, (2) Hilbert transform along variable curves, and, taking the center stage, (3) Bilinear Hilbert transform and bilinear maximal operators along curves. All these classes of operators will be studied in the context of hybrid curves with no quadratic resonances.

The above study is interposed between two naturally derived topics:

i) A prologue providing a first rigorous account on how the presence/absence of a higher order modulation invariance property interacts with and determines the nature of the method employed for treating operators with such a property.

ii) An epilogue revealing how several key ingredients within our present study can blend and inspire a short, intuitive new proof of the smoothing inequality that plays the central role in the analysis of the curved version of the triangular Hilbert transform treated in \cite{CDR20}.
\end{abstract}

\setcounter{tocdepth}{2}
\tableofcontents
\section{Introduction}

\subsection{A unified theory: context}\label{cont}
In this paper we continue the complex program focusing on the interplay between zero and non-zero curvature features in harmonic analysis. Among the fundamental themes that stay at the heart of this program we mention the study of:
\begin{itemize}
	\item (A) the \emph{(generalized) Carleson operator} -- in relation to Luzin's conjecture on the pointwise convergence of Fourier series;
	\item (B) the \emph{Hilbert transform and maximal operator along variable curves} -- in relation to Zygmund's conjecture on differentiation along vector fields; and
	\item (C) the \emph{bilinear Hilbert transform along curves} -- stemming from Calderon's study of the Cauchy transform on Lipschitz curves.
\end{itemize}

In the zero curvature context, besides translation and dilation symmetries, all of the above mentioned operators also enjoy some suitable modulation invariant structure. It is this last feature that: 1) is responsible for the level of difficulty and, in many cases, still open character of some of the problems addressing the items (A), (B) or (C), and 2) requires the involvement of (generalized/higher order) wave-packet analysis in order to achieve any relevant progress on such problems.

In contrast with the first scenario, in the non-zero curvature context the modulation invariance is absent, thus conveying a simpler though, for some of the most interesting problems, still difficult (or yet unknown) treatment. Originally arising independently of the zero-curvature context--as is the case of the item (B) above\footnote{For more on this please see the extended historical review provided in Section 1.6. of \cite{Lie19}.}, the non-zero curvature themes developed later in parallel with, and often serving as toy-models for, their zero-curvature counterparts. However, this latter direction was very limited in scope, since, until very recently, the known approaches for the non-zero curvature problems were quite different and often ad-hoc in nature relative to their zero-curvature analogues, thus providing very little insight into the modulation invariant situations.

It is thus natural to seek a unified framework for the seemingly quite distinct approaches to the zero/nonzero curvature end-points of the spectrum that would offer a natural treatment for the intricate class of \emph{hybrid} problems, that is, problems that include both zero and non-zero curvature features.
As will be detailed soon in Section \ref{Unif}, such a unified framework involves two possible resolutions according to the nature of the modulation invariance properties of the operator under analysis:

\begin{itemize}
	\item \textsf{[\underline{Linear but \textbf{no} higher modulation invariance}: LGC methodology]}:

Introduced in \cite{Lie19}, this method combines several key ingredients: (1) a time-frequency partition of the \emph{ambient} universe that has as an effect a \emph{linearized} behavior of the phase of the multiplier; (2) a time-frequency discretization of the \emph{input} function(s) using \emph{Gabor frames} (windowed Fourier decomposition/micro-local analysis); (3) the analysis of some suitably defined \emph{time-frequency correlation set} derived via a \emph{sparse-uniform dichotomy}. This method provides a unified approach for problems having zero and/or non-zero curvature features that \emph{do not} have generalized (higher than linear) modulation invariance properties---for some revelatory examples, see I and II in Section \ref{compdisc};

	\item \textsf{[\underline{Linear \textbf{and} higher modulation invariance}: Relational time-frequency analysis]}:

This method was first introduced in the realm of wave-packet analysis---and in particular modulation invariant operators---in \cite{Lie09}, \cite{Lie09Thesis} and \cite{Lie20} in connection with the study of the Polynomial Carleson operator; the origin of this approach goes back to the work of C. Fefferman on the space-phase localization properties of differential operators, \cite{Fefunc}. At the heart of the method lies a \emph{generalized Heisenberg principle interpretation associated with higher order wave-packets} that proves essential in addressing hybrid problems that have both linear and higher order modulation invariance properties---for a prototypical example, see III in Section \ref{compdisc}.
\end{itemize}

\subsection{A brief outlook}

In the previous section we have seen the motivation behind developing a unified theory for the zero/nonzero curvature paradigm in the context brought by the three central themes stated as (A), (B) and (C). Since the historical background and evolution of these themes was addressed in great detail in the introductory section of  \cite{Lie19} we will not provide here any further historical context, motivation or references; instead, we directly address the main goals of our current paper:

A first, preliminary and more philosophical goal is to substantiate the dichotomy above and provide a first rigorous account of how the presence of quadratic (higher) order modulation invariance property impacts and conditions the nature of the approach. This will be achieved via the discussions in Sections \ref{Unif} and \ref{H2discussion} (see also Section \ref{Hywithout} and Observation \ref{quadratict} therein).

The second, more concrete goal of this paper is to supply further evidence on the versatility and unifying character of the LGC-method by offering a global treatment for the three classes of operators mentioned earlier---see (A), (B) and (C)---in the hybrid, no-quadratic resonance case (for a concrete description of the results the reader is invited to consult the next section). This contrasts with the previous work in  \cite{Lie19} in two ways: (1) therein the author studies the purely non-zero curvature case as opposed to the hybrid curve case here, and, (2), there the proof focuses on the classes of operators (A) and (B), while in the present paper our main interest will be in the (C) theme (Sections \ref{BilHilb}--\ref{BiMax}) with only brief accounts for the remaining themes (A) and (B) (Section \ref{Hywithout}).

Finally, a third goal of our paper is to show how several main constitutive ingredients of the LGC-method serve as inspiration for and become consequential in providing a new, concise, and self-contained proof of the key smoothing inequality encapsulating the behavior of the main oscillatory component of the curved triangular Hilbert transform analyzed in \cite{CDR20}.

\subsection{Main results}

As briefly mentioned before, the central aim of our paper is to present a unifying perspective on three main classes of operators---see (A), (B) and (C) at the beginning of the Introduction---in the case of \emph{hybrid} curves with no quadratic (or higher order) resonances.

Generically speaking, if $\g$ is a planar, finitely piecewise smooth curve, we say that $\g$ is \emph{hybrid} iff there exist $a_0,\,a_{+},\,a_{-}\in \R\setminus\{0\}$ and $\a_0,\a_{+},\a_{-}\in (0,\infty)$ such that
	\begin{equation}\label{asym}
\exists\:\lim_{t\rightarrow 0}\frac{\g(t)}{t^{\a_0}}=a_0\qquad\textrm{and}\qquad \exists\:\lim_{t\rightarrow \pm\infty}\frac{\g(t)}{t^{\a_{\pm}}}=a_{\pm}\,,
	\end{equation}
and either $[\a_0=1\:\textrm{and}\:1\not\in\{\a_{-},\a_{+}\}]$ or $[\a_0\not=1\:\textrm{and}\:1\in\{\a_{-},\a_{+}\}]$.

The main result of our paper focuses on the boundedness properties of
\begin{itemize}
	\item the bilinear Hilbert transform\footnote{Throughout this paper for notational simplicity we omit the principal value symbol.} along $\g$:
	\begin{equation}\label{bht-gamma}
		H_{\g}(f,g)(x):= \int_{\R} f(x-t)\, g(x+\g(t))\,\frac{dt}{t}\:,
	\end{equation}
	\item the bilinear Maximal operator along $\g$:
	\begin{equation}\label{bmt-gamma}
		M_{\g}(f,g)(x):= \sup_{\ep>0} \frac{1}{2\ep}\,\int_{|t|\leq \ep} |f(x-t)\, g(x+\g(t))|\,dt\:.
	\end{equation}
\end{itemize}
 More precisely, we have

\begin{theorem}\label{main1} Let $a\in\R\setminus\{-1\}$, $b\in \R$ and\footnote{For a discussion on the necessity of imposing the restriction $\a\not=2$ please see Observation \ref{alfa2} and Section \ref{H2discussion} below.} $\a\in(0,\,\infty)\setminus\{1,2\}$. Then, taking $\g(t):=a t\,+\,b t^{\a}$, we have that the operator
	\begin{equation}\label{bht}
		H_{a,b}^\a(f,g)(x):=\int_\R f(x-t)\, g(x+a\,t+b\,t^\a)\,\frac{dt}{t}
	\end{equation}
	obeys the bound
	\begin{equation}\label{mainrmon}
		\|H_{a,b}^\a(f,g)\|_{L^r}\lesssim_{a,b,\a,p,q} \|f\|_{L^p}\,\|g\|_{L^q}\:,
	\end{equation}
	where here $\frac{1}{p}+\frac{1}{q}=\frac{1}{r}$, $p,\,q\geq 1$ and $r>\frac{2}{3}$.
\end{theorem}

We also have the maximal analogue of Theorem \ref{main1}:

\begin{theorem} \label{main2} Let $a,b, \a$ and $\g$ be as in the previous statement. Then, defining
	\begin{equation}\label{mht}
		M_{a,b}^\a(f,g)(x):=\sup_{\ep>0}\frac{1}{2\ep}\,\int_{-\ep}^{\ep} |f(x-t)\, g(x+a\,t+b\,t^\a)|\,dt\:,
	\end{equation}
	we have that
	\begin{equation}\label{mainrmonmax}
		\|M_{a,b}^\a(f,g)\|_{L^r}\lesssim_{a,b,\a,p,q} \|f\|_{L^p}\,\|g\|_{L^q}\:,
	\end{equation}
	where as before $\frac{1}{p}+\frac{1}{q}=\frac{1}{r}$, $p,\,q\geq 1$ and $r>\frac{2}{3}$.
\end{theorem}

Moreover, based on the ideas in \cite{Lie19} and the present paper, and at the expense of some unavoidable technicalities that will not be detailed here but which are worked out in extenso in the context of \cite{Lie19}, one can obtain the following

\begin{corollary}\label{extens} Fix $d\in\N$ and consider the planar curve $\g(t)=\sum_{k=0}^d b_k\,t^{\a_k}$ with $\a_0=1$, $\{\a_k\}_{k=1}^d\subset (0,\infty)\setminus\{1,2\}$,  $b_0\in \R\setminus\{-1\}$ and  $\{b_k\}_{k=1}^d\subset\R$. If $p,\,q,\,r$ obey $\frac{1}{p}+\frac{1}{q}=\frac{1}{r}$, with $p,\,q> 1$ and $r> \frac{d}{d+1}$ then the operators \eqref{bht} and \eqref{mainrmon} extend continuously from $L^p(\R)\times L^q(\R)$ into $L^r(\R)$.
\end{corollary}

\begin{remark} \label{main3}
i) The boundedness ranges in both Theorem \ref{main2} and Corollary \ref{extens} are sharp up to endpoints. This extends the result in \cite{LX16} which treated the purely curved case $b_0=0$ and $\a_k=k$ for $k\in\{1,\ldots, d\}$. Notice that the endpoint $\frac{d}{d+1}$ is dictated by the number of the nonzero fewnomials/monomials in the definition of $\g$ and not by the degree of $\g$. For a hint justifying this latter distinction it might be useful to consult Observation 51 in \cite{Lie19}.

ii) The result in  Corollary \ref{extens} can be further extended to more general types of curves---see e.g. \cite{Lie15}, \cite{GL20}, \cite{Lie19}---which allow for example linear combinations of terms of the form $t^{\a} (\log t)^{\b}$ or/and remove the restriction that $\{\a_k\}_{k=1}^d$ are non-negative. In the interest of brevity we will not elaborate here on either of these situations.
\end{remark}

Next, we approach the themes (A) and (B). Based on the intuition presented in Section \ref{Unif}, in Section \ref{Hywithout} we will provide a brief argumentation for the following

\begin{theorem}\label{main4} Fix $d\in\N$ and let $\a_0=1$, $\{\a_k\}_{k=1}^d\subset (0,\infty)\setminus\{1,\,2\}$ be pairwise distinct and $\{a_k(\cdot)\}_{k=0}^{d}$ be some arbitrary real measurable functions. Then, given the variable planar curve $\g(x,t)=\sum_{k=0}^d a_k(x)\,t^{\a_k}$, we have that
\begin{itemize}
\item the $\g-$Carleson operator as defined as
		\begin{equation}\label{C-gamma}
			C_{\g}(f)(x):=\int_{\R} f(x-t)\,e^{i\,\g(x,t)}\,\frac{dt}{t}
		\end{equation}
is bounded from $L^p$ to $L^p$ for any $1<p<\infty$.
\item the linear Hilbert transform along $\g$ defined as
		\begin{equation}\label{HL-gamma}
			\H_{\g}(f)(x,y):= \int_{\R} f(x-t,\, y+\g(x,t))\,\frac{dt}{t}
		\end{equation}
is bounded from $L^2$ to $L^2$.		
\end{itemize}
\end{theorem}

Finally, building on some key elements of the LGC method, in particular on the sparse-uniform dichotomy used in the analysis of the time-frequency correlation set, in Section \ref{TrHC}--see Proposition \ref{MSTrHC} therein--we provide a simple, concise and intuitive proof of the main ingredient used in order to establish the following

\begin{theorem}  (\cite{CDR20}) \label{consec} Let $\g(t)=t^2$ and consider the following (parabolic) curved model for the triangular Hilbert transform:
	\begin{equation}\label{Trcurv}
		T^{\bigtriangleup}_{\g}(f,g)(x,y):=\,\int_{\R} f(x+t, y)\, g(x,\,y+\,\g(t))\,\frac{dt}{t}\:,\qquad\textrm{with}\:\:f,\,g\in S(\R^2)\;.
	\end{equation}
	Then $T^{\bigtriangleup}_{\g}$ extends to a bounded operator from $L^p\times L^q\,\rightarrow\,L^r$ with $\frac{1}{p}+\frac{1}{q}=\frac{1}{r}$, $1<p,\,q<\infty$ and $r\in[1,2)$.
\end{theorem}

We end this section with several commentaries:

\begin{remark} \label{ExtTr}[\textsf{Curved model for the triangular Hilbert transform--Extensions}] In a soon to follow work, \cite{HsL24}, the authors therein introduce a continuous and further simplified version of the discretized approach to Proposition \ref{MSTrHC} presented here. This continuous version adds some more flexibility in the choice of the curves $\g$ for which the stated bounds on \eqref{Trcurv} in Theorem \ref{consec} remain valid.
\end{remark}

\begin{observation} \label{alfa2}[\textsf{Quadratic resonances}]
	In both Theorems \ref{main1} and \ref{main2} the condition $\a\not=2$---see also the analogous situation for Corollary \ref{extens} and Theorem \ref{main4}---is not a mere technical restriction but rather the landmark of a conceptual, philosophical difficulty derived from the presence of a \emph{(quasi) quadratic resonance}\footnote{See Remark \ref{Quasi}.} that also manifests when treating\
	\begin{itemize}
		\item the Polynomial Carleson operator--see the distinction between the presence of the quadratic monomial (\cite{Lie09},\cite{Lie20}) versus the lack thereof in \cite{Lie19} and the present paper;
		\item the Bilinear Hilbert--Carleson operator $BC^{\a}$, which in the nonresonant case $\a\in(0,\infty)\setminus\{1,2\}$ is treated in \cite{BBLV21}, while in the resonant case $BC^{2}$ appears as an open problem in Section 1.4 of the same paper.
	\end{itemize}
	 Moreover, as we will show in Section \ref{H2discussion}, the case $\a=2$ cannot follow the same strategy as that developed for the proof of our main results above. This is because both Theorems \ref{main1} and \ref{main2} rely on the $m$-decaying estimate provided in Proposition \ref{L2mdecay} which fails to be true for $\a=2$. For a revealing discussion on the topic of the presence/absence of quadratic resonances one is invited to consult Section \ref{Unif}.
\end{observation}

Finally, in light of this last observation, we are naturally brought to discussing two unusual but very intriguing open problems:
\medskip

\noindent\textbf{Open Problem 1.} \emph{Show that both Theorems \ref{main1} and \ref{main2} (or for that matter---with the obvious adaptations---Corollary \ref{extens}) remain true in the case $\a=2$}.
\medskip

Based on the elements provided in Section \ref{cont}, Section \ref{Unif}, and most importantly Section \ref{H2discussion},  it is expected that a positive resolution to the above problem will rely on a relational time-frequency analysis approach in the spirit of \cite{Lie09} and \cite{Lie20} and not on the LGC methodology.

Next, for $c\in\R$, we define the (kernel) $c-$shifted expression
\begin{equation}\label{bhtrefc}
H_{\mathcal{P}_{[c]}}(f,g)(x):= \int_\R f(x-t)\, g(x+t^2)\,\frac{dt}{t-c}\:,
\end{equation}
and, denoting the translation symmetry by $\tau_{s}f(x):=f(x-s)$, we notice that
\begin{equation}\label{bhtref}
		H_{2,1}^2(f,g)(x)=H_{\mathcal{P}_{[1]}}(\tau_{-1}f,\tau_{1}g)(x)\:.
\end{equation}
Thus, regarded via formulation \eqref{bhtref}, the first open problem above is the source of inspiration for, as well as a very particular case of the following striking
\medskip

\noindent\textbf{Open Problem 2.} \emph{Study the boundedness properties of the \emph{bilinear Hilbert $C$-shifted maximal operator along} $\g$, defined by
\begin{equation}\label{bhtrefc}
\mathcal{H}_{\g}^{C}(f,g)(x):=\sup_{c\in C} \left|H_{\g_{[c]}}(f,g)(x)\right| := \sup_{c\in C} \left|\int_\R f(x-t)\, g(x+\g(t))\,\frac{dt}{t-c}\right|\:,
\end{equation}
where here $\g$ is some suitable piecewise smooth curve and $C$ is a subset\footnote{It might be worth mentioning here that if one considers the case $C=\R$, then, depending on the curvature of $\g$, the function $c(x)$ obtained after linearizing the supremum in \eqref{bhtrefc} is expected to obey some suitable nondegeneracy and regularity conditions in order for $\mathcal{H}_{\g}^{\R}$ to be well behaved.} of $\R$. Of particular interest are---in the increasing order of complexity---the cases $\g(t)=t^3$, $\g(t)=t^2$ and $\g(t)=t$.}
\medskip

Remark here that \eqref{bhtrefc} gives birth to a completely new category of objects in which, in a bilinear context, one creates a maximal structure relative to the \emph{translation} symmetry as opposed to the more familiar \emph{modulation} symmetry. For the latter situation, significant progress has been made only recently through the works on the bilinear Hilbert-Carleson operator\footnote{For its definition and an outline of the proof in \cite{BBLV21} see Section \ref{BHCalfa} in the present paper.} $BC^{\a}$ introduced in \cite{BBLV21} and on the curved trilinear Hilbert transform, \cite{HL23}, see in particular Theorem 4.3. therein.
$\newline$

\noindent\textbf{Acknowledgments.} The authors would like to thank Martin Hsu for carefully reading parts of this paper, for several useful comments--see in particular Observation \ref{Unifb}, and for pointing out a gap--together with a simple and elegant solution to it--in an earlier version of our work's Epilogue. The second author was partially supported by the NSF grant DMS-1900801.

\section{Prologue: A unified perspective on the zero/non-zero curvature paradigm}\label{Unif}

\subsection{The two end-points of the spectrum: a philosophical overview}

We open this section with a brief discussion\footnote{This follows the spirit of the description made in Section 1.2 of \cite{Lie19}.}
on the generic strategy involved in treating the two extreme situations within the zero/nonzero curvature spectrum. For streamlining our discussion we take $T$ to be our generic operator whose boundedness properties are investigated in what follows:
$\newline$

\noindent{\textsf{Panoramic view on the strategy involved for treating the zero/nozero curvature cases} }

\begin{itemize}
	
	\item \textsf{the zero curvature case}: This corresponds to the situation when our operator $T$ has a modulation invariance property on top of possible other symmetries (\emph{e.g.} dilation or translation). As a consequence, we are required to use wave-packet analysis in order to capture the time-frequency properties of the input/output of $T$. Thus, one is naturally led to a phase-space discretization of $T$ guided by the Heisenberg localization principle, thus achieving a representation of $T$ as a superposition of well-localized time-frequency suboperators, \textit{i.e.}
	\begin{equation}\label{zdec}
		T=:\sum_{P\in\P} T_{P}\:.
	\end{equation}
	Once at this point, one needs to identify the relevant quantities that govern the behavior of each suboperator $T_P$--see e.g. the concepts of mass and size--and then accomplish the difficult task of recombining all the elementary building blocks $\{T_P\}_{P\in\P}$ based on the behavior of these quantities in order to obtain a global estimate on the original operator $T$. In this reconstruction process a key role is played by the geometry and combinatorics of the tiles $P$ enclosing the (Heisenberg) time-frequency representation of $T_P$ with elementary structures such as trees playing a fundamental role. The basic outline of the reconstruction process may now be summarized as follows:
	\begin{itemize}
		\item Firstly, one divides the family of all tiles, referred to as $\P$ in \eqref{zdec}, into subfamilies having uniform mass and/or size, \textit{i.e.}
		\begin{equation}\label{zdec1}
			\P=\bigcup_{n\in\N}\P_n\qquad\textrm{with}\qquad\P_n:=\{P\in\P\,|\,\textrm{mass}\,(P)\approx 2^{-n}\}\:.
		\end{equation}
		\item Next, each such subfamily is further organized in forests, that is, suitable collections of almost disjoint trees.
		
		\item  Finally, by further exploiting the almost orthogonality among the forests within each $\P_n$, one obtains the control over $T_{\P_n}:=\sum_{P\in\P_n} T_P$ by showing that for a given $1<p<\infty$ there exists $\d(p)>0$ such that
		\begin{equation}\label{zdec2}
			\|T_{\P_n}\|_{L^p}\lesssim_p 2^{-\d(p)\,n}\:.
		\end{equation}
		Combining \eqref{zdec}, \eqref{zdec1} and \eqref{zdec2}  one derives the global control over $T$ via a simple telescoping argument.
	\end{itemize}
	$\newline$
	
	\item \textsf{the nonzero curvature case}:  in this situation our operator $T$ has no (generalized) modulation invariance and as a consequence no translation invariance symmetry for its associated symbol. This in turn provides the $0$ frequency with a favorite (central) role in the time-frequency analysis of $T$. As a further consequence we have the natural decomposition
	\begin{equation}\label{nonzdec}
		T=:T_{0}\,+\,T_{osc}\:,
	\end{equation}
	where here
	\begin{itemize}
		\item $T_0$ represents the \emph{low frequency} component whose multiplier--as the name suggests--has essentially no oscillation. Usually the study of $T_0$ can be reduced to a well understood maximal/singular operator appearing in the classical harmonic analysis literature;

		\item $T_{osc}$ represents the \emph{high frequency} component corresponding to a multiplier whose phase is highly oscillatory. Then, one decomposes
		\begin{equation}\label{nonzdec1}
			T_{osc}=\sum_{m\in\N} T_{m}\:,
		\end{equation}
		with each $T_m$ representing the component that has the property that the phase of multiplier associated with $T_{osc}$ has the size $\approx 2^{\mu m}$, where here $\mu>0$ is some parameter that depends on the specific structure of the original $T$. Once here, the generic goal is to show that there exists a $\d(\mu,p)>0$ such that
		\begin{equation}\label{nonzdec2}
			\|T_{m}\|_{L^p}\lesssim_p 2^{-\d(\mu,p)\,m}\:.
		\end{equation}
		This latter relation is achieved in the classical approach via stationary phase analysis, $TT^{*}$-method/almost orthogonality arguments and interpolation techniques. Once at this point, as in the zero-curvature case, the control over $T_{osc}$ is achieved via a telescoping argument.
	\end{itemize}
\end{itemize}

Now, as mentioned in the introductory section, even if originally some specific topics within the zero and non-zero curvature realms evolved mutually independently and with distinct motivations, in more recent years most of the problems considered within the non-zero curvature realm were formulated in direct relation with their zero-curvature counterparts and often intended as toy-models for the latter. However, the methods involved for treating the non-zero curvature situations were until recently all based on basic methods in the spirit of Calderon-Zygmund theory (wavelet expansion). Such an approach becomes obstructive when one intends to design non-zero curvature problems as toy models for some suitable zero-curvature problems since the latter are known to be incompatible with wavelet type decompositions. The situation becomes even more obscure when the considered problems are of \emph{hybrid} nature--see e.g. the case of the Polynomial Carleson operator, of the  Bilinear Hilbert-Carleson operator or that of the problem(s) addressed in the present paper.\footnote{Another relevant example of a problem of hybrid nature is that of the maximally modulated singular Radon transform discussed in Section 1.3.4. of \cite{BBLV21}.} Indeed, in such problems the zero and nonzero curvature features coexist and interact each other requiring thus a method that is compatible with and, moreover, unifies both facets.

Motivated by all these and consistent with the discussion in Section \ref{cont}, we are now naturally led towards this next topic.

\subsection{A comparative discussion}\label{compdisc}

In what follows we discuss antithetically the two unifying methods employed in order to provide a global approach to the zero-nonzero curvature paradigm:
\begin{itemize}\label{dichotomy}
	\item the \emph{LGC method} involving linear wave-packet analysis;
	
	\item the \emph{relational time-frequency analysis} relying on higher order wave-packet analysis.
\end{itemize}

In order to exhibit an enhanced contrast between the two methods above we will use as a prototype for our discussion the following
\medskip

\noindent \underline{\textsf{Model hybrid problem.}} \emph{Investigate the $L^2$ behavior of the Polynomial Carleson type operator:
\begin{equation}\label{CarlPolyn}
	C_{a,b}^{\a} f(x):=Tf(x)=\int_{\R} f(x-t)\,e^{i\,(a(x)\,t\,+\,b(x)\,t^{\a})}\,\frac{dt}{t}\:,
\end{equation}
where here $a(\cdot)$ and $b(\cdot)$ are measurable real functions and $\a\in\{2,3\}$.}
\medskip

Our approach of \eqref{CarlPolyn} will focus for most of the discussion on the applicability of the LGC method depending on the values of the parameter $\a$  and of $a(x)$, saving thus only for the very end the references to/necesity of the second method.

We start our analysis of \eqref{CarlPolyn} by focusing on the nonlinear component of the phase and, thus,  apply the natural decomposition
\begin{equation}\label{DecT}
	T=:T_0\,+\,\sum_{{m\in\N}\atop{j\in\Z}}T_{j,m}\:,
\end{equation}
where, for $m\in\N$ and $j\in\Z$, we set
\begin{equation}\label{CarlPolyn1}
	T_{j,m} f(x):=\left(\int_{\R} f(x-t)\,e^{i\,(a(x)\,t\,+\,b(x)\,t^{\a})}\,\rho_j(t)\,dt\right)\,\phi\left(\frac{b(x)}{2^{m+\a j}}\right)\:,
\end{equation}
with $\rho$, $\phi\in C_0^{\infty}(\R)$, $\textrm{supp}\,\rho,\,\textrm{supp}\,\phi\subset\{\frac{1}{4}<|t|<4\}$,
$\rho_j(t)=2^j\,\rho(2^j t)$ and
$$\forall\:t\in \R\setminus\{0\}\,:\qquad\frac{1}{t}=\sum_{j\in\Z} \rho_j(t)\qquad\textrm{and}\qquad 1=\sum_{j\in\Z} \phi(2^j t)\,.$$

Once at this point, we notice that $T_0$ can be essentially reduced to the classical Carleson operator and hence, it is enough to focus our attention on $T_{j,m}$, which, in the multiplier setting, can be rewritten as
$$T_{j,m} f(x):=\int_{\R} \hat{f}(\xi)\,\mathfrak{m}_j(x,\xi)\,e^{i\,\xi\,x}\,d\xi\,,$$
with
\begin{equation}\label{multipl}
	\mathfrak{m}_j(x,\xi):=\left(\int_{\R} e^{-i \xi t}\,e^{i\,(a(x)\,t\,+\,b(x)\,t^{\a})}\,\rho_j(t)\,dt\right)\,\phi\left(\frac{b(x)}{2^{m+\a j}}\right)\:.
\end{equation}
Now, an application of the stationary phase principle gives
\begin{equation}\label{multipl1}
	\mathfrak{m}_j(x,\xi)=\frac{1}{2^{\frac{m}{2}}}\,e^{i\,c_{\a}\, \frac{(\xi-a(x))^{\frac{\a}{\a-1}}}{b(x)^{\frac{1}{\a-1}}}} \,\tilde{\rho}\left(\frac{2^j\,(\xi-a(x))^{\frac{1}{\a-1}}}{(\a\,b(x))^{\frac{1}{\a-1}}}\right)\,\phi\left(\frac{b(x)}{2^{m+\a j}}\right)\,+\,\textrm{Error Term}\:,
\end{equation}
where here $c_{\a}\in\R$ is a nonzero constant depending only on $\a$, and $\tilde{\rho}$ is a function having similar properties with $\rho$.

Denoting now the phase of our multiplier by $\Psi_x(\xi):=c_{\a}\, \frac{(\xi-a(x))^{\frac{\a}{\a-1}}}{b(x)^{\frac{1}{\a-1}}}$  we notice that, in order to have $\mathfrak{m}_j(x,\xi)\not=0$, the following must hold:
\begin{equation}\label{multipl2}
	|\Psi_x(\xi)|\approx 2^m\qquad\textrm{and}\qquad \left|\frac{\partial^2}{\partial \xi^2}\Psi_x(\xi)\right|\approx 2^{-m-2j}\:.
\end{equation}

Implementing now the \emph{first stage} of the LGC method--the phase linearization of the multiplier--we deduce from \eqref{multipl2} the necessity of dividing the frequency axis into intervals of length $2^{\frac{m}{2}+j}$, which, via the Heisenberg principle, is equivalent with dividing the time (space) axis into intervals of length $2^{-\frac{m}{2}-j}$.  Thus, in order to apply this linearization procedure in space, we rewrite \eqref{CarlPolyn1} in the form

\begin{equation}\label{CarlPolyn01}
	T_{j,m} f(x):=\left(\int_{\R} f(t)\,e^{i\,(a(x)\,(x-t)\,+\,b(x)\,(x-t)^{\a})}\,\rho_j(x-t)\,dt\right)\,\phi\left(\frac{b(x)}{2^{m+\a j}}\right)\:,
\end{equation}
and, as explained, divide the $t$-spatial space representation in intervals of length $2^{-\frac{m}{2}-j}$ or, equivalently, write
$$\rho_j(x-t)=2^j\,\rho(2^j(x-t))\approx 2^j\sum_{p\sim 2^{\frac{m}{2}}} \rho(2^{j+\frac{m}{2}}(x-t)-p)\,.$$
Now, on the support of $\rho(2^{j+\frac{m}{2}}(x-t)-p)$, the phase linearization produces:
\begin{equation}\label{Spphase}
\end{equation}
$$\Omega_x(t):=a(x)\,(x-t)\,+\,b(x)\,(x-t)^{\a}$$
$$=a(x)\,(x-t)\,+\,\a\,\left(\frac{p}{2^{j+\frac{m}{2}}}\right)^{\a-1}\,b(x)\,(x-t)\,+\,\Omega_x^{j,m}(p)\,+\,O(1)\,,$$
where here $\Omega_x^{j,m}(p)$ stands for a suitable function depending only on $j,\,m,\, p$ and $x$.

At this point we reach the \emph{second stage} of the LGC method--the adapted Gabor decomposition of the input function--which amounts to the linear wave-packet discretization
\begin{equation}\label{Gabor}
	f(t)=\sum_{l,w\in\Z} \langle f,  \check{\phi}_{l}^w\rangle\, \check{\phi}_{l}^w(t)\,,
\end{equation}
where here we let
\begin{equation}\label{Gabor1}
	\phi_{l}^w(\xi):=\frac{1}{2^{\frac{j}{2}+\frac{m}{4}}}\,\phi\Big(\frac{\xi}{2^{j+\frac{m}{2}}}-w\Big)\,
	e^{i\,\frac{\xi}{2^{j+\frac{m}{2}}}\,l}\,.
\end{equation}
Putting together \eqref{CarlPolyn01}--\eqref{Gabor1} we deduce
\begin{equation}\label{CarlPolynest}
\end{equation}
$$|T_{j,m} f(x)|\lesssim 2^{\frac{3j}{2}+\frac{m}{4}}\,\sum_{{l,w\in\Z}\atop{p\sim 2^{\frac{m}{2}}}} | \langle f,  \check{\phi}_{l}^w\rangle|\,\phi\left(\frac{b(x)}{2^{m+\a j}}\right)\times$$
$$\left|\int_{\R} \check{\phi}(2^{j+\frac{m}{2}}\,t-l)\,e^{i\,2^{j+\frac{m}{2}}\,w\,t}\,e^{i\,(a(x)\,(x-t)\,+\,\a\,\left(\frac{p}{2^{j+\frac{m}{2}}}\right)^{\a-1}\,b(x)\,(x-t))}\,
\rho(2^{j+\frac{m}{2}}(x-t)-p)\,dt\right|\,.$$
Now, after a change of variable and integration by parts, we reach the \emph{third stage} of the LGC-methodology, expressed in the following time-frequency correlation governing the behavior of our operator:

\begin{equation}\label{CarlPolynestKey}
	|T_{j,m} f(x)|\lesssim \sum_{{l,w\in\Z}\atop{p\sim 2^{\frac{m}{2}}}} \frac{2^{\frac{j}{2}-\frac{m}{4}}\,|\langle f,  \check{\phi}_{l}^w\rangle|\,\rho(2^{j+\frac{m}{2}}x-l-p)}
	{\left\lfloor\frac{a(x)\,+\,\a\,\big(\frac{p}{2^{j+\frac{m}{2}}}\big)^{\a-1}\,b(x)\,-\,w\,2^{j+\frac{m}{2}}}{2^{j+\frac{m}{2}}}\right\rfloor^2}\,
	\phi\left(\frac{b(x)}{2^{m+\a j}}\right)\:.
\end{equation}

Finally, one can remove the presence of the $p-$parameter by rewriting \eqref{CarlPolynestKey} as

\begin{equation}\label{CarlPolynestKeynop}
	|T_{j,m} f(x)|\lesssim \sum_{l,w\in\Z} \frac{2^{\frac{j}{2}-\frac{m}{4}}\,|\langle f,  \check{\phi}_{l}^w\rangle|\,\rho(2^{j}x-\frac{l}{2^{\frac{m}{2}}})}
	{\left\lfloor\frac{a(x)\,+\,\a\,\big(x-\frac{l}{2^{j+\frac{m}{2}}}\big)^{\a-1}\,b(x)\,-\,w\,2^{j+\frac{m}{2}}}{2^{j+\frac{m}{2}}}\right\rfloor^2}\,
	\phi\left(\frac{b(x)}{2^{m+\a j}}\right)\:.
\end{equation}
Once at this point, it becomes apparent that the above expression enjoys suitable spatial and frequency almost-orthogonality properties that allow us to restrict our problem to the regime $x\in [2^{-j},2^{-j+1}]$ and  $l,w\sim 2^{\frac{m}{2}}$. We can now reinterpret our problem in the following manner\footnote{The informal presentation below---which, for expository reasons,  is stripped here of any technical definitions and detailed reasonings---will be made precise in Section \ref{Hywithout}.}:
\begin{itemize}
\item We are given a collection of pairs  $\tilde{\r}_m=\{(l,w)\}$, $\#\tilde{\r}_m\sim 2^m$ with each such pair associated to a local Fourier (Gabor) coefficient $\langle f,  \check{\phi}_{l}^w\rangle$.
\item Via a standard normalization procedure, we have the $L^2$ information
$$\sum_{(l,w)\in \tilde{\r}_m}|\langle f,  \check{\phi}_{l}^w\rangle|^2\leq 1\,.$$
\item Then, our task of bounding the local $L^2$-average
\begin{equation}\label{L2norm}
\int_{[2^j, 2^{j+1}]}|T_{j,m} f|^2
\end{equation}
essentially amounts to a good control over a sequence of weights $\{a_{(l,w)}\}_{(l,w)\in\tilde{\r}_m}$ with each $a_{(l,w)}$ representing the $x-$average behavior of the expression in the RHS of \eqref{CarlPolynestKeynop} that multiplies the term $|\langle f,  \check{\phi}_{l}^w\rangle|$.
\end{itemize}
The above line of thought invites us to naturally consider a \emph{light-heavy dichotomy}:
\begin{itemize}
\item we say that $(l,w)\in \tilde{\r}_m$ is a \emph{light} pair if the associated weight $a_{(l,w)}$ is suitably ``small"; in such a situation, the corresponding local Fourier coefficient $|\langle f,  \check{\phi}_{l}^w\rangle|$ has a small contribution to \eqref{L2norm}. Most of the Fourier coefficients of $f$ will end up in this light case.

\item naturally, we say that $(l,w)\in \tilde{\r}_m$ is a \emph{heavy} pair if the associated weight $a_{(l,w)}$ is suitably ``large"; in such a situation the corresponding local Fourier coefficient $|\langle f,  \check{\phi}_{l}^w\rangle|$ has a significant contribution to \eqref{L2norm}. Our aim is to show that there are only ``few" such Fourier coefficients of $f$.
\end{itemize}

Now, as it turns out, it is easy to treat \eqref{L2norm} in the light case, and thus, the main interest focusses on the heavy case: notice that a pair $(l,w)$/coefficient $\langle f,  \check{\phi}_{l}^w\rangle$ is heavy iff, essentially, for ``most" of
$x\in [2^{-j},2^{-j+1}]$ the following holds:
\begin{equation}\label{weight1}
\left|\frac{a(x)\,+\,\a\,\big(x-\frac{l}{2^{j+\frac{m}{2}}}\big)^{\a-1}\,b(x)\,-\,w\,2^{j+\frac{m}{2}}}{2^{j+\frac{m}{2}}}\right|\approx O(1)\,.
\end{equation}
Based on \eqref{weight1} we can further reinterpret the task of evaluating the number of heavy pairs as a (weighted) ``variable curve"-``point" discretized incidence problem in which our duty is to control the number of ``heavy incidences"---\textit{i.e.}, incidences that remain stable for ``most" $x'$s---between
\begin{itemize}
\item $2^{j+\frac{m}{2}}-$neighborhoods of $x-$variable\footnote{Throughout this section by $x-$variable curves (graphs) we understand a curve of the form $\g_x(t)$ with $t$ being the actual variable and $x$ being a parameter. Thus in the case when $\a\in\N$ the curve $\g_x(t)$ may be identified with a real polynomial in the $t$ variable having as coefficients real measurable functions in $x$.} families of graphs of the form
\begin{equation}\label{famcurv}
\left\{\left(t,a(x)\,+\,\a\,\big(x-t\big)^{\a-1}\,b(x)\,-\,w\,2^{j+\frac{m}{2}}\right)\Big|_{t\sim 2^{-j}}\right\}_{w\sim 2^{\frac{m}{2}}}\;
\end{equation}
\item and a family of $2^{-j-\frac{m}{2}}\times 2^{j+\frac{m}{2}}$ rectangles of the form
\begin{equation}\label{fampoint}
\{I_l\times J_w\}_{(l,w)\in \tilde{\r}_m}
\end{equation}
with $I_l=[\frac{l}{2^{j+\frac{m}{2}}},\,\frac{l+1}{2^{j+\frac{m}{2}}}]$ and $J_w=[w 2^{j+\frac{m}{2}},\, (w+1)2^{j+\frac{m}{2}}]$.
\end{itemize}

As expected, the key role in controlling these heavy incidences is played by the ``amount" of \emph{transversality} present among the various families of curves within \eqref{famcurv}, which depends fundamentally on the \emph{structure} of these curves.

With these done, we are now ready to advance our antithetical discussion of the two unifying methods by gradually increasing the level of complexity:

$\newline$
\noindent\textbf{I. The (purely) non-zero curvature case: $a(x)\equiv 0$, $\a\in\{2,3\}$.} \textsf{[LGC applies]}
$\newline$

\vspace{-.5cm}
\begin{figure}[h]\centering
	\resizebox{0.63\textwidth}{!}{
		\begin{picture}(300,650)
			\put(25,335){\includegraphics[scale=.75]{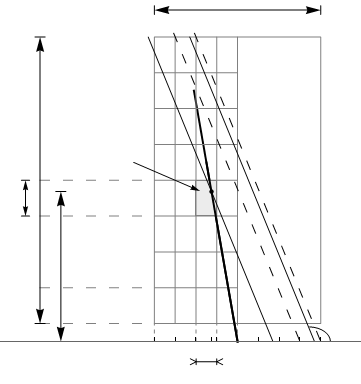}}
			\put(183,625){\text{\Large$ |I| \sim 2^{-j} $}}		\put(33,605){\text{\huge $\alpha=2$}}
			\put(160,600){\text{\Large $\tau$}}
			\put(5,520){\text{\Large$ |J| {\sim} 2^{j+m} $}}
			\put(100,508){\text{\Large$\langle f,  \check{\phi}_{l}^w\rangle$}}
			\put(25,472){\text{\Large$ J_w $}}
			\put(75,420){\text{\large $\approx\! 2 b(x)(x{-}t)$}}
			\put(270,405){\text{\large $\tan\beta_x{=}{-}2b(x)$}}
			\put(273,377){\text{\Large $\beta_x$}}
			\put(182,357){\text{\Large$ t $}}
			\put(228,357){\text{\Large$ x $}}
			\put(280,330){\text{\large $J_w{=}[w\, 2^{j+\frac{m}2},(w{+}1)2^{j+\frac{m}2}]$}}
			\put(285,310){\text{\large $I_l{=}\Big[\frac{l}{2^{j+\frac{m}2}},\frac{l+1}{2^{j+\frac{m}2}}\Big]$}}
			\put(174,335){\text{\Large$ |I_l|\sim 2^{-j-\frac{m}2 }$}}
			\put(26,0){\includegraphics[scale=.75]{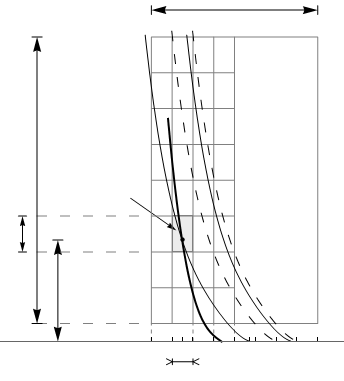}}
			\put(183,288){\text{\Large$ |I| \sim 2^{-j} $}}
			\put(32,270){\text{\huge $\alpha=3$}}
			\put(160,265){\text{\Large $\tau$}}
			\put(4,190){\text{\Large$ |J| {\sim} 2^{j+m} $}}
			\put(98,145){\text{\Large$\langle f, \check{\phi}_{l}^w\rangle$}}
			\put(23,107){\text{\Large$ J_w $}}
			\put(72,68){$\approx\! 3 b(x)(x{-}t)^2$}
			\put(161,20){\text{\Large$ t $}}
			\put(210,20){\text{\Large$ x $}}
			\put(156,0){\text{\large$ |I_l|\sim 2^{-j-\frac{m}2 }$}}
		\end{picture}
	}
	\caption{\textsf{Few heavy incidences:} if two curves corresponding to two well separated values of $x's$ are incident to the same rectangle $I_l\times J_w$ then they must intersect transversally.}\label{fig1}
\end{figure}
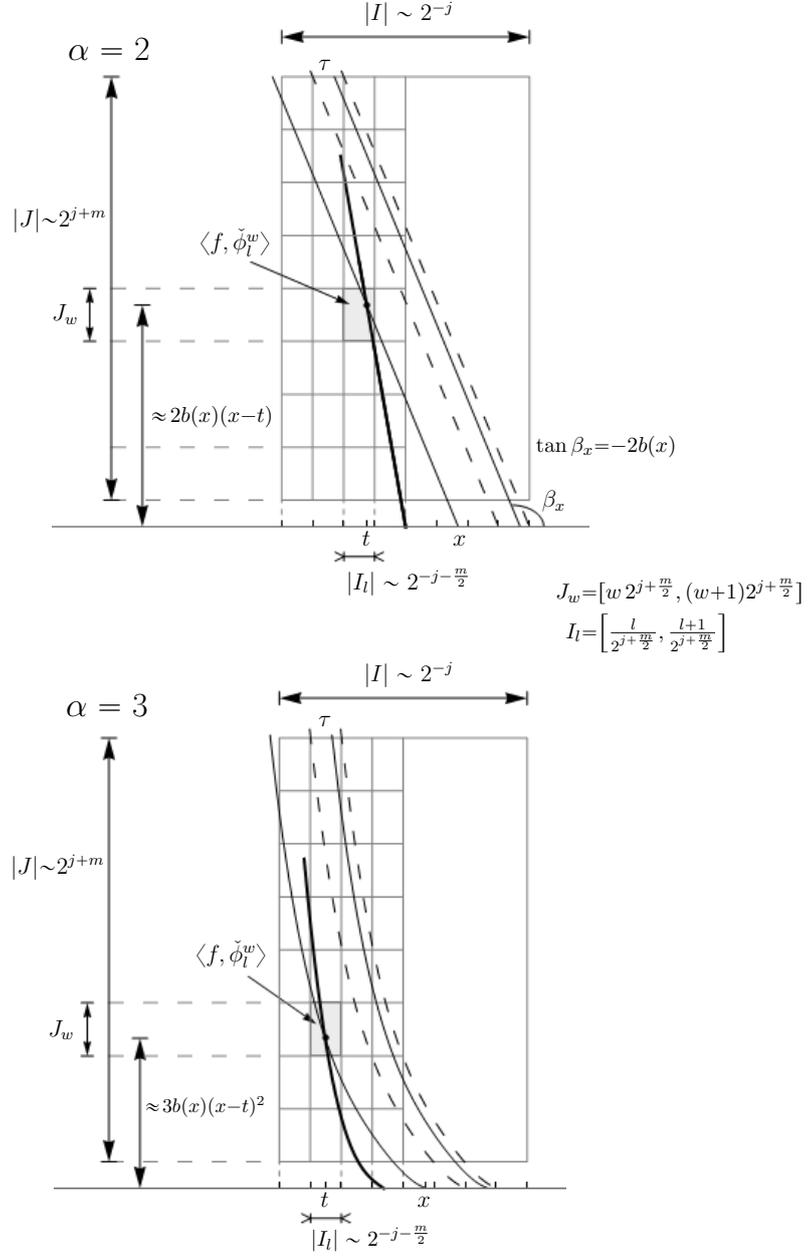

\medskip
In this situation the $a(x)-$vertical shift term for the graphs in \eqref{famcurv} is missing:

\begin{equation}\label{famcurv1}
\left\{\left(t,\a\,\big(x-t\big)^{\a-1}\,b(x)\,-\,w\,2^{j+\frac{m}{2}}\right)\Big|_{t\sim 2^{-j}}\right\}_{w\sim 2^{\frac{m}{2}}}\,,
\end{equation}
and thus---see Figure \ref{fig1}---irrespective of the value of $\a\in\{2,3\}$, if two variable curves corresponding to two distinct $x_1$ and $x_2$ in \eqref{famcurv1} are incident to the same rectangle $I_l\times J_w$, and, if $x_1$ and $x_2$ are well separated, \textit{i.e.} $\textrm{dist}\,(x_1,x_2)>>2^{-j-\frac{m}{2}}$, then the two curves are transversal. This key property enforces the desired goal: there are only few heavy incidences and hence only few heavy pairs $(l,w)$.

This last statement exploits in a crucial fashion the structure of the time-frequency correlations, as revealed by the more general, purely non-zero curvature setting treated in Section 10 of \cite{Lie19} (see also Section \ref{Hywithout} in the present paper). As a consequence, one immediately deduces

\begin{proposition} \label{purelycurv} Assume that $a(x)\equiv0$ in \eqref{CarlPolyn} and set
\begin{equation}\label{CarlPolynestKey1}
	\mathcal{L}_{j,m}^{curv} f(x):=\sum_{l,w\in\Z} \frac{2^{\frac{j}{2}-\frac{m}{4}}\,|\langle f,  \check{\phi}_{l}^w\rangle|\,\rho(2^{j}x-\frac{l}{2^{\frac{m}{2}}})}
	{\left\lfloor\frac{\a\,\left(x-\frac{l}{2^{j+\frac{m}{2}}}\right)^{\a-1}\,b(x)\,-\,w\,2^{j+\frac{m}{2}}}{2^{j+\frac{m}{2}}}\right\rfloor^2}\,
	\phi\left(\frac{b(x)}{2^{m+\a j}}\right)\:.
\end{equation}
Then, there exists $\d>0$ such that uniformly in $j\in\Z$ and $m\in\N$, one has that
\begin{equation}\label{CarlPolynestKey1purelycurv}
		\|\mathcal{L}_{j,m}^{curv} f\|_2\lesssim 2^{-\d\, m}\,\|f*\check{\phi}_{j+m}\|_2\:,
\end{equation}
	where here $\phi_{m+j}(\xi):=\phi(\frac{\xi}{2^{j+m}})$. Moreover,
	\begin{equation}\label{CarlPolynestKey1purelycurvG}
		\|\sum_{{j\in\Z}\atop{m\in\N}}T_{j,m} f\|_2\lesssim \|\sum_{{j\in\Z}\atop{m\in\N}} \mathcal{L}_{j,m}^{curv} f\|_2\lesssim \|f\|_2\:,
	\end{equation}
	and hence
	\begin{equation}\label{CarlPolynestKey1purelycurvG0}
		\| C_{0,b}^{\a} f\|_2\lesssim \|f\|_2\:.
	\end{equation}
\end{proposition}

\noindent\textbf{II. The hybrid case \underline{without} higher order modulation invariance: $\a=3$.} \textsf{[LGC is still effective]}
$\newline$

In this situation the $a(x)-$vertical shift term is present in the family \eqref{famcurv}, but this comes together with a non-zero curvature feature in the $t$-variable:
\begin{equation}\label{famcurv1}
\left\{\left(t,a(x)+3\,\big(x-t\big)^{2}\,b(x)\,-\,w\,2^{j+\frac{m}{2}}\right)\Big|_{t\sim 2^{-j}}\right\}_{w\sim 2^{\frac{m}{2}}}\,.
\end{equation}
As a consequence---see Figure \ref{fig2}---the same conclusion as in the first case above holds: if two variable curves in \eqref{famcurv1} corresponding to two distinct and well separated $x_1$ and $x_2$ are incident to the same rectangle, then the two curves are transversal, forcing thus only few heavy incidences.

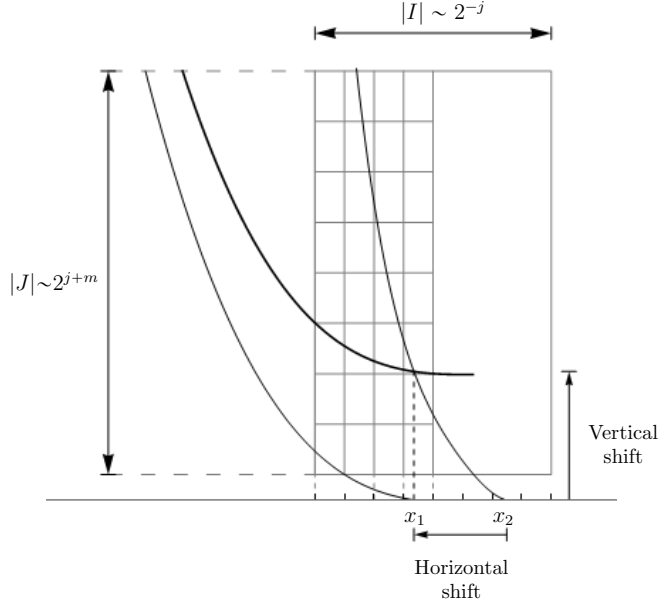
\begin{figure}[h]\centering
	\resizebox{.6\textwidth}{!}{
		\begin{picture}(300,310)
			\put(0,10){\includegraphics[scale=.75]{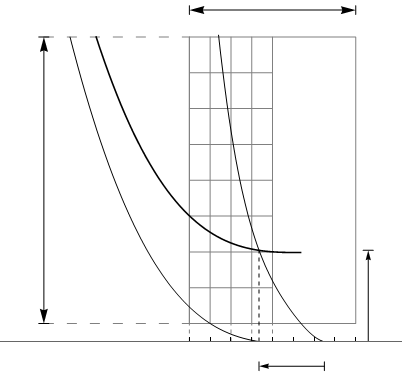}}
			\put(188,305){\text{\Large$ |I| \sim 2^{-j} $}}
			\put(-20,160){\text{\Large$ |J| {\sim} 2^{j+m} $}}
			\put(190,36){\text{\Large$ x_1 $}}
			\put(237,36){\text{\Large$ x_2 $}}
			\put(278,75){\text{\parbox{2cm}{\centering\large Vertical shift}}}
			\put(192,3){\text{\parbox{2cm}{\centering\large Horizontal shift}}}
		\end{picture}
	}
	\caption{\textsf{Few heavy incidences:} due to the non-zero curvature in the $t$ parameter, the presence of the vertical shift does not alter the transversal intersection property for two curves arising from well separated values of $x$.}\label{fig2}
\end{figure}

Hence, by properly quantifying the time-frequency correlations--see Section \ref{Hywithout}--we obtain the following

\begin{proposition} \label{hybcurv3} Let $\a\not=2$ in \eqref{CarlPolyn} and set
\begin{equation}\label{CarlPolynestKey2}
	\mathcal{L}_{j,m}^{hyb,3} f(x):=\sum_{l,w\in\Z} \frac{2^{\frac{j}{2}-\frac{m}{4}}\,|\langle f,  \check{\phi}_{l}^w\rangle|\,\rho(2^{j}x-\frac{l}{2^{\frac{m}{2}}})}
	{\left\lfloor\frac{a(x)\,+\,3\,\left(x-\frac{l}{2^{j+\frac{m}{2}}}\right)^{2}\,b(x)\,-\,w\,2^{j+\frac{m}{2}}}{2^{j+\frac{m}{2}}}\right\rfloor^2}\,
	\phi\left(\frac{b(x)}{2^{m+3j}}\right)\:.
\end{equation}
Then, there exists $\d>0$ such that, uniformly in $j\in\Z$ and $m\in\N$, one has
\begin{equation}\label{CarlPolynestKey1hybcurv3}
		\|\mathcal{L}_{j,m}^{hyb,3} f\|_2\lesssim 2^{-\d\, m}\,\|f\|_2\:.
\end{equation}
Moreover,
	\begin{equation}\label{CarlPolynestKey1purelycurvGH}
		\|\sum_{{j\in\Z}\atop{m\in\N}}T_{j,m} f\|_2\lesssim \|\sum_{{j\in\Z}\atop{m\in\N}} \mathcal{L}_{j,m}^{hyb,3} f\|_2\lesssim \|f\|_2\:,
	\end{equation}
and hence
	\begin{equation}\label{CarlPolynestKey1purelycurvGH0}
		\| C_{a,b}^{3} f\|_2\lesssim \|f\|_2\:.
	\end{equation}
\end{proposition}

$\newline$
\noindent\textbf{III. The hybrid case \underline{with} higher order modulation invariance: $\a=2$.} \textsf{[Relational time-frequency analysis needed; (Rank I) LGC no longer applicable]}\footnote{Throughout this section, unless otherwise explicitly stated, the LGC method refers to Rank I LGC--for meaning/definitions please consult \cite{HL23}.}
$\newline$

\begin{figure}[h]\centering
	\resizebox{.6\textwidth}{!}{
		\begin{picture}(300,310)
			\put(0,10){\includegraphics[scale=.75]{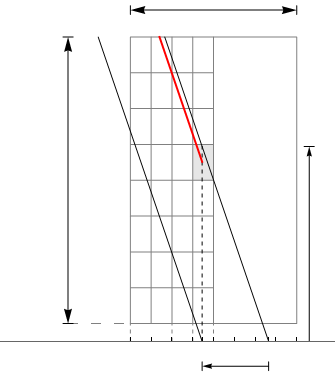}}
			\put(141,298){\text{\Large$ |I| \sim 2^{-j} $}}
			\put(-8,160){\text{\Large$ |J| \sim 2^{j+m} $}}
			\put(150,31){\text{\Large$ x_1 $}}
			\put(198,31){\text{\Large$ x_2 $}}
			\put(230,110){\text{\parbox{2cm}{\centering\large Vertical shift}}}
			\put(150,0){\text{\parbox{2cm}{\centering\large Horizontal shift}}}
		\end{picture}
	}
	\caption{\textsf{Many heavy incidences are possible:} due to the zero curvature in the $t$ parameter, the presence of the vertical shift allows non-transversal intersection for curves arising from well separated values of $x$.}\label{fig3}
\end{figure}
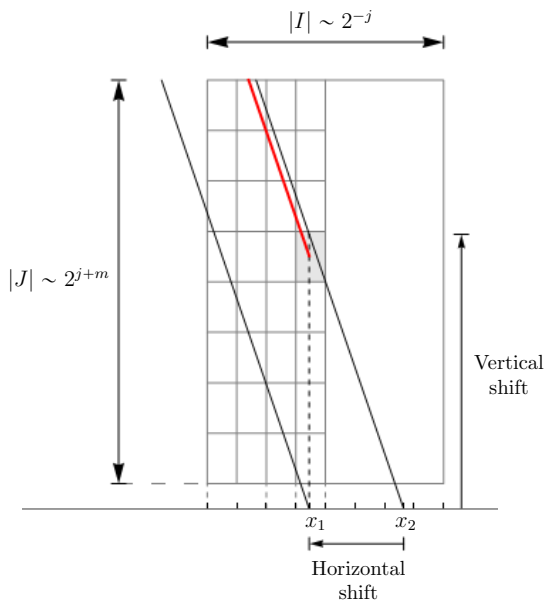

In this situation, the family of curves in  \eqref{famcurv} encapsulates both the $a(x)-$vertical shift term and the zero curvature feature in the $t$-variable:
\begin{equation}\label{famcurv2}
\left\{\left(t,a(x)+2\,\big(x-t\big)\,b(x)\,-\,w\,2^{j+\frac{m}{2}}\right)\Big|_{t\sim 2^{-j}}\right\}_{w\sim 2^{\frac{m}{2}}}\,.
\end{equation}
As a consequence---see Figure \ref{fig3}---the transversality property for curves associated with well separated $x'$s is no longer true and a scenario with ``many" heavy incidences is possible. This is consistent with the fact that the analysis of the time-frequency correlations stops being efficient in controlling the number of heavy pairs $(l,w)$---see Observation \ref{quadratict}. Thus, the LGC-method is no longer applicable.\footnote{This conclusion holds \emph{per se} for Rank I LGC method. If one allows non-absolutely summable models then, with a highly increased level of complexity, Rank II LGC might still be applicable--for a more detailed discussion please see \cite{HL23}.}
\medskip

In contrast with all of the above, one can embrace a relational time-frequency point of view (see \cite{Lie09},\cite{Lie09Thesis} \cite{Lie20}) and apply a different discretization of the original operator $C_{a,b}^{2}$ in \eqref{CarlPolyn}--recall the spirit of \eqref{zdec}--as revealed by the algorithm below:
\begin{itemize}
\item Consider the time-frequency plane with the horizontal axis representing the spatial (time) variable and  the vertical axis representing the Fourier (frequency) variable and choose on both axis the canonical dyadic grid.
\item Let $\P$ be the collection of all triples $P=(\a,\o, I)$ with $\a, \o$ dyadic frequency intervals and $I$ a dyadic spatial interval such that $|\a|=|\o|=|I|^{-1}$; notice that by slightly abusing notation one can naturally identify each such $P=(\a,\o,I)$ with the unique parallelogram $\tilde{P}$ in the time-frequency plane such that i) the dyadic frequency intervals $\a, \o$  represent the projection onto the frequency axis of the vertical edges of $\tilde{P}$, and, ii) the dyadic interval $I$ represents the projection onto the time axis of the other two edges of $\tilde{P}$.
\item Denote by $l_x(t):=a(x)+2 b(x) t$ the line obtained by taking the first derivative in $t$ of the phase appearing in \eqref{CarlPolyn} and introduce the following convention\footnote{Here and throughout the paper we are using the following standard notation: if $I$ is a (dyadic) interval then $l(I)$ and $r(I)$ stand for the left and right endpoint of $I$, respectively.}:
\begin{equation}\label{lineP}
l_x\in P=(\a,\o, I)\in\P\quad\textrm{iff}\quad l_x(l(I))\in \a\:\:\wedge\:\:l_x(r(I))\in\o\,.
\end{equation}
Notice that $l_x\in P$ is geometrically equivalent with the assertion that the line $l_x$ intersects both vertical edges of the parallelogram $P$.
\item For each $P=(\a,\o, I)\in\P$ define the set $E(P):=\{x\in I\,|\,l_x\in P\}$ and introduce the mass of a tile $P$ as
\begin{equation}\label{massP}
A_0(P):=\frac{|E(P)|}{|I|}\,.
\end{equation}
\item Finally, using the above notations and conventions, we introduce the discretization
\begin{equation}\label{zdecquad}
		C_{a,b}^{2}=\sum_{P\in\P} C_{P}\:,
\end{equation}
where, if $P=(\a,\o, I)\in\P$ with $|I|=2^{-j}$, then
\begin{equation}\label{zdecquad1}
		C_{P}f(x):=\left(\int_{\R} f(x-t)\,e^{i\,(a(x)\,t\,+\,b(x)\,t^{2})}\,\rho_j(t)\,dt\right)\chi_{E(P)}(x)\,.
\end{equation}
\end{itemize}

Now, taking into account the magnitude of the mass parameter \eqref{massP} one properly groups the tiles $P\in\P$ (and hence their respective associated operators $C_P$) within certain structured collections--\textit{i.e.}, trees and then forests--in order to gradually obtain good control over increasingly larger collection of tiles until one is able to conclude the desired global bounds on the whole operator $C_{a,b}^{2}$.  This corresponds essentially to the process \eqref{zdec}--\eqref{zdec2} described at the beginning of this section.

Putting together all the elements presented within this case III, to which we also add the conclusion of \cite{Lie09}, we have the following stark contrast:

\begin{proposition} \label{hybcurv2} Let $\a=2$ in \eqref{CarlPolyn} and set
\begin{equation}\label{CarlPolynestKey3}
	\mathcal{L}_{j,m}^{hyb,2} f(x):=\sum_{l,w\in\Z} \frac{2^{\frac{j}{2}-\frac{m}{4}}\,|\langle f,  \check{\phi}_{l}^w\rangle|\,\rho(2^{j}x-\frac{l}{2^{\frac{m}{2}}})}
	{\left\lfloor\frac{a(x)\,+\,2\,\left(x-\frac{l}{2^{j+\frac{m}{2}}}\right)\,b(x)\,-\,w\,2^{j+\frac{m}{2}}}{2^{j+\frac{m}{2}}}\right\rfloor^2}\,
	\phi\left(\frac{b(x)}{2^{m+2j}}\right)\:.
\end{equation}
Then, uniformly in $j\in\Z$ and $m\in\N$, one has
	\begin{equation}\label{CarlPolynestKey1hybcurv2}
		1\lesssim \|T_{j,m}\|_{2\rightarrow 2}\lesssim \|\mathcal{L}_{j,m}^{hyb,2} \|_{2\rightarrow 2}\lesssim 1\:.
	\end{equation}
Moreover,
	\begin{equation}\label{CarlPolynestKey1purelycurvGH2(L)}
		\|\sum_{{j\in\Z}\atop{m\in\N}} \mathcal{L}_{j,m}^{hyb,2}\|_{2\rightarrow 2}=\infty\:,
	\end{equation}
	while
	\begin{equation}\label{CarlPolynestKey1purelycurvGH2(T)}
		\|\sum_{{j\in\Z}\atop{m\in\N}} T_{j,m}\|_{2\rightarrow 2},\:\|C_{a,b}^{2}\|_{2\rightarrow 2}\lesssim 1\:.
	\end{equation}
\end{proposition}

\subsection{Conclusion}

We end this comparative discussion by addressing the natural question arising as a consequence of Proposition \ref{hybcurv2}:
\medskip

\noindent\underline{\textsf{Question.}} \emph{Why is there this sharp contrast between the suitability of the relational time-frequency method as opposed to the lack thereof for the LGC-method when approaching the Quadratic Carleson operator?}
\medskip

\vspace{-0.3cm}
\begin{figure}[h]\centering
	\resizebox{.7\textwidth}{!}{
		\begin{picture}(300,400)
			\put(100,0){\includegraphics[scale=.6]{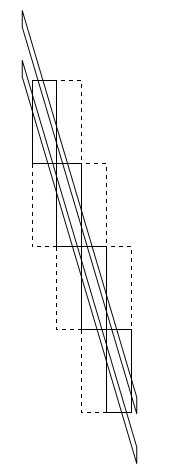}}
		\end{picture}
	}
	\caption{The LGC localization is blurring the relational time-frequency Heisenberg localization of a generalized wave-packet thus destroying the almost orthogonality between the thinner tubes.}\label{fig4}
\end{figure}

\noindent\underline{\textsf{Answer:}} The key relies in the fabric of the time-frequency localization cuttings associated to a generalized wave-packet\footnote{Here, $x\in I$ is treated as a fixed parameter.}
\begin{equation}\label{quadwavep}
\psi_{x}(t):=e^{i\,(a(x)\,t\,+\,b(x)\,t^{2})}\,\rho_j(t)\qquad\textrm{and}\:\:|b(x)|\sim 2^{m+2j}
\end{equation}
with the former corresponding precisely to the $t$-integrand in \eqref{zdecquad1}:

\begin{itemize}
\item On the one hand, in the LGC method the underlying philosophy rests on a \emph{linearization procedure} that has as a direct effect the partition of the spatial interval $I$ having length $2^{-j}$--see the first picture in Figure \ref{fig1}--into $2^{\frac{m}{2}}$ intervals of equal length. Via the Heisenberg principle, this forces a time-frequency localization region for \eqref{quadwavep} that consists of $2^{\frac{m}{2}}$ boxes of size $2^{-j-\frac{m}{2}}\times 2^{j+\frac{m}{2}}$ arranged along the central line\footnote{Of course, this is precisely the first derivative of the phase in \eqref{quadwavep}.} $l_x(t)=a(x)+2 b(x) t$. Thus, in a nutshell, the LGC localization of \eqref{quadwavep} represents a
    $2^{j+\frac{m}{2}}-$thick tube around $l_x$ that is spatially restricted to $I$.

\item On the other hand, in the relational time-frequency approach---see the detailed explanations in Section 2 of \cite{Lie09} and in particular Figure 5 therein---the time-frequency localization of \eqref{quadwavep} is obtained as
    a $2^{j}-$thick tube around $l_x$ that is spatially restricted to $I$.
\end{itemize}

As a consequence, the LGC derived localization enlarges by a $2^{\frac{m}{2}}$ factor the true area one Heisenberg localization region obtained via the relational time-frequency approach. This $2^{\frac{m}{2}}-$blurring effect destroys the almost orthogonality between ``nearby" generalized wave-packets--see Figure \ref{fig4} above, almost orthogonality that, in the absence of the linearizing procedure, is the expression of Lemma 0 in \cite{Lie09}.

Thus, the main concluding---and otherwise completely intuitive---informal message of the above discussion is that in the time-frequency decomposition of a generalized modulation invariant operator $T$ the constitutive elementary building blocks must involve the same type of generalized modulation structure, that is:
\begin{itemize}
 \item if the operator $T$ has no higher than linear modulation invariance then the elementary building blocks are chosen to be standard (linear) wave-packets and hence the LGC method applies;

\item if the operator $T$ enjoys higher than linear modulation invariance then the elementary building blocks must encapsulate same order generalized  wave-packets and thus the relational time-frequency approach is required.
\end{itemize}
\medskip
Finally, the entire discussion of this section may be regarded as an embodiment of the philosophical principle stated in \cite{Lie20}:
\medskip

\noindent\underline{\textsf{A heuristic symmetry principle:}} \emph{The classes of symmetries of an operator are responsible for the nature of the approach/techniques to be involved in the analysis of its boundedness properties}.

\section{The Bilinear Hilbert transform $H_{a,b}^{\a}$, $\a\in \R_{+}\setminus\{1,2\}$: the initial decomposition of its associated form $\L=\L^{NSI}\,+\,\L^{LO}\,+\,\L^{NS}\,+\,\L^{S}$}\label{BilHilb}

In what follows we focus our attention on the bilinear Hilbert function along $\g(t)=at+bt^\a$ defined by
\begin{equation}\label{maintwo}
	H(f,g)(x)\equiv H_{a,b}^{\a}(f,g)(x):= \int_\R f(x-t)\,g(x+at+bt^\a)\frac{dt}{t}\,,
\end{equation}
where $a\in\R\setminus\{-1\}$, $b\in\R\setminus\{0\}$ and $\a\in\R_+\setminus\{1,\,2\}$. Our later reasonings will exploit a certain symmetry between $H$ and the corresponding dual bilinear operators and hence it will be convenient to perform our analysis directly on the trilinear form
\begin{equation}\label{trilin}
	\L(f,g,h):=\int H(f,g)(x)\,h(x)\,dx= \int_{\R^2} f(x-t)\,g(x+at+bt^\a)\,h(x)\,\frac{dt}{t}\,dx\,.\quad
\end{equation}

Notice that on the Fourier side, \eqref{trilin} may be rewritten as
\begin{equation}\label{multform}
	\L(f,g,h)=\int_\R\int_\R \hat{f}(\xi)\,\hat{g}(\eta)\,\hat{h}(\xi+\eta)\,m(\xi,\eta)\,d\xi \,d\eta\,,
\end{equation}
with the multiplier given by
\begin{equation}\label{mult}
	m(\xi,\eta):=\int_\R e^{-i\xi t}e^{i\eta (at+bt^\a)}\frac{dt}{t}.
\end{equation}

In its first stage, the analysis of our form concerns three types of discretization corresponding to:
\begin{itemize}
\item the spatial scale of the kernel -- this accounts for the $j$-parameter below;
\item the structure of the multiplier's phase -- this requires a more elaborate reasoning based on (non)stationary phase analysis that involves two other parameters $m$ and $n$;
\item the phase-space localization properties of the subforms resulted from the decompositions performed at the first two items -- this amounts to a discussion on the relative positions of the parameters $j$, $m$ and $n$.
\end{itemize}

\begin{observation}\label{const}[\textsf{Hybrid regime}, I] We end this brief introductory paragraph by mentioning that while the first two items each focus on only one side of the phase-space analysis (\textit{i.e.} space for the first and frequency for the second) is the third item that combines the two sides as a direct consequence of the competing behavior between the linear and nonlinear terms in $\g(t)=at+bt^\a$. Thus, while the discretization addressing the first two items can be traced to the purely non-zero curvature case treated in \cite{Li13} and \cite{Lie15} the one corresponding to the third item marks a novel feature of our present paper that captures the transition from the non-zero to the zero curvature regimes.

Finally, it is worth saying the following:
\begin{itemize}
\item on the one hand, the analysis developed in our present paper uses as a black-box the results addressing the purely zero regime $b=0$ (see \cite{LT97} and \cite{LT99}) and thus we are not adding anything new in this regime;

\item on the other hand, our methods provide a straightforward handling of the purely non-zero regime $a=0$ treated earlier in \cite{Li13} and \cite{Lie15}.
\end{itemize}

As a consequence of the above, throughout the reminder of the paper, we will always assume
\begin{equation}\label{abnl}
		b\not=0\,.
	\end{equation}
Regarding the parameter $a\in\R\setminus\{-1\}$ we will focus on the case $a\not=0$ as this is the truly novel part of our Theorem \ref{main1}, though, with just trivial modifications, our proof also covers the case $a=0$.
\end{observation}

\begin{observation}\label{Unifb}[\textsf{Hybrid regime}, II]\footnote{This observation is due to Martin Hsu.} Let $a\in\R\setminus\{-1\}$ and $\a\in(0,\,\infty)\setminus\{1\}$ be fixed and, for $\l>0$, denote by $D_{\l}$ the $L^{\infty}-$normalized dilation symmetry $D_{\l}f(x):=f(\l x)$, $x\in\R$. Given now any $b\in\R$ and based on \eqref{maintwo} we immediately have
\begin{equation}\label{dillambda}
		D_{\l} H_{a,b}^\a(f,g)= H_{a, \l^{\a-1} b} (D_{\l} f, D_{\l} g)\,.
\end{equation}
Deduce then that by taking any two $b_1, b_2\in \R\setminus\{0\}$ with the same signature and any triple $(p,q,r)$ as in the statement of Theorem \ref{main1} we have
\begin{equation}\label{bunif}
		\|H_{a,b_1}^\a\|_{L^{p}\times L^{q}\rightarrow L^{r}}=\|H_{a,b_2}^\a\|_{L^{p}\times L^{q}\rightarrow L^{r}}\,.
\end{equation}
As a consequence of \eqref{bunif} we further deduce:
\begin{itemize}
\item the bounds obtained in both \eqref{mainrmon} and \eqref{mainrmonmax} are \emph{uniform} in the $b$ parameter.

\item via a Fatou argument, for any $b\in\R$ we have $\|H_{a,0}^\a\|_{L^{p}\times L^{q}\rightarrow L^{r}}\leq \|H_{a,b}^\a\|_{L^{p}\times L^{q}\rightarrow L^{r}}$.
\end{itemize}
The last item above reinforces the hybrid regime label used in this paper since any of the members within the class $\{H_{a,b}^\a\}_{b\in \R}$ controls and encapsulates the behavior of the classical Bilinear Hilbert transform $H_{a,0}^\a$.

Finally, based on the above comments, without loss of generality we may assume from now that
$$b=1\,.$$
\end{observation}

\subsection{The spatial scale decomposition}

In this section we perform a spatial (time) decomposition relative to the $t$-parameter. For this,
we decompose the kernel $\frac{1}{t}$ as
\begin{equation}\label{1overt}
	\frac{1}{t}=\sum_{j\in\Z}\rho_j(t) \quad \text{for all } t\in\R\setminus\{0\},
\end{equation}
with $\rho_j(t)=2^j\,\rho(2^j t)$ a suitable $\rho\in C^\infty_0(\R)$ odd function such that $\supp\rho\subseteq\{t\in\R\:|\:\frac{1}{4}<|t|<4\}$ obeys the \emph{mean zero condition}
\begin{equation}\label{mean-zero}
	\int_{\R}\rho(t)\,dt=0\,.
\end{equation}

Now putting together \eqref{mult} and  \eqref{1overt} we have
\begin{equation*}
	m(\xi,\eta)=\sum_{j\in\Z} m_j(\xi,\eta)\,
\end{equation*}
with
\begin{equation}\label{rhoj}
	m_j(\xi,\eta):=\int_\R e^{-i(\xi-a\eta) t}e^{i \eta t^\a}\rho_j(t)dt=\int_\R e^{-i\frac{\xi-a\eta}{2^j}  t}e^{i\frac{\eta}{2^{\a j}}t^\a}\rho(t)dt\:.
\end{equation}

Letting $\L_j$ be the form whose multiplier is given by $m_j$, we notice that
\begin{equation}\label{trilinj}
	\L_j(f,g,h)=\int_\R\int_\R f\Big(x-\frac{t}{2^j}\Big)\,g\Big(x+a\frac{t}{2^j}+\frac{t^\a}{2^{\a j}}\Big)\,h(x)\,\rho(t)\,dt\,dx\,.
\end{equation}

Once at this point, we make the following

\begin{observation}\label{trim}[\textsf{Split between non-singular and singular behavior}] In what follows, for a suitable choice of a constant $C_0>0$ depending only on the curve $\g$ (\textit{i.e.} on $a,b$ and $\a$),  we decompose our form $\L$ in two components\footnote{The specific choice of $C_0$ will be clarified in Section \ref{NSIterm}. This split will help us to avoid some technical difficulties when treating Theorem \ref{thm:mS}.}
\begin{itemize}

\item the non-singular form

\begin{equation}\label{nsing}
\L^{NSI}:=\sum_{|j|<C_0} \L_j\,;
\end{equation}

\item the singular form

\begin{equation}\label{sing}
\L^{SI}:=\sum_{|j|\geq C_0} \L_j\,.
\end{equation}
\end{itemize}
\end{observation}

With these we obtain the initial decomposition
\begin{equation}\label{lj}
	\L=\sum_{j\in\Z} \L_j=\L^{NSI}\,+\,\L^{SI}\,.
\end{equation}

\subsection{The phase-dependent decomposition}\label{Phdepdec}

In this section we perform a second decomposition that focuses exclusively on the frequency properties of the multiplier associated to $\L$. To this end we decompose
\begin{equation}\label{keydec}
	\L^{SI}=\L^{LO}\,+\,\L^{NS}\,+\,\L^{S}\,,
\end{equation}
where
\begin{itemize}
	
\item $\L^{LO}$ stands for the \emph{low oscillatory} component: this treats the case when the phase of the multiplier has essentially no oscillation;
	
\item $\L^{NS}$ stands for the \emph{non-stationary} component: this is the regime in which the phase of the multiplier does oscillate but has no stationary points;
	
\item $\L^{S}$ stands for the \emph{stationary} component: this corresponds to the regime when the phase of the multiplier oscillates and has stationary points.
\end{itemize}

Given the context provided by the above itemization,\footnote{Throughout this section we assume that $|j|\geq C_0$.} we isolate the phase of the multiplier $m_j$ in \eqref{rhoj}:
\begin{equation}\label{phasefunction}
	\varphi_{j,\xi,\eta}(t):=-\frac{\xi-a\eta}{2^j}\,t + \frac{\eta}{2^{\a j}}\,t^\a\,.
\end{equation}
Taking its derivative
\begin{equation*}
	\frac{d}{dt}\varphi_{j,\xi,\eta}(t)=-\frac{\xi-a\eta}{2^j}+\frac{\a\,\eta}{2^{\a j}}\,t^{\a-1}
\end{equation*}
we notice that at the heuristic level
\begin{equation}\label{derivative-heuristic}
	\frac{d}{dt}\varphi_{j,\xi,\eta}(t)\approx -\frac{\xi-a\eta}{2^j}+\frac{\a\,\eta}{2^{\a j}}.
\end{equation}
Expression \eqref{derivative-heuristic} invites us to consider the magnitude of the terms $-\frac{\xi-a\eta}{2^j}$ and $\frac{\a,\eta}{2^{\a j}}$, and as a matter of convenience we will accomplish this relative to powers of two (dyadic scale).

In order to do so, we let $\phi$ be a non-negative even Schwartz function with $\supp\phi\subseteq\{x\in\R\,|\,\frac{1}{4}<|x|<4\}$ that generates the partition of unity
\begin{equation}\label{ref:unity-phi}
	\sum_{n\in\Z}\phi\Big(\frac{x}{2^n}\Big)=1\qquad\textrm{for all}\:x\in\R\setminus\{0\}\,.
\end{equation}

Using now \eqref{ref:unity-phi}, we decompose each component $m_j$ of the multiplier $m$ as
\begin{equation}\label{decmjn}
	m_j(\xi,\eta)=\sum_{(m,n)\in \Z^2}\hspace{-.2cm} m_{j,m,n}(\xi,\eta) :=\sum_{(m,n)\in \Z^2} \hspace{-.2cm}  m_j(\xi,\eta)\, \phi\Big(\frac{\xi-a\eta}{2^{j+m}}\Big)\,\phi\Big(\frac{\a\,\eta}{2^{\a j+n}}\Big).
\end{equation}

Once at this point, for the same $C_0$ defined earlier, we let
\begin{equation}\label{diag}
\Delta:=\{(m,n)\in\Z^2:m,n\ge0,|m-n|\le C_0\}\,,
\end{equation}
and decompose our multiplier into three components:
\begin{itemize}
	\item the \textit{low oscillatory} component:
	\begin{align}
		m_j^{LO}&:=\sum_{(m,n)\,\in\, \Z_{-}\times \Z_-} m_{j,m,n}\,;\label{ld}
	\intertext{\item the \textit{non-stationary} component:}
		m_j^{NS}&:=\sum_{(m,n)\,\in\, \Z^2\setminus ((\Z_{-}\times \Z_-)\cup \Delta)} m_{j,m,n}\,;\label{lnd-s}
	\intertext{\item the \textit{stationary} component:}\label{highfreqon}
		m_j^{S}&:=\sum_{(m,n)\in\Delta} m_{j,m,n}\,.
	\end{align}
\end{itemize}

With this, defining $\L_{j,m,n}$ as the form whose multiplier is given my $m_{j,m,n}$, and similarly for $\L_{j}^{LO}$, $\L_{j}^{NS}$, $\L_{j}^{S}$ we finally set
\begin{equation}\label{formdefc}
	\L^{*}(f,g,h):=\sum_{|j|\geq C_0} \L_{j}^{*}(f,g,h)\,,
\end{equation}
where $\{*\}\in\{LO, NS, S\}$.

This achieves the desired decomposition \eqref{keydec}.

\subsection{The phase-space decomposition}\label{sphase}

In this section we will reorganize the information carried within the forms $\L^{LO}$, $\L^{NS}$, $\L^{S}$ depending on the relative position of the parameters $j,m,n$ on the real axis. The philosophy behind this reorganization is revealed by the following:
\begin{observation}\label{dominance}[\textsf{Asymptotic behavior}] Let $j\in\Z$, $m,\,n\in\N$ and assume $\a>1$. Departing from the formula
\begin{equation}\label{mjmn}
	m_{j,m,n}(\xi,\eta):=\left(\int_\R e^{-i\frac{\xi-a\eta}{2^j}  t}e^{i\frac{\eta}{2^{\a j}}t^\a}\rho(t)dt\right)\,\phi\Big(\frac{\xi-a\eta}{2^{j+m}}\Big)\phi\Big(\frac{\a\,\eta}{2^{\a j+n}}\Big)
\end{equation}
we envision three possible scenarios that mold the behavior of the associated trilinear form $\L_{j,m,n}$:\footnote{When $m$ or $n$ (or both) take negative values we will need to properly adjust this reasoning - see the approach for the terms $\L^{NS}$ and $\L^{LO}$.}
\begin{itemize}
		\item  If $j$ is suitably small depending on $m$ and $n$ then the linear factor $\frac{t}{2^j}$ is dominated by the nonlinear expression $\frac{t^\a}{2^{\a j}}$;\footnote{This is further explained/detailed in \eqref{NLS}--\eqref{saljm20n1}.} such a situation is referred to as the \underline{nonlinear} regime;
        \item If $j$ is suitably large depending on $m$ and $n$ then the factor $\frac{t}{2^j}$ dominates $\frac{t^\a}{2^{\a j}}$ at the expense of a suitable nonlinear modulation factor;\footnote{This is to be soon clarified in \eqref{HS}--\eqref{saljm21}.} such a situation is referred to as the \underline{hybrid} regime;
        \item  If $j$ is neither part of the hybrid or nonlinear regimes then no relative dominance can be exploited between the linear $\frac{t}{2^j}$ and nonlinear $\frac{t^\a}{2^{\a j}}$ factors and thus we leave the expression of $\L_{j,m,n}$ unaltered;  this is referred to as the \underline{transitional} regime.
	\end{itemize}
\end{observation}

\begin{remark}\label{settlealpha} Because the asymptotics of the polynomial $at+t^\a$ are reversed when $\a<1$, the analysis of our operator in this latter case is obtained by interchanging our arguments for the hybrid and the nonlinear regimes in the former case $\a>1$. As a consequence of this, throughout the reminder of this paper we restrict our attention to the case
	\begin{equation}\label{al}
		\a>1\,.
	\end{equation}
\end{remark}

We are now going to clarify the heuristic presented in Observation \ref{dominance}. For expository reasons, we start our analysis in the reverse order relative to the presentation made in Section \ref{Phdepdec}:

$\newline$
\noindent\textbf{The stationary component $\L^{S}$}\label{ls}
$\newline$

From \eqref{diag}, \eqref{highfreqon} and \eqref{formdefc} without loss of generality we may assume that
\begin{equation}\label{formdefcS}
	\L^{S}(f,g,h)\approx \sum_{|j|\geq C_0}\,\sum_{m\in\N} \L_{j,m,m}(f,g,h)\,.
\end{equation}

We now identify three regimes:

$\newline$
\noindent\textsf{The nonlinear stationary component $\L^{S}_{NL}$}
$\newline$

This is defined as
\begin{equation}\label{NLS}
	\L^{S}_{NL}(f,g,h)=\sum_{m\in\N}\L^{S}_{NL,m}(f,g,h):=\sum_{m\in\N} \sum_{{|j|\geq C_0}\atop{(\a-1)j<-m}} \L_{j,m,m}(f,g,h)\,.
\end{equation}

In this regime, fixing $m\in\N$ and $(\a-1)j<-m$, we have that
\begin{equation}\label{nsaljmm1}
m_{j,m,m}(\xi,\eta)\approx\left(\int_\R e^{-i\frac{\xi-a\eta}{2^j}  t}e^{i\frac{\eta}{2^{\a j}}t^\a}\rho(t)dt\right)\,
\phi\Big(\frac{\xi}{2^{j+m}}\Big)\,\phi\Big(\frac{\a\,\eta}{2^{\a j+m}}\Big)\,\phi\Big(\frac{\xi+\eta}{2^{j+m}}\Big)\,.
\end{equation}
Moreover, we have that the linear term is dominated by the nonlinear term in the expression of $\g$ appearing in the argument of $g$, \textit{i.e.}:
\begin{equation}\label{redaln1}
(g*\check{\phi}_{\a j+m})\Big(x+a\frac{t}{2^j}+\frac{t^\a}{2^{\a j}}\Big)\approx
(g*\check{\phi}_{\a j+m})\Big(x+\frac{t^\a}{2^{\a j}}\Big)\:.
\end{equation}
Deduce from \eqref{nsaljmm1} and \eqref{redaln1} that
\begin{align}
	\nonumber\\[-8ex]
	\nonumber\intertext{\begin{equation}
			\L_{j,m,m}(f,g,h)\approx \nonumber \end{equation}} \nonumber\\[-4ex]
	&\int_\R\int_\R  (f*\check{\phi}_{j+m})\Big(x-\frac{t}{2^j}\Big)\,
	(g*\check{\phi}_{\a j+m})\Big(x+\frac{t^\a}{2^{\a j}}\Big)\,
	(h*\check{\phi}_{j+m})(x)\,\rho(t)\,dt\,dx\,.\label{saljm20n1}
\end{align}

$\newline$
\noindent\textsf{The hybrid stationary component $\L^{S}_{H}$}\label{lm}
$\newline$

This is defined as
\begin{equation}\label{HS}
	\L^{S}_{H}(f,g,h):=\sum_{m\in\N}\L^{S}_{H,m}(f,g,h):=\sum_{m\in\N} \sum_{{|j|\geq C_0}\atop{(\a-1)j>m}} \L_{j,m,m}(f,g,h)\,.
\end{equation}

In this regime, fixing $m\in\N$ and $(\a-1)j>m$ with $j\geq C_0$, we have that
\begin{equation}\label{saljmm0}
		m_{j,m,m}(\xi,\eta)\approx\sum_{l\sim_{b,\a} 2^{(\a-1)j}} m_{j,m,m}^{l}(\xi,\eta)
\end{equation}
with\footnote{Here we abuse the notation and write the same $\phi\in C_{0}^{\infty}(\R)$ for the frequency location in $\xi$, $\eta$ and $\xi+\eta$ though, of course, strictly speaking these should be represented by three distinct smooth, compactly supported functions whose locations depend on the parameters $a, b$ and $\a$.}
\begin{equation}\label{saljmm12}
m_{j,m,m}^{l}(\xi,\eta)\!:=\!\Big(\int_\R\! e^{-i\frac{\xi-a\eta}{2^j}  t}e^{i\frac{b \eta}{2^{\a j}}t^\a}\!\!\rho(t)dt\Big)
	\phi\Big(\frac{\xi}{2^{j+m}}-a l\Big)\phi\Big(\frac{ \eta}{2^{j+m}}-l\Big)\phi\Big(\frac{\xi{+}\eta}{2^{j+m}}-l\left(a+1\right)\Big),
\end{equation}
and
\begin{equation}
(g*\check{\phi}_{j+m}^l)\Big(x+a\frac{t}{2^j}+\frac{t^\a}{2^{\a j}}\Big)\approx
e^{i\,\frac{l}{2^{(\a-1)j}}\,2^m\,t^\a}\,(g*\check{\phi}_{j+m}^l)\Big(x+a\frac{t}{2^j}\Big)\,,
\end{equation}
where here we set
\begin{equation}
		\check{\phi}_{j+m}^l(x):=\int_{\R} \phi\Big(\frac{\xi}{2^{j+m}}-l\Big)\,e^{i\,\xi\,x}\,d\xi=
		2^{j+m}\,e^{i\,2^{j+m}\,l\,x}\,\check{\phi}(2^{j+m}\,x)\,.
\end{equation}

As a consequence, we deduce
\begin{equation}\label{saljmsumm1}
\L_{j,m,m}(f,g,h)=\sum_{l\sim 2^{(\a-1)j}}\L^{l}_{j,m,m}(f,g,h)\,,
\end{equation}
with
\begin{align}
	\nonumber\\[-8ex]
	\nonumber\intertext{\begin{equation}
			\L^{l}_{j,m,m}(f,g,h)\approx \label{saljm21} \end{equation}} \nonumber\\[-4ex]
	&\int_{\R^2}  (f*\check{\phi}_{j+m}^{al})\Big(x-\sdfrac{t}{2^j}\Big)
	(g*\check{\phi}_{j+m}^l)\Big(x+a\sdfrac{t}{2^j}\Big)
	(h*\check{\phi}_{j+m}^{(a+1)l})(x) e^{i\,\frac{l}{2^{(\a-1)j}}\,2^m\,t^\a}\!\!\rho(t)\,dt\,dx\,. \nonumber
\end{align}

$\newline$
\noindent\textsf{The transitional stationary component $\L^{S}_{TR}$}
$\newline$

This is now simply defined as
\begin{equation}\label{TRS0}
	\L^{S}_{TR}(f,g,h):= \L^{S}_{TR,+}(f,g,h)\,+\,\L^{S}_{TR,-}(f,g,h)\,,
\end{equation}
where
\begin{align}
	\L^{S}_{TR,\pm}(f,g,h)= \sum_{m\in\N}\L^{S}_{TR,\pm,m}(f,g,h):=\sum_{m\in\N}
	\sum_{\substack{|j|\geq C_0\\ -m\leq (\a-1)j\leq m\\ \sgn j=\pm}}
	\hspace{-.4cm} \L_{j,m,m}(f,g,h)\,.\label{TRS1}
\end{align}

This way we have obtained the desired decomposition

\begin{equation}\label{formdefcSDec}
	\L^{S}=\L^{S}_{NL}\,+\,\L^{S}_{TR}\,+\L^{S}_{H}\,.
\end{equation}

$\newline$
\noindent\textbf{The non-stationary component $\L^{NS}$}\label{lns}
$\newline$

In this situation, recalling \eqref{rhoj} and \eqref{decmjn}, and following the outlined philosophy motivating the decomposition in \eqref{keydec}, we exploit the lack of stationary points of the phase in \eqref{phasefunction} and integrate by parts once in order to derive (the heuristic) formulation\footnote{Throughout the paper we set $m_{+}:=\max\{m,0\}$.}
\begin{equation}\label{decmjnT}
	m_{j,m,n}(\xi,\eta)\approx\frac{1}{2^{|m_{+}-n_{+}|}}\, m_j(\xi,\eta)\,\phi\Big(\frac{\xi-a\eta}{2^{j+m}}\Big)\,\phi\Big(\frac{\a\,\eta}{2^{\a j+n}}\Big).
\end{equation}

Once we obtain the extra decaying factor in \eqref{decmjnT}, one can track the motivation provided earlier for $\L^{S}$ and achieve the analogue decomposition
\begin{equation}\label{formdefcNSDec}
	\L^{NS}=\L^{NS}_{NL}\,+\,\L^{NS}_{TR}\,+\L^{NS}_{H}\,.
\end{equation}

In order to make explicit the nonlinear, transitional and hybrid nonstationary components for ${*}\in\{NL,\,TR,\,H\}$ we will write
\begin{equation}
 \L^{NS}_{*}=\L^{NS,I}_{*}\,+\,\L^{NS, II}_{*}\,+\,\L^{NS, III}_{*}\,+\,\L^{NS, IV}_{*}\,,
\end{equation}
and respectively
\begin{equation}
 \L^{NS}=\L^{NS,I}\,+\,\L^{NS, II}\,+\,\L^{NS, III}\,+\,\L^{NS, IV}\,,
\end{equation}
with
\begin{align}
\L^{NS,I} &=\L^{NS,I}_{NL}\,+\,\L^{NS,I}_{TR}\,+\,\L^{NS,I}_{H}=\sum_{n\geq 0}  \L^{NS,I}_{NL,n}\,+\,\sum_{n\geq 0}  \L^{NS,I}_{TR,n}\,+\,\sum_{n\geq 0}\L^{NS,I}_{H,n}\nonumber\\[1ex]
&:=\sum_{n\geq 0}\hspace{-.2cm} \sum_{\substack{m<0\\[.1em]|j|\ge C_0\\[.1em](\a-1)j\leq -n}}\hspace{-.3cm} \L_{j,m,n}\,\,+\,\,
\sum_{n\geq 0}\hspace{-.2cm} \sum_{\substack{m<0\\[.1em]|j|\ge C_0\\[.1em] -n<(\a-1)j<0}}\hspace{-.3cm} \L_{j,m,n}\,\,+\,\,\sum_{n\geq 0} \sum_{\substack{m<0\\[.1ex]|j|\ge C_0\\[.1ex] (\a-1)j\geq 0}} \hspace{-.1cm} \L_{j,m,n}\,,\raisetag{1\baselineskip}\label{ns11}\\[1ex]
\L^{NS,II}&=\L^{NS,II}_{NL}\,+\,\L^{NS,II}_{TR}\,+\,\L^{NS,II}_{H}\nonumber\\[1ex]
&=\sum_{0\leq m<n-C_0}  \L^{NS,II}_{NL,m,n}\quad+\,\,\sum_{0\leq m<n-C_0}  \L^{NS,II}_{TR,m,n}\quad+\,\,\sum_{0\leq m<n-C_0}  \L^{NS,II}_{H,m,n}\nonumber\\
&:=\hspace{-.1cm}\sum_{0\leq m<n-C_0}
\left(\:
\sum_{\substack{|j|\ge C_0\\[.5ex](\a-1)j\leq-n }}\hspace{-.2cm} \L_{j,m,n}+
\sum_{\substack{|j|\ge C_0\\[.5ex]-n<(\a-1)j<m}}\hspace{-.2cm} \L_{j,m,n}+
\sum_{\substack{|j|\geq C_0\\[.5ex] (\a-1)j\geq m}}\hspace{-.2cm} \L_{j,m,n}
\right)\,, \label{ns21}\\[1ex]
\L^{NS,III}&=\L^{NS,III}_{L}\,+\,\L^{NS,III}_{TR}\,+\,\L^{NS,III}_{H}\nonumber\\[1ex]
&=\sum_{0\leq n<m-C_0}  \L^{NS,III}_{NL,m,n}\quad+\,\,\sum_{0\leq n<m-C_0}  \L^{NS,III}_{TR,m,n}\quad+\,\,\sum_{0\leq n<m-C_0}  \L^{NS,III}_{H,m,n}\nonumber\\
&:=\sum_{0\leq n<m-C_0}
\left(\:
\sum_{\substack{|j|\ge C_0\\[.5ex](\a-1)j\leq-n}}\hspace{-.2cm} \L_{j,m,n}\,+\,
\sum_{\substack{|j|\ge C_0\\[.5ex]-n<(\a-1)j<m}}\hspace{-.3cm} \L_{j,m,n}\,+\,
\sum_{\substack{|j|\ge C_0\\[.5ex](\a-1)j\geq m}}\hspace{-.3cm} \L_{j,m,n}
\right)\,,\label{ns31}\\[1ex]
\L^{NS,IV}&=\L^{NS,IV}_{L}+\L^{NS,IV}_{TR}+\L^{NS,IV}_{H}=\sum_{m\geq 0}\L^{NS,IV}_{NL,m}+\sum_{m\geq 0}\L^{NS,IV}_{TR,m} +\sum_{m\geq 0}\L^{NS,IV}_{H,m} \nonumber\\[1ex]
&:=\sum_{m\geq 0}\,
\sum_{\substack{n<0\\[.2em]|j|\ge C_0\\[.2em] (\a-1)j\leq 0}}\hspace{-.1cm} \L_{j,m,n}\,\,+\,\,
\sum_{m\geq 0}\,\sum_{\substack{n<0\\[.2em]|j|\ge C_0\\[.2em] 0<(\a-1)j<m}}\hspace{-.2cm} \L_{j,m,n}\,\,+\,\,
\sum_{m\geq 0}\,\sum_{\substack{n<0\\[.2ex]|j|\ge C_0\\[.2ex] (\a-1)j\geq m}} \hspace{-.2cm} \L_{j,m,n}\,.\raisetag{1\baselineskip}\label{ns41}
\end{align}

$\newline$
\noindent\textbf{The low oscillatory component $\L^{LO}$}
$\newline$

In this situation, exploiting the fact that the phase $\varphi_{j,\xi,\eta}(t)$ is essentially constant (\emph{i.e.}, no oscillation is present), we will be able to reduce $m_j^{LO}$ to a sum of the form\footnote{For more on this the reader is invited to consult Section \ref{LOf}.}
\begin{equation}\label{lodescr}
	m_j^{LO}(\xi,\eta)\approx\phi\Big(\frac{\xi-a\eta}{2^{j}}\Big)\, \psi\Big(\frac{\a\,\eta}{2^{\a j}}\Big)\,+\,\psi\Big(\frac{\xi-a\eta}{2^{j}}\Big)\, \phi\Big(\frac{\a\,\eta}{2^{\a j}}\Big)\,,
\end{equation}
where here $\psi\in C_0^{\infty}(\R)$ is a function supported on $(-2,2)$ defined as
\begin{equation}\label{ref:unity-psi}
	\psi(x):=\sum_{n\in\Z_-} \phi\Big(\frac{x}{2^n}\Big)\,.
\end{equation}

Thus, in accordance with \eqref{lodescr}, it is natural to split
\begin{align}
\L^{LO} &=\L^{LO}_{D}\,+\,\L^{LO}_L\label{llo}\\[1ex]
&:=\sum_{j\le -C_0}\,\sum_{m,n<0}\, \L_{j,m,n}\,+\,\sum_{j\ge C_0}\,\sum_{m,n<0}\, \L_{j,m,n}\,,\nonumber
\end{align}
where
\begin{itemize}
\item $\L^{LO}_{D}$ stands for the \emph{degenerate low frequency} component that behaves essentially as a paraproduct;

\item $\L^{LO}_{L}$ stands for the \emph{linear low frequency} component that essentially represents a truncated version of the Bilinear Hilbert transform.
\end{itemize}

This concludes our initial decomposition, that can be summarized as follows:
\begin{center}
	\begin{forest}
		for tree={tier/.wrap pgfmath arg={tier #1}{level()},
			edge path={\noexpand\path[\forestoption{edge}] (\forestOve{\forestove{@parent}}{name}.parent anchor) -- +(0,-12pt)-| (\forestove{name}.child anchor)\forestoption{edge label};}
		}
		[$\L$
			[$\L^{NSI}$, before computing xy={s'-=120pt}
			]
			[$\L^{SI}$, before computing xy={s'=60pt}
				[$\L^{LO}$
					[$\L^{LO}_D$
					]
					[$\L^{LO}_L$
					]
				]
				[$\L^{NS}$
					[$\L^{NS}_{NL}$
					]
					[$\L^{NS}_{TR}$
					]
					[$\L^{NS}_H$
					]
				]
				[$\L^{S}$
					[$\L^S_{NL}$
					]
					[$\L^S_{TR}$
					]
					[$\L^S_H$
					]
				]
			]
		]
	\end{forest}
\end{center}
\if
and thus write
\begin{align}\label{keydeclnl}
	\L&=& &\L^{NSI}& &+& &\qquad\phantom{\L^{SI}}& &\phantom{+}& &\qquad\quad\L^{SI}& &\phantom{+}&  &\qquad\quad\phantom{\L^S}&\nonumber\\
	\,&=& &\L^{NSI}& &+& &\qquad\L^{LO}& &+& &\qquad\quad\L^{NS}& &+&  &\qquad\quad\L^S&\nonumber\\
	\,&=&&\L^{NSI}& &+&  &\L^{LO}_D\,+\,\L^{LO}_{L}&   &+& &\L^{NS}_{NL}+ \L^{NS}_{TR}+\L^{NS}_{H}& &+& &\L^S_{NL}+ \L^S_{TR}+\L^S_{H}\,.&\nonumber\\
\end{align}
\fi

In what follows we will separately discuss each of the terms in the decomposition
\begin{equation}
	\L=\L^{NSI}\,+\,\L^{LO}\,+\,\L^{NS}\,+\,\L^S\;.
\end{equation}

\section{The analysis of the stationary term $\L^{S}$: Implementation of the LGC methodology.}\label{LGCimpl}

In this section we focus on the most delicate component of our operator, that is the stationary term $\L^{S}$, which addresses the situation when the phase of the multiplier has stationary points. At the heart of the analysis of $\L^{S}$ stays the (Rank I) LGC-methodology introduced in \cite{Lie19}.

\subsection{Statements of the key results}\label{MR1}

The main result of this section is given by the following

\begin{theorem}\label{thm:mS}
With the previous notations\footnote{Recall \eqref{trilin} and \eqref{formdefcS}.} we have that for any $a\in\R\setminus\{-1\}$,  $\a\in (0,\infty)\setminus\{1,2\}$ and $1<p,\,q<\infty$ with $\frac{1}{p}+\frac{1}{q}=\frac{1}{r}$ and $r>\frac{2}{3}$ the following holds
\begin{equation}\label{thm:mSS}
	|\L^S(f,g,h)|\lesssim_{a,\a,p,q} \|f\|_p\,\|g\|_q\,\|h\|_{r'}\,.
\end{equation}
\end{theorem}

This theorem will be an immediate consequence of the two next two propositions once one applies standard multilinear interpolation for sublinear forms (see \textit{e.g.} \cite{MS13}) and a telescoping sum argument.

In order to state these propositions we will make a brief detour to introduce some notation: recalling \eqref{highfreqon}, for expository reasons, we may assume without loss of generality that $m=n$ and hence that
\begin{equation}\label{mjS}
m_j^S(\xi,\eta)\approx\sum_{m\in\N}m_{j,m,m}(\xi,\eta):=
\sum_{m\in\N}m_j(\xi,\eta)\, \phi\Big(\frac{\xi-a\eta}{2^{j+m}}\Big)\,
\phi\Big(\frac{\a\,\eta}{2^{\a j+m}}\Big).
\end{equation}

Now, for $m\in\N$ fixed, we let
\begin{equation}\label{sjm}
	\L^S_{m}(f,g,h):=\sum_{|j|\ge C_0}\int_\R \int_\R \hat{f}(\xi)\,\hat{g}(\eta)\,\hat{h}(\xi+\eta)\,m_{j,m,m}(\xi,\eta)\,d\xi\,d\eta\,.
\end{equation}

With this, we have

\begin{proposition}\label{L2mdecay}

There exists $\d>0$ such that for any $a\in\R\setminus\{-1\}$ and $\a\in (0,\infty)\setminus\{1,2\}$ we have uniformly in $m\in\N$ that
\begin{equation}\label{LSm}
	|\L^S_{m}(f,g,h)|\lesssim_{a,\a} 2^{-\d m}\,\|f\|_2\,\|g\|_2\,\|h\|_{\infty}\,.
\end{equation}
\end{proposition}

The proof of the $m-$exponentially decaying bound in \eqref{LSm} will be presented in Section \ref{Propmdec}. The second proposition treated in Section \ref{Propgen} provides us with the general $L^p\times L^q\times L^{r'}$ range up to an admissible $m-$polynomial loss:

\begin{proposition}\label{Lgennodecay}
Let $m\in\N$ and $F,\,G,\,H$ be any measurable sets with finite (nonzero) Lebesgue measure. Then
\begin{equation}\label{Hprexist}
\exists\:H'\subseteq H\:\:\textrm{measurable with}\:\:|H'|\geq \frac{1}{2} |H|\,,
\end{equation}
such that, for any triple of functions $f,\,g$ and $h$ obeying
\begin{equation}\label{restrcitedweak}
|f|\leq \chi_{F},\qquad |g|\leq \chi_{G},\qquad |h|\leq \chi_{H'}\,,
\end{equation}
and any triple $(p,q,r')$ with $1<p,q\leq \infty$, $r>\frac{2}{3}$ and $\frac{1}{p}+\frac{1}{q}+\frac{1}{r'}=1$, one has that
\begin{equation}\label{goallsl1}
|\L^{S}_{m}(f,g,h)|\lesssim_{a,\a,p,q} m\, |F|^{\frac{1}{p}}\,|G|^{\frac{1}{q}}\,|H|^{\frac{1}{r'}}\,.
\end{equation}
\end{proposition}

\subsection{The $L^2\times L^2\times L^\infty$ bound: Proof of Proposition \ref{L2mdecay}}\label{Propmdec}

An interesting aspect that differentiates the proof of this proposition relative to the treatment of Proposition \ref{Lgennodecay} and that of the component $\L^{NS}$ is that the present analysis is performed for the \emph{global} cases
\begin{itemize}
\item $j\ge C_0$ covering the component $\L^{S}_{m,+}:=\L^{S}_{H,m}+\L^{S}_{TR,+,m}$, and

\item $j\le -C_0$ covering the component $\L^{S}_{m,-}:=\L^{S}_{NL,m}+\L^{S}_{TR,-,m}$,
\end{itemize}
with no further subdivisions adapted to the relative positions of $j$ and $m$.

\subsubsection{Preparatives: analysis of the multiplier and reduction to the main term}

As a preliminary step, we would like to isolate the main term within our form $\L_m^{S}$. In order to do so, we recall \eqref{rhoj} and \eqref{mjS} and focus our attention on the structure of $m_{j,m,m}^S(\xi, \eta)$ assuming, of course, that $\xi,\, \eta\in \textrm{supp}\,m_{j,m,m}^S$. Let
\begin{equation}
t_c:=2^j\Big(\frac{\xi-a\eta}{\a\,\eta}\Big)^{\frac1{\a-1}}
\end{equation}
be the critical point of the phase function defined in \eqref{phasefunction} by
\begin{equation*}
\varphi_{j,\xi,\eta}(t)=-\frac{\xi-a\eta}{2^j}t+\frac{\eta}{2^{\a j}}t^\a.
\end{equation*}

Notice that since $|\varphi_{j,\xi,\eta}''(t)|\approx_\a 2^m$ for all $t\in\supp\rho$, we have
\begin{align*}
|\varphi_{j,\xi,\eta}(t)-\varphi_{j,\xi,\eta}(t_c)|=(t-t_c)^2\,\Big|\int_0^1\int_0^1 r\,\varphi_{j,\xi,\eta}''(t_c+s\,r\,(t-t_c))\, dr\,ds\,\Big|\\[1ex]
\le \|\varphi_{j,\xi,\eta}''\|_{L^\infty(\supp\rho)}\,(t-t_c)^2\lesssim_{\a,\rho} 2^m\,(t-t_c)^2,
\end{align*}
Thus $|\varphi_{j,\xi,\eta}(t)-\varphi_{j,\xi,\eta}(t_c)|\lesssim 1$ whenever $|t-t_c|\lesssim 2^{-\frac{m}2}$.

Next, we choose $\vartheta, \tilde{\vartheta} \in C_0^\infty(\R)$ with $\supp\vartheta\subseteq [-10,10]$ and $ \supp\tilde{\vartheta}\subseteq\{t|\frac{1}{100}<|t|<100\}$ satisfying
\[
1=\vartheta(t)+\sum_{k\in\N}\tilde\vartheta(2^{-k}t)\,\qquad \forall\, t\in\R.
\]
With these done, we write
\begin{equation}
m_{j,m,m}=\mathcal{A}_{j,m}(\xi,\eta)+\sum_{k\in\N}\B_{j,m}^k(\xi,\eta),
\end{equation}
where
\begin{align*}
\mathcal{A}_{j,m}(\xi,\eta)&=\bigg(\int_\R \vartheta(2^{\frac{m}{2}}(t-t_c)) e^{i\varphi_{j,\xi,\eta}(t)}\, \rho(t)dt\bigg)\,\phi\Big(\frac{\xi-a\,\eta}{2^{j+m}}\Big)\,\phi\Big(\frac{\a\,\eta}{2^{\a j+m}}\Big)\\
\intertext{and}
\B_{j,m}^k(\xi,\eta)&=\bigg(\int_\R \tilde{\vartheta}(2^{\frac{m}{2}-k}(t-t_c)) e^{i\varphi_{j,\xi,\eta}(t)}\, \rho(t)dt\bigg)\,\phi\Big(\frac{\xi-a\,\eta}{2^{j+m}}\Big)\,\phi\Big(\frac{\a\,\eta}{2^{\a j+m}}\Big).
\end{align*}

Following similar reasoning with those in \cite{Lie15} we deduce that there exist functions $\{\zeta_k\}_{k\ge 0}$ with $\zeta_k:[\frac1{10},10]\times[\frac1{10},10]\mapsto \R$ and $\|\zeta_k\|_{C^4}\lesssim_{\a,a} 2^{-k}$ such that
\begin{align}
m_{j,m,m}(\xi,\eta)=\sum_{k\ge 0} 2^{-\frac{m}{2}}\,e^{i\varphi_{j,\xi,\eta}(t_c)}\,\zeta_k\Big(\frac{\xi-a\eta}{2^{j+m}},\frac{\a\,\eta}{2^{\a j+m}}\Big) \phi\Big(\frac{\xi-a\eta}{2^{j+m}}\Big)\phi\Big(\frac{\a\,\eta}{2^{\a j+m}}\Big)\,.\raisetag{.3\baselineskip}
\end{align}

Once at this point, standard reasonings reduce matters to the main term $k=0$. Moreover, via a Fourier series argument, one can show that the function $\zeta_0$ can be essentially replaced by the constant function 1.\footnote{For more details on this, the reader in invited to consult Section 5.1 in \cite{Lie15}.} Hence, from now on, we may assume wlog that the multiplier $m^S_j$ is in fact given by the formula
\begin{equation*}
m_j^S(\xi,\eta)=\sum\limits_{m\in\N} m_{j,m,m}(\xi,\eta)=:\sum\limits_{m\in\N} m_{j,m}(\xi,\eta)\,,
\end{equation*}
with
\begin{equation}\label{mjm}
m_{j,m}(\xi,\eta) \approx 2^{-\frac{m}{2}}\,e^{i\,\varphi_{j,\xi,\eta}(t_c)}\,\phi\Big(\frac{\xi-a\eta}{2^{j+m}}\Big)\,\phi\Big(\frac{\a\,\eta}{2^{\a j+m}}\Big)\,.
\end{equation}

\subsubsection{The treatment of $\L^{S}_{m,+}$}\label{LGCplus}
In this section we focus on proving Proposition \ref{L2mdecay} with $\L^S_{m}$ replaced  in \eqref{LSm} by $\L^{S}_{m,+}$. For this purpose, fixing $j,\,m\in\N$ with $j\geq C_0$, we intend to obtain suitable $m$-decaying bounds for the single piece $\L^{S}_{j,m}$.

With this preface, we are now ready to implement the LGC methodology, that in the context of our paper--similar with the work in \cite{Lie19} and in contrast with the approaches in \cite{BBLV21} and \cite{HL23}--embraces a linearization procedure on the frequency side:

$\newline$
\noindent\textbf{Step I: Phase linearization.} Our methodology here starts with a frequency linearization process applied directly to the phase of our multiplier $m_{j,m}$ in \eqref{mjm}.

For this, we first isolate the phase function $\varphi_{j,\xi,\eta}(t_c)$ which, up to the constant $c_{\a}:=-\frac{\a-1}{\a}\cdot \frac{1}{(\a)^{\frac{1}{\a-1}}}$, may be written as
\begin{equation}\label{mainphase}
\psi(\xi,\eta):=\frac{(\xi-a\eta)^{\frac{\a}{\a-1}}}{\eta^{\frac{1}{\a-1}}}\:.
\end{equation}
Next, we analyze its second derivative behavior; for this, we first compute the first order derivatives
\begin{align}
	\partial_{\xi}\psi(\xi,\eta)&=\frac{\a}{\a-1}\,\frac{(\xi-a\eta)^{\frac{1}{\a-1}}}{\eta^{\frac{1}{\a-1}}}\,,\label{mainphase1}\\
	\intertext{and}
	\partial_{\eta}\psi(\xi,\eta)&=-\frac{a\,\a}{\a-1}\,\frac{(\xi-a\eta)^{\frac{1}{\a-1}}}{ \eta^{\frac{1}{\a-1}}}-\frac{1}{\a-1}\,\frac{(\xi-a\eta)^{\frac{\a}{\a-1}}}{\eta^{\frac{\a}{\a-1}}}\,,\label{mainphase2}
\end{align}
followed by the second order derivatives
\begin{align}
	\partial_{\xi}^2 \psi(\xi,\eta)&=\frac{\a}{(\a-1)^2}\,\frac{(\xi-a\eta)^{\frac{2-\a}{\a-1}}}{\eta^{\frac{1}{\a-1}}}\,,\label{mainphase21}\\[1ex]
	\partial_{\xi\eta}^2\psi(\xi,\eta)&=- \frac{\a}{(\a-1)^2}\:\frac{\xi}{\eta}\:\frac{(\xi-a\eta)^{\frac{2-\a}{\a-1}}}{\eta^{\frac{1}{\a-1}}}\,,\label{mainphase22}
	\intertext{and}
\partial_{\eta}^2\psi(\xi,\eta)&=\frac{\a}{(\a-1)^2}\:\frac{\xi^2}{\eta^2}\:\frac{(\xi-a\eta)^{\frac{2-\a}{\a-1}}}{\eta^{\frac{1}{\a-1}}}\,.\label{mainphase23}
\end{align}
Assume now that $\xi,\,\eta\in \textrm{supp}\,m_{j,m}$. Then, from \eqref{mjm} and \eqref{mainphase21}--\eqref{mainphase23} and based on our assumptions $j\in\N$ and $\a>1$ (recall Remark \ref{settlealpha})  we deduce
\begin{equation}\label{mainphasebd}
|\partial_{\xi}^2\psi(\xi,\eta)|,\,|\partial_{\xi\eta}^2\psi(\xi,\eta)|,\,|\partial_{\eta}^2\psi(\xi,\eta)|\lesssim_{a} \frac{|\a|+1}{(\a-1)^2}\,2^{-m-2j}\,.
\end{equation}
In order for the second order term in the Taylor expansion of the phase to be $O(1)$ we must impose
\begin{equation}\label{freqstep}
|\Delta \xi|\,,\,|\Delta \eta|\lesssim 2^{j+\frac{m}{2}}\,.
\end{equation}

Now, in accordance with \eqref{freqstep}, we discretize the frequency plane as follows:
\begin{align}
	\phi\Big(\frac{\xi-a\eta}{2^{j+m}}\Big)&\approx\sum_{u\sim 2^{\frac{m}{2}}} \phi\Big(\frac{\xi-a\eta}{2^{j+\frac{m}{2}}}-u\Big)\,,\label{discr1}
	\intertext{and}
	\phi\Big(\frac{\a\,\eta}{2^{\a j+m}}\Big)&\approx\sum_{\substack{\\v\sim_{{}_{\a}} 2^{\frac{m}{2}+(\a-1)j}}} \phi\Big(\frac{\eta}{2^{j+\frac{m}{2}}}-v\Big)\,.\label{discr2}
\end{align}
Consequently, from \eqref{mjm}, \eqref{discr1} and \eqref{discr2}, we deduce that
\begin{equation}\label{mjmndisc}
m_{j,m}(\xi,\eta) = \sum_{u\sim 2^{\frac{m}{2}}}\, \sum_{v\sim_{{}_{\a}} 2^{\frac{m}{2}+(\a-1)j}} m_{j,m,u,v}(\xi,\eta)\,,
\end{equation}
with
\begin{equation}
\label{mjmndisc1}
	m_{j,m,u,v}(\xi,\eta)\approx 2^{-\frac{m}{2}}\,e^{i\,c_{\a}\,\psi(\xi,\eta)}\,\phi\Big(\frac{\xi}{2^{j+\frac{m}{2}}}-(u+a v)\Big)\,\phi\Big(\frac{\eta}{2^{j+\frac{m}{2}}}-v\Big)\,.
\end{equation}

At this point we write
\begin{equation}\label{mainphases-jpos}
\psi(\xi,\eta)=2^{j+\frac{m}{2}}\,\frac{(\frac{\xi-a\eta}{2^{j+\frac{m}{2}}})^{\frac{\a}{\a-1}}}{(\frac{ \eta}{2^{j+\frac{m}{2}}})^{\frac{1}{\a-1}}}=:\,2^{j+\frac{m}{2}}\,\underline{\psi}\Big(\frac{\xi}{2^{j+\frac{m}{2}}}, \frac{\eta}{2^{j+\frac{m}{2}}}\Big)\,.
\end{equation}

Now on the support of the multiplier $m_{j,m,u,v}$ from the choice of our discretization -- see \eqref{mainphasebd} and \eqref{freqstep}
-- we deduce via a Taylor series development\footnote{Here we assume that $(\xi,\eta)$ satisfies the condition $m_{j,m,u,v}(\xi,\eta)\not=0$.} around the point $(u+a v, v)$ that\footnote{Here the function $\tilde{\psi}_{j,m}$ is a suitable (piecewise) smooth real function on $\R^2$.}
\begin{align}
\varphi_{j,\xi,\eta}(t_c)
&=c^1_{\a,a}\,\xi\,\frac{u^{\frac{1}{\a-1}}}{v^{\frac{1}{\a-1}}}-\eta\bigg(a\, c^1_{\a,a}\,\frac{u^{\frac{1}{\a-1}}}{v^{\frac{1}{\a-1}}}+c^2_{\a,a}\,\frac{u^{\frac{\a}{\a-1}}}{v^{\frac{\a}{\a-1}}}\bigg)+\tilde{\psi}_{j,m}(u,v)+O(1)\,.\label{keyphaseexpr}\raisetag{.3\baselineskip}
\end{align}
Finally, from \eqref{mjmndisc1} and \eqref{keyphaseexpr} we conclude that

\begin{align}
\nonumber\\[-11ex]
\nonumber\intertext{\begin{equation}
		m_{j,m,u,v}(\xi,\eta)\approx \nonumber \end{equation}} \nonumber\\[-4.5ex]
&2^{-\frac{m}{2}}c_{\a,a,u,v}\, e^{i\Big(c_1\, \xi\,\frac{u^{\frac{1}{\a-1}}}{v^{\frac{1}{\a-1}}}
	-\eta\,\big(a c_1\,\frac{u^{\frac{1}{\a-1}}}{v^{\frac{1}{\a-1}}}+
	c_2\,\frac{u^{\frac{\a}{\a-1}}}{v^{\frac{\a}{\a-1}}}\big)\!\Big)}
\phi\Big(\frac{\xi}{2^{j+\frac{m}{2}}}-(u+a v)\Big)\phi\Big(\frac{\eta}{2^{j+\frac{m}{2}}}-v\Big),\label{mjmndisc10}
\end{align}
where here, for notational simplicity we dropped the dependence of the constants $c_1$, $c_2$ on the parameters $\a,\, a$.

Then the form $\L_{j,m}$ with multiplier defined in \eqref{mjm} may be written as
\begin{align}
	\intertext{\begin{equation}\label{tjm}
			\L_{j,m}(f,g,h):= \end{equation}} \\[-4ex]
	&2^{-\frac{m}{2}}\hspace{-.5cm}
	\sum_{\substack{\\u\sim 2^{\frac{m}{2}}\\[.3ex] v\sim_{{}_{\a}} 2^{\frac{m}{2}+(\a-1)j }}}\hspace{-.5cm}
	\int_{\R^3}h(x)\,\hat{f}(\xi)\,\hat{g}(\eta) \: e^{i\Big(c_1\,\xi\,		\frac{u^{\frac{1}{\a-1}}}{v^{\frac{1}{\a-1}}} 		{-}\eta\Big(a c_1\,\tfrac{u^{\frac{1}{\a-1}}}{v^{\frac{1}{\a-1}}} {+}	c_2\frac{u^{\frac{\a}{\a-1}}}{v^{\frac{\a}{\a-1}}}\Big)\Big)}\,\times\nonumber\\[-4ex]
	&\hspace{4cm}\times\,e^{i(\xi{+}\eta)x}\, \phi\Big(\sdfrac{\xi}{2^{j{+}\frac{m}{2}}}{-}(u{+}a v)\Big)
	\phi\Big(\sdfrac{\eta}{2^{j{+}\frac{m}{2}}}{-}v\Big)\,d\xi\, d\eta\, dx.\nonumber
\end{align}

$\newline$
\noindent\textbf{Step II: Adapted Gabor frame discretization: conversion of the curvature into the time-frequency localization of the wave-packets}
$\newline$

Define the Gabor frame $\{\phi_{u}^p\}_{p,u\in\Z}$ by
\begin{equation}\label{GF}
\phi_{u}^p(\xi):=\frac{1}{2^{\frac{j}{2}+\frac{m}{4}}}\,\phi\Big(\frac{\xi}{2^{j+\frac{m}{2}}}-u\Big)\,e^{i\,\frac{\xi}{2^{j+\frac{m}{2}}}\,p}\,,
\end{equation}
and hence
\begin{equation}\label{GFf}
\widehat{\phi_{u}^p}(x)=2^{\frac{j}{2}+\frac{m}{4}}\hat{\phi}\big(2^{j+\frac{m}{2}}x-p\big)\,e^{-i\,(2^{j+\frac{m}{2}}x-p)u}\,.
\end{equation}


With this we decompose
\begin{align}
	\hat{f}(\xi)&=\sum_{w,p\, \in\Z} \langle f,  \check{\phi}_w^p\rangle\, \phi_w^p (\xi)\,,\label{gdf}
	\intertext{and}
	\hat{g}(\eta)&=\sum_{z,r\, \in\Z} \langle g,  \check{\phi}_z^r\rangle\, \phi_z^r (\eta)\,.\label{gdg}
\end{align}

Replacing \eqref{gdf} and \eqref{gdg} in \eqref{tjm} we get
\begin{align}
\intertext{\begin{equation}\label{tjm0}
		\L_{j,m}(f,g,h)\approx 2^{-\frac{m}{2}}
		\sum_{\substack{\\u\sim 2^{\frac{m}{2}}\\ p\in\Z}}\, \sum_{\substack{\\v\sim_{{}_{\a}} 2^{\frac{m}{2}+(\a-1)j}\\ r\in\Z}}\langle f,\check{\phi}_{u+a v}^p\rangle\, \langle g,  \check{\phi}_{v}^r\rangle\, \times \nonumber\end{equation}} \\[-4ex]
\intertext{\begin{equation}
		\times\int_{\R^3} h(x)\,\phi_{u+av}^p(\xi)\,\phi_v^r(\eta)\, e^{i\,\big(c_1\, \xi\,\tfrac{u^{\frac{1}{\a-1}}}{v^{\frac{1}{\a-1}}} -\eta\,\big(a c_1\,\tfrac{u^{\frac{1}{\a-1}}}{v^{\frac{1}{\a-1}}} + c_2\,\tfrac{u^{\frac{\a}{\a-1}}}{v^{\frac{\a}{\a-1}}}\big) \big)}e^{i(\xi+\eta)x} d\xi\, d\eta\, dx\nonumber \end{equation}}\nonumber\\[-4ex]
&\!\!=2^{-\frac{m}{2}}
\hspace{-.8cm}
\sum_{\substack{p, r\in\Z\\[.3ex] u\sim 2^{\tfrac{m}{2}}\\[.3ex] \:v\sim_{{}_{\a}}2^{\tfrac{m}{2}+(\a-1)j}}}\hspace{-.7cm}
\langle f,\check{\phi}_{u+a v}^p\rangle\,\langle g,\check{\phi}_{v}^r\rangle\! \int_{\R}\! h(x)\, \check{\phi}_{u+av}^p\big(\text{\scriptsize\({-}x{-}c_1\frac{u^{\frac{1}{\a{-}1}}}{v^{\frac{1}{\a{-}1}}}\)}\big)\, \check{\phi}_v^r\big(\text{\scriptsize\({-}x{+}a c_1\,\tfrac{u^{\frac{1}{\a{-}1}}}{v^{\frac{1}{\a{-}1}}}{+}
	c_2\,\tfrac{u^{\frac{\a}{\a{-}1}}}{v^{\frac{\a}{\a{-}1}}}\)}\big)\,dx\,. \raisetag{1.5\baselineskip}\label{ljm}
\end{align}

Now from \eqref{GFf} and the last line above we deduce that the main term in \eqref{tjm0} is achieved when the following
\emph{time-frequency correlation} condition takes place

\begin{equation}\label{tjm1}
	\boxed{\qquad -p-2^{j+\frac{m}{2}}\,c_1\,\frac{u^{\frac{1}{\a-1}}}{v^{\frac{1}{\a-1}}}=
	-r+2^{j+\frac{m}{2}}\,a\,c_1\,\frac{u^{\frac{1}{\a-1}}}{v^{\frac{1}{\a-1}}}+ 2^{j+\frac{m}{2}}\,c_2\,\frac{u^{\frac{\a}{\a-1}}}{v^{\frac{\a}{\a-1}}}\,+O(1)\:.}
\end{equation}

Now, based on \eqref{tjm1}, we conclude that the study of \eqref{ljm} reduces to\footnote{Here and throughout the reminder of the paper any sub/super index attached to a wave-packet is understood in the sense of entire part, \textit{i.e.}  $\phi_{e}^d\equiv \phi_{e}^{\lfloor d\rfloor}$.}

\begin{align}
&|\L_{j,m}(f,g,h)|\lesssim
2^{\frac{j}{2}{-}\frac{m}{4}}\,\hspace{-.2cm}\sum_{\substack{\\u\sim 2^{\frac{m}{2}}\\ r\in\Z}} \sum_{\substack{\\v\sim_{{}_{\a}} 2^{\frac{m}{2}{+}j(\a{-}1)}}}
\hspace{-.4cm}\Big|\Big\langle f,  \check{\phi}_{u+ av}^{r-2^{j+\frac{m}{2}}\,(1+a) c_1\,\frac{u^{\frac{1}{\a-1}}}{v^{\frac{1}{\a-1}}}-
	2^{j+\frac{m}{2}}\,c_2\,\frac{u^{\frac{\a}{\a-1}}}{v^{\frac{\a}{\a-1}}}}\Big\rangle\Big|\,\nonumber\\[-2ex]
&\hspace{5cm}\times
|\langle g,  \check{\phi}_{v}^r\rangle|\,
\Big|\Big\langle h,  \check{\phi}_{u+(1+a) v}^{r- 2^{j+\frac{m}{2}}\,a c_1\,\frac{u^{\frac{1}{\a-1}}}{v^{\frac{1}{\a-1}}}-
	2^{j+\frac{m}{2}}c_2\,\frac{u^{\frac{\a}{\a-1}}}{v^{\frac{\a}{\a-1}}}}\Big\rangle\Big|\,.\label{ljm0a}
\end{align}

\begin{remark}[\textsf{Zero-nonzero curvature paradigm via time-frequency correlations}]\label{tfrcor} We take now a moment in order to elaborate on the expressions \eqref{tjm1} and \eqref{ljm0a}. A preliminary inspection of \eqref{tjm1} reveals the presence of a suitable \emph{correlation} among the spatial parameters $p,\,r$ and the frequency parameters $u,\,v$. This time-frequency intertwining, is a direct expression of the zero--nonzero curvature paradigm encapsulated within the shape of $\g(t)=a\,t+b\,t^\a$ appearing in the original definition of $H_{a,b}^\a$---see \eqref{bht}. Indeed, we have the following:
\begin{itemize}
\item In the \emph{asymptotical/purely zero curvature regime} conceived as the limiting situation when $\a$ tends to $1$ (morally equivalent with the case $b=0$) we notice that the spatial and frequency parameters \emph{decouple} with \eqref{tjm1} becoming the simple spatial condition $p=r$. The latter,  transforms \eqref{ljm0a} into a suitably averaged and truncated form of the discretized wave-packet model corresponding to the classical Bilinear Hilbert transform. Of course, as expected, such a situation produces no decay in the parameter $m$ and thus no meaningful analog of Proposition \ref{L2mdecay} is possible.

\item In the \emph{purely non-zero curvature regime} obtained by setting $a=0$, the key time-frequency intertwining in \eqref{tjm1} remains valid while in \eqref{ljm0a} the frequency locations of the Gabor coefficients for $f$ and $g$ \emph{decouple}. In such a situation, the statement and proof of Proposition \ref{L2mdecay} hold true with many of the reasonings therein admitting significant simplifications.

\item In the \emph{hybrid curvature regime} when both $a$ and $b$ are non-zero (and $\a\in \R\setminus\{1\}$ fixed) the full force of our methods becomes employed; however, in a quite remarkable twist, and, in analogy to the situations encountered in the case of Polynomial Carleson (\cite{Lie09},\cite{Lie20}) or the Bilinear Hilbert--Carleson operators (\cite{BBLV21}), in order for Proposition \ref{L2mdecay} to hold we have to exclude the quadratic case $\a=2$. For more on the latter, the reader is invited to consult Section \ref{H2discussion}.
\end{itemize}

\end{remark}

$\newline$
\noindent\textbf{Step III: Cancelation via $TT^{*}$-method, phase-level set analysis and time-frequency correlations}\label{CaL}
$\newline$

We start our section with the following commentary: in view of Observations \ref{const}, \ref{Unifb} and Remark \ref{tfrcor} our main focus relies on the genuine hybrid case when both $a\not=0$ and $b\not=0$, and thus, without loss of generality,  we will assume from now on
\begin{equation}\label{abnlone}
		a=1\qquad\textrm{in addition to}\qquad b=1\,.
\end{equation}
As a consequence, in this new setting, relation \eqref{ljm0a} becomes:
\begin{align}
&|\L_{j,m}(f,g,h)|\lesssim
2^{\frac{j}{2}{-}\frac{m}{4}}\,\hspace{-.2cm}\sum_{\substack{\\u\sim 2^{\frac{m}{2}}\\ r\in\Z}} \sum_{\substack{\\v\sim_{{}_{\a}} 2^{\frac{m}{2}{+}j(\a{-}1)}}}
\hspace{-.4cm}\Big|\Big\langle f,  \check{\phi}_{u+v}^{r-2\cdot2^{j+\frac{m}{2}}c_1\,\frac{u^{\frac{1}{\a-1}}}{v^{\frac{1}{\a-1}}}-
	2^{j+\frac{m}{2}}\,c_2\,\frac{u^{\frac{\a}{\a-1}}}{v^{\frac{\a}{\a-1}}}}\Big\rangle\Big|\,\nonumber\\[-2ex]
&\hspace{5cm}\times
|\langle g,  \check{\phi}_{v}^r\rangle|\,
\Big|\Big\langle h,  \check{\phi}_{u+2v}^{r- 2^{j+\frac{m}{2}}\,c_1\,\frac{u^{\frac{1}{\a-1}}}{v^{\frac{1}{\a-1}}}-
	2^{j+\frac{m}{2}}c_2\,\frac{u^{\frac{\a}{\a-1}}}{v^{\frac{\a}{\a-1}}}}\Big\rangle\Big|\,,\label{ljm0}
\end{align}
We now perform the following change of variable\footnote{Throughout this section the constants $c_1,\,c_2$ are allowed to change from line to line.} (in the prescribed order)
$$r\longmapsto r+ 2^{j+\frac{m}{2}}\,c_1\,\frac{u^{\frac{1}{\a-1}}}{v^{\frac{1}{\a-1}}}+
2^{j+\frac{m}{2}}\,c_2\,\frac{u^{\frac{\a}{\a-1}}}{v^{\frac{\a}{\a-1}}}$$
$$\textrm{and}\qquad u+2v\longmapsto v\qquad\:\&\:\qquad u+v\longmapsto u $$
which, after a suitable splitting of the summation domain in $u$ and $v$, recasts the term $|\L_{j,m}(f,g,h)|$ as a superposition of at most $c(\a)$ terms of the form
\begin{align}
2^{\frac{j}{2}-\frac{m}{4}}\sum_{\substack{\\|v|\in \mathfrak{I}^{m,j}}} \sum_{\substack{\\|2u-v|\in \mathfrak{I}^{m,0}\\ r\in\Z}} &
\Big|\Big\langle f,  \check{\phi}_u^{r-2^{j+\frac{m}{2}}c_1\frac{(2u-v)^{\frac{1}{\a-1}}}{(v-u)^{\frac{1}{\a-1}}}}\Big\rangle\Big|\times\nonumber\\[-2ex]
&\Big|\Big\langle g,  \check{\phi}_{v-u}^{r+2^{j+\frac{m}{2}}c_1\frac{(2u-v)^{\frac{1}{\a-1}}}{(v-u)^{\frac{1}{\a-1}}}+
	2^{j+\frac{m}{2}}\,c_2\,\frac{(2u-v)^{\frac{\a}{\a-1}}}{(v-u)^{\frac{\a}{\a-1}}}}\Big\rangle\Big|\,
| \langle h,  \check{\phi}_{v}^{r}\rangle|\,,\nonumber
\end{align}
where here we set $\mathfrak{I}^{m,j}:=[2^{\frac{m}{2}+(\a-1)j},\,2^{\frac{m}{2}+(\a-1)j+1}]$ when $j\geq 0$.

Once at this point we remark
\begin{itemize}
\item the almost disjointness of the set of frequencies given by the set of indices $\big\{(u,v-u)\big\}_{(|2u-v|,|v|)\in \mathfrak{I}^{m,0}\times \mathfrak{I}^{m,j,\ell}}$\, as $\ell\in \{1,\ldots, 2^{(\a-1)j}\}$ where here $\mathfrak{I}^{m,j,\ell}:=\big[2^{\frac{m}{2}+(\a-1)j}+ 2^{\frac{m}{2}}(\ell-1),\,2^{\frac{m}{2}+(\a-1)j}+2^{\frac{m}{2}}\ell\big]$;

\item the almost disjointness of the set of spatial locations given by
\[
\Big\{\big(r- 2^{j{+}\frac{m}{2}}c_1\frac{(2u{-}v)^{\frac{1}{\a{-}1}}}{(v{-}u)^{\frac{1}{\a{-}1}}},\,r+2^{j{+}\frac{m}{2}}c_1\frac{(2u{-}v)^{\frac{1}{\a{-}1}}}{(v{-}u)^{\frac{1}{\a{-}1}}} +
	2^{j{+}\frac{m}{2}}c_2\frac{(2u{-}v)^{\frac{\a}{\a{-}1}}}{(v{-}u)^{\frac{\a}{\a{-}1}}}\big)\Big\}_{\!r\in R(r_0)}\,,
\]
\end{itemize}
as $r_0$ spans $\Z$, where $R(r_0):=[ 2^{\frac{m}{2}}r_0,\,2^{\frac{m}{2}}(r_0+1)]$.

Based on the above remark, setting $R=R(0)$ and
\begin{equation*}\label{IJ}
\mathfrak{J}:=\mathfrak{I}^{m,j,1}\!\quad\!\textrm{and}\quad
\tilde{\mathfrak{J}}:=\big[\tfrac{1}{2}c(\mathfrak{J})-2^{\frac{m}{2}+2}, \tfrac{1}{2}c(\mathfrak{J})-2^{\frac{m}{2}+1}\big]\cup \big[\tfrac{1}{2}c(\mathfrak{J})+2^{\frac{m}{2}+1}, \tfrac{1}{2}c(\mathfrak{J})+2^{\frac{m}{2}+2}\big]\,,
\end{equation*}
we deduce that it is enough to estimate the term
\begin{align}
\L_{j,m}^{R,\mathfrak{J}}(f,g,h)&:=2^{\frac{j}{2}{-}\frac{m}{4}}\sum_{\substack{u\in \tilde{\mathfrak{J}}\\ v\in \mathfrak{J}\\r\in R}}
\Big|\Big\langle f,  \check{\phi}_u^{r-2^{j+\frac{m}{2}}c_1\frac{(2u-v)^{\frac{1}{\a-1}}}{(v-u)^{\frac{1}{\a-1}}}}\Big\rangle\Big|\times\label{ljm1}\\[-1ex]
&\hspace{2cm}\times
\Big|\Big\langle g,  \check{\phi}_{v-u}^{r+2^{j+\frac{m}{2}}c_1\frac{(2u-v)^{\frac{1}{\a-1}}}{(v-u)^{\frac{1}{\a-1}}}+
	2^{j+\frac{m}{2}}\,c_2\,\frac{(2u-v)^{\frac{\a}{\a-1}}}{(v-u)^{\frac{\a}{\a-1}}}}\Big\rangle\Big|\,
| \langle h,  \check{\phi}_{v}^{r}\rangle|\,.\nonumber
\end{align}
The following single scale $m-$decaying estimates constitutes the key ingredient for proving Theorem \ref{thm:mS}:
$\newline$

\noindent\textbf{Main Proposition (+)} \emph{Fix $j, m\in\N$ with $j\geq C_0$ and let $R$, $\mathfrak{J}$ and $\tilde{\mathfrak{J}}$ be intervals as above and
\begin{equation}\label{DeltaIK}
\Delta_{j,\frac{m}{2}}^{R,\mathfrak{J}}{f}:=\left(\:\sum_{\substack{u\in \mathfrak{J}\\[.1ex] r\in R}}\:|\langle f,\check{\phi}_u^r\rangle|^2\:\right)^{\frac12}\,.	
\end{equation}
 Then, for $\L_{j,m}^{R,\mathfrak{J}}$ defined by \eqref{ljm1}, we have that there exist
$\d_0\in(0,\frac{1}{4})$ and $c=c(c_1,c_2)>0$ absolute constants such that the following holds uniformly in $j,m\in\N$:
\begin{equation}\label{ljmkL}
	\L_{j,m}^{R,\mathfrak{J}}(f,g,h)\lesssim_{\a} 2^{-\d_0 m}\, 2^{\frac{j}{2}}
	\Big(\Delta_{j,\frac{m}{2}}^{c R, 3\tilde{\mathfrak{J}}}f\Big)
	\Big(\Delta_{j,\frac{m}{2}}^{c R,3\tilde{\mathfrak{J}}} g\Big) \Big(\Delta_{j,\frac{m}{2}}^{R,\mathfrak{J}}h \Big)\:.	
\end{equation}
}

The proof of the above proposition will be performed in several steps.

$\newline$
\noindent\textbf{Step III.1. Frequency fiber foliation: a sparse--uniform dichotomy}\label{3p1}
$\newline$

Fix $\d\in (0,\frac{1}{4})$ to be chosen later and define the set of \emph{$(\d,f)$-sparse frequency fibers} associated to the spatial location $R$ and the frequency location $\mathfrak{J}$ as
\begin{equation}\label{corel}
	\boxed{\mathcal{S}_{\d}^{R,\mathfrak{J}}(f):=\big\{u\in \mathfrak{J}\,|\,\max_{r\in R}\,  |\langle f,  \check{\phi}_{u}^{r}\rangle|\geq 2^{-\d m}\,\Delta_{j,\frac{m}{2}}^{R, \mathfrak{J}}(f)\big\}\,.}
\end{equation}

Notice that from Chebyshev's inequality one has
\beq\label{corelcard}
\#\mathcal{S}_{\d}^{c R,\tilde{\mathfrak{J}}}(f)\leq 2^{2 \d m}\:.
\eeq

We now split
\beq\label{ljm2}
\L_{j,m}^{R,\mathfrak{J}}(f,g,h)= \L_{j,m,\textrm{sparse}}^{R,\mathfrak{J}}(f,g,h)\,+\,\L_{j,m,\textrm{unif}}^{R,\mathfrak{J}}(f,g,h)\,,
\eeq
where
\begin{itemize}
	\item the \emph{sparse} (heavy) component is defined as
\begin{align}
	\L_{j,m,\textrm{sparse}}^{R,\mathfrak{J}}(f,g,h)&:=2^{\frac{j}{2}-\frac{m}{4}}\,\sum_{\substack{u\in \mathcal{S}_{\d}^{c R,\tilde{\mathfrak{J}}}(f)\\[.1ex] v\in \mathfrak{J}\,,\: r\in R}}\Big(\ldots\Big)  \,,\label{ljm3}
	\intertext{\item the \emph{uniform} (light) component is given by}
	\L_{j,m,\textrm{unif}}^{R,\mathfrak{J}}(f,g,h)&:=2^{\frac{j}{2}-\frac{m}{4}}\,\sum_{\substack{u\in \mathcal{U}_{\d}^{cR,\tilde{\mathfrak{J}}}(f)\\[.1ex] v\in \mathfrak{J}\,,\: r\in R}}\Big(\ldots\Big)\,,\label{ljm4}
\end{align}
\end{itemize}
where here we set $\mathcal{U}_{\d}^{cR,\tilde{\mathfrak{J}}}(f):=\tilde{\mathfrak{J}}\:\setminus\: \mathcal{S}_{\d}^{c R,\tilde{\mathfrak{J}}}(f)$.

We start with the treatment of the easier term:

$\newline$
\noindent\textbf{Step III.2. The sparse component}\label{Th}
$\newline$

Applying first a Cauchy--Schwarz in $u$ followed by a Cauchy--Schwarz in $v,r$ and then followed by the change of variable
$r\longmapsto r + 2^{j+\frac{m}{2}}\,c_1\,\frac{(2u-v)^{\frac{1}{\a-1}}}{(u-v)^{\frac{1}{\a-1}}}$
we get
\beq\label{corrterm}
\L_{j,m,\textrm{sparse}}^{R,\mathfrak{J}}(f,g,h)\leq 2^{\frac{j}{2}-\frac{m}{4}}\,\left( \#\mathcal{S}_{\d}^{c R,\tilde{\mathfrak{J}}}(f)\right)^{\frac{1}{2}}\bigg(\sum_{{v\in \mathfrak{J}}\atop{r\in R}}  |\langle h,  \check{\phi}_{v}^{r}\rangle|^2\bigg)^{\frac{1}{2}}\times
\eeq
\[
\times\left(\sum_{\substack{u\in \tilde{\mathfrak{J}}\\r\in c R}}
|\langle f, \check{\phi}_{u}^{r}\rangle|^2\,
\left(\sum_{v\in \mathfrak{J}} \Big|\Big\langle g,  \check{\phi}_{v-u}^{r+2^{j+\frac{m}{2}+1}\,c_1\,\frac{(2u-v)^{\frac{1}{\a-1}}}{(v-u)^{\frac{1}{\a-1}}}+
2^{j+\frac{m}{2}}\,c_2\,\frac{(2u-v)^{\frac{\a}{\a-1}}}{(v-u)^{\frac{\a}{\a-1}}}}\Big\rangle\Big|^2\,\right)\right)^{\frac{1}{2}}.
\]

Using now \eqref{corelcard} we conclude that

\begin{equation}\label{ljmkLclust}
\L_{j,m,\textrm{sparse}}^{R,\mathfrak{J}}(f,g,h)\lesssim_{\a} 2^{-(\frac{1}{4}-\d) m}\, 2^{\frac{j}{2}}\,\Delta_{j,\frac{m}{2}}^{c R, 3\tilde{\mathfrak{J}}}\!(f)\: \Delta_{j,\frac{m}{2}}^{c R,3\tilde{\mathfrak{J}}}\!(g)\:\Delta_{j,\frac{m}{2}}^{R,\mathfrak{J}}\!(h)\:.	
\end{equation}

$\newline$
\noindent\textbf{Step III.3. The uniform component}\label{Tu}
$\newline$

The treatment of the uniform term relies fundamentally on the properties of the time-frequency correlation set \eqref{tfreqcor1} which are derived via a suitable phase level set analysis. The latter becomes the quintessential ingredient in our approach meant to exploit the non-zero curvature features\footnote{See also Remark \ref{tfrcor}.} of $\g$.

To properly set the context, we first consider an arbitrary (measurable) function
\beq\label{vmeas}
v:\,R\,\longrightarrow\,\mathfrak{J}
\eeq
whose meaning will become transparent momentarily.

$\newline$
\noindent\textsf{Step III.3.1. Time-frequency correlations}\label{TFC}
$\newline$

For $u\in \mathcal{U}_{\d}^{cR,\tilde{\mathfrak{J}}}\!(f)$  and $\kappa\in c R$ we define the $(u,\kappa)$-sets
\beq\label{tfreqcor}
\I_{u,\kappa}^{R,\mathfrak{J}}:=\left\{r\in R\,\Bigg|\,\Big|r- c_1\,2^{j+\frac{m}{2}}\,\frac{(2u-v(r))^{\frac{1}{\a-1}}}{(v(r)-u)^{\frac{1}{\a-1}}}-\kappa\Big|\leq 10\right\}\:.
\eeq
Let now $\ep\in (0,\frac{\d}{2})$ small (to be chosen later) and define the \emph{time-frequency correlation set} (depending on $v(\cdot)$) as
\beq\label{tfreqcor1}
\boxed{\mathcal{C}_{\ep,\d,+}^{R,\mathfrak{J}}:=\big\{u\in \mathcal{U}_{\d}^{cR,\tilde{\mathfrak{J}}}\!(f)\,|\,\exists\:\kappa=\kappa(u)\in c R\:\:\:\textrm{s.t.}\:\:\:\#\I_{u,\kappa}^{R,\mathfrak{J}}\gtrsim 2^{(\frac{1}{2}-\ep)\,m}\big\}\:.}
\eeq

Notice in particular that the above may be interpreted as the existence of a \emph{time-frequency correlation function}  $\kappa$
with
$$\kappa:\,\mathcal{U}_{\d}^{cR,\tilde{\mathfrak{J}}}\!(f)\,\mapsto\,c R\:\:(\textrm{measurable})$$
such that the size of the set $\I_{u,\kappa}^{R,\mathfrak{J}}$ is (suitably) large as $u$ covers $\mathcal{C}_{\ep,\d,+}^{R,\mathfrak{J}}$\,.

$\newline$
\noindent\textsf{Step III.3.2. Time-frequency correlation set analysis: exploiting the curvature}\label{TFCsize}
$\newline$

The crux of our argument relies on the following claim: \emph{uniformly in the choice of $v(\cdot)$ in \eqref{vmeas}, the size of the time-frequency correlation set is suitably small}, i.e.
\beq\label{keyrel}
\boxed{\exists\:\mu=\mu(\ep)>\frac{1}{2}-2\d\:\:\quad\textrm{s.t.}\quad\forall\:v\:\quad\#\mathcal{C}_{\ep,\d,+}^{R,\mathfrak{J}}\lesssim 2^{(\frac{1}{2}-\mu)\,m}\:.}
\eeq

\begin{proof}
In what follows we will show that there exists\footnote{Throughout this section the parameter $C$ is allowed to change from line to line.} $C=C(\a)>1$ such that
\beq\label{keyrel1}
\#\mathcal{C}_{\ep,\d,+}^{R,\mathfrak{J}}\leq C\,2^{4\ep\,m}\:,
\eeq
and hence one may choose in \eqref{keyrel} $\mu=\frac{1}{2}-4\ep$.

Assume by contradiction that \eqref{keyrel1} fails. If this were indeed the case, one could verify the hypothesis of Lemma 50 in \cite{Lie19} for $I_l \longrightarrow \tilde{\I}_{u_l,\kappa_l}^{R,\mathfrak{J}}(v)$ (with the latter playing the role of the $2^{-\frac{m}{2}}$ normalization of the set $\I_{u_l,\kappa_l}^{R,\mathfrak{J}}(v)$), $M=2^{\ep m}$,\, $N=\#\mathcal{C}_{\ep,\d,+}^{R,\mathfrak{J}}\geq C\,2^{4\ep\,m}$, and $n=2$ in order to deduce that there exists $u_1,\,u_2\in \mathcal{U}_{\d}^{cR,\tilde{\mathfrak{J}}}\!(f)$ with $|u_1-u_2|\geq C\,2^{2\ep m}$ such that
\beq\label{keyrel2}
\#\Big(\I_{u_1,\kappa_1}^{R,\mathfrak{J}}(v)\cap \I_{u_2,\kappa_2}^{R,\mathfrak{J}}(v)\Big)\geq 2^{(\frac{1}{2}-2\ep)\,m-1}\:.
\eeq
From this we further deduce that
\beq\label{KEY0}
\textit{there exist}\:\:|u_1-u_2|\geq C\,2^{2\ep m}\:\:\textit{and}\:\:|r_1-r_2|\geq 2^{(\frac{1}{2}-2\ep)\,m-1}\:\:\textit{such that}
\eeq
\beq\label{keyrel3}
\left|r_s- c_1\,2^{j+\frac{m}{2}}\,\frac{(2u_l-v(r_s))^{\frac{1}{\a-1}}}{(v(r_s)-u_l)^{\frac{1}{\a-1}}}-\kappa_l\right|\leq 10\qquad \forall\:s,l\in\{1,2\}\:.
\eeq

Subtracting for each fixed $l$ the corresponding two relations in $s$ we obtain:
\beq\label{keyrel4}
\left|r_2-r_1+ c_1\,2^{j+\frac{m}{2}}\,\sdfrac{(2u_l-v(r_2))^{\frac{1}{\a-1}}}{(v(r_2)-u_l)^{\frac{1}{\a-1}}}-
c_1\,2^{j+\frac{m}{2}}\,\sdfrac{(2u_l-v(r_1))^{\frac{1}{\a-1}}}{(v(r_1)-u_l)^{\frac{1}{\a-1}}}\right|\leq 20\:,\:\:\:\:l\in\{1,2\}\:.
\eeq
Define now the function $\zeta:\,\tilde{\mathfrak{J}}\,\longrightarrow\,\R$ by
\beq\label{keyrel5}
\zeta(u):=2^{j+\frac{m}{2}}\,\frac{(2u-v(r_2))^{\frac{1}{\a-1}}}{(v(r_2)-u)^{\frac{1}{\a-1}}}-
2^{j+\frac{m}{2}}\,\frac{(2u-v(r_1))^{\frac{1}{\a-1}}}{(v(r_1)-u)^{\frac{1}{\a-1}}}\:.
\eeq
Notice that \eqref{keyrel4} can be rewritten now as
\beq\label{keyrel6}
\left|r_2-r_1+ c_1\,\zeta(u_l)\right|\leq 20\:,\:\:\:\:l\in\{1,2\}\:,
\eeq
while if we subtract the two relations in \eqref{keyrel6} we deduce
\beq\label{keyrel7}
\left|\zeta(u_1)-\zeta(u_2)\right|\lesssim 1\:.
\eeq
At this point we remark that
\beq\label{keyrel8}
\zeta'(u):=\sdfrac{2^{j+\frac{m}{2}}}{\a-1}\,\sdfrac{v(r_2)}{(v(r_2)-u)^2}\,\sdfrac{(2u-v(r_2))^{\frac{2-\a}{\a-1}}}{(v(r_2)-u)^{\frac{2-\a}{\a-1}}}-
\sdfrac{2^{j+\frac{m}{2}}}{\a-1}\,\sdfrac{v(r_1)}{(v(r_1)-u)^2}\,\sdfrac{(2u-v(r_1))^{\frac{2-\a}{\a-1}}}{(v(r_1)-u)^{\frac{2-\a}{\a-1}}}\:.
\eeq
Define further the function
\beq\label{f1}
\vartheta(v):=\frac{v}{(v-u)^2}\,\frac{(2u-v)^{\frac{2-\a}{\a-1}}}{(v-u)^{\frac{2-\a}{\a-1}}}\,,
\eeq
and notice that
\begin{align}
\vartheta'(v)&=-\frac{v+u}{(v-u)^3}\,\Big(\frac{2u-v}{v-u}\Big)^{\frac{2-\a}{\a-1}}-\,\frac{2-\a}{\a-1}\, \frac{u\,v}{(v-u)^4}\, \Big(\frac{2u-v}{v-u}\Big)^{\frac{2-\a}{\a-1}-1}\,.\label{f3}
\end{align}

Now based on the range for our parameters $u,v$ we have
\beq\label{f4}
|v-u|,\,u,\,v\approx 2^{\frac{m}{2}+(\a-1)j}\:\:\:\textrm{and}\:\:\:|2u-v|\approx 2^{\frac{m}{2}}\,.
\eeq
Since we are in the setting $\a>1$ the last term in \eqref{f3} is the dominant one and
\beq\label{f5}
|\vartheta'(v)|\approx \frac{|2-\a|}{|\a-1|}\,\frac{1}{2^{m+j}}\,.
\eeq
Applying twice the mean value theorem and using \eqref{keyrel8}--\eqref{f5} we have that \eqref{keyrel7} implies
\beq\label{key10}
|u_1-u_2|\,|v(r_1)-v(r_2)|\lesssim_{\a} 2^{\frac{m}{2}}\,.
\eeq
Similar reasonings show that
\beq\label{keyrel10}
|\zeta(u)|\approx 2^{j+\frac{m}{2}}\,|v(r_2)-v(r_1)|\,\frac{|u|}{|v_0-u|^2}\frac{|2u-v_0|^{\frac{1}{\a-1}-1}}{|v_0-u|^{\frac{1}{\a-1}-1}}\:\:\:\textrm{for some}\:\:v_0\in[v(r_1),v(r_2)]\:.
\eeq
Using \eqref{f4}, one deduces
\beq\label{keyrel11}
|\zeta(u)|\approx |v(r_2)-v(r_1)|\:.
\eeq
Inserting now \eqref{keyrel11} in \eqref{keyrel6} and using \eqref{KEY0}, we have that
\beq\label{keyrel12}
|v(r_2)-v(r_1)|\approx|r_2-r_1|\:,
\eeq
and thus from \eqref{keyrel12} and \eqref{key10}, we conclude
\beq\label{KEY}
|u_1-u_2|\,|r_1-r_2|\lesssim_{\a} 2^{\frac{m}{2}}\,.
\eeq
However \eqref{KEY} contradicts the choice in \eqref{KEY0} (for a suitable choice of $C$) hence  \eqref{keyrel1} must hold.
\end{proof}

$\newline$
\noindent\textsf{Step III.3.3. Completing the puzzle}\label{Tuniffinal}
$\newline$

In what follows we list the key remaining arguments for completing our proof:

\begin{itemize}
\item Cauchy--Schwarz in the $u$ variable:
\begin{align}
	&\L_{j,m,\text{unif}}^{R,\mathfrak{J}}(f,g,h)\le 2^{\frac{j}{2}-\frac{m}{4}}\,\sum_{{v\in \mathfrak{J}}\atop{r\in R}} \left(\sum_{u\in \mathcal{U}_{\d}^{cR,\tilde{\mathfrak{J}}}(f)}  \Big|\Big\langle f,  \check{\phi}_{u}^{r- c_1\,2^{j+\frac{m}{2}}\,\frac{(2u-v)^{\frac{1}{\a-1}}}{(v-u)^{\frac{1}{\a-1}}}}\Big\rangle\Big|^2\right)^{\frac{1}{2}}\times\nonumber\\
	&\hspace{1cm}\times\left(\sum_{u\in \mathcal{U}_{\d}^{cR,\tilde{\mathfrak{J}}}(f)} \Big|\Big\langle g,  \check{\phi}_{v-u}^{r+2^{j+\frac{m}{2}}\,c_1\,\frac{(2u-v)^{\frac{1}{\a-1}}}{(v-u)^{\frac{1}{\a-1}}}+
		2^{j+\frac{m}{2}}\,c_2\,\frac{(2u-v)^{\frac{\a}{\a-1}}}{(v-u)^{\frac{\a}{\a-1}}}}\Big\rangle\Big|^2\right)^{\frac{1}{2}}\, |\langle h,  \check{\phi}_{v}^{r}\rangle|\,.\label{ljmunif1}
\end{align}

 \item $\ell^{\infty}\times \ell^2\times \ell^2$ H\"older inequality in the $v$ parameter:
\begin{align}
&\L_{j,m,\text{unif}}^{R,\mathfrak{J}}(f,g,h)\leq 2^{\frac{j}{2}-\frac{m}{4}}\,\sum_{r\in R} \left(\sup_{v\in \mathfrak{J}}\left(\sum_{u\in \mathcal{U}_{\d}^{cR,\tilde{\mathfrak{J}}}(f)}\hspace{-.3cm}  |\langle f,  \check{\phi}_{u}^{r- c_1\,2^{j+\frac{m}{2}}\,\frac{(2u-v)^{\frac{1}{\a-1}}}{(v-u)^{\frac{1}{\a-1}}}}\rangle|^2\right)^{\frac{1}{2}}\right)\times\nonumber\\
&\times\left(\sum_{v\in \mathfrak{J}}\sum_{u\in \tilde{\mathfrak{J}}} |\langle g,  \check{\phi}_{v-u}^{r+2^{j+\frac{m}{2}}\,c_1\,\frac{(2u-v)^{\frac{1}{\a-1}}}{(v-u)^{\frac{1}{\a-1}}}+
	2^{j+\frac{m}{2}}\,c_2\,\frac{(2u-v)^{\frac{\a}{\a-1}}}{(v-u)^{\frac{\a}{\a-1}}}}\rangle|^2\right)^{\frac{1}{2}}\, \left(\sum_{v\in \mathfrak{J}}|\langle h,  \check{\phi}_{v}^{r}\rangle|^2\right)^{\frac{1}{2}}\,.\label{ljmunif2}
\end{align}

\item for all $r\in R$ one has that
\beq\label{ljmunifg}
\sum_{v\in \mathfrak{J}}\sum_{u\in \tilde{\mathfrak{J}}} |\langle g,  \check{\phi}_{v-u}^{r+2^{j+\frac{m}{2}}\,c_1\,\frac{(2u-v)^{\frac{1}{\a-1}}}{(v-u)^{\frac{1}{\a-1}}}+
2^{j+\frac{m}{2}}\,c_2\,\frac{(2u-v)^{\frac{\a}{\a-1}}}{(v-u)^{\frac{\a}{\a-1}}}}\rangle|^2\lesssim \left(\Delta_{j,\frac{m}{2}}^{c R, 3\tilde{\mathfrak{J}}}(g)\right)^2\,.
\eeq
\end{itemize}
Indeed, in order to verify this claim one first applies the change of variable  $v\longmapsto v-u$ and then  taking the function
$$\vartheta_v(u):=c_1\,2^{j+\frac{m}{2}}\,\Big(\frac{u-v}{v}\Big)^{\frac{1}{\a-1}}\,+
\,c_2\,2^{j+\frac{m}{2}}\,\Big(\frac{u-v}{v}\Big)^{\frac{\a}{\a-1}}$$
under the assumption $v\sim 2^{\frac{m}{2}+(\a-1)\,j}$ and $u-v\sim 2^{\frac{m}{2}}$ one verifies that
\beq\label{growth}
\left|\frac{\partial}{\partial u}\,\vartheta_v(u)\right|\approx 1\,.
\eeq
It is worth noticing here that the validity of \eqref{growth} is a consequence of the condition $|j|\ge C_0$ for some $C_0$ large enough relative to the values of $c_1,\,c_2$ and $\a$.

\begin{itemize}
\item  Cauchy--Schwarz in the $r$ variable:
 \end{itemize}
\begin{equation}\label{ljmunif4}
\L_{j,m,\text{unif}}^{R,\mathfrak{J}}(f,g,h)\leq 2^{\frac{j}{2}{-}\frac{m}{4}}
\Delta_{j,\frac{m}{2}}^{c R,3\tilde{\mathfrak{J}}}\:\!\!(g)\, \Delta_{j,\frac{m}{2}}^{R,\mathfrak{J}} (h) \hspace{-.1cm}
\left(\!\!\sum_{\substack{u\in \mathcal{U}_{\d}^{cR,\tilde{\mathfrak{J}}}\!(f)\\r\in R }} \hspace{-.5cm}
|\langle f,  \check{\phi}_{u}^{\hspace{-.1cm} r{-} c_1 2^{j{+}\frac{m}{2}}\frac{(2u{-}v(r))^{\frac{1}{\a{-}1}}}{(v(r){-}u)^{\frac{1}{\a{-}1}}}
}\rangle|^2\! \right)^{\hspace{-.1cm}\frac{1}{2}},
\end{equation}
where here $v(\cdot)$ is a/the measurable function that attains the supremum in \eqref{ljmunif2}.

\begin{itemize}
\item We split the analysis depending on the relative position of the frequency location $u$ to the set $\mathcal{C}_{\ep,\d,+}^{R,\mathfrak{J}}$:
\end{itemize}

Denote by
\begin{align}
	S^{+}&:=\sum_{r\in R}\,\sum_{u\in \mathcal{U}_{\d}^{cR,\tilde{\mathfrak{J}}}\!(f)}  |\langle f,  \check{\phi}_{u}^{r- c_1\,2^{j+\frac{m}{2}}\, \frac{(2u-v(r))^{\frac{1}{\a-1}}}{(v(r)-u)^{\frac{1}{\a-1}}}}\rangle|^2\nonumber
	\intertext{and split}
	S^{+}&=S_1^{+}\,+\,S_2^{+}:=\sum_{r\in R}\,\sum_{u\in \mathcal{C}_{\ep,\d,+}^{R,\mathfrak{J}}}\left(\ldots\right)\,+\,\sum_{r\in R}\,\sum_{u\in \big(\mathcal{U}_{\d}^{cR,\tilde{\mathfrak{J}}}\!(f)\big)\setminus \mathcal{C}_{\ep,\d,+}^{R,\mathfrak{J}}}\left(\ldots\right)\:.\label{Ssplit}
	\intertext{Now}
	\qquad\qquad\qquad S_1^{+}&\lesssim 2^{\frac{m}{2}}\,(\#\mathcal{C}_{\ep,\d,+}^{R,\mathfrak{J}})\,2^{-2\d m}\,\left(\Delta_{j,\frac{m}{2}}^{c R, 3\tilde{\mathfrak{J}}}(f)\right)^2\:,\label{S1}
\end{align}
and
\beq\label{S2}
S_2^{+}\lesssim 2^{(\frac{1}{2}-\ep)\,m} \,\left(\Delta_{j,\frac{m}{2}}^{c R, 3\tilde{\mathfrak{J}}}(f)\right)^2\:.
\eeq

\begin{itemize}
\item  A closing variational argument:
\end{itemize}

Putting together \eqref{keyrel} and \eqref{ljmunif4}--\eqref{S2} we conclude
\beq\label{unifconcl}
\L_{j,m,\text{unif}}^{R,\mathfrak{J}}(f,g,h)\lesssim (2^{(\frac{1}{4}-\frac{\mu}{2}-\d)m}\,+\, 2^{-\frac{\ep}{2} m})\, 2^{\frac{j}{2}}\,\Delta_{j,\frac{m}{2}}^{c R, 3\tilde{\mathfrak{J}}}\!(f)\: \Delta_{j,\frac{m}{2}}^{c R,3\tilde{\mathfrak{J}}}\:\!(g)\: \Delta_{j,\frac{m}{2}}^{R,\mathfrak{J}} (h)\:,
\eeq
which for appropriate values of $\ep,\,\mu,\d$ gives us the desired exponential decay in $m$.

Indeed, from \eqref{ljm2}, \eqref{ljmkLclust} and \eqref{unifconcl} and for the choice $\mu=\frac{1}{2}-4\ep$, we deduce
\beq\label{CRUX}
|\L_{j,m}^{R,\mathfrak{J}}(f,g,h)|\lesssim  (2^{-m(\frac{1}{4}-\d)}\,+\,2^{(2\ep-\d)m}\,+\, 2^{-\frac{\ep}{2} m})\,2^{\frac{j}{2}}\,\Delta_{j,\frac{m}{2}}^{c R, 3\tilde{\mathfrak{J}}}\!(f)\: \Delta_{j,\frac{m}{2}}^{c R,3\tilde{\mathfrak{J}}}\:\!(g)\:\Delta_{j,\frac{m}{2}}^{R,\mathfrak{J}} (h)\,,\quad
\eeq
which, after a variational argument, for the choices $\d=\frac{5}{24}$ and $\ep=\frac{1}{12}$, implies
\beq\label{CRUX1}
|\L_{j,m}^{R,\mathfrak{J}}(f,g,h)|\lesssim_{\a}  2^{-\frac{m}{24}}\,2^{\frac{j}{2}}\,\Delta_{j,\frac{m}{2}}^{c R, 3\tilde{\mathfrak{J}}}(f)\: \Delta_{j,\frac{m}{2}}^{c R,3\tilde{\mathfrak{J}}}\:(g)\:\Delta_{j,\frac{m}{2}}^{R,\mathfrak{J}} (h)\,.
\eeq

\begin{remark}\label{l2l2intol1} It is now easy to see that our Main Proposition (+) immediately implies the analogue of \eqref{LSm} in Proposition \ref{L2mdecay} with $\L^{S}_{m,+}$ instead of $\L^S_{m}$. Indeed, based on a Cauchy-Schwarz argument, we have that
\beq\label{hsup}
\Delta_{j,\frac{m}{2}}^{R,\mathfrak{J}} (h)\lesssim 2^{-\frac{j}{2}}\, \|h\|_{\infty}\:,
\eeq
and hence from \eqref{CRUX1} one deduces that
\beq\label{CRUX10}
|\L_{j,m}^{R,\mathfrak{J}}(f,g,h)|\lesssim_{\a}  2^{-\frac{m}{24}}\,\Delta_{j,\frac{m}{2}}^{c R, 3\tilde{\mathfrak{J}}}\!(f)\: \Delta_{j,\frac{m}{2}}^{c R,3\tilde{\mathfrak{J}}}\:\!(g)\:\|h\|_{\infty}\,.
\eeq
Once at this point we use Cauchy-Schwarz and Parseval in order to get
\begin{align*}
	|\L_{j,m}(f,g,h)|&\leq \sum_{\substack{|R|=2^{\frac{m}{2}}\\[.1ex]\mathfrak{J}\subseteq \mathfrak{I}^{m,j}}}|\L_{j,m}^{R,\mathfrak{J}}(f,g,h)|\\[-2ex]
	&\lesssim_{\a}  2^{-\frac{m}{24}}
	\left( \sum_{{|R|=2^{\frac{m}{2}}}\atop{\mathfrak{J}\subseteq \mathfrak{I}^{m,j}}}|\Delta_{j,\frac{m}{2}}^{c R, 3\tilde{\mathfrak{J}}}(f)|^2\right)^{\frac{1}{2}} \left(\sum_{{|R|=2^{\frac{m}{2}}}\atop{\mathfrak{J}\subseteq \mathfrak{I}^{m,j}}}|\Delta_{j,\frac{m}{2}}^{c R,3\tilde{\mathfrak{J}}}\:(g)|^2\right)^{\frac{1}{2}}\:\|h\|_{\infty}\\[1ex]
	&\lesssim 2^{-\frac{m}{24}}\,\|\hat{f}\|_{L^2[2^{\frac{m}{2}+\a j-10}, 2^{\frac{m}{2}+\a j+10}]}\,\|\hat{g}\|_{L^2[2^{\frac{m}{2}+\a j-10}, 2^{\frac{m}{2}+\a j+10}]}\|h\|_{\infty}\,.
\end{align*}

Finally, by another Cauchy-Schwarz application, we get
\begin{equation}\label{Lcs}
	|\L^{S}_{m,+}(f,g,h)|\leq \sum_{j\ge  C_0}|\L_{j,m}(f,g,h)|\lesssim_{\a} 2^{-\frac{m}{24}}\,\|f\|_2\,\|g\|_2\,\|h\|_{\infty}\,.
\end{equation}
\end{remark}


\subsubsection{The treatment of $\L^{S}_{m,-}$}\label{LGCminus}

We pass now to the proof of Proposition \ref{L2mdecay} with $\L^S_{m}$ replaced in \eqref{LSm} by $\L^{S}_{m,-}$. In this situation, we fix $m\in\N$ and $j\in\Z_{-}$ with $j\leq - C_0$, and, for a fixed $\L^{S}_{j,m}$ component, we implement the (Rank I) LGC-methodology, choosing, as before, to perform our linearization process on the frequency side.

$\newline$
\noindent\textbf{Step I: Phase linearization}
$\newline$

In this situation, recalling \eqref{mainphase21}--\eqref{mainphase23}, we have that\footnote{Recall that throughout this section we assume $\a>1$ and $j<0$.}
\begin{equation}\label{mainphasebdneg}
	|\partial_{\xi}^2\psi(\xi,\eta)|\approx 2^{-m-2j},\quad |\partial_{\xi\eta}^2\psi(\xi,\eta)|\approx 2^{-m-(1+\a)j}\quad \text{and} \quad |\partial_{\eta}^2\psi(\xi,\eta)|\approx 2^{-m-2\a j}\,.
\end{equation}
We deduce thus that in the current context in order for the second order term in the Taylor expansion of the phase to be $O(1)$ we must impose
\beq\label{freqstepneg}
|\triangle \xi|,\,|\triangle \eta|\lesssim 2^{\a j+\frac{m}{2}}\,.
\eeq

In accordance to \eqref{freqstepneg}, we discretize our frequency space as follows:
\begin{align}
	\phi\Big(\frac{\xi-\eta}{2^{j+m}}\Big)&\approx\sum_{u\sim 2^{\frac{m}{2}-(\a-1)j}} \phi\Big(\frac{\xi-\eta}{2^{\a j+\frac{m}{2}}}-u\Big)\,,\label{discr1n}
	\intertext{and}
	\phi\Big(\frac{\a\, \eta}{2^{\a j+m}}\Big)&\approx\sum_{v\sim_{\a} 2^{\frac{m}{2}}} \phi\Big(\frac{\eta}{2^{\a j+\frac{m}{2}}}-v\Big)\,.\label{discr2n}
\end{align}
Consequently, from \eqref{mjm}, \eqref{discr1n} and \eqref{discr2n} we deduce that
\beq\label{mjmndiscn}
m_{j,m}(\xi,\eta) = \sum_{u\sim 2^{\frac{m}{2}-(\a-1)j}}\: \sum_{v\sim_{\a} 2^{\frac{m}{2}}} m_{j,m,u,v}(\xi,\eta)
\eeq
with
\beq\label{mjmndisc1n}
m_{j,m,u,v}(\xi,\eta)\approx 2^{-\frac{m}{2}} \,e^{i\,c_{\a}\,\psi(\xi,\eta)} \,\phi\Big(\frac{\xi}{2^{\a j+\frac{m}{2}}}-(u+v)\Big)\,\phi\Big(\frac{\eta}{2^{\a j+\frac{m}{2}}}-v\Big)\,.
\eeq

\if
At this point we write
\beq\label{mainphases-jneg}
\psi(\xi,\eta)= c_{\a,a}\,2^{\a j+\frac{m}{2}}\,\frac{(\frac{\xi-a\eta}{2^{\a j+\frac{m}{2}}})^{\frac{\a}{\a-1}}}{(\frac{a \eta}{2^{\a j+\frac{m}{2}}})^{\frac{1}{\a-1}}}=: c_{\a,a}\,2^{\a j+\frac{m}{2}}\,\underline{\psi}(\frac{\xi}{2^{\a j+\frac{m}{2}}}, \frac{a\eta}{2^{\a j+\frac{m}{2}}})\:.
\eeq

Now on the support of the multiplier $m_{j,m,u,v}$ from the choice of our discretization -- see \eqref{mainphasebd} and \eqref{freqstep} -- we have via a Taylor series argument that

\begin{align*}
\psi(\xi,\eta)&=\psi(2^{\a j+\frac{m}{2}}(u+v),\,2^{\a j+\frac{m}{2}}v)\,+(\xi-2^{\a j+\frac{m}{2}}(u+v))\,\partial_{\xi}\psi(2^{\a j+\frac{m}{2}}(u+v),\,2^{\a j+\frac{m}{2}} v)\,\\
&+\,(\eta-2^{\a j+\frac{m}{2}}v)\,\partial_{\eta}\psi(2^{\a j+\frac{m}{2}}(u+v),\,2^{\a j+\frac{m}{2}} v)\,+\,O(1)\,,
\end{align*}
and hence that
\begin{align*}
\psi(\xi,\eta)&= c_{\a,a}\,2^{\a j+\frac{m}{2}}\,\frac{u^{\frac{\a}{\a-1}}}{v^{\frac{1}{\a-1}}} \,+(\xi-2^{\a j+\frac{m}{2}}(u+v))\,c_{\a,a}\,\frac{\a}{\a-1}\,\frac{u^{\frac{1}{\a-1}}}{v^{\frac{1}{\a-1}}}\\
&-\,(\eta-2^{\a j+\frac{m}{2}}v)\,c_{\a,a}\,(\frac{a\a}{\a-1}\,\frac{u^{\frac{1}{\a-1}}}{v^{\frac{1}{\a-1}}}+
\frac{a}{\a-1}\,\frac{u^{\frac{\a}{\a-1}}}{v^{\frac{\a}{\a-1}}})\,+\,O(1)\,.
\end{align*}
\fi

From this, proceeding as for the $+$ component, we deduce that
\beq\label{keyphase exprn}
\psi(\xi,\eta)= c_{\a}^1\, \xi\,\frac{u^{\frac{1}{\a-1}}}{v^{\frac{1}{\a-1}}}
-\eta\Big(c_{\a}^1\,\frac{u^{\frac{1}{\a-1}}}{v^{\frac{1}{\a-1}}}+
c_{\a}^2\,\frac{u^{\frac{\a}{\a-1}}}{v^{\frac{\a}{\a-1}}}\Big)+\,\tilde{\psi}_{j,m}(u,v)\,+\,O(1)\,.
\eeq
Finally, from \eqref{mjmndisc1n} and \eqref{keyphase exprn}, we conclude that
\begin{align}
	\nonumber\\[-8ex]
	\nonumber\intertext{\begin{equation}
			m_{j,m,u,v}(\xi,\eta)\approx \nonumber \end{equation}} \\[-4ex]
			\nonumber\\[-4ex]
	& 2^{-\frac{m}{2}}c_{\a,u,v}\, e^{i\,\big(c_1\, \xi\,\frac{u^{\frac{1}{\a-1}}}{v^{\frac{1}{\a-1}}} -\eta\,\big( c_1\frac{u^{\frac{1}{\a-1}}}{v^{\frac{1}{\a-1}}}+
		c_2\frac{u^{\frac{\a}{\a-1}}}{v^{\frac{\a}{\a-1}}}\big)\big)}
	\phi\Big(\frac{\xi}{2^{\a j+\frac{m}{2}}}{-}(u+v)\Big)\phi\Big(\frac{a\, \eta}{2^{\a j+\frac{m}{2}}}{-}v\Big)\,,\label{mjmndisc10n}
\end{align}
where here, for notational simplicity we dropped the dependence of the constants $c_1$, $c_2$ on the parameter $\a$.

$\newline$
\noindent\textbf{Step II: Adapted Gabor frame discretization: conversion of the curvature into the time-frequency localization of the wave-packets}
$\newline$

We now define the Gabor frame $\{\phi_{u}^p\}_{p,u\in\Z}$ by
\begin{align}
	\phi_{u}^p(\xi)&:=\frac{1}{2^{\frac{\a j}{2}+\frac{m}{4}}}\,\phi\Big(\frac{\xi}{2^{\a j+\frac{m}{2}}}-u\Big)\,
	e^{i\,\frac{\xi}{2^{\a j+\frac{m}{2}}}\,p}\,,\label{GFn}
	\intertext{and hence}
	\hat{\phi}_{u}^p(x)&=2^{\frac{\a j}{2}+\frac{m}{4}}\,
	\hat{\phi}\big(2^{\a j+\frac{m}{2}}x-p\big)\,
	e^{-i\,\big(2^{\a j+\frac{m}{2}}x-p\big)\,u}\,.\label{GFfn}
\end{align}

Recalling now \eqref{tjm} and using the Gabor decompositions
\beq\label{gdfn}
\hat{f}(\xi)\approx\sum_{w,p\in\Z} \langle f,  \check{\phi}_{u}^p\rangle\, \phi_{u}^p (\xi)\,,
\eeq
\beq\label{gdgn}
\hat{g}(\eta)\approx\sum_{z,r\in\Z} \langle g,  \check{\phi}_{z}^r\rangle\, \phi_{z}^r (\eta)\,,
\eeq
we get
\begin{align}
	\nonumber\\[-8ex]
	\nonumber\intertext{\begin{equation}
			\L_{j,m}(f,g,h)\approx 2^{-\frac{m}{2}}\,\sum_{\substack{u\sim 2^{\frac{m}{2}-(\a-1)j}\\[.1ex] p\in\Z}}\: \sum_{\substack{v\sim_{\a} 2^{\frac{m}{2}}\\[.1ex] r\in\Z}}  \langle f,  \check{\phi}_{u+v}^p\rangle\,
			\langle g,  \check{\phi}_{v}^r\rangle \:\times\qquad\nonumber \end{equation}} \\[-4ex]
	\nonumber\\[-4ex]
	&\hspace{1cm}\times \int_{\R} h(x)\, \check{\phi}_{u+v}^p\Big(-x-c_1\,\frac{u^{\frac{1}{\a-1}}}{v^{\frac{1}{\a-1}}}\Big)\,\check{\phi}_{v}^r\Big(-x +c_1\,\frac{u^{\frac{1}{\a-1}}}{v^{\frac{1}{\a-1}}}+
	c_2\,\frac{u^{\frac{\a}{\a-1}}}{v^{\frac{\a}{\a-1}}}\Big)\,dx\:.\label{tjm0n}
\end{align}

Now from \eqref{GFfn} we deduce that the dominant term in \eqref{tjm0n} requires the \emph{time-frequency correlation} condition
\beq\label{tjm1n}
\boxed{-p-2^{\a j+\frac{m}{2}}\,c_1\,\frac{u^{\frac{1}{\a-1}}}{v^{\frac{1}{\a-1}}}=
-r+2^{\a j+\frac{m}{2}}\,c_1\,\frac{u^{\frac{1}{\a-1}}}{v^{\frac{1}{\a-1}}}+ 2^{\a j+\frac{m}{2}}\,c_2\,\frac{u^{\frac{\a}{\a-1}}}{v^{\frac{\a}{\a-1}}}+O(1)\:.}
\eeq
Based on the above we deduce
\beq\label{ljm0n}
|\L_{j,m}(f,g,h)|\lesssim
2^{\frac{\a j}{2}-\frac{m}{4}}\,\sum_{{u\sim 2^{\frac{m}{2}-(\a-1)
			j
		}}\atop{r\in\Z}} \sum_{v\sim_{\a} 2^{\frac{m}{2}}}
\eeq
$$ \Big|\langle f,  \check{\phi}_{u+v}^{r- 2^{\a j+\frac{m}{2}}\Big(2c_1\,\frac{u^{\frac{1}{\a-1}}}{v^{\frac{1}{\a-1}}}+
\,c_2\,\frac{u^{\frac{\a}{\a-1}}}{v^{\frac{\a}{\a-1}}}\Big)}\rangle\Big|\,
|\langle g,  \check{\phi}_{v}^r\rangle|\,\Big| \langle h,  \check{\phi}_{u+2v}^{r- 2^{\a j+\frac{m}{2}}\Big(c_1\,\frac{u^{\frac{1}{\a-1}}}{v^{\frac{1}{\a-1}}}+
\,c_2\,\frac{u^{\frac{\a}{\a-1}}}{v^{\frac{\a}{\a-1}}}\Big)}\rangle\Big|\:.$$

$\newline$
\noindent\textbf{Step III: Cancelation via $TT^{*}$-method, phase-level set analysis and time-frequency correlation}

In \eqref{ljm0n} we apply the change of variable $r\longmapsto r+ 2^{\a j+\frac{m}{2}}\,c_1\,\frac{u^{\frac{1}{\a-1}}}{v^{\frac{1}{\a-1}}}+
2^{\a j+\frac{m}{2}}\,c_2\,\frac{u^{\frac{\a}{\a-1}}}{v^{\frac{\a}{\a-1}}}$ followed by $u\longmapsto u-2v$ and then, after performing a suitable splitting of the $(u,v)-$summation domain, the term $|\L_{j,m}(f,g,h)|$ may be expressed as a superposition of at most $c(\a)$ terms of the form
\begin{align}
	2^{\!\!^{\frac{\a j}{2}-\frac{m}{4}}}\hspace{-.3cm}
	\sum_{\substack{u\in \mathfrak{I}^{m,j}\\[.1ex] v\in \mathfrak{I}^{m,0}\\[.1ex] r\in\Z }}\hspace{-.25cm}
	|\langle f, \check{\phi}_{u-v}^{\!r- c_1 2^{\a j+\frac{m}{2}}\!\frac{(u-2v)^{\frac{1}{\a-1}}}{v^{\frac{1}{\a-1}}}}\!\rangle|
	|\langle g, \check{\phi}_{v}^{\!r+2^{\a j+\frac{m}{2}}\! \big(c_1\frac{(u-2v)^{\frac{1}{\a-1}}}{v^{\frac{1}{\a-1}}}+
		c_2\frac{(u-2v)^{\frac{\a}{\a-1}}}{v^{\frac{\a}{\a-1}}}\big)}
	\!\rangle|
	|\langle h,\check{\phi}_{u}^{r}\rangle|\,,\raisetag{.9\baselineskip}\label{ljm1n}
\end{align}
where here we recall the notation
$\mathfrak{I}^{m,j}:=[2^{\frac{m}{2}-(\a-1)j},\,2^{\frac{m}{2}-(\a-1)j+1}]$.

As in Section \ref{CaL}, we remark
\begin{itemize}
\item the almost disjointness of the set of frequencies given by the set of indices $\{(u-v, v)\}_{(u,v)\in \mathfrak{I}^{m,j,l}\times \mathfrak{I}^{m,0}}$ as $l\in \{1,\ldots, 2^{-(\a-1)j}\}$ where here $\mathfrak{I}^{m,j,l}:=[2^{\frac{m}{2}-(\a-1)j}+(l-1) 2^{\frac{m}{2}},\,2^{\frac{m}{2}-(\a-1)j}+l\, 2^{\frac{m}{2}}]$;

\item the almost disjointness of the set of spatial locations given by
\[
\Big\{\big(r- 2^{\a j+\frac{m}{2}}c_1 \tfrac{(u{-}2v)^{\frac{1}{\a-1}}}{v^{\frac{1}{\a-1}}}, \,r+2^{\a j{+}\frac{m}{2}} c_1\tfrac{(u{-}2v)^{\frac{1}{\a-1}}}{v^{\frac{1}{\a-1}}}+
2^{\a j{+}\frac{m}{2}}c_2 \tfrac{(u{-}2v)^{\frac{\a}{\a-1}}}{v^{\frac{\a}{\a-1}}}\big)\Big\}_{r\in R(r_0)}\,,
\]
\end{itemize}
as $r_0$ spans $\Z$, where, as before,  $R(r_0):=[ 2^{\frac{m}{2}}r_0,\,2^{\frac{m}{2}}(r_0+1)]$.

Setting now $\mathfrak{J}:= \mathfrak{I}^{m,j,1}$, $\mathfrak{M}:=\mathfrak{I}^{m,0}$, it is enough to estimate the term
\begin{align}
	\intertext{\begin{equation}
			\L_{j,m}^{R,\mathfrak{J}}(f,g,h):= \nonumber\end{equation}} \\[-4ex]
	&2^{\!\!^{\frac{\a j}{2}{-}\frac{m}{4}}}\hspace{-.2cm} \sum_{{{u\in \mathfrak{J}}\atop{v\in \mathfrak{M}}}\atop{r\in R}}\!\!\!
	|\langle f, \check{\phi}_{u-v}^{\!r- c_1\,2^{\a j+\frac{m}{2}}\frac{(u{-}2v)^{\frac{1}{\a-1}}}{v^{\frac{1}{\a-1}}}}\!\!\rangle|
	|\langle g,  \check{\phi}_{v}^{\!r+2^{\a j+\frac{m}{2}}\! \big(c_1\frac{(u{-}2v)^{\frac{1}{\a-1}}}{v^{\frac{1}{\a-1}}}+
		c_2\frac{(u{-}2v)^{\frac{\a}{\a-1}}}{v^{\frac{\a}{\a-1}}}\big)}\!\rangle|
	|\langle h,  \check{\phi}_{u}^{r}\rangle|\,. \raisetag{\baselineskip}\label{ljm1n-1}
\end{align}

Mirroring the procedure for the $+$ component, the main ingredient will be provided by the $m-$exponentially decaying single scale estimate
$\newline$

\noindent\textbf{Main Proposition (-)} \emph{Fix $j\in\Z_{-}$ with $j\leq - C_0$, $m\in\N$ and let $R$, $\mathfrak{J}$ and $\mathfrak{M}$ as above. Given a function $f$, we set
\begin{equation}\label{DeltaIKmin}
\Delta_{\a j,\frac{m}{2}}^{R,\mathfrak{J}}{f}:=\Bigg(\sum_{\substack{u\in \mathfrak{J}\\r\in R}}|\langle f,\check{\phi}_u^r\rangle|^2 \Bigg)^{\frac12}\,.	
\end{equation}
Then, for $\L_{j,m}^{R,\mathfrak{J}}$ defined by \eqref{ljm1n}, we have that there exist
$\d_0\in(0,\frac{1}{4})$ and $c=c(c_1,c_2)>0$ absolute constants such that uniformly in $j$ and $m$ the following holds:
\begin{equation}\label{ljmkLmin}
\L_{j,m}^{R,\mathfrak{J}}(f,g,h)\lesssim_{\a} 2^{-\d_0 m}\, 2^{\frac{\a j}{2}}\Big(\Delta_{\a j,\frac{m}{2}}^{c R, 3\mathfrak{J}}f\Big) \Big(\Delta_{\a j,\frac{m}{2}}^{c R,\mathfrak{M}}g\Big)\Big(\Delta_{\a j,\frac{m}{2}}^{R, \mathfrak{J}} h\Big)\:.	
\end{equation}}

As in Section \ref{CaL}, we enumerate the main steps of our approach:

$\newline$
\noindent\textbf{Step III.1. Frequency fiber foliation:  a sparse--uniform dichotomy}\label{3p1m}
$\newline$

 We fix $\d\in (0,\frac{1}{4})$ to be chosen later and define the set of \emph{$(\d,g)$-sparse frequency fibers} associated to spatial location $R$ and frequency location $\mathfrak{M}$ as
\begin{equation}\label{corelmin}
	\boxed{\mathcal{S}_{\d}^{R,\mathfrak{M}}(g):=\Big\{v\in \mathfrak{M}\,|\,\max_{r\in R}\,  |\langle g,  \check{\phi}_{v}^{r}\rangle|\geq 2^{-\d m}\,\Delta_{\a j,\frac{m}{2}}^{R, \mathfrak{M}}(g)\Big\}}
\end{equation}
and notice that
\beq\label{corelcardmin}
\#\mathcal{S}_{\d}^{c R,\mathfrak{M}}(g)\leq 2^{2 \d m}\:.
\eeq

Next, we split
\beq\label{ljm2n}
\L_{j,m}^{R,\mathfrak{J}}(f,g,h)= \L_{j,m,\textrm{sparse}}^{R,\mathfrak{J}}(f,g,h)\,+\,\L_{j,m,\textrm{unif}}^{R,\mathfrak{J}}(f,g,h)\,,
\eeq
where
\begin{itemize}
\item the \emph{sparse} (heavy) component is defined as
\begin{align}
	\L_{j,m,\textrm{sparse}}^{R,\mathfrak{J}}(f,g,h)&:=
	\sum_{\substack{v\in \mathcal{S}_{\d}^{c R,\mathfrak{M}}(g)\\[.4ex] u\in \mathfrak{J}\,,\:r\in R }} \left(\ldots\right)\,,\label{ljm3n}
	\intertext{\item the \emph{uniform} (light) component is given by}
	\L_{j,m,\textrm{unif}}^{R,\mathfrak{J}}(f,g,h)&:=
	\sum_{\substack{v\in\mathcal{U}_{\d}^{c R,\mathfrak{M}}(g)\\[.4ex] u\in \mathfrak{J}\,,\:r\in R }}
	 \left(\ldots\right)\,,\label{ljm4n}
\end{align}
with $\mathcal{U}_{\d}^{c R,\mathfrak{M}}(g):=\mathfrak{M}\:\setminus\: \mathcal{S}_{\d}^{c R,\mathfrak{M}}(g)$.
\end{itemize}

$\newline$
\noindent\textbf{Step III.2. The sparse component}\label{Thm}
$\newline$

Applying first a Cauchy--Schwarz in $v$ and then in $u,r$, followed by the change of variable
$r\longmapsto r-2^{\a j+\frac{m}{2}}\,c_1\,\frac{(u-2v)^{\frac{1}{\a-1}}}{v^{\frac{1}{\a-1}}}-
2^{\a j+\frac{m}{2}}\,c_2\,\frac{(u-2v)^{\frac{\a}{\a-1}}}{v^{\frac{\a}{\a-1}}}$,
we get
\beq\label{corrtermn}
\L_{j,m,\textrm{sparse}}^{R,\mathfrak{J}}(f,g,h)\leq 2^{\frac{\a j}{2}-\frac{m}{4}}\,(\#\mathcal{S}_{\d}^{c R,\mathfrak{M}}(g))^{\frac{1}{2}}\,\Bigg(\sum_{{\atop{u\in \mathfrak{J}}}\atop{r\in R}}  |\langle h,  \check{\phi}_{u}^{r}\rangle|^2\Bigg)^{\frac{1}{2}}\times
\eeq
$$\times\,\left(\,\sum_{\substack{v\in \mathfrak{M}\\[.2ex] r\in c R}}  |\langle g,  \check{\phi}_{v}^{r}\rangle|^2\,
\left(\sum_{u\in \mathfrak{J}} |\langle f,  \check{\phi}_{u-v}^{r-2^{j+\frac{m}{2}+1}\,c_1\,\frac{(u-2v)^{\frac{1}{\a-1}}}{v^{\frac{1}{\a-1}}}-
2^{j+\frac{m}{2}}\,c_2\,\frac{(u-2v)^{\frac{\a}{\a-1}}}{v^{\frac{\a}{\a-1}}}}\rangle|^2\,\right)\right)^{\frac{1}{2}}\:.$$
Applying the same reasonings as for \eqref{ljmkLclust}, we deduce
\beq\label{corrterm1n}
\L_{j,m,\textrm{sparse}}^{R,\mathfrak{J}}(f,g,h)\lesssim 2^{-m(\frac{1}{4}-\d)}\, 2^{\frac{\a j}{2}}\,\Delta_{\a j,\frac{m}{2}}^{c R, 3\mathfrak{J}}(f)\: \Delta_{\a j,\frac{m}{2}}^{c R,\mathfrak{M}}(g)\:\Delta_{\a j,\frac{m}{2}}^{R,\mathfrak{J}} (h)\:.
\eeq

$\newline$
\noindent\textbf{Step III.3. The uniform component}\label{Tum}
$\newline$

As before, the treatment of the uniform term relies on the analysis of the time-frequency correlations:
\begin{itemize}
\item  Consider an arbitrary (measurable) function
\beq\label{vmeasmin}
u(\cdot):\,R\,\longrightarrow\,\mathfrak{J}\:.
\eeq
\item For $v\in \mathcal{U}_{\d}^{c R,\mathfrak{M}}(g)$  and $\kappa\in c R$ we define the $(v,\kappa)$-sets as
\begin{equation}\label{tfreqcorn}
	\J_{v,\kappa}^{R,\mathfrak{M}}(u)\!:=\!\left\{r\in R\,\Big|\,\big|r+2^{\a j+\frac{m}{2}} c_1\tfrac{(u(r)-2v)^{\frac{1}{\a-1}}}{v^{\frac{1}{\a-1}}}+
	2^{\a j+\frac{m}{2}} c_2\tfrac{(u(r)-2v)^{\frac{\a}{\a-1}}}{v^{\frac{\a}{\a-1}}} -\kappa\big|\leq 10\right\}\:.
\end{equation}
\item Let now $\e\in (0,\frac{\d}{2})$ small and define the \emph{time-frequency correlation set} (relative to $u(\cdot)$) as
\beq\label{tfreqcor1n}
\boxed{\mathcal{C}_{\ep,\d,-}^{R,\mathfrak{M}}:=\big\{v\in \mathcal{U}_{\d}^{c R,\mathfrak{M}}(g) \,|\,\exists\:\kappa=\kappa(v)\in c R\:\:\:\textrm{s.t.}\:\: \#\J_{v,\kappa}^{R,\mathfrak{M}}(u)\gtrsim 2^{(\frac{1}{2}-\ep)\,m}\big\}\:.}
\eeq
Then the key observation is that uniformly in $u(\cdot)$ given by \eqref{vmeasmin}, we have
\beq\label{keyreln}
\boxed{\exists\:\mu=\mu(\ep)>\frac{1}{2}-2\d\:\:\quad\textrm{s.t.}\quad\#\mathcal{C}_{\ep,\d,-}^{R,\mathfrak{M}}\lesssim 2^{(\frac{1}{2}-\mu)\,m}\:.}
\eeq
The proof of \eqref{keyreln} follows  similar (technical) reasonings with the ones used for the proof of \eqref{keyrel} and thus we will skip it here.\footnote{Again, here one has to use the condition $j\leq - C_0$ for $C_0>0$ large enough.}

\item Perform a Cauchy--Schwarz in the $v$ variable:
\begin{align}
	\L&_{j,m,\textrm{unif}}^{R,\mathfrak{J}}(f,g,h)\leq 2^{\frac{\a j}{2}-\frac{m}{4}}\,\sum_{\substack{u\in \mathfrak{J}\\[.2ex] r\in R }} \left(\sum_{v\in \mathcal{U}_{\d}^{c R,\mathfrak{M}}(g)}  |\langle f,  \check{\phi}_{u-v}^{r- c_1\,2^{\a j+\frac{m}{2}}\,\frac{(u-2v)^{\frac{1}{\a-1}}}{v^{\frac{1}{\a-1}}}}\rangle|^2\right)^{\frac{1}{2}}\times\nonumber\\
	&\times \left(\sum_{v\in \mathcal{U}_{\d}^{c R,\mathfrak{M}}(g)} |\langle g,  \check{\phi}_{v}^{r+2^{\a j+\frac{m}{2}}\,c_1\,\frac{(u-2v)^{\frac{1}{\a-1}}}{v^{\frac{1}{\a-1}}}+
		2^{\a j+\frac{m}{2}}\,c_2\, \frac{(u-2v)^{\frac{\a}{\a-1}}}{v^{\frac{\a}{\a-1}}}}\rangle|^2\right)^{\frac{1}{2}} |\langle h,  \check{\phi}_{u}^{r}\rangle|\:.\label{ljmunif1n}
\end{align}

\item Apply the $l^2\times l^{\infty}\times l^2$ H\"older inequality in the $u$ parameter:
\begin{align}
	&\L_{j,m,\textrm{unif}}^{R,\mathfrak{J}}(f,g,h)\leq 2^{\frac{\a j}{2}-\frac{m}{4}}\,\sum_{r\in R}
	\left(\sum_{\substack{v\in \mathcal{U}_{\d}^{c R,\mathfrak{M}}(g)\\[.2ex] u\in \mathfrak{J}}}
	|\langle f,  \check{\phi}_{u-v}^{r- c_1\,2^{j+\frac{m}{2}}\,\frac{(u-2v)^{\frac{1}{\a-1}}}{v^{\frac{1}{\a-1}}}}\rangle|^2\right)^{\frac{1}{2}}\times\nonumber\\
	&\times\! \left(\sup_{u\in \mathfrak{J}}\!\left(
	\sum_{v\in \mathcal{U}_{\d}^{c R,\mathfrak{M}}(g)}
	\hspace{-.5cm}
	|\langle g,  \check{\phi}_{v}^{r+2^{j+\frac{m}{2}}c_1\frac{(u-2v)^{\frac{1}{\a-1}}}{v^{\frac{1}{\a-1}}}+
		2^{j+\frac{m}{2}}c_2\frac{(u-2v)^{\frac{\a}{\a-1}}}{v^{\frac{\a}{\a-1}}}}\rangle|^2\!\right)^{\!\!\!\frac{1}{2}}\right)\!\!
	\left(\sum_{u\in \mathfrak{J}}|\langle h,  \check{\phi}_{u}^{r}\rangle|^2\!\right)^{\!\!\!\frac{1}{2}}.\raisetag{.5\baselineskip} \label{ljmunif2n}
\end{align}
At this point we remark that uniformly in $r\in R$ one has that
\beq\label{ljmunifgn}
\sum_{\substack{v\in \mathcal{U}_{\d}^{c R,\mathfrak{M}}(g)\\[.2ex] u\in \mathfrak{J}}}
  |\langle f,  \check{\phi}_{u-v}^{r- c_1\,2^{j+\frac{m}{2}}\,\frac{(u-2v)^{\frac{1}{\a-1}}}{v^{\frac{1}{\a-1}}}}\rangle|^2\lesssim
\big(\Delta_{\a j,\frac{m}{2}}^{cR, 3\mathfrak{J}} (f)\big)^2\,.
\eeq

\item Combine \eqref{ljmunifgn} with an application of Cauchy--Schwarz in order to deduce
\beq\label{ljmunif4n}
\L_{j,m,\textrm{unif}}^{R,\mathfrak{J}}(f,g,h)\leq  2^{\frac{\a j}{2}-\frac{m}{4}}\,\Delta_{\a j,\frac{m}{2}}^{cR, 3\mathfrak{J}} (f)\, \Delta_{\a j,\frac{m}{2}}^{R, \mathfrak{J}} (h)\,\times
\eeq
$$\times\left(\sum_{r\in R}\,\sum_{v\in \mathcal{U}_{\d}^{c R,\mathfrak{M}}(g)} |\langle g,  \check{\phi}_{v}^{r+2^{j+\frac{m}{2}}\,c_1\,\frac{(u(r)-2v)^{\frac{1}{\a-1}}}{v^{\frac{1}{\a-1}}}+
2^{j+\frac{m}{2}}\,c_2\,\frac{(u(r)-2v)^{\frac{\a}{\a-1}}}{v^{\frac{\a}{\a-1}}}}\rangle|^2\right)^{\frac{1}{2}}\:,$$
where $u(\cdot)$ is the measurable function that attains the supremum in \eqref{ljmunif2n}.

\item Denote with $$S^{-}:=\sum_{r\in R}\,\sum_{v\in \mathcal{U}_{\d}^{c R,\mathfrak{M}}(g)} \Big|\langle g,  \check{\phi}_{v}^{r+2^{j+\frac{m}{2}}\,c_1\,\frac{(u(r)-2v)^{\frac{1}{\a-1}}}{v^{\frac{1}{\a-1}}}+
2^{j+\frac{m}{2}}\,c_2\,\frac{(u(r)-2v)^{\frac{\a}{\a-1}}}{v^{\frac{\a}{\a-1}}}}\rangle\Big|^2$$ and split
\beq\label{Ssplitn}
S^{-}=S_1^{-}\,+\,S_2^{-}:=\sum_{r\in R}\,\sum_{v\in \mathcal{C}_{\ep,\d,-}^{R,\mathfrak{M}}}\left(\ldots\right)\,+\,\sum_{r\in R}\,\sum_{v\in
	\big(\mathcal{U}_{\d}^{c R,\mathfrak{M}}(g)\big)\setminus \mathcal{C}_{\ep,\d,-}^{R,\mathfrak{M}}}\left(\ldots\right)\:.
\eeq
Now
\beq\label{S1n}
S_1^{-}\lesssim 2^{\frac{m}{2}}\,(\#\mathcal{C}_{\ep,\d,-}^{R,\mathfrak{M}})\,2^{-2\d m}\,\big(\Delta_{\a j,\frac{m}{2}}^{c R,\mathfrak{M}}(g)\big)^2\:,
\eeq
and
\beq\label{S2n}
S_2^{-}\lesssim 2^{(\frac{1}{2}-\ep)\,m} \,\big(\Delta_{\a j,\frac{m}{2}}^{c R,\mathfrak{M}}(g)\big)^2\:.
\eeq
\item Putting together \eqref{keyreln} and \eqref{ljmunif4n}--\eqref{S2n} we conclude
\beq\label{unifconcln}
\L_{j,m,\textrm{unif}}^{R,\mathfrak{J}}(f,g,h)\lesssim (2^{(\frac{1}{4}-\frac{\mu}{2}-\d)m}\,+\, 2^{-\frac{\ep}{2} m})\,2^{\frac{\a j}{2}}\,\Delta_{\a j,\frac{m}{2}}^{c R, 3\mathfrak{J}}(f)\,\Delta_{\a j,\frac{m}{2}}^{c R,\mathfrak{M}}(g)\:\Delta_{\a j,\frac{m}{2}}^{R, \mathfrak{J}} (h)
\eeq
which for appropriate values of $\ep,\,\mu,\d$--the same as for the case $j>0$--gives us the desired exponential decay in $m$.
\end{itemize}

\begin{remark}\label{l2l2intol1min} In order to obtain the analogue of \eqref{LSm} for $\L^{S}_{m,-}$ one can follow similar steps with the ones in Remark \ref{l2l2intol1} and replace \eqref{hsup} by
\beq\label{hsupmin}
\Delta_{\a j,\frac{m}{2}}^{R,3\mathfrak{J}} (h)\lesssim 2^{-\frac{\a j}{2}}\, \|h\|_{\infty}\:.
\eeq
\end{remark}

\subsection{The extended boundedness range: Proof of Proposition \ref{Lgennodecay}} \label{Propgen}

$\quad$In what follows we split our analysis of $\L^{S}_{m}$ according to the three components $\L^{S}_{H,m}$, $\L^{S}_{NL,m}$ $\L^{S}_{TR,m}$.

\subsubsection{The treatment of $\L^{S}_{H,m}$}\label{LGCLfull}

From the definition of the hybrid term in Section \ref{lm} we have that\footnote{For notational convenience, throughout the reminder of the section we preserve the assumption \eqref{abnlone} and also omit specifying the $\a$-dependencies.}
\begin{equation}\label{lsl}
\L^S_H:=\sum_{m\in\N}  \L^{S}_{H,m}:=\sum_{m\in\N}  \sum_{{(\a-1) j\geq m}\atop{j\geq C_0}} \L_{j,m}\,,
\end{equation}
where the multiplier for each $\L_{j,m}$ (with $(\a-1) j\geq m$ and $j\ge C_0$) is given by
\begin{equation}\label{saljmm}
m_{j,m}(\xi,\eta):=\Big(\int_\R e^{-i\frac{\xi-\eta}{2^j}  t}e^{i\frac{\eta}{2^{\a j}}t^\a}\rho(t)dt\Big)
		\phi\Big(\frac{\xi-\eta}{2^{j+m}}\Big)\phi\Big(\frac{\eta}{2^{\a j+m}}\Big)\,.
\end{equation}

For reader's convenience, we elaborate more on \eqref{saljmm0}--\eqref{saljm21} in the setting prescribed by \eqref{abnlone};
for $\phi$ a suitable smooth compactly supported function, we have	
\begin{equation}\label{saljmm0r}
		m_{j,m}(\xi,\eta)\approx\sum_{l\sim 2^{(\a-1)j}} m_{j,m,l}(\xi,\eta)
\end{equation}
with
\begin{equation}\label{saljmm1r}
	m_{j,m,l}(\xi,\eta):=\left(\int_\R e^{-i\frac{\xi-\eta}{2^j}  t}e^{i\frac{\eta}{2^{\a j}}t^\a}\rho(t)dt\right)\,
	\phi\Big(\frac{\xi}{2^{j+m}}-l\Big)\,\phi\Big(\frac{\eta}{2^{j+m}}-l\Big)\,\phi\Big(\frac{\xi+\eta}{2^{j+m}}-2l\Big)\,.
\end{equation}

As a consequence, we deduce
\begin{equation}\label{saljmsummr}
\L_{j,m}(f,g,h)=\sum_{l\sim 2^{(\a-1)j}}\L^{l}_{j,m}(f,g,h)\,,
\end{equation}
with
\begin{equation}\label{saljm2r}
\L^{l}_{j,m}(f,g,h)\!:=\!\int_{\R^2}\! (f*\check{\phi}_{j+m}^l)\big(x-\tfrac{t}{2^j}\big)\,(g*\check{\phi}_{j+m}^l)\big(x+\tfrac{t}{2^j}+\tfrac{t^\a}{2^{\a j}}\big)\, (h*\check{\phi}_{j+m}^{2l})(x)\rho(t)\,dt\,dx\,.
\end{equation}

Now given that
\begin{equation}\label{s2jm3}
		\check{\phi}_{j+m}^l(x):=\int_{\R} \phi\big(\tfrac{\xi}{2^{j+m}}-l\big)\,e^{i\,\xi\,x}\,d\xi=
		2^{j+m}\,e^{i\,2^{j+m}\,l\,x}\,\check{\phi}(2^{j+m}\,x)\,
\end{equation}
and taking into account that $(\a-1)j\geq m$ we have
	\begin{align}
		&(g*\check{\phi}_{j+m}^l)\big(x+\tfrac{t}{2^j}+\tfrac{t^{\a}}{2^{\a j}}\big)\label{redg}\\
		&=\int_{\R} g(y)\,2^{j+m}\,e^{i\,2^{j+m}\,l\,(x+\frac{t}{2^j}+\frac{t^{\a}}{2^{\a j}}-y)}\,
		\check{\phi}\big(2^{j+m}\big(x+\tfrac{t}{2^j}-y\big)+\tfrac{t^{\a}}{2^{(\a-1)j-m}}\Big)\,dy\nonumber\\
		&\approx e^{i\,\frac{l}{2^{(\a-1)j}}\,2^m\,t^{\a}}\!\int_{\R} g(y)\,2^{j+m}\,e^{i\,2^{j+m}\,l\,(x+\frac{t}{2^j}-y)}\,
		\check{\phi}\big(2^{j+m}\big(x+\frac{t}{2^j}-y\big)\big)\,dy\nonumber
	\end{align}
from which we deduce\footnote{We ignore here the error term(s) derived from a standard Taylor series argument.}
\begin{equation}\label{redal}
(g*\check{\phi}_{j+m}^l)\Big(x+\frac{t}{2^j}+\frac{t^\a}{2^{\a j}}\Big)\approx
e^{i\,\frac{l}{2^{(\a-1)j}}\,2^m\,t^\a}(g*\check{\phi}_{j+m}^l)\Big(x+\frac{t}{2^j}\Big)\:.
\end{equation}
Thus, for $l\sim 2^{(\a-1)j}$, we have that \eqref{saljm2r} may be reduced to the simpler form:
\begin{align}
\intertext{\begin{equation}\label{saljm200}
			|\L^{l}_{j,m}(f,g,h)|\lesssim
			 \end{equation}} \\[-4ex]
	&=\int_{\R^2}  \big|(f*\check{\phi}_{j+m}^l)\big(x-\tfrac{t}{2^j}\big)\big|
	\big|(g*\check{\phi}_{j+m}^l)\big(x+\tfrac{t}{2^j}\big)\big|\,
	\big|(h*\check{\phi}_{j+m}^{2l})(x)\big|\,|\rho(t)|\,dt\,dx\nonumber\\[1ex]
	&\approx\frac{1}{2^m}\sum_{s\sim 2^m}\int_{\R}  \Big|(f*\check{\phi}_{j+m}^l)\big(x{-}\frac{s}{2^{j+m}}\big)\Big|\Big|(g*\check{\phi}_{j+m}^l)\big(x{+}\frac{s}{2^{j+m}}\big)\Big|
	\big|(h*\check{\phi}_{j+m}^{2l})(x)\big|\,dx\nonumber\\[1ex]
	&\approx\frac{1}{2^m} \sum_{\substack{z\in\Z\\s\sim 2^m }} \big|(f*\check{\phi}_{j+m}^l)\big(\frac{z-s}{2^{j+m}}\big)\big|\,\big|(g*\check{\phi}_{j+m}^l)\big(\frac{z+s}{2^{j+m}}\big)\big|\,
	\big|(h*\check{\phi}_{j+m}^{2l})(\frac{z}{2^{j+m}})\big|\,\frac{1}{2^{j+m}}\,.\nonumber
\end{align}
Set now
$$I_{j+m}^z:=\bigg[\frac{z}{2^{j+m}},\,\frac{z+1}{2^{j+m}}\bigg)\,,\qquad\o_{j+m}^l:=[l\,2^{j+m},\,(l+1)\,2^{j+m})\,,$$

$$P_{j+m}(z,l):=I_{j+m}^z \times \o_{j+m}^l\,,$$
and
$$\P_{j+m}:=\left\{P_{j+m}(z,l)\right\}_{{z\in\Z}\atop{l\sim 2^{(\a-1)j}}}\,,\qquad \qquad\P=\bigcup_{j\geq 0} \P_{j+m}\:.$$

With the above notations we can write (with the obvious adaptations for the analogue expressions involving $g$ and $h$)
\begin{equation}\label{fcoef}
(f*\check{\phi}_{j+m}^l)\Big(\frac{z-s}{2^{j+m}}\Big)\sim \frac{1}{|I_{j+m}|^{\frac{1}{2}}}\,\big\langle f,  \phi_{P_{j+m}(z-s,l)}\big\rangle\,,
\end{equation}
where $I_{j+m}^0$ is denoted by $I_{j+m}$, and $\phi_{P_{j+m}(z,l)}$ is an $L^2$ normalized wave-packet adapted to $P_{j+m}(z,l)$.

Putting now together \eqref{saljmsummr}, \eqref{saljm200} and \eqref{fcoef} we deduce that
\begin{align}
	\intertext{\begin{equation}
			|\L_{j,m}(f,g,h)|\lesssim\label{saljmsumwp}
	\end{equation}} \\[-4ex]
	&\frac{1}{2^m}\sum_{s\sim 2^m}\sum_{z\in\Z}\sum_{l\sim 2^{(\a-1)j}}\frac{1}{|I_{j+m}|^{\frac{1}{2}}}\,|\langle f,  \phi_{P_{j+m}(z-s,l)}\rangle|\,
	|\langle g,  \phi_{P_{j+m}(z+s,l)}\rangle|\,|\langle h,  \phi_{P_{j+m}(z,2l)}\rangle|\,.\nonumber
\end{align}
Combining now \eqref{saljmsumwp} with \eqref{lsl} we reach the final form\footnote{Here, once we exploited the approximation in \eqref{redal}, we slightly abuse the summation in \eqref{lsl} and, for mere notational convenience, we include in its range the first $\max\{C_0,\frac{m}{\a-1}\}$ terms.}
\begin{align}
	\intertext{\begin{equation}
			|\L^{S}_{H,m}(f,g,h)|\lesssim\label{lsl1}
	\end{equation}} \\[-5ex]
	& \frac{1}{2^m}\sum_{s\sim 2^m}\sum_{j\geq 0}\sum_{{z\in\Z}\atop{l\sim 2^{(\a-1)j}}}\frac{1}{|I_{j+m}|^{\frac{1}{2}}}\,|\langle f,  \phi_{P_{j+m}(z-s,l)}\rangle|\,
	|\langle g,  \phi_{P_{j+m}(z+s,l)}\rangle|\,|\langle h,  \phi_{P_{j+m}(z,2l)}\rangle|\,.\nonumber
\end{align}

Recalling now the goal stated in Proposition \ref{Lgennodecay}, we assume that $F,G$ and $H$ satisfy the hypothesis therein with $f$, $g$ obeying \eqref{restrcitedweak} and focus on proving \eqref{goallsl1}.

Take now
\begin{equation}\label{om}
\O:=\left\{M \chi_{F}\geq 100 \frac{|F|}{|H|}\right\}\cup \left\{M \chi_{G}\geq 100 \frac{|G|}{|H|}\right\}\quad \textrm{and}\quad H':=H\setminus \O\,.
\end{equation}
It is now immediate to see that $H'$ satisfies the requirements in \eqref{Hprexist}. Consider a function $h$ satisfying \eqref{restrcitedweak}. With these done, we rewrite \eqref{lsl1} as

\begin{equation}\label{lslre}
|\L^{S}_{H,m}(f,g,h)|\lesssim\sum_{j\geq 0}\sum_{{u\in\Z}\atop{l\sim 2^{(\a-1)j}}} \L^{S}[j,m,u,l](f,g,h)\,,
\end{equation}
where
\begin{align}
	\intertext{\begin{equation}\label{lslre1}
			\L^{S}[j,m,u,l](f,g,h):=
	\end{equation}} \\[-4ex]
	& \frac{1}{2^{\frac{m}{2}}}\,
	\sum_{I_{j+m}^v\subseteq I_{j}^u} \frac{1}{|I_{j}^u|^{\frac{1}{2}}}\,|\langle f,  \phi_{P_{j+m}(v,l)}\rangle|\,
	\sum_{s\sim 2^m} |\langle g,  \phi_{P_{j+m}(v+2s,l)}\rangle|\,|\langle h,  \phi_{P_{j+m}(v+s,2l)}\rangle|\,.\nonumber
\end{align}
Letting now
$$\Delta_{j,m}^{l,[I_j^u]} (g):=\Bigg(\sum_{I_{j+m}^v\subseteq I_{j}^u} |\langle g,  \phi_{P_{j+m}(v,l)}\rangle|^2 \Bigg)^{\frac{1}{2}}\,,$$
we deduce that \eqref{lslre1} may be reduced via a Cauchy--Schwarz argument to
\begin{equation}\label{lslre2}
\L^{S}[j,m,u,l](f,g,h)\lesssim
\left(\frac{1}{|I_{j}^u|}\,\int |f|\,\tilde{\chi}_{I_{j}^u}\right)\, \Delta_{j,m}^{l,[5 I_j^u]} (g)\: \Delta_{j,m}^{2l,[5 I_j^u]} (h)\:.
\end{equation}
Now, due to the symmetry in the input functions $f,g$ and $h$ in \eqref{lslre1}, one can obtain the obvious analogues of \eqref{lslre2}, which, after applying a geometric mean argument, upgrades  \eqref{lslre2} to
\begin{align}
	\L^{S}[j,m,u,l](f,g,h)&\lesssim \bigg(\frac{1}{|I_{j}^u|}\int |f|\,\tilde{\chi}_{I_{j}^u}\bigg)^{\theta_1}
	\bigg(\frac{1}{|I_{j}^u|}\int |g|\,\tilde{\chi}_{I_{j}^u}\bigg)^{\theta_2}
	\bigg(\frac{1}{|I_{j}^u|}\int |h|\,\tilde{\chi}_{I_{j}^u}\bigg)^{\theta_3}\times\nonumber\\[1ex]
	&\hspace{1cm}\times \big(\Delta_{j,m}^{l,[5 I_j^u]} (f)\big)^{1-\theta_1}\,\big(\Delta_{j,m}^{l,[5 I_j^u]} (g)\big)^{1-\theta_2}\,\big(\Delta_{j,m}^{2l,[5 I_j^u]} (h)\big)^{1-\theta_3}\,,\label{lslkey}
\end{align}
for any $0\leq \theta_1,\,\theta_2,\,\theta_3\leq 1$ with $\theta_1+\theta_2+\theta_3=1$.

Recalling now \eqref{om}, for $\b\in\N$, we define
\begin{equation}\label{setI}
\I_{\b}:=\bigg\{I_{j}^u\,\Big|\,1+\frac{\textrm{dist} (I_{j}^u,\,\O^{c}) }{|I_{j}^u|}\approx 2^{\b}\bigg\}
\end{equation}
and deduce that
\begin{equation}\label{lslreal}
|\L^{S}_{H,m}(f,g,h)|\lesssim\sum_{\b\in\N} \L^{S,\b}_{H,m}(f,g,h)\,,
\end{equation}
where
\begin{equation}\label{lslrealf}
 \L^{S,\b}_{H,m}(f,g,h):=\sum_{j\geq 0}\sum_{{I_j^u\in\I_{\b}}\atop{l\sim 2^{(\a-1)j}}} \L^{S}[j,m,u,l](f,g,h)\,.
\end{equation}

Now from \eqref{om} and \eqref{setI} we have that for any $I_j^u\in\I_{\b}$ the following holds:
\begin{equation}\label{FGal}
\frac{1}{|I_{j}^u|}\,\int |f|\,\tilde{\chi}_{I_{j}^u}\lesssim 2^{\b}\,\frac{|F|}{|H|}\,,\qquad\frac{1}{|I_{j}^u|}\,\int |g|\,\tilde{\chi}_{I_{j}^u}\lesssim 2^{\b}\,\frac{|G|}{|H|}\,,
\end{equation}
and for any $N\in\N$
\begin{equation}\label{Hal}
\frac{1}{|I_{j}^u|}\,\int |h|\,\tilde{\chi}_{I_{j}^u}\lesssim_{N} 2^{-\b\,N}\,.
\end{equation}
Putting now together \eqref{lslkey}--\eqref{Hal}, we deduce
\begin{align}
	\intertext{\begin{equation}
			 \L^{S,\b}_{H,m}(f,g,h)\lesssim\label{lslrealf1}
	\end{equation}} \\[-4ex]
	&\bigg( 2^{\b}\frac{|F|}{|H|}\bigg)^{\!\!\theta_1}\! \bigg( 2^{\b}\frac{|G|}{|H|}\bigg)^{\!\!\theta_2}\!2^{-\b\,N\,\theta_3}
	\hspace{-.4cm}\sum_{\substack{j\geq 0\\u\in\Z\\l\sim 2^{(\a-1)j}}}\hspace{-.3cm}
	\big(\Delta_{j,m}^{l,[5 I_j^u]} (f)\big)^{\!1-\theta_1}\big(\Delta_{j,m}^{l,[5 I_j^u]} (g)\big)^{\!1-\theta_2}\big(\Delta_{j,m}^{2l,[5 I_j^u]} (h)\big)^{\!1-\theta_3}\nonumber\\
	&\lesssim \left(\frac{|F|}{|H|}\right)^{\theta_1}\,\left(\frac{|G|}{|H|}\right)^{\theta_2}\,2^{\b\,(\theta_1+\theta_2-N\,\theta_3)}\,
	\Bigg(\sum_{j\geq 0}\sum_{{u\in\Z}\atop{l\sim 2^{(\a-1)j}}}
	\big(\Delta_{j,m}^{l,[5 I_j^u]} (f)\big)^2\Bigg)^{\frac{1-\theta_1}{2}}\times\nonumber\\
	&\hspace{1cm}\times \Bigg(\sum_{j\geq 0}\sum_{{u\in\Z}\atop{l\sim 2^{(\a-1)j}}}
	\big(\Delta_{j,m}^{l,[5 I_j^u]} (g)\big)^2\Bigg)^{\frac{1-\theta_2}{2}}\,
	\Bigg(\sum_{j\geq 0}\sum_{{u\in\Z}\atop{l\sim 2^{(\a-1)j}}}
	\big(\Delta_{j,m}^{2l,[5 I_j^u]} (h)\big)^2\Bigg)^{\frac{1-\theta_3}{2}}\:.\nonumber
\end{align}

Thus, appealing to Parseval, we have that
\begin{align}
	\L^{S,\b}_{H,m}(f,g,h)&\lesssim_{N} 2^{\b\,(\theta_1+\theta_2-N\,\theta_3)}\,\left(\frac{|F|}{|H|}\right)^{\theta_1}\,\left(\frac{|G|}{|H|}\right)^{\theta_2}\,
	|F|^{\frac{1-\theta_1}{2}}\,|G|^{\frac{1-\theta_2}{2}}\,|H|^{\frac{1-\theta_3}{2}}\nonumber\\[1ex]
	&\lesssim 2^{\b\,(1-(N+1)\,\theta_3)}\,
	|F|^{\frac{1+\theta_1}{2}}\,|G|^{\frac{1+\theta_2}{2}}\,|H|^{-\frac{1-\theta_3}{2}}\,.\label{lslrealf2}
\end{align}

Finally, from \eqref{lslreal} and \eqref{lslrealf2}, we conclude that\footnote{Choosing $N$ large enough.}
\begin{equation}\label{lslconcl}
|\L^{S}_{H,m}(f,g,h)|\lesssim |F|^{\frac{1+\theta_1}{2}}\,|G|^{\frac{1+\theta_2}{2}}\,|H|^{-\frac{1-\theta_3}{2}}\,,
\end{equation}
for any $0<\theta_3\leq 1$,\, $0\leq\theta_1,\,\theta_2<1$ with $\theta_1+\theta_2+\theta_3=1$ proving thus the analogue of \eqref{goallsl1} with $\L^{S}_{m}$ replaced by $\L^{S}_{H,m}$.

\begin{remark}\label{Shybrid} It is worth mentioning here that from the three components of $\L^{S}_{m}$, that is, $\L^{S}_{H,m}$, $\L^{S}_{NL,m}$ $\L^{S}_{TR,m}$, it is precisely the hybrid component $\L^{S}_{H,m}$ that is responsible for the restricted range stated in Proposition \ref{Lgennodecay}. Indeed, for the other two remaining components one can prove a better/maximal range as revealed in Proposition \ref{Lgennodecaynl} below. Also, in contrast with \eqref{goallsl1}, for $\L^{S}_{H,m}$ one can prove directly\footnote{\textit{I.e.}, not involving multilinear interpolation.} bounds that do not depend on the parameter $m$.
\end{remark}

\subsubsection{The treatment of $\L^{S}_{NL,m}$}\label{LGCNLfull} As already alluded, in this section we will prove that the non linear component $\L^{S}_{NL,m}$ obeys a wider boundedness range:

\begin{proposition}\label{Lgennodecaynl}
Let $m\in\N$ and $F,\,G,\,H$ be any measurable sets with finite (nonzero) Lebesgue measure. Then
\begin{equation}\label{Hprexistnl}
\exists\:H'\subseteq H\:\:\textrm{measurable with}\:\:|H'|\geq \frac{1}{2} |H|\,,
\end{equation}
such that for any triple of functions $f,\,g$ and $h$ obeying
\begin{equation}\label{restrcitedweaknl}
|f|\leq \chi_{F},\qquad |g|\leq \chi_{G},\qquad |h|\leq \chi_{H'}\,,
\end{equation}
and any triple $(p,q,r')$ with $1<p,q\leq \infty$, $r>\frac{1}{2}$ and $\frac{1}{p}+\frac{1}{q}+\frac{1}{r'}=1$ one has that
\begin{equation}\label{goallsl1nl}
|\L^{S}_{NL,m}(f,g,h)|\lesssim_{\a,p,q} m\,|F|^{\frac{1}{p}}\,|G|^{\frac{1}{q}}\,|H|^{\frac{1}{r'}}\,.
\end{equation}
\end{proposition}

We start with the definition of the nonlinear dominant behavior in Section \ref{Phdepdec}
\begin{equation}\label{nlsl}
\L^{S}_{NL}:=\sum_{m\in\N} \L^{S}_{NL,m}:=\sum_{m\in\N}\sum_{{(\a-1) j\leq -m}\atop{j\leq -C_0}} \L_{j,m}\,,
\end{equation}
where, recalling that $j<0$, $\a>1$ and we are in the setting \eqref{abnlone}, the multiplier for each $\L_{j,m}$ is given by
\begin{equation}\label{nsaljmm}
m_{j,m}(\xi,\eta)\approx\bigg(\int_\R e^{-i\frac{\xi-\eta}{2^j}  t}e^{i\frac{\eta}{2^{\a j}}t^\a}\rho(t)dt\bigg)
\phi\Big(\frac{\xi}{2^{j+m}}\Big)\phi\Big(\frac{\eta}{2^{\a j+m}}\Big)\phi\Big(\frac{\xi+\eta}{2^{j+m}}\Big)\,.
\end{equation}
As a consequence, we have
\begin{equation}\label{nsaljm2}
\L_{j,m}(f,g,h)\approx\int\limits_{\R^2}\! (f*\check{\phi}_{j+m})\Big(x-\sdfrac{t}{2^j}\Big)
(g*\check{\phi}_{\a j+m})\Big(x+\sdfrac{t}{2^j}+\sdfrac{t^\a}{2^{\a j}}\Big) (h*\check{\phi}_{j+m})(x) \rho(t)\,dt\,dx\,.
\end{equation}
Using that $(\a-1)\,j\leq -m$ and following similar reasonings with the ones in \eqref{redg}, we have
\begin{equation}\label{redaln}
(g*\check{\phi}_{\a j+m})\bigg(x+\frac{t}{2^j}+\frac{t^\a}{2^{\a j}}\bigg)\approx
(g*\check{\phi}_{\a j+m})\bigg(x+\frac{t^\a}{2^{\a j}}\bigg)\:.
\end{equation}
Deduce from here that
\begin{equation}\label{saljm20n}
	|\L_{j,m}(f,g,h)|\lesssim
\end{equation}
\\[-4ex]
$$\int_{\R^2}  \Big|(f*\check{\phi}_{j+m})\Big(x-\sdfrac{t}{2^j}\Big)\Big|\,
\Big|(g*\check{\phi}_{\a j+m})\Big(x+\sdfrac{t^\a}{2^{\a j}}\Big)\Big|\,
\Big|(h*\check{\phi}_{j+m})(x)\Big|\,|\rho(t)|\,dt\,dx$$
\\[-3ex]
$$\approx\!\sdfrac{1}{2^m}\!\!\sum_{\substack{s\sim 2^m\\[.2ex] z\in\Z}}\! \Big|(f*\check{\phi}_{j+m})\Big(\sdfrac{z-s}{2^{j+m}}\Big)\Big|
\Big|(g*\check{\phi}_{\a j+m})\Big(\sdfrac{z}{2^{j+m}}+\sdfrac{s^{\a}}{2^{\a j+\a m}}\Big)\Big|
\Big|(h*\check{\phi}_{j+m})\Big(\sdfrac{z}{2^{j+m}}\Big)\Big|\sdfrac{1}{2^{j+m}}\,.$$
Keeping the same notations as in  Section \ref{LGCLfull} and setting for notational simplicity $P_{j+m}(v)\equiv P_{j+m}(v,1)$ we have
\begin{equation}\label{lnslre}
|\L^{S}_{NL,m}(f,g,h)|\lesssim \frac{1}{2^m}\sum_{s\sim 2^m} \L^{S,s}_{NL,m}(f,g,h)\,,
\end{equation}
where
\begin{equation}\label{lnslre1}
	\L^{S,s}_{NL,m}(f,g,h):=
\end{equation}
$$\sum_{j<0}\sum_{u\in\Z}
\frac{1}{|I_{\a j+m}^u|^{\frac{1}{2}}}\,|\langle g,  \phi_{P_{\a j+m}(u+\frac{s^{\a}}{2^{(\a-1) m}})}\rangle|\,\sum_{I_{j+m}^v\subseteq I_{\a j+ m}^u} \hspace{-.3cm}
|\langle f,  \phi_{P_{j+m}(v-s)}\rangle|\,|\langle h,  \phi_{P_{j+m}(v)}\rangle|\,.$$

Using now a Fubini followed by an $l^1-l^{\infty}$ H\"older in the $s$ parameter, we deduce

\begin{align}\label{maxshift}
|\L^{S}_{NL,m}(f,g,h)|\lesssim \sum_{j<0}\sum_{u\in\Z}
\frac{1}{|I_{\a j+m}^u|^{\frac{1}{2}}}\,\left(\frac{1}{2^m}\sum_{s\sim 2^m} |\langle g,  \phi_{P_{\a j+m}(u+\frac{s^{\a}}{2^{(\a-1)m}})}\rangle|\right)\nonumber\\
\left(\sum_{I_{j+m}^v\subseteq I_{\a j+ m}^u}
|\langle f,  \phi_{P_{j+m}(v)}\rangle|\,|\langle h,  \phi_{P_{j+m}(v+s(I_{\a j+ m}^u))}\rangle|\right)\,,
\end{align}
where here
\begin{equation}\label{meas}
s(\cdot):\,\mathcal{I}\,\mapsto\,[2^m, 2^{m+1}]\cap \N\quad\textrm{measurable function}\,.
\end{equation}

Assume now $f,g,h$ obey \eqref{restrcitedweaknl} and set $H':=H\setminus \O_{m}$ where
\begin{equation}\label{oms}
\O_{m}:=\left\{M \chi_{G}\geq C\,\frac{|G|}{|H|}\right\}
\cup\left\{M \chi_{H}\geq C\,|H|\right\}\,.
\end{equation}
Notice that for a large enough $C>0$ the set $H'$ satisfies the requirements in \eqref{Hprexistnl}.

Next, for $\b\in\N$, we set
\begin{equation}\label{setIn}
\I_{\b}:=\left\{I_{\a j}^w\,\Big|\,1+\frac{\textrm{dist} (I_{\a j}^w,\,(\O_{m})^c) }{|I_{\a j}^w|}\approx 2^{\b}\right\}
\end{equation}
and deduce that
\begin{equation}\label{lslrealn}
|\L^{S}_{NL,m}(f,g,h)|\lesssim\sum_{\b\in\N} \L^{S,\b}_{NL,m}(f,g,h)\,,
\end{equation}
where
\begin{align}\label{maxshift1}
\L^{S,\b}_{NL,m}(f,g,h):=\qquad\qquad\qquad\qquad\qquad\qquad\qquad\qquad\nonumber\\
\sum_{j<0}\sum_{I_{\a j}^w\in \I_{\b,m}}
\left(\frac{1}{|I_{\a j}^w|} \int_{I_{\a j}^{w}}|g|\right)
\left(\sum_{{I_{j+m}^v\subseteq I_{\a j+ m}^u}\atop{I_{\a j+ m}^u}\subseteq I_{\a j}^w}
|\langle f,  \phi_{P_{j+m}(v)}\rangle|\,|\langle h,  \phi_{P_{j+m}(v+s(I_{\a j+ m}^u))}\rangle|\right)\,.
\end{align}
Let $\I$ be a collection of dyadic intervals and $s(\cdot)$ any function obeying \eqref{meas}. We define now the \emph{maximal $m$-shifted square function} relative to the pair $(\I,s)$ by
\begin{equation}\label{maxshift}
 S_{\I,s}^{max}(h)(x):=\left(\sum_{j<0}\sum_{I_{\a j}^w\in \I}
\sum_{{I_{j+m}^v\subseteq I_{\a j+ m}^u}\atop{I_{\a j+ m}^u}\subseteq I_{\a j}^w}
\frac{|\langle h,  \phi_{P_{j+m}(v+s(I_{\a j+ m}^u))}\rangle|^2}{|I_{j+m}^v|}\chi_{I_{j+m}^v}(x)\right)^{\frac{1}{2}}\,.
\end{equation}
From \eqref{setIn}, \eqref{maxshift1} and \eqref{maxshift} we immediately deduce that
\begin{equation}\label{maxshift2}
\L^{S,\b}_{NL,m}(f,g,h)\lesssim \max\left\{1,\,2^{\b}\,\frac{|G|}{|H|}\right\}\,\int_{\R} S(f)(x)\,S_{\I_{\b},s}^{max}(h)(x)\,dx\,.
\end{equation}
Proceeding now in the similar spirit with the approach in \cite{Lie18} one can show that
\begin{equation}\label{maxshift3}
\|S_{\I_{\b},s}^{max}(h)\|_{p}\lesssim_{n} m\, 2^{-n\b}\,|H|^{\frac{1}{p}}\,
\end{equation}
for any $n\in\N$ and $1\leq p<\infty$.

Putting together \eqref{setIn}--\eqref{maxshift3} we deduce that for any $\theta,\,\mu\in (0,1)$
\begin{equation}\label{keyn02}
\L^{S}_{NL,m}(f,g,h)\lesssim_{\theta\,\mu}\,\min\left\{m\,|F|^{\theta}\,|G|\,|H|^{-\theta},\,|F|^{\mu}\,|H|^{1-\mu}\right\}\,.
\end{equation}

The desired bound \eqref{goallsl1nl} follows now via multilinear interpolation.

\subsubsection{The treatment of $\L^{S}_{TR,m}$}\label{Trm}
For this it is enough to notice that--for a detailed account of this please see Section \ref{NSIterm}--\emph{we have uniformly in $j$} that
\begin{equation}\label{flrange}
|\L^{S}_{j,m}(f,g,h)|\lesssim_{\g,p,q} \|f\|_p\,\|g\|_q\,\|h\|_{r'}\,,
\end{equation}
for any $\frac{1}{p}+\frac{1}{q}+\frac{1}{r'}=1$ with $1\leq p,\,q\leq \infty$ and $r>\frac{1}{2}$.

\section{The analysis of the non-stationary term $\L^{NS}$}\label{HNL-2}

Recalling the notation and decompositions introduced in Section \ref{lns}, we will treat one by one each of the terms
$\L^{NS,I},\,\L^{NS, II},\,\L^{NS, III}$ and $\L^{NS, IV}$ that form $\L^{NS}$ by essentially reducing the matter--via an integration by parts procedure--to the techniques described in Section \ref{Propgen}.

\subsection{Treatment of $\boldmath{\L^{NS,I}}$}\label{lns1} In this situation, recalling \eqref{ns11}, we have

\begin{equation}\label{ns10}
\L^{NS,I}:=\sum_{n\geq 0}\L^{NS,I}_{n}:=\sum_{n\geq 0}\sum_{|j|\geq C_0} \L^{NS,I}_{j,n}:=\sum_{n\geq 0}\sum_{|j|\geq C_0}\:\sum_{m<0} \L_{j,m,n}\,.
\end{equation}
Notice now that for a fixed $n\in\N$, the multiplier associated with $\L^{NS,I}_{j,n}$ obeys
\begin{equation}
	m_{j,n}^{NS,I}(\xi,\eta)=\Big(\int_\R e^{-i\frac{\xi-\eta}{2^j}t+i\frac{\eta}{2^{\a j}}t^\a}\rho(t)dt\Big)\psi\Big(\frac{\xi-\eta}{2^j}\Big)\phi\Big(\frac{\a\,\eta}{2^{\a j+n}}\Big)\nonumber\,.
\end{equation}
Integrating now by parts once we obtain\footnote{Again, we ignore here the error terms.}
\begin{equation}\label{ns1}
	m_{j,n}^{NS,I}(\xi,\eta)\approx\frac{1}{2^n}\,\Big(\int_\R e^{-i\frac{\xi-\eta}{2^j}t+i\frac{\eta}{2^{\a j}}t^\a}\rho(t)dt\Big)\psi\Big(\frac{\xi-\eta}{2^j}\Big)\phi\Big(\frac{\a\,\eta}{2^{\a j+n}}\Big)\,.
\end{equation}

\subsubsection{The $\L^{NS,I}_{H}$ component}\label{lns1h}

In this situation, given that
\begin{align}\label{ns10L}
\L^{NS,I}_{H,n}:=\sum_{{(\a-1)j\geq 0}\atop{j\geq C_0}} \L^{NS,I}_{j,n}:=\sum_{{(\a-1)j\geq 0}\atop{j\geq C_0}}\:\sum_{m<0} \L_{j,m,n}\,,
\end{align}
we exploit the information  $(\a-1)j\geq 0\geq-n$ in order to discretize the multiplier in \eqref{ns11} as
\begin{equation}\label{saljmm0nl1}
		m_{j,n}^{NS,I}(\xi,\eta)\approx\sum_{l\sim 2^{(\a-1)j+n}} m_{j,n,l}^{NS,I}(\xi,\eta)
\end{equation}
with
\begin{equation}\label{saljmm1nl}
		m_{j,n,l}^{NS,I}(\xi,\eta):=\frac{1}{2^n}\left(\int_\R e^{-i\frac{\xi-\eta}{2^j}  t}e^{i\frac{\eta}{2^{\a j}}t^\a}\rho(t)dt\right)
		\phi\Big(\sdfrac{\xi}{2^{j}}-l\Big)\,\phi\Big(\sdfrac{\eta}{2^{j}}-l\Big)\,\phi\Big(\sdfrac{\xi+\eta}{2^{j}}-2l\Big)\,.
\end{equation}

As a consequence we have
\begin{equation}\label{saljmsumnl}
\L^{NS,I}_{j,n}(f,g,h)=\sum_{l\sim 2^{(\a-1)j+n}}\L^{NS,I,l}_{j,n}(f,g,h)\,,
\end{equation}
with\footnote{Recall the notation for $\phi_{j}^l$ in \eqref{s2jm3}.}
\begin{equation}\label{saljm2nl}
	\L^{NS,I,l}_{j,n}(f,g,h)\approx\frac{1}{2^n}\int_{\R^2}  (f*\check{\phi}_{j}^l)\Big(x-\sdfrac{t}{2^j}\Big)\,(g*\check{\phi}_{j}^l)\Big(x+\sdfrac{t}{2^j}+\sdfrac{t^\a}{2^{\a j}}\Big)\,(h*\check{\phi}_{j}^{2l})(x)\,\rho(t)\,dt\,dx\,.
\end{equation}

Once at this point, we can apply similar reasonings with the ones presented in \eqref{s2jm3}--\eqref{redal} to deduce that
\begin{equation}\label{saljm20nl}
	|\L^{NS,I,l}_{j,n}(f,g,h)(f,g,h)|\lesssim
\end{equation}
$$\frac{1}{2^n}\,\int_{\R^2}  \Big|(f*\check{\phi}_{j}^l)\Big(x-\sdfrac{t}{2^j}\Big)\Big|\,\Big|(g*\check{\phi}_{j}^l)\Big(x+\sdfrac{t}{2^j}\Big)\Big|\,
\left|(h*\check{\phi}_{j}^{2l})(x)\right|\,|\rho(t)|\,dt\,dx$$
$$\approx \sum_{z\in\Z}\frac{1}{2^n} \Big|(f*\check{\phi}_{j}^l)\Big(\frac{z}{2^{j}}\Big)\Big|\,\Big|(g*\check{\phi}_{j}^l)\Big(\frac{z}{2^{j}}\Big)\Big|\,
\Big|(h*\check{\phi}_{j}^{2l})\Big(\frac{z}{2^{j}}\Big)\Big|\,\frac{1}{2^{j}}\,.$$

Now, with the notations from Section \ref{LGCLfull}, we further deduce that
\begin{equation}\label{lnl}
|\L^{NS,I}_{H,n}(f,g,h)|\lesssim \frac{1}{2^n}\,\sum_{j\geq 0}\sum_{{z\in\Z}\atop{l\sim 2^{(\a-1)j+n}}}\frac{1}{|I_{j}|^{\frac{1}{2}}}\,|\langle f,  \phi_{P_{j}(z,l)}\rangle|\,
|\langle g,  \phi_{P_{j}(z,l)}\rangle|\,|\langle h,  \phi_{P_{j}(z,2l)}\rangle|\,.
\end{equation}

At this point it remains to notice that \eqref{lnl} may be regarded as the analogue of \eqref{lsl1} in the situation $m=0$ and with an extra decaying factor of $\frac{1}{2^n}$. Thus, following the same reasoning as in Section \ref{LGCLfull}, one concludes that
\begin{equation}\label{flrangenll}
|\L^{NS,I}_{H,n}|\lesssim_{p,q} \frac{1}{2^n}\,\|f\|_p\,\|g\|_q\,\|h\|_{r'}
\end{equation}
for any $\frac{1}{p}+\frac{1}{q}+\frac{1}{r'}=1$ with $1\leq p,\,q\leq \infty$ and $r>\frac{2}{3}$.

\subsubsection{The $\L^{NS,I}_{NL}$ component}
In this situation we have
\begin{align}\label{ns10NL}
\L^{NS,I}_{NL,n}:=\sum_{(\a-1)j<-n} \L^{NS,I}_{j,n}:=\sum_{(\a-1)j<-n}\:\sum_{m<0} \L_{j,m,n}\,,
\end{align}
and thus, since $(\a-1)j< -n$, the multiplier in \eqref{ns11} may be reduced to
\begin{equation}\label{ns11NL}
m_{j,n}^{NS,I}(\xi,\eta)\approx\frac{1}{2^n}\,\Big(\int_\R e^{-i\frac{\xi}{2^j}t+i\frac{\eta}{2^{\a j}}t^\a}\rho(t)\,dt\Big)\psi\Big(\frac{\xi}{2^j}\Big)\phi\Big(\frac{\eta}{2^{\a j+n}}\Big)\,.
\end{equation}
After some standard Taylor series arguments we deduce that the main term in \eqref{ns11NL} is given by an expression of the form\footnote{For notational simplicity, here as in several other places before, we abuse the notation and leave the functions $\psi$, $\phi$ unchanged, though, strictly speaking, this need not be the case.}
\begin{equation}\label{ns11NL1}
m_{j,n}^{NS,I}(\xi,\eta)\approx\frac{1}{2^n}\,\psi\Big(\frac{\xi}{2^j}\Big)\,\phi\Big(\frac{\eta}{2^{\a j+n}}\Big)\,.
\end{equation}
Once at this point, we appeal to an Abel summation argument in order to deduce
\begin{align}
\sum_{(\a-1)j<-n}\hspace{-.4cm} 2^n\,m_{j,n}^{NS,I}(\xi,\eta)&=\sum_{s\le -\frac{n}{\a-1}} \hspace{-.2cm} \theta\Big(\frac{\xi}{2^s}\Big) \sum_{j\le s}\,\phi\Big(\frac{\eta}{2^{\a j+n}}\Big)  + \psi\Big(\frac{\xi}{2^{-\frac{n}{\a-1}+1}}\Big)\hspace{-.2cm} \sum_{j\le -\frac{n}{\a-1}} \hspace{-.2cm} \phi\Big(\frac{\eta}{2^{\a j+n}}\Big)\nonumber\\
&= \sum_{s\le -\frac{n}{\a-1}} \theta\Big(\frac{\xi}{2^s}\Big) \nu\Big(\frac{\eta}{2^{s\a}}\Big)\,+\,\psi\Big(\frac{\xi}{2^{-\frac{n}{\a-1}+1}}\Big)\,
\nu\Big(\frac{\eta}{2^{-\frac{\a n}{\a-1}}}\Big)\nonumber\\
&=:2^n\, m^{NS,I,1}_{n}(\xi,\eta)\,+\, 2^n\,m^{NS,I,2}_{n}(\xi,\eta)\,,\label{mNLIIns}
\end{align}
where here we used the notation
\begin{align*}
	\psi\Big(\frac{t}{2^j}\Big) &= \sum_{s=j}^{-\frac{n}{\a-1}} \theta\Big(\frac{t}{2^s}\Big) + \psi\Big(\frac{t}{2^{-\frac{n}{\a-1}+1}}\Big)\,,
	\intertext{and}
	\nu(\eta)&:=\sum_{j\le 0} \phi\Big(\frac{\eta}{2^{\a j+n}}\Big)\,.
\end{align*}

From this point one can apply standard paraproduct theory in order to deduce that the analogue of
\eqref{flrangenll} holds with $\L^{NS,I}_{L,n}$ replaced by $\L^{NS,I}_{NL,n}$.

\subsubsection{The $\L^{NS,I}_{TR}$ component}\label{lns1tr}
This is a direct consequence of the work in Section \ref{NSIterm}. Indeed, given the range $-n\leq (\a-1)j<0$,  it is enough to notice that, \emph{uniformly for $j\in \Z$}, we have that
\begin{equation}\label{flrangenl}
|\L^{NS,I}_{j,n}|\lesssim_{p,q} \frac{1}{2^n}\,\|f\|_p\,\|g\|_q\,\|h\|_{r'}
\end{equation}
for any $n\in\N$ and $\frac{1}{p}+\frac{1}{q}+\frac{1}{r'}=1$ with $1\leq p,\,q\leq \infty$ and $r>\frac{1}{2}$.

\subsection{Treatment of $\L^{NS,II}$}\label{lns2}

In this situation we know that $0\leq m<n-C_0$ and hence applying an integration by parts on the integrand component of the multiplier
\begin{equation}\label{mjmn-NSII}
 m_{j,m,n}(\xi,\eta):=\left(\int_\R e^{-i\frac{\xi-\eta}{2^j}  t}e^{i\frac{\eta}{2^{\a j}}t^\a}\rho(t)dt \right) \phi\Big(\frac{\xi-\eta}{2^{j+m}}\Big)\phi\Big(\frac{\eta}{2^{\a j+n}}\Big)
\end{equation}
we get that
%
\begin{equation}
m_{j,m,n}(\xi,\eta)=\frac{1}{2^n}\sum\limits_{\ell=0}^\infty\frac{1}{2^{\ell(n-m)}} \Big(\int\limits_\R e^{i\varphi_{j,\xi,\eta}(t)}i\,\rho_{\a,\ell}(t)dt\Big)\widetilde{\phi_\ell}\big(\frac{\xi-\eta}{2^{j+m}}\big)\widetilde{\phi}_{_{-(\ell+1)}}\big(\frac{\eta}{2^{\a j+n}}\big)
\end{equation}
where $\tilde\phi_\ell(x):=x^\ell \phi(x)$ and $\rho_{\a,\ell}$ is a smooth functions with compact support on $\supp\rho$, and satisfies $\|\rho_{\a,\ell}\|_{L^\infty}\lesssim_{\a,\ell,\rho} C^\ell$ for some suitable $C>1$. Now due to the fast decay in $\ell$ it will be enough to only consider the term $\ell=0$. In conclusion it remains to focus on
\begin{equation}\label{mjmnaNS}
m_{j,m,n}(\xi,\eta)\approx\frac{1}{2^n} \left(\int_\R e^{-i\frac{\xi-\eta}{2^j}  t}e^{i\frac{\eta}{2^{\a j}}t^\a}\rho(t)dt\right)\,\phi\Big(\frac{\xi-\eta}{2^{j+m}}\Big)\phi\Big(\frac{\eta}{2^{\a j+n}}\Big)\:.
\end{equation}	

\subsubsection{The $\L^{NS,II}_{H}$ component}

In this situation we have that
\begin{align}\label{ns10NSL}
\L^{NS,II}_{H,m,n}:=\sum_{{(\a-1)j\geq m}\atop{j\geq C_0}} \L_{j,m,n}\,,
\end{align}
and proceeding as in Section \ref{LGCNLfull} we discretize the multiplier in \eqref{mjmnaNS} as
\begin{equation}\label{saljmm0nl1NSL}
		m_{j,m,n}(\xi,\eta)\approx\sum_{l\sim 2^{(\a-1)j+n-m}} m_{j,m,n,l}(\xi,\eta)
\end{equation}
with
\begin{equation}\label{saljmm1nlNSL}
		m_{j,m,n,l}(\xi,\eta)\!:=\frac{1}{2^n}\!\left(\int_\R e^{-i\frac{\xi-\eta}{2^j}  t}e^{i\frac{\eta}{2^{\a j}}t^\a}\!\rho(t)dt\right)\!
		\phi\Big(\sdfrac{\xi}{2^{j+m}}-l\Big)\phi\Big(\sdfrac{\eta}{2^{j+m}}-l\Big)\phi\Big(\sdfrac{\xi+\eta}{2^{j+m}}-2l\Big)\,.
\end{equation}
Since $(\a-1)j\geq m$, we have
\begin{equation}\label{saljmsumnlNSL}
\L_{j,m,n}(f,g,h)=\sum_{l\sim 2^{(\a-1)j+n-m}}\L_{j,m,n,l}(f,g,h)\,,
\end{equation}
with
\begin{equation}\label{saljm2nlNSL}
\L_{j,m,n,l}(f,g,h)\approx\frac{1}{2^n}\!\int_{\R^2} \! (f*\check{\phi}_{j+m}^l)\Big(x-\sdfrac{t}{2^j}\Big)(g*\check{\phi}_{j+m}^l)\Big(x+\sdfrac{t}{2^j}\Big)(h*\check{\phi}_{j+m}^{2l})(x)\rho(t)\,dt\,dx\,.
\end{equation}

Applying now similar reasonings with those in Section \ref{LGCNLfull} we deduce that
\begin{equation}\label{lnl2}
	\L^{NS,II}_{H,m,n}\lesssim
\end{equation}
$$\frac{1}{2^n} \cdot\frac{1}{2^m} \sum_{s\sim 2^m}\sum_{\substack{(\a-1)j\geq m\\[.1ex] z\in\Z\\[.1ex] l\sim 2^{(\a-1)j+n-m}}} \frac{1}{|I_{j+m}|^{\frac{1}{2}}}\, |\langle f,  \phi_{P_{j+m}(z-s,l)}\rangle|\,
|\langle g,  \phi_{P_{j+m}(z+s,l)}\rangle|\,|\langle h,  \phi_{P_{j+m}(z,2l)}\rangle|\,.$$

Not surprisingly, \eqref{lnl2} serves as the direct analogue of \eqref{lsl} with the expression $|\L^{S}_{H}(f,g,h)|$ therein modified such that $l\sim 2^{(\a-1)j}$ is substituted by $l\sim 2^{(\a-1)j+n-m}$ with an additional $\frac{1}{2^n}$ decaying factor. Thus, applying the same reasoning as in Section \ref{LGCLfull}, one concludes that
\begin{equation}\label{flrangenll1}
|\L^{NS,II}_{H,m,n}|\lesssim_{p,q} \frac{1}{2^n}\,\|f\|_p\,\|g\|_q\,\|h\|_{r'}
\end{equation}
for any $\frac{1}{p}+\frac{1}{q}+\frac{1}{r'}=1$ with $1\leq p,\,q\leq \infty$ and $r>\frac{2}{3}$.

\subsubsection{The $\L^{NS,II}_{NL}$ component}

In this situation we have that
\begin{align}\label{ns10NSNL}
\L^{NS,II}_{NL,m,n}:=\sum_{{(\a-1)j<-n}\atop{j\leq -C_0}} \L_{j,m,n}\,.
\end{align}
Based on \eqref{mjmnaNS} and $(\a-1)j<-n$ we deduce
\begin{equation}\label{saljm2nlNSNL}
\L_{j,m,n}(f,g,h)\approx\!\frac{1}{2^n}\!\int_{\R^2}\!  (f*\check{\phi}_{j+m})\Big(x-\sdfrac{t}{2^j}\Big)(g*\check{\phi}_{\a j+n})\Big(x+\sdfrac{t^{\a}}{2^{\a j}}\Big)(h*\check{\phi}_{j+m})(x)\rho(t)\,dt\,dx\,.
\end{equation}
From this point on, one can follow the steps as in Section \ref{LGCNLfull} and write
\begin{equation}\label{saljm20NSNL}
	|\L_{j,m,n}(f,g,h)|\lesssim
\end{equation}
$$\sdfrac{1}{2^{2n}}\!\!\sum_{\substack{z\in\Z\\s\sim 2^n}}\! \Big|(f*\check{\phi}_{j+m})\big(\sdfrac{z}{2^{j+m}}{-}\sdfrac{s}{2^{j+n}}\big)\Big|
\Big|(g*\check{\phi}_{\a j+n}) \big(\sdfrac{z}{2^{j+m}}{+}\sdfrac{s^{\a}}{2^{\a j+\a n}}\big)\Big|
\Big|(h*\check{\phi}_{j+m})\big(\sdfrac{z}{2^{j+m}}\big)\Big| \sdfrac{1}{2^{j+m}}\,,$$
which further implies
\begin{equation}\label{lnslrNSNL}
|\L^{NS,II}_{NL,m,n}(f,g,h)|\lesssim \frac{1}{2^n}\sum_{s\sim 2^n} \L^{NS,II,s}_{NL,m,n}(f,g,h)\,,
\end{equation}
where
\begin{equation}\label{lnslreNSNL}
	\L^{NS,II,s}_{NL,m,n}(f,g,h):=
\end{equation}
\\[-4.5ex]
$$\frac{1}{2^n}\sum_{\substack{(\a-1)j<-n\\ u\in\Z}}\hspace{-.3cm}
\frac{|\langle g,  \phi_{P_{\a j+n}(u+\frac{s^{\a}}{2^{(\a-1) n}})}\rangle|}{|I_{\a j+n}^u|^{\frac{1}{2}}} \hspace{-.2cm}\sum_{I_{j+m}^v\subseteq I_{\a j+ n}^u}\hspace{-.2cm}
|\langle f,  \phi_{P_{j+m}(v-\frac{s}{2^{n-m}})}\rangle|\,|\langle h,  \phi_{P_{j+m}(v)}\rangle|\,.$$

Now one can apply the same reasonings as in Section \ref{LGCNLfull}, see the analogy between \eqref{lnslrNSNL}-\eqref{lnslreNSNL} and \eqref{lnslre1}-\eqref{lnslre}, in order to obtain the desired result.

\subsubsection{The $\L^{NS,II}_{TR}$ component}
As before we get uniformly in $0\leq m\leq n-C_0$ and $j\in \Z$ that
\begin{equation}\label{flrangenlq}
|\L^{NS,II}_{j,,m,n}|\lesssim_{p,q} \frac{1}{2^n}\,\|f\|_p\,\|g\|_q\,\|h\|_{r'}
\end{equation}
for any $\frac{1}{p}+\frac{1}{q}+\frac{1}{r'}=1$ with $1\leq p,\,q\leq \infty$ and $r>\frac{1}{2}$. Conclude that
\begin{equation}\label{flrangenlq0}
|\L^{NS,II}_{TR,n}|\lesssim_{p,q} \frac{n}{2^n}\,\|f\|_p\,\|g\|_q\,\|h\|_{r'}\;,
\end{equation}
which trivially implies our desired control over $\L^{NS,II}_{TR}$.

\subsection{Treatment of $\L^{NS,III}$} In this situation we know that $0\leq n<m-C_0$ and hence applying the same steps as in Section \ref{lns2} (with the same notation) we have in a first instance
\begin{equation}
	m_{j,m,n}(\xi,\eta)=-\sdfrac{1}{2^m}\sum\limits_{\ell=0}^\infty\sdfrac{1}{2^{\ell(m-n)}} \bigg(\int\limits_\R e^{i\varphi_{j,\xi,\eta}(t)}\tilde\rho_{\a,\ell}(t)dt\bigg)\widetilde{\phi}_{-(\ell+1)}\Big(\sdfrac{\xi-\eta}{2^{j+m}}\Big) \widetilde{\phi_\ell}\Big(\sdfrac{\eta}{2^{\a j+n}}\Big)
\end{equation}
which, via standard reasonings, can be reduced to the main term expressed as
\begin{equation}\label{mjmnaNS4}
	m_{j,m,n}(\xi,\eta)\approx\frac{1}{2^m} \left(\int_\R e^{-i\frac{\xi-\eta}{2^j}  t}e^{i\frac{\eta}{2^{\a j}}t^\a}\rho(t)dt\right)\,\phi\Big(\frac{\xi-\eta}{2^{j+m}}\Big)\phi\Big(\frac{\eta}{2^{\a j+n}}\Big)\:.
\end{equation}	

From this point on the treatments of $\L^{NS,III}_{H}$, $\L^{NS,III}_{NL}$, $\L^{NS,III}_{TR}$ are similar to the ones of $\L^{NS,II}_{H}$, $\L^{NS,II}_{NL}$ and $\L^{NS,III}_{TR}$, respectively.

\subsection{Treatment of $\L^{NS,IV}$}
In this situation, for a fixed $m\in\N$, we have
\begin{align}\label{nsIV}
\L^{NS,IV}_{m}:=\sum_{j\in\Z} \L^{NS,IV}_{j,m}:=\sum_{j\in\Z}\:\sum_{n<0} \L_{j,m,n}\,.
\end{align}
Now, proceeding as at the beginning of Section \ref{lns1}, we may assume
\begin{equation}\label{nsNSL}
	m_{j,m}^{NS,IV}(\xi,\eta)(\xi,\eta)\approx\frac{1}{2^m}\,\Big(\int_\R e^{-i\frac{\xi-\eta}{2^j}t+i\frac{\eta}{2^{\a j}}t^\a}\rho(t)dt\Big)\phi\Big(\frac{\xi-\eta}{2^{j+m}}\Big)
\psi\Big(\frac{\eta}{2^{\a j}}\Big)\,.
\end{equation}

\subsubsection{The $\L^{NS,IV}_{H}$ component}

In this situation, we write
\begin{align}\label{ns4NSL}
\L^{NS,IV}_{H,m}:=\sum_{{(\a-1)j\geq m}\atop{j\geq C_0}} \L^{NS,IV}_{j,m}:=\sum_{{(\a-1)j\geq m}\atop{j\geq C_0}}\:\sum_{n<0} \L_{j,m,n}\,,
\end{align}

Following similar steps with those in Section \ref{lns1h}, one concludes
\begin{equation}\label{lnlNSNL}
	|\L^{NS,IV}_{L,m}(f,g,h)|\lesssim
\end{equation}
$$\frac{1}{2^{2m}}\sum_{s\sim 2^m}\sum_{\substack{(\a-1)j\geq m\\[.2ex] z\in\Z\\[.2ex] l\sim 2^{(\a-1)j-m} }}
\frac{1}{|I_{j+m}|^{\frac{1}{2}}}\,|\langle f,  \phi_{P_{j+m}(z-s,l)}\rangle|\,
|\langle g,  \phi_{P_{j+m}(z+s,l)}\rangle|\,|\langle h,  \phi_{P_{j+m}(z,2l)}\rangle|\,,$$
which immediately implies
\begin{equation}\label{flrangenllNSNL}
|\L^{NS,IV}_{L,m}|\lesssim_{p,q} \frac{1}{2^m}\,\|f\|_p\,\|g\|_q\,\|h\|_{r'}
\end{equation}
for $p,q$ and $r$ in the same range as in the statement of Proposition \ref{Lgennodecay}.

\subsubsection{The $\L^{NS,IV}_{NL}$ component}

In this situation we have
\begin{align}\label{ns4NSNL}
\L^{NS,IV}_{NL,m}:=\sum_{{(\a-1)j\leq 0}\atop{j\leq - C_0}} \L^{NS,IV}_{j,m}:=\sum_{{(\a-1)j\leq 0}\atop{j\leq - C_0}}\:\sum_{n<0} \L_{j,m,n}\,.
\end{align}
Using $(\a-1)j\leq 0$ the multiplier in \eqref{nsNSL} may be reduced to
\begin{equation}\label{ns4NL}
m_{j,m}^{NS,IV}(\xi,\eta)\approx\frac{1}{2^m}\,\Big(\int_\R e^{-i\frac{\xi}{2^j}t+i\frac{\eta}{2^{\a j}}t^\a}\rho(t)dt\Big)\phi\Big(\frac{\xi}{2^{j+m}}\Big)\psi\Big(\frac{\eta}{2^{\a j}}\Big)\,.
\end{equation}
Proceeding as in Section \ref{LGCNLfull}, we have that
\begin{equation}\label{saljm20n-1}
	|\L^{NS,IV}_{j,m}(f,g,h)|\lesssim
\end{equation}
$$\frac{1}{2^m}\,\int_{\R^2}  \left|(f*\check{\phi}_{j+m})\left(x-\sdfrac{t}{2^j}\right)\right|\,
\left|(g*\check{\psi}_{\a j})\left(x+\sdfrac{t^\a}{2^{\a j}}\right)\right|\,
\left|(h*\check{\phi}_{j+m})(x)\right|\,|\rho(t)|\,dt\,dx$$
$$\approx\frac{1}{2^{2m}}\sum_{{s\sim 2^m}\atop{z\in\Z}}\, \left|(f*\check{\phi}_{j+m})\left(\sdfrac{z-s}{2^{j+m}}\right)\right|\,\left|(g*\check{\psi}_{\a j})\left(\sdfrac{z}{2^{j+m}}\right)\right|\,
\left|(h*\check{\phi}_{j+m})(\frac{z}{2^{j+m}})\right|\,\frac{1}{2^{j+m}}\,.$$

From this point on one can follow a similar strategy with the one exposed at the previous steps in order to obtain the desired conclusion. We leave all these details to the interested reader.

\subsubsection{The $\L^{NS,IV}_{TR}$ component}

One follows the same approach as the one in Section \ref{lns1tr}.

%

\section{The analysis of the low oscillatory term $\L^{LO}$}\label{LOf}

\subsection{Multiplier analysis}
In this section we address the operator with multiplier defined as in \eqref{ld}. Since in the current context we have $|\frac{\a\,\eta}{2^{\a j}}|\lesssim 1$ one can apply a Taylor expansion, together  with \eqref{ref:unity-phi} and \eqref{ref:unity-psi} in order to write
\begin{align}
	m_j^{LO}&=\sum_{(m,n)\in \Z_{-}\times\Z_{-}}\hspace{-.5cm} m_j(\xi,\eta)\phi\Big(\sdfrac{\xi-\eta}{2^{j+m}}\Big)\phi\Big(\sdfrac{\eta}{2^{\a j+n}}\Big)\!=\!\sum_{n\in \Z_-} m_j(\xi,\eta)\,\psi\Big(\sdfrac{\xi-\eta}{2^{j}}\Big)\,\phi\Big(\sdfrac{\eta}{2^{\a j+n}}\Big)\nonumber\\
	&=\sum_{n\in\Z_-}\sum_{\ell=0}^\infty \frac{i^\ell}{\ell\,!}\,2^{n\ell} \Big(\int_\R e^{-i\frac{\xi-\eta}{2^j}t} \rho_{\a,\ell}(t)dt\Big)\, \psi\Big(\frac{\xi-\eta}{2^{j}}\Big)\,\widetilde{\phi}_\ell\Big(\frac{\eta}{2^{\a j+n}}\Big)\nonumber\\
	&=\hat{\rho} \Big(\sdfrac{\xi-\eta}{2^j}\Big)\psi\Big(\sdfrac{\xi-\eta}{2^{j}}\Big)\psi\Big(\sdfrac{\eta}{2^{\a j}}\Big)
	+ \sum_{\substack{n\in\Z_-\\ \ell\ge 1}}\! \frac{i^\ell}{\ell\,!}\, 2^{n\ell}\, \widehat{\rho_{\a,\ell}} \Big(\sdfrac{\xi-\eta}{2^j}\Big)\,\psi\Big(\sdfrac{\xi-\eta}{2^{j}}\Big)\, \widetilde{\phi}_\ell\Big(\sdfrac{\eta}{2^{\a j+n}}\Big), \raisetag{1\baselineskip}\label{ref:Taylor}
\end{align}
where $\rho_{\a,\ell}(t):=\big(\frac{t^\a}{\a}\big)^\ell \rho(t)$ and $\widetilde{\phi}_\ell(x)=x^\ell \phi(x)$. Since $|\widehat{\rho_{\a,\ell}}|\le C_{\a}^\ell$ for some positive constant $C_{\a}$, and $n\in\Z_-$, the sum in \eqref{ref:Taylor} is absolutely convergent. In particular, we write the multiplier $m_j^{LO}$ as
\begin{equation}\label{Taylor-mjL0}
	m_j^{LO}(\xi,\eta)=m_j^{LO,0}(\xi,\eta)+m_j^{LO,1}(\xi,\eta)\,,\\
\end{equation}
where
\begin{align}
	m_j^{LO,0}(\xi,\eta)&:=\hat{\rho} \Big(\frac{\xi-\eta}{2^j}\Big)\,\psi\Big(\frac{\xi-\eta}{2^{j}}\Big)\, \psi\Big(\frac{\eta}{2^{\a j}}\Big)\,,\label{ref:mjL0}
\end{align}
and
\begin{align}
	m_j^{LO,1}(\xi,\eta)&:=\sum_{n\in\Z_-} \sum_{\ell=1}^\infty \frac{i^\ell}{\ell\,!} 2^{n\ell}\, \widehat{\rho_{\a,\ell}} \Big(\frac{\xi-\eta}{2^j}\Big)\,\psi\Big(\frac{\xi-\eta}{2^{j}}\Big)\, \widetilde{\phi}_\ell\Big(\frac{\eta}{2^{\a j+n}}\Big)\,.\label{ref:mjLO1}
\end{align}

Due to the extra decay in $\ell$ in \eqref{ref:mjLO1}, and in $n$, one may assume without loss of generality that $\ell=1$ and $n=-1$, and thus, ignoring absolute constants, that\footnote{We maintain the same notation $\phi$ and disregard the dependence on $\ell=1$.}
\begin{equation}
	m_j^{LO,1}(\xi,\eta)=\widehat{\rho_{\a,1}} \Big(\frac{\xi-\eta}{2^j}\Big)\,\psi\Big(\frac{\xi-\eta}{2^{j}}\Big)\, \phi\Big(\frac{\eta}{2^{\a j}}\Big)\,.
\end{equation}

Finally, since $\rho$ obeys the mean zero condition, standard considerations reduce the analysis of the above multipliers to the following two situations:\footnote{Throughout this section we allow the functions $\psi$ and $\phi$ to change from line to line with the preservation of their key properties: both are smooth, compactly supported, and $0\notin\textrm{supp}\,\phi$.}
\begin{align}
	m_j^{LO,0}(\xi,\eta)&:=\phi\Big(\frac{\xi-\eta}{2^{j}}\Big)\, \psi\Big(\frac{\eta}{2^{\a j}}\Big)\,,\label{ref:mjL00}
	\intertext{and}
	m_j^{LO,1}(\xi,\eta)&:=\psi\Big(\frac{\xi-\eta}{2^{j}}\Big)\, \phi\Big(\frac{\eta}{2^{\a j}}\Big)\,.\label{ref:mjL01}
\end{align}

With the obvious correspondences we write
\begin{equation}
\L_{L}^{LO}=\L_{L}^{LO,0}\,+\,\L_{L}^{LO,1}\qquad\textrm{and}\qquad \L_{D}^{LO}=\L_{D}^{LO,0}\,+\,\L_{D}^{LO,1}\,.\label{LOtr}
\end{equation}

\subsection{Treatment of $\L_{L}^{LO}$}
The first component $\L_{L}^{LO,0}$ is essentially a ``scale-truncated" version of the  bili\-near Hilbert transform while the second component $\L_{L}^{LO,1}$ may be essentially reduced to a paraproduct (or, alternatively, one may proceed as for the treatment of $\L^{S}_{L,m}$, for $m=0$, in Section \ref{LGCLfull}).

In what follows we will only provide a very brief outline of the above.

\subsubsection{The $\L_{L}^{LO,0}$ component} In this case we write\footnote{Here we abuse the notation by allowing the first $C_0$ terms in the definition of \eqref{llo}.}
\begin{equation}\label{losl0}
\L_{L}^{LO,0}=\sum_{j\geq 0} \L_{L,j}^{LO,0}\,,
\end{equation}
where the multiplier for each $\L_{L,j}^{LO,0}$ is given by \eqref{ref:mjL00}. We next decompose
\begin{equation}\label{salojmm0}
		m_j^{LO,0}(\xi,\eta)\approx\sum_{|l|\leq 2^{(\a-1)j}} m_{j,l}^{LO}(\xi,\eta)
\end{equation}
with
\begin{equation}\label{saljmm1lo}
		m_{j,l}^{LO}(\xi,\eta):=
		\phi\Big(\frac{\xi}{2^{j}}-l-1\Big)\,\phi\Big(\frac{\eta}{2^{j}}-l\Big)\,\phi\Big(\frac{\xi+\eta}{2^{j}}-2l\Big)\,,
\end{equation}
where $\phi$ is a suitable smooth compactly supported function.	

As a consequence we deduce
\begin{equation}\label{saljmsumlo}
\L_{L,j}^{LO,0}(f,g,h)=\sum_{|l|\leq 2^{(\a-1)j}}\L_{L,j,l}^{LO,0}(f,g,h)\,,
\end{equation}
with
\begin{equation}\label{saljm2-2}
\L_{L,j,l}^{LO,0}(f,g,h):=\int_{\R^2}  (f*\check{\phi}_{j}^{l+1})\Big(x-\sdfrac{t}{2^j}\Big)\,(g*\check{\phi}_{j}^l)\Big(x+\sdfrac{t}{2^j}\Big)\,(h*\check{\phi}_{j}^{2l})(x)\,\rho(t)\,dt\,dx\,.
\end{equation}

Now, with the notations from Section \ref{LGCLfull}, we have
\begin{align}\label{saljmsumwplo}
|\L_{L,j}^{LO,0}(f,g,h)|\lesssim\sum_{\substack{z\in\Z\\[.2ex] |l|\leq 2^{(\a-1)j}}} \frac{1}{|I_{j}|^{\frac{1}{2}}}\,|\langle f,  \phi_{P_{j}(z,l+1)}\rangle|\,
|\langle g,  \phi_{P_{j}(z,l)}\rangle|\,|\langle h,  \phi_{P_{j}(z,2l)}\rangle|\,,\raisetag{1\baselineskip}
\end{align}
which is very close in spirit to \eqref{saljmsumwp} for the case $m=0$ with the key distinction that in the present situation $|l|\leq 2^{(\a-1)j}$ instead of $l\sim 2^{(\a-1)j}$ therein. This latter observation brings in the present situation an overlapping of the time-frequency localization of the input functions as we move through the scales hence the necessity to appeal to the tree structures. This is of course not surprising since $\L_{L}^{LO,0}$ defined by \eqref{losl0}--\eqref{saljm2-2} encapsulates precisely the Bilinear Hilbert transform type behavior.

From this point on one can apply the standard techniques---see e.g. \cite{LT97}, \cite{LT99}, \cite{MS13}---in order to obtain the desired bounds for $\L_{L}^{LO,0}$.

\subsubsection{The $\L_{L}^{LO,1}$ component}

In this case we have
\begin{equation}\label{losl01}
\L_{L}^{LO,1}=\sum_{j\geq 0} \L_{L,j}^{LO,1}\,,
\end{equation}
where the multiplier for each $\L_{L,j}^{LO,1}$ is given by \eqref{ref:mjL01}. We next decompose
\begin{equation}\label{salojmm01}
		m_j^{LO,1}(\xi,\eta)\approx\sum_{l\sim 2^{(\a-1)j}} m_{j,l}^{LO}(\xi,\eta)
\end{equation}
with $ m_{j,l}^{LO}$ defined by \eqref{saljmm1lo}.

Now following the same steps as the ones from the previous section we have
\begin{equation}\label{saljmsumwplo1}
|\L_{L}^{LO,1}(f,g,h)|\lesssim\sum_{j\ge 0} \sum_{\substack{z\in\Z\\l\sim 2^{(\a-1)j}}}\frac{1}{|I_{j}|^{\frac{1}{2}}}\,|\langle f,  \phi_{P_{j}(z,l+1)}\rangle|\,
|\langle g,  \phi_{P_{j}(z,l)}\rangle|\,|\langle h,  \phi_{P_{j}(z,2l)}\rangle|\,,
\end{equation}
which corresponds precisely to $\L^{S}_{L,m}$ for $m=0$ as displayed in \eqref{lsl1}.

\subsection{Treatment of $\L_D^{LO}$}

\subsubsection{The $\L_D^{LO,0}$ component}
Since in this situation we focus on the range $j<0$ (and also recall that we assume wlog that $\a>1$) we immediately deduce that
\begin{align}
	m_j^{LO,0}(\xi,\eta)\approx\phi\Big(\frac{\xi}{2^{j}}\Big)\, \psi\Big(\frac{\eta}{2^{\a j}}\Big)\,.\label{ref:mjL001}
\end{align}

Once here one can see an immediate parallelism with the situation treated in Section \ref{LGCNLfull}. Indeed, it is simple to see that our form $\L_{NL}^{LO,0}$ may be identified with $\L^{S,s}_{D,m}$ in \eqref{lnslre1} for the case $s=m=0$. From this point on one can follow line by line the reasoning therein in order to conclude the desired bounds.

\subsubsection{The $\L_D^{LO,1}$ component} In this situation one can reduce the shape of $m_j^{LO,1}$ in \eqref{ref:mjL01} to

\begin{equation}
	m_j^{LO,1}(\xi,\eta)=\psi\Big(\frac{\xi}{2^{j}}\Big)\, \phi\Big(\frac{\eta}{2^{\a j}}\Big)\,.\label{ref:mjL011}
\end{equation}
This is a direct analogue of the multiplier $\eqref{ns11NL1}$ in the situation $n=0$. Therefore, by applying an Abel summation argument we deduce that
\begin{align}
	\sum_{j<0} m_j^{LO,1}(\xi,\eta)&= \sum_{s\le 0} \theta\Big(\frac{\xi}{2^s}\Big) \nu\Big(\frac{\eta}{2^{s\a}}\Big)\,+\,\psi\Big(\frac{\xi}{2}\Big)\,\nu(\eta)\nonumber\\
	&=:m_{D,I}^{LO,1}(\xi,\eta)+m_{D,II}^{LO,1}(\xi,\eta)\,,\label{mD-I-II}
\end{align}
where
\begin{equation}\label{theta}
	\theta(t):=\psi(t)-\psi\Big(\frac{t}{2}\Big)\,,
\end{equation}
satisfies $\theta\in C_0^\infty$ with $\supp\theta\subset\{t:\frac{1}{4}<|t|<4\}$.

Also we set
\begin{equation}\label{nu-jNeg}
	\nu(t):=\sum_{j<0}\phi(2^{-\a j}\,t)\,,
\end{equation}
and notice that $\supp\nu\subseteq \{t:|t|<4\}$.

%

Now it is immediate to see that the term corresponding to the multiplier $m_{D,I}^{LO,1}$  can be treated in a similar fashion with $\L_{D}^{LO,0}$ (indeed, this is a consequence of simply inspecting the expression of $m_{D,I}^{LO,1}$ with that provided by \eqref{ref:mjL001}) while the second term corresponding to $m_{D,II}^{LO,1}$ is essentially behaving as $f(x)\,Tg(x)$ with $T$ a suitable Calderon-Zygmund operator.

\section{The analysis of the non-singular term $\L^{NSI}$}\label{NSIterm}

The main result of this section is the following

\begin{theorem}\label{mainsinglsecale} Let $a\in\R\setminus\{-1\}$, $b\in\R$, $\a\in(0,\,\infty)\setminus\{1\}$ and $j\in \Z$. Defining
\begin{equation}\label{bhtj}
H_{j,a,b}^\a(f,g)(x):= \int_\R f(x-t)\, g(x+a\,t+b\,t^\a)\,2^j\,\rho(2^j t)\,dt\:,
\end{equation}
we have that
\begin{equation}\label{mainrmonj}
		\|H_{j,a,b}^\a(f,g)\|_{L^r}\lesssim_{a,b,\a,p,q} \|f\|_{L^p}\,\|g\|_{L^q}\:,
\end{equation}
holds uniformly in $j$, where here $\frac{1}{p}+\frac{1}{q}=\frac{1}{r}$, $p,\,q\geq 1$ and $r>\frac{1}{2}$.
\end{theorem}

\begin{proof}

Without loss of generality we may assume that $b\not=0$ as otherwise the proof below becomes straightforward. Define now the key quantity
\begin{equation}\label{K}
K(t):=a+1+\a\,b\,t^{\a-1}\:.
\end{equation}
By decomposing $\rho$ as a superposition of at most $10$ similar behaving functions we may assume wlog that
\begin{equation}\label{suprho}
\textrm{supp}\,\rho\subset \Big[\frac{3}{4},\,\frac{5}{4}\Big]\:.
\end{equation}
Since $K'(t)=\a(\a-1)\,b\,t^{\a-2}$ has no zero in $(0,\infty)$ the function $K$ is strictly monotone on $\R_{+}$ and thus there exists at most one $t_0>0$ such that $K(t_0)=0$; consequently, there exists at most one $j_0=j_0(a,b,\a)\in\Z$ such that $t_0\in \frac{11}{10}\,\big[\frac{3}{4}\,2^{-j_0},\,\frac{5}{4}\,2^{-j_0}\big]$.
\vspace{.2cm}

Fix now $j\in\Z$ and, given \eqref{suprho}, assume throughout our proof that
\begin{equation}\label{suprho1}
t\in\textrm{supp}\,\rho_{j}(\cdot)\subseteq J_j:=[2^{-j-1},\,2^{-j+1}]\:.
\end{equation}

Inspecting the arguments of $f$ and $g$ in \eqref{bhtj}, that is $x-t$ and $x+a\,t+b\,t^\a$, standard orthogonality arguments imply that for proving \eqref{mainrmonj} it is enough to restrict the range of $x$ within an interval $I_j$ of length $\max\{2^{-j},\,2^{-\a\,j}\}$ where in the last reasoning we made use of \eqref{suprho1}. Assume for the moment that $j\not=j_0$. Then, applying a Cauchy-Schwarz argument and the change of variable $y=x-t$ and $z=x+a\,t+b\,t^\a$, we deduce that\footnote{Here $c>0$ is some absolute constant depending only on $a,\,b,\,\a$. Also, for notational simplicity, throughout this section we ignore all the dependencies on the parameters $a,\,b,\,\a$.}
\begin{align}
	A_{j}:=\left(\int_{I_j}|H_{j,a,b}^\a(f,g)(x)|^{\frac{1}{2}}\,dx\right)^2
	&\lesssim |I_j|\,\iint\limits_{I_j\times J_j} |f(x-t)|\, |g(x+a\,t+b\,t^\a)|\,2^j\rho(2^j t)\,dt\,dx\nonumber\\[1ex]
	&\lesssim 2^j\,|I_j|\,\iint_{c I_j\times c I_j} |f(y)|\, |g(z)|\,\frac{dy\,dz}{|K(t(y,z)|)}\,,\label{AI}
\end{align}
where in the last line we used that the Jacobian satisfies
$$\mathcal{J}:=\left|\frac{\partial (y,z)}{\partial (x,t)}\right|=|K(t)|\,,$$
while $t=t(y,z)$ stands for the implicit function resulted from our change of variable.

We split now our discussion in three cases:

$\newline$
\noindent\textbf{Case 1:} $\:\:j<j_0$
$\newline$

In this situation, since $t\sim 2^{-j}$, we have that as $j$ approaches $-\infty$, $t$ becomes increasingly large and hence $t^\a$ dominates $t$. This further implies that
\begin{equation}\label{keyI}
|t|\lesssim 2^{-\a\,j} \approx |a\,t+b\,t^{\a}|\quad\textrm{hence}\quad |I_j|\approx 2^{-\a j} \quad\textrm{and}\quad |K(t)|\approx 2^{-(\a-1)\,j}\:.
\end{equation}

Putting together \eqref{AI} and \eqref{keyI} we get
\begin{equation}\label{AI1}
\|H_{j,a,b}^\a(f,g)\|_{\frac{1}{2}}\lesssim \|f\|_1\,\|g\|_1\,.
\end{equation}

$\newline$
\noindent\textbf{Case 2:} $\:\:j>j_0$
$\newline$

In this situation, as $j$ approaches $+\infty$, $t$ becomes increasingly small and hence $t^\a$ is dominated by $t$. Thus
\begin{equation}\label{keyII}
|t|\approx 2^{-\,j}\gtrsim |a\,t+b\,t^{\a}|\quad\textrm{hence}\quad |I_j|\approx 2^{-j}\quad\textrm{and}\quad |K(t)|\approx |a+1|\approx 1\:.
\end{equation}
Putting together \eqref{AI} and \eqref{keyII} we get the analogue of \eqref{AI1}
\begin{equation}\label{AI10}
\|H_{j,a,b}^\a(f,g)\|_{\frac{1}{2}}\lesssim \|f\|_1\,\|g\|_1\,.
\end{equation}

$\newline$
\noindent\textbf{Case 3:} $\:\:j=j_0$
$\newline$

This is a slightly more difficult case than the previous two due to the cancellation of the Jacobian at $t=t_0$. This requires a further Whitney decomposition of the expression $K(t)$ relative to its root $t_0$. Indeed using the fact that
\begin{equation}\label{Ks}
|K(t)-K(t_0)|=|\a\,b|\,|t^{\a-1}-t_0^{\a-1}|\approx |t-t_0|\, |t_0|^{\a-2}\approx |t-t_0|\:,
\end{equation}
for $s\in\N$, we set
\begin{equation}\label{Ds}
J_{j_0}^{s}:=\{t\in J_{j_0}\,|\,|t-t_0|\approx 2^{-s-j_0}\}\:,
\end{equation}
and notice that for $t\in J_{j_0}^{s}$ we have
\begin{equation}\label{keyIII}
|t|\approx 2^{-j_0}\approx 1\approx|a\,t+b\,t^{\a}|\quad\textrm{hence}\quad |I_{j_0}|\approx 1\quad\textrm{and}\quad |K(t)|\approx 2^{-s-j_0}\,2^{-j_0\,(\a-2)}\approx 2^{-s}\:.
\end{equation}
Let now
\begin{equation}\label{AIII}
B_{j_0}^{s}:= \iint_{I_{j_0}\times J_{j_0}^s} |f(x-t)|\, |g(x+a\,t+b\,t^\a)|\,2^{j_0}\,\rho(2^{j_0} t)\,dt\,dx\:.
\end{equation}
Then
\begin{itemize}
\item on the one hand, for any $1\leq p\leq \infty$ with $\frac{1}{p}+\frac{1}{p'}=1$, we have
\begin{equation}\label{AIII1}
B_{j_0}^{s}\lesssim \int_{J_{j_0}^{s}} \left(\int_{c I_{j_0}} |f(x)|^p\,dx\right)^{\frac{1}{p}}\,\left(\int_{c I_{j_0}} |g(x)|^{p'}\,dx\right)^{\frac{1}{p'}}\,2^{j_0}\,\rho(2^{j_0} t)\,dt\lesssim 2^{-s}\,\|f\|_p\,\|g\|_{p'}\:.
\end{equation}

\item on the other hand, proceeding as in cases 1 and 2, we have
\begin{equation}\label{AIII2}
B_{j_0}^{s}\lesssim  2^{j_0}\,|I_{j_0}|\,2^{s+(\a-1)j_0}\,\iint_{c I_j\times c I_j} |f(y)|\, |g(z)|\,dy\,dz \lesssim 2^{s}\, \|f\|_1\,\|g\|_{1}\:,
\end{equation}
from which we deduce
\begin{equation}\label{AIII3}
A_{j_0}^{s}:=\left(\int_{J_{j_0}^{s}}|H_{j_0,a,b}^\a(f,g)(x)|^{\frac{1}{2}}\,dx\right)^2
	\lesssim |J_{j_0}^{s}|\,B_{j_0}^{s}\lesssim \|f\|_1\,\|g\|_{1}\,.
\end{equation}
\end{itemize}
Applying now multilinear interpolation we deduce that there exists $\d=\d(p,q)>0$ such that
\begin{equation}\label{AIII4}
		\|H_{j_0,a,b}^\a(f,g)\|_{L^r(J_{j_0}^{s})}\lesssim_{a,b,\a,p,q} 2^{-\d\,s}\,\|f\|_{L^p}\,\|g\|_{L^q}\:,
\end{equation}
for any $\frac{1}{p}+\frac{1}{q}=\frac{1}{r}$, $p,\,q\geq 1$ and $r>\frac{1}{2}$.

Putting together \eqref{AI1}, \eqref{AI10} and \eqref{AIII4} we conclude the proof of \eqref{mainrmonj}.
\end{proof}

\section{The Bilinear Maximal Function $M_{a,b}^\a$, $\a\in\R_+\setminus\{1,\,2\}$}\label{BiMax}
In this section we study the (sub)bilinear Maximal function along $\bar{\g}(t)=a\,t+b\,t^\a$ defined in \eqref{mht} by
\begin{equation*}
	M_{a,b}^\a(f,g)(x):= \sup_{\e>0}\,\frac{1}{2\,\e}\int_{-\e}^\e |f(x-t)\, g(x+at+bt^\a)|dt\,,
\end{equation*}
where $a\in\R\setminus\{-1\}$, $b\in\R\setminus\{0\}$ and $\a\in\R_+\setminus\{1,\,2\}$.

Without lost of generality, we will assume that
$f$ and $g$ are non-negative. It will then be enough to study our maximal function with the supremum ranging over dyadic numbers, \textit{i.e.}, letting  $\e\sim 2^{j+1}$ with $j\in\Z$, we have that
\begin{equation}
	M_{a,b}^\a(f,g)(x)\approx\sup\limits_{j\in\Z}\,\frac{1}{2^{j+1}}\int\limits_{2^j\le|t|\le2^{j+1}}\hspace{-.2cm} f(x-t)\,g(x+at+bt^\a)\,dt.
\end{equation}

Let now $\urho\in C^\infty_0(\R)$ be a nonnegative, even function with $\supp \urho \subseteq\{t\in\R|\frac{1}{4}<|t|<4\}$ and
\begin{equation*}
	\int\urho(t)dt=1\,.
\end{equation*}

Then\footnote{For notational simplicity we omit the presence of $a, b$ and $\a$ in the definition of $M_j$.}
\begin{equation*}
	M_{a,b}^\a(f,g)(x)\approx\sup\limits_{j\in\Z}\,\int_\R f(x-t)g(x+at+bt^\a)\urho_j(t)dt =: \sup\limits_{j\in\Z}\:M_j(f,g)(x) \,,
\end{equation*}
where for $j\in\Z$ we define  $\urho_j(t)=2^j\urho(2^j t)$ (with $j\in\Z$).

On the Fourier side, we have that
\begin{equation*}\label{Mabalpha}
	M_{a,b}^\a(f,g)(x)=\sup\limits_{j\in\Z}\:M_j(f,g)(x)=\sup\limits_{j\in\Z}\: \int_\R\int_\R \hat{f}(\xi)\hat{g}(\eta)\um_j(\xi,\eta)e^{i\xi x}e^{i\eta x}d\xi\,,
\end{equation*}
where the multiplier is given by
\begin{equation}
	\um_j(\xi,\eta):=\int_\R e^{-i(\xi-a\eta) t}e^{i b\eta t^\a} \urho_j(t)dt=\int_\R e^{-i\frac{\xi-a\eta}{2^j} t}e^{i\frac{b\,\eta}{2^{\a j}}t^\a} \urho(t)dt.
\end{equation}

Via a linearization procedure we write
\begin{equation*}
	M(f,g)(x):=M_{a,b}^\a(f,g)(x)\approx M_{\jx}(f,g)(x)
\end{equation*}
where $\jx:\R\to\Z$ is a measurable function who assigns for each point $x\in\R$ a value $j\in \Z$ for which
\[
\int_\R\int_\R \hat{f}(\xi)\,\hat{g}(\eta)\,\um_j(\xi,\eta)\,e^{i\xi x}e^{i\eta x}d\xi d\eta
\]
is at least half of the value of $M(f,g)(x)$.

Like in the case of bilinear Hilbert transform, for a suitable choice of $C_0>0$ depending only on $a,b$ and $\a$, we decompose our operator in two components:
\begin{itemize}
	\item the non-singular component
	\begin{equation*}
		M^{NSI}:=\sup_{|j|<C_0} M_j\,;
	\end{equation*}
	\item the singular component
	\begin{equation*}\label{sing-m}
		M^{SI}:=\sup_{|j|\geq C_0} M_j\,.
	\end{equation*}
\end{itemize}

With these, we obtain the initial decomposition
\begin{equation*}
	M\le M^{NSI}\,+\,M^{SI}\,.
\end{equation*}

\subsection{The phase-dependent decomposition.} In this section we apply a further decomposition relative to the frequency of the multiplier. Consider the partition of unity defined on \eqref{ref:unity-phi}. Then, for every $j\in\Z$
\begin{equation*}
	\um_j=\sum_{m,n\in\Z} \um_{j,m,n},
\end{equation*}
where
\begin{equation}\label{mjmnM}
	\um_{j,m,n}(\xi,\eta):=\um_j(\xi,\eta)\,\phi\Big(\frac{\xi-a\eta}{2^{j+m}}\Big)\,\phi\Big(\frac{\a\,b\,\eta}{2^{\a j+n}}\Big)\,.
\end{equation}

Then, for  $C_0$ as above, we let
\begin{equation}\label{delta}
	\Delta:=\{(m,n)\in\Z_{+}^2\,|\,|m-n|\le C_0\}\,,
\end{equation}
and decompose our generic multiplier (for a fixed $j$) into three components:
\begin{itemize}
	\item the \textit{low oscillatory} component:\footnote{Here we include the case $(m,n)\in\Z_+\times \Z_-$ within the low oscillation component, though, strictly speaking, this regime corresponds to the non-stationary  oscillatory component.}
	\begin{equation}\label{max-mjL}
		\um_j^{LO}:=\sum_{(m,n)\in \Z\times\Z_{-}} \um_{j,m,n}\,;
	\end{equation}
	\item the \textit{non-stationary} component:
	\begin{equation}\label{max-mjNS}
		\um_j^{NS}:=\sum_{(m,n)\in (\Z\times\Z_{+})\setminus\Delta} \um_{j,m,n}\,;
	\end{equation}
	\item the \textit{stationary} component:
	\begin{equation}\label{max-mjS}
		\um_j^{S}:=\sum_{(m,n)\in\Delta} \um_{j,m,n}\,.
	\end{equation}
\end{itemize}

With this done, we define
\begin{equation*}
	M^*(f,g)(x):=\sup_{|j|\geq C_0} \big|M_j^\star(f,g)(x)\big|:= \bigg|\int_\R\int_\R \hat{f}(\xi)\hat{g}(\eta)\um_j^*(\xi,\eta)e^{i\xi x}e^{i\eta x}d\xi d\eta\bigg|\,,
\end{equation*}
with $\{*\}\in\{LO,NS,S\}$, and denote $M_{j,m,n}$ the operator with multiplier $\um_{j,m,n}$. Then
\begin{equation}
	M^{SI}\le M^{LO}+M^{NS}+M^S\,.
\end{equation}

Finally, as in Section \ref{sphase}, without loss of generality we may assume $\a>1$ and rearrange each component as follows:

$\newline$
\noindent\textbf{The stationary component $M^S$:}
$\newline$

From \eqref{delta} and \eqref{max-mjS} we have
\begin{equation}\label{MSmm}
	M^S(f,g)(x)\lesssim \sum_{m\in\N} \sup_{|j|\geq C_0} |M_{j,m,m}(f,g)(x)|=:\sum_{m\in\N} M^S_m(f,g)(x)\,.
\end{equation}

In approaching $M^S$ we rely on the fact that the mean zero condition of the function $\rho$ plays no role in the proof of Theorem \ref{thm:mS}. Consequently, one can identify the operator $M_{j,m}:=M_{j,m,m}$ with the operator having as a multiplier \eqref{mjm}. This will allow us later to transfer the theorems from Section \ref{LGCimpl} to our current settings. In order to prepare this ground, we adapt the decomposition of  $\L^S$ to our setting:
\begin{itemize}
	\item The \emph{hybrid stationary} component $M^{S}_{H}$: defined as
	\begin{equation}\label{MSH-dec}
		M^{S}_{H}:=\sum_{m\in\N} M^{S}_{H,m}:=\sum_{m\in\N} \sup_{{|j|\geq C_0}\atop{(\a-1)j>m}} |M_{j,m}|\,.
	\end{equation}
	\item The \emph{nonlinear stationary} component $M^{S}_{NL}:$
	\begin{equation}\label{MNLS}
		M^{S}_{NL}:=\sum_{m\in\N}M^{S}_{NL,m}:=\sum_{m\in\N} \sup_{{|j|\geq C_0}\atop{(\a-1)j<-m}}\hspace{-.2cm} |M_{j,m}|\,.
	\end{equation}
	\item The \emph{transitional stationary} component $M^{S}_{TR}$:
	\begin{equation*}
		M^{S}_{TR}= M^{S}_{TR,+}\,+\,M^{S}_{TR,-}\,,
	\end{equation*}
	where
	\begin{equation}\label{MTRS1}
		M^{S}_{TR,\pm}:= \sum_{m\in\N}M^{S}_{TR,\pm,m}:=\sum_{m\in\N} \sup_{{{|j|\geq C_0}\atop{-m\leq (\a-1)j\leq m}}\atop{\textrm{sgn}\,j=\pm}}\hspace{-.3cm} |M_{j,m}|\,.
	\end{equation}
	
	This way we have obtained
	\begin{equation}
		M^{S}\le M^{S}_{NL}\,+M^{S}_{TR}\,+M^{S}_{H}\,.
	\end{equation}
\end{itemize}

$\newline$
\noindent\textbf{The non-stationary component $M^{NS}$:}
$\newline$

As with the stationary case, we notice that the mean zero condition of the function $\rho$ is irrelevant in the analysis of $m^{NS}$. Consequently, we will be able to appeal to the theorems in Section \ref{HNL-2} which motivates the decomposition below:
\begin{equation*}
	M^{NS}\le M^{NS,I}\,+\,M^{NS, II}\,+\,M^{NS, III}\,,
\end{equation*}
where
\begin{align}
	&M^{NS,I}
	\le M^{NS,I}_{NL}\,+\,M^{NS,I}_{TR}\,+\,M^{NS,I}_{H}
	=\sum_{n\ge 0}\left(M^{NS,I}_{NL,0,n} \,+\, M^{NS,I}_{TR,0,n}\,+\, M^{NS,I}_{H,0,n}\right)\nonumber\\[1ex]
	&:=\sum_{n\ge 0}\Bigg(
	\sup_{\substack{|j|\ge C_0\\[.5ex](\a-1)j\le -n\\}}\hspace{-.25cm} |\sum_{\substack{m<0}}M_{j,m,n}|\,+
	\hspace{-.5cm} \sup_{\substack{|j|\ge C_0\\[.5ex]-n<(\a-1)j<0\\}} \hspace{-.25cm} |\sum_{\substack{m<0}}M_{j,m,n}|+
	\hspace{-.1cm} \hspace{-.15cm}\sup_{\substack{|j|\ge C_0\\[.5ex](\a-1)j\geq 0\\}}\hspace{-.2cm} |\sum_{\substack{m<0}}M_{j,m,n}|\hspace{-.05cm}\Bigg),\nonumber\\[3ex]
	&M^{NS,II}
	\le M^{NS,II}_{NL}{+}M^{NS,II}_{TR}{+}M^{NS,II}_{H}
	=\hspace{-.7cm} \sum_{\substack{(m,n)\in (\Z\times\Z_{+})\setminus\Delta\\[.5ex]0\leq m<n-C_0}}
	\hspace{-.2cm} \Bigg(M^{NS,II}_{NL,m,n}{+}M^{NS,II}_{TR,m,n}{+}M^{NS,II}_{H,m,n} \Bigg)\nonumber
	\\[1ex]
	&:=\hspace{-.2cm}\sum_{\substack{(m,n)\in (\Z\times\Z_{+})\setminus\Delta\\[.5ex] 0\leq m<n-C_0}}\Bigg(
	\sup_{\substack{|j|\geq C_0\\[.5ex](\a-1)j\le-n\\}} \hspace{-.2cm}  |M_{j,m,n}|+
	\hspace{-.4cm}\sup_{\substack{|j|\geq C_0\\[.5ex]-n<(\a-1)j<m\\}} \hspace{-.4cm} |M_{j,m,n}| +
	\hspace{-.2cm}\sup_{\substack{|j|\geq C_0\\[.5ex](\a-1)j\ge m\\}} \hspace{-.2cm} |M_{j,m,n}|
	\Bigg)\nonumber\\[3ex]
	&M^{NS,III}
	\le M^{NS,III}_{NL}{+}M^{NS,III}_{TR}{+}M^{NS,III}_{H}\!
	=\hspace{-.7cm} \sum_{\substack{(m,n)\in (\Z\times\Z_{+})\setminus\Delta\\[.5ex] 0\leq m<n-C_0}}\hspace{-.75cm} \left(M^{NS,III}_{NL,m,n}{+}M^{NS,III}_{TR,m,n}{+}M^{NS,III}_{H,m,n}
	\right)\nonumber\\[1ex]
	&:=\sum_{\substack{(m,n)\in (\Z\times\Z_{+})\setminus\Delta\\[.5ex] 0\leq n<m-C_0}}\Bigg(
	\sup_{\substack{|j|\geq C_0\\[.5ex](\a-1)j\le-m\\}} \hspace{-.2cm}  |M_{j,m,n}|+
	\hspace{-.4cm}\sup_{\substack{|j|\geq C_0\\[.5ex]-m<(\a-1)j<n\\}} \hspace{-.4cm} |M_{j,m,n}| +
	\hspace{-.2cm}\sup_{\substack{|j|\geq C_0\\[.5ex](\a-1)j\ge n\\}} \hspace{-.2cm} |M_{j,m,n}|
	\Bigg)\,.\label{MNS-dec}
\end{align}

$\newline$
\noindent\textbf{The low oscillatory component $M^{LO}$}
$\newline$

The component $M^{LO}$ essentially represents a truncated version of the Bilinear Maximal function - for more details, the reader is invited to consult Section \ref{mlo}.

\subsection{The analysis of the stationary component $M^S$}\label{MS-full}
In this section we discuss the  component whose phase of the multiplier has stationary points. The main result of this section is the following
\begin{theorem}\label{thm:MS}
	For any $\a\in (0,\infty)\setminus\{1,2\}$ and any $1<p,\,q\le \infty$ with $\frac{1}{p}+\frac{1}{q}=\frac{1}{r}$ and $r>\frac{2}{3}$, the following holds
	\begin{equation}
		\|M^S(f,g)\|_r\lesssim_{\a,a,b,p,q} \|f\|_p\,\|g\|_q\,.
	\end{equation}
\end{theorem}

Theorem \ref{thm:MS} is an immediate consequence of the two propositions below.

Before, stating these propositions though, we prepare the ground by recalling the result in Section \ref{LGCplus}: For any $j\in\Z$ and $m\in\Z$ there is $\delta=\delta(\a)>0$ such that
\begin{equation}\label{Tjm-constant}
	\|M_{j,m}(f,g)\|_1\lesssim_{\a,a,b} 2^{-\delta m}\,\|f\|_2\,\|g\|_2\,.
\end{equation}

We claim that one can get an extension of \eqref{Tjm-constant} to the \emph{\underline{variable} case}, that is:

\begin{proposition}\label{prop:MS2}
	Let $m\in\N$, $\a\in (0,\infty)\setminus\{1,2\}$ and $M_m^S$ as defined in \eqref{MSmm}. Then, the following holds uniformly in $m\in\N$:
	\begin{equation}
		\|M_m^S(f,g)\|_{L^1(dx)}\lesssim_{\a} 2^{-\delta m}\, \|f\|_2\,\|g\|_2\,.
	\end{equation}
\end{proposition}

In order to state the second proposition we need yet some more preparation: we first notice that for each $m\in\N$ and $x\in\R$
\begin{equation*}
	\,M^S_m(f,g)(x)\,\le \sum_{\{\star\}\in\{NL,\{TR,\pm\},H\}}\,M^S_{\star,m}(f,g)(x)\,,
\end{equation*}
and given $\{\star\}\in\{NL,\{TR,\pm\},H\}$ we have
\begin{equation*}\label{MSstarm-l2bound}
	M^S_{\star,m}(f,g)(x)\le\Big(\sum_{j\in\B_{\star,m}} |M_{j,m}(f,g)(x)|^2\Big)^{\frac12}\,,
\end{equation*}
where $\B_{\star,m}$ stands for the range of $j$ in the definitions of $M^{S}_{\star,m}$\,.

Now, for a given value of $m\in\N$ and $x\in\R$, one has
\begin{equation*}
	\Big(\sum_{j\in\B_{\star,m}} |M_{j,m}(f,g)(x)|^2\Big)^{\frac12} =\sup_{\substack{\{\e_j(x)\}\in\ell^2\\ \|\e_j(x)\|_{\ell^2}\le 1}}\: \sum_{j\in \B_{\star,m}} M_{j,m}(f,g)(x)\,\e_j(x)\,.
\end{equation*}

Then, there exist $\{\e_{\star,j}(x)\}\in\ell^2$ with $\big(\sum\limits_{j\in\B_{\star,m}} |\e_{\star,j}(x)|^2\,\big)^{\frac12}\le 1$, such that
\begin{equation*}
	M^S_m(f,g)(x)\,\le \sum_{\{\star\}\in\{NL,\{TR,\pm\},H\}}\: \sum_{j\in \B_{\star,m}} M_{j,m}(f,g)(x)\,\e_{\star,j}(x)\,.
\end{equation*}
Define now
\begin{equation}\label{uLstarS}
	\uL_\star^S(f,g,h)=\sum_{m\in\N}\:\sum_{j\in\B_{\star,m}}  \uL_{j,m}(f,g,h)=:\sum_{m\in\N}\uL_{\star,m}^S(f,g,h)\,,
\end{equation}
where\footnote{We omit the term ${*}$ in $\e_{\star,j}$ since we will treat each case separately.}
\begin{align}
	\uL_{j,m}(f,g,h)&:=\int_\R M_{j,m}(f,g)(x)\e_j(x)h(x)dx\nonumber\\
	&=\int_\R\int_\R \hat{f}(\xi)\,\hat{g}(\eta)\,\widehat{\e_j h}(\xi+\eta) \,\um_{j,m}(\xi,\eta)\,d\xi\,d\eta\,.\label{uLjm}
\end{align}

Let $\uL_m^S=\uL_{NL,m}^S+\uL_{TR,m}^S+\uL_{S,m}^S$; we will show the analogue of Proposition \ref{Lgennodecay}:

\begin{proposition}\label{prop:MSfull}
	Let $m\in\N$ and $F,\,G,\,H$ be measurable sets with finite (nonzero) Lebesgue measure. Then
	\begin{equation}\label{sets_hyp}
		\exists\:H'\subseteq H\:\:\textrm{measurable with}\:\:|H'|\geq \frac{1}{2} |H|\,,
	\end{equation}
	such that for any triple of functions $f,\,g$ and $h$ obeying
	\begin{equation}\label{func_hyp}
		|f|\leq \chi_{F},\qquad |g|\leq \chi_{G},\qquad |h|\leq \chi_{H'}\,,
	\end{equation}
	and any triple $(p,q,r')$ with $1<p,q\leq \infty$, $r>\frac{2}{3}$ and $\frac{1}{p}+\frac{1}{q}+\frac{1}{r'}=1$ one has that
	\begin{equation}\label{prop:MSfull-bound}
		|\uL^{S}_m(f,g,h)|\lesssim_{\a,a,b,p,q} m\,|F|^{\frac{1}{p}}\,|G|^{\frac{1}{q}}\,|H|^{\frac{1}{r'}}\,.
	\end{equation}
\end{proposition}

\subsubsection{The $L^2\times L^2\times L^\infty$ $m$-decaying bound: Proof of Proposition \ref{prop:MS2}}

$\newline$

For simplicity we consider $\xi>0$ and $\eta>0$. The reasonings for the remaining cases are of similar nature. As with the proof of Proposition \ref{L2mdecay}, the analysis of the operator is performed for the \emph{global} cases
\begin{itemize}
	\item $j\ge C_0$ covering the component $M^{S}_{m,+}:=M^{S}_{H,m}+M^{S}_{TR,+,m}$, and
	
	\item $j\le -C_0$ covering the component $M^{S}_{m,-}:=M^{S}_{NL,m}+M^{S}_{TR,-,m}$\,.
\end{itemize}

For $j\ge0$, we have\footnote{Wlog, we assume in what follows that $|a|>1$.}
\[
\supp\hat{f}\subset\Big[\frac a8\cdot\frac{2^{\a j+m}}{\a\,b},8a\cdot \frac{2^{\a j+m}}{\a\,b}\Big]\quad\text{ and }\quad\supp\hat{g}\subset \Big[\frac14\cdot\frac{2^{\a j+m}}{\a\,b},4 \frac{2^{\a j+m}}{\a\,b}\Big]\,.
\]
Then, for $j\in\N$,  we rewrite \eqref{Tjm-constant} as
\begin{equation}\label{Tjm-pos-supp}
	\|M_{j,m}(f,g)\|_1\lesssim_{\a,a,b} 2^{-\delta m} \|f*\check{\phi}_{\a j+m}\|_2\,\|g*\check{\phi}_{\a j+m}\|_2\,,
\end{equation}
where  $\phi\in C_0^\infty$ with $\supp\phi\subset[\frac18,8]$.

Thus, using the almost orthogonality of the family $\{\phi_{\a j+m}\}_{j}$, we have
\begin{align}
	\|M^{S}_{m,+}(f,g)\|_1&\le \sum_{j\ge C_0} \|M_{j,m}(f,g)\|_1
	\lesssim_\a 2^{-\delta m}\sum_{j\ge 0} \|f*\check{\phi}_{\a j+m}\|_2\,\|g*\check{\phi}_{\a j+m}\|_2 \nonumber\\
	&\le 2^{-\delta m}
	\Big(\sum_{j\ge 0}  \|f*\check{\phi}_{\a j+m}\|_2^2\Big)^{\frac12}\, \Big(\sum_{j\ge 0}  \|g*\check{\phi}_{\a j+m}\|_2^2\Big)^{\frac12}
	\lesssim 2^{-\delta m}\|f\|_2\|g\|_2\,.\nonumber
\end{align}
We turn our attention towards the case $j\le -C_0$. Here we have
\[
\supp\hat{f}\subset\big[\frac18\cdot2^{j+m},8\cdot 2^{j+m}\big]\quad\text{ and }\quad\supp\hat{g}\subset  \Big[\frac14\cdot\frac{2^{\a j+m}}{\a\,b},4 \frac{2^{\a j+m}}{\a\,b}\Big]\,.
\]
From \eqref{mjmnM} and \eqref{Tjm-constant}, in the setting $j\in\Z_-$, we have
\begin{equation}
	\|M_{j,m}(f,g)\|_1\lesssim_{\a,a,b} 2^{-\delta m}\,\|f*\check{\phi}_{j+m}\|_2\,\|g*\check{\phi}_{\a j+m}\|_2\,,
\end{equation}
and hence
\begin{align}
	&\qquad\qquad\qquad\|M^{S}_{m,-}(f,g)\|_1\le \sum_{j\le -C_0} \|M_{j,m}(f,g)\|_1\nonumber\\
\le_{\a,a,b} & 2^{-\delta m}\Big(\sum_{j\in\Z_-}  \|f*\check{\phi}_{j+m}\|_2^2\Big)^{\frac12}\,\Big(\sum_{j\in\Z_-} \|g*\check{\phi}_{\a j+m}\|_2^2\Big)^{\frac12}\lesssim 2^{-\delta m}\|f\|_2\|g\|_2\,.\label{Tjm-}
\end{align}

Putting these results together we deduce that
\begin{equation}
	\|M_m^S(f,g)\|_{L^1}\le \|M^{S}_{m,+}(f,g)\|_1+\|M^{S}_{m,-}(f,g)\|_1\nonumber\\
	\lesssim_{\a,a,b}2^{-\delta m}\|f\|_{L^2}\|g\|_{L^2}\,.\label{TjxS2}
\end{equation}

\subsubsection{The extended boundedness range: Proof of Proposition \ref{prop:MSfull}.}
$\newline$

We split the analysis of $\uL^{S}_{m}$ corresponding to $\uL^{S}_{H,m}$, $\uL^{S}_{NL,m}$ $\uL^{S}_{TR,m}$.

\paragraph{The $\uL^{S}_{H,m}$ term.}\label{uLSHm-range} Recall from \eqref{MSH-dec} and \eqref{uLstarS} that
\begin{equation*}
	\uL^S_H=\sum_{m\in\N}  \sum_{{(\a-1) j\geq m}\atop{j\geq C_0}} \uL_{j,m}=\sum_{m\in\N}  \uL^{S}_{H,m}\,.
\end{equation*}

Following the same steps as in Section \ref{LGCLfull}, we get
\vspace{-.7cm}\begin{align}
	\nonumber\intertext{\begin{equation}
			|\uL^{S}_{H,m}(f,g,h)|\lesssim\label{uLSHm}
	\end{equation}} \\[-6ex]
	& \frac{1}{2^m}\sum_{\substack{j\ge 0\\s\sim 2^m}}\sum_{{z\in\Z}\atop{l\sim 2^{(\a-1)j}}}\frac{1}{|I_{j+m}|^{\frac{1}{2}}}\,|\langle f,  \phi_{P_{j+m}(z-s,l)}\rangle|\,
	|\langle g,  \phi_{P_{j+m}(z+s,l)}\rangle|\,|\langle \e_j h,  \phi_{P_{j+m}(z,2l)}\rangle|\,.\nonumber
\end{align}

Let now $F,\,G,\,H$ be measurable sets satisfying \eqref{sets_hyp}, and let $f$, $g$ and $h$ functions obeying \eqref{func_hyp}.  Continuing the reasoning in Section \ref{LGCLfull}, one has for any $0<\theta_3\leq 1$, $0\leq\theta_1,\,\theta_2<1$ with $\theta_1+\theta_2+\theta_3=1$ and any $N\in\N$
\begin{align}
	&\qquad\qquad\qquad\qquad\qquad\qquad\qquad|\uL^{S}_{H,m}(f,g,h)|\lesssim\\
	&\sum_{\beta\in\N}2^{\b\,(\theta_1+\theta_2-N\,\theta_3)}\,
	|F|^{\frac{1+\theta_1}{2}}\,|G|^{\frac{1+\theta_2}{2}}\,|H|^{-(\theta_1+\theta_2)}
	\left(\sum_{j\geq 0}\sum_{{u\in\Z}\atop{l\sim 2^{(\a-1)j}}}
	(\Delta_{j,m}^{2l,[5 I_j^u]}\, (\e_j h))^2\right)^{\frac{1-\theta_3}{2}}\,.\nonumber
\end{align}
Our desired estimate follows if we show that
\begin{equation}
	\sum_{j\geq 0}\sum_{{u\in\Z}\atop{l\sim 2^{(\a-1)j}}}
	(\Delta_{j,m}^{2l,[5 I_j^u]}\, (\e_j h))^2\lesssim |H|\,.
\end{equation}
The latter is a direct consequence of
\begin{align*}
	&\sum_{j\geq 0}\sum_{{u\in\Z}\atop{l\sim 2^{(\a-1)j}}}
	(\Delta_{j,m}^{2l,[5 I_j^u]}\, (\e_j h))^2=
	\sum_{j\geq 0}\sum_{{u\in\Z}\atop{l\sim 2^{(\a-1)j}}}\sum_{I_{j+m}^v\subseteq 5I_{j}^u} |\langle \e_j h,  \phi_{P_{j+m}(v,2l)}\rangle|^2 \\
	&\lesssim \sum_{j\geq 0}\, \int\limits_{\xi\approx 2^{\a j+m}} \hspace{-.3cm} |(\widehat{\e_j h})(\xi)|^2d\xi \lesssim \sum_{j\geq 0} \int |(\e_j h)(x)|^2dx\lesssim \|h\|_2^2\le |H|\,,
\end{align*}
where in the last line above we used $\|\{\e_j\}\|_{\ell^2}\le 1$ and  $|h|\le \chi_{H'}$\,.

\paragraph{The $\uL^{S}_{NL,m}$ term.}\label{SNLm-range} From the definition of the nonlinear dominant behavior, following the same reasonings as in Section \ref{LGCNLfull},  we have
\begin{equation}
	\uL^{S}_{NL}=\sum_{m\in\N} \sum_{{|j|\geq C_0}\atop{(\a-1)j<-m}} \uL_{j,m}=\sum_{m\in\N}\uL^{S}_{NL,m}\,,
\end{equation}
with
\begin{equation*}
	|\uL_{NL,m}^S(f,g,h)|\lesssim\frac1{2^m}\sum_{s\sim 2^m}\uL_{NL,m}^{S,s}(f,g,h)\,
\end{equation*}
where
\begin{equation*}
\uL_{NL,m}^{S,s}(f,g,h):=\sum_{\substack{j<0\\u\in\Z}}
	\frac{|\langle g,  \phi_{P_{\a j+m}(u+\frac{s^{\a}}{2^{(\a-1) m}})}\rangle|}{|I_{\a j+m}^u|^{\frac{1}{2}}}\,\sum_{I_{j+m}^v\subseteq I_{\a j+ m}^u}
	|\langle f,  \phi_{P_{j+m}(v-s)}\rangle|\,|\langle \e_j h,  \phi_{P_{j+m}(v)}\rangle|\,.
\end{equation*}

Assuming now we are in the setting described by \eqref{sets_hyp} and \eqref{func_hyp}, we can follow (with the obvious modifications) the reasonings from \eqref{maxshift}--\eqref{maxshift1}, in order\footnote{Throughout this section we preserve the same notation from Section \ref{LGCLfull}.} to obtain the analogue of \eqref{maxshift2}:
\begin{equation}\label{maxshift2A}
\uL^{S,\b}_{NL,m}(f,g,h)\lesssim \max\left\{1,\,2^{\b}\,\frac{|G|}{|H|}\right\}\,\int_{\R} S(f)(x)\,S_{\I_{\b},s}^{max}(\e,h)(x)\,dx\,,
\end{equation}
where
\begin{equation}\label{maxshiftA}
S_{\I,s}^{max}(\e,h)(x):=\left(\sum_{j<0}\sum_{I_{\a j}^w\in \I}
\sum_{{I_{j+m}^v\subseteq I_{\a j+ m}^u}\atop{I_{\a j+ m}^u}\subseteq I_{\a j}^w}
\frac{|\langle \e_j h,  \phi_{P_{j+m}(v+s(I_{\a j+ m}^u))}\rangle|^2}{|I_{j+m}^v|}\chi_{I_{j+m}^v}(x)\right)^{\frac{1}{2}}\,.
\end{equation}
It is not hard to see now that the analogue of \eqref{maxshift3} holds, \textit{i.e.}
\begin{equation}\label{maxshift3A}
\|S_{\I_{\b},s}^{max}(\e,h)\|_{p}\lesssim_{n} m\, 2^{-n\b}\,|H|^{\frac{1}{p}}\qquad\forall\:n\in\N\:\:\:\&\:\:\:1\leq p<\infty\,.
\end{equation}

The desired bound \eqref{prop:MSfull-bound} follows now via multilinear interpolation.

\paragraph{The $\uL^{S}_{TR,m}$ term.}\label{STRm-range} From \eqref{MTRS1}, \eqref{MSstarm-l2bound} and Section \ref{NSIterm} we have
\begin{equation}
	\|M^S_{TR,\pm,m}(f,g)\|_r\le \sum_{\substack{|j|\geq C_0\\-m\leq (\a-1)j\leq m}} \hspace{-.5cm} \|M^S_{j,m}(f,g)\|_r\lesssim_{p,q} m\, \|f\|_p\,\|g\|_q
\end{equation}
for any $\frac{1}{p}+\frac{1}{q}+\frac{1}{r'}=1$ with $1\leq p,\,q\leq \infty$.

Proposition \ref{prop:MSfull} is now an immediate consequence of the main results proved in Sections \ref{uLSHm-range}, \ref{SNLm-range} and \ref{STRm-range}\,.

\subsection{The analysis of the non-stationary component $M^{NS}$}
In this section we treat each of the terms $M^{NS,I}$, $M^{NS,II}$ and $M^{NS,III}$. As before, given that the mean zero condition of $\rho$ is not required in the analysis of $\L_{NS}$, one can apply the results from Sections \ref{HNL-2} and \ref{MS-full} to our current setting.

We can identify the operators $M_{j,n}$ and $M_{j,m,n}$ defined below with the corresponding operators defined at Section \ref{HNL-2}:
\begin{itemize}
	\item for $n\in \Z_+$
	\begin{equation*}
		M_{j,0,n}(f,g)(x):=\sum_{m<0}\int_\R\int_\R\hat{f}(\xi)\hat{g}(\eta)\um_{j,m,n}(\xi,\eta)e^{i\xi x}e^{i\eta x} d\xi d\eta;
	\end{equation*}
	\item for values $(m,n)\in(\Z_+\times \Z_+)\setminus\Delta$
	\begin{equation*}
		M_{j,m,n}(f,g)(x):=\int_\R\int_\R\hat{f}(\xi)\hat{g}(\eta)\um_{j,m,n}(\xi,\eta)e^{i\xi x}e^{i\eta x} d\xi d\eta\,.
	\end{equation*}
\end{itemize}

Given $(m,n)\in (\{0\}\times\Z_+)\cup ((\Z_+\times \Z_+)\setminus\Delta)$, we define
\begin{equation}
	M^{NS}_{m,n}(f,g)(x):= \sup_{|j|\geq C_0} |M_{j,m,n}(f,g)(x)|\,.
\end{equation}

Recalling the notation in \eqref{MNS-dec} we have for any $x\in\R$
\begin{equation*}
	M^{NS}_{m,n}(f,g)(x)\,\le \sum_{\substack{\I\in\{I,II,III\}\\ \{\star\}\in\{NL,TR,H\}}}  M^{NS,I}_{\star,m,n}(f,g)(x)\,.
\end{equation*}

Next, given $\{\star\}\in\{NL,TR,H\}$, $\I\in\{I,II,III\}$ and $(m,n)\in\mathcal R_\I$, we have\footnote{Here $\mathcal R_I=\{(0,n):n\in\Z_+\}$ and for $\I\in\{II,III\}$, $\mathcal R_\I$ stands for the range of the pairs $(m,n)$ in the definitions of the terms $M^{NS,\I}$\,.}
\begin{equation}\label{MNSIstarmn-l2bound}
	M^{NS,\I}_{\star,m,n}(f,g)(x)\le\Big(\sum_{j\in\B_{\star,m,n}^\I} |M_{j,m,n}(f,g)(x)|^2\Big)^{\frac12}\,,
\end{equation}
where $\B_{\star,m,n}^\I$ stands for the range of $j$ in the definitions of $M^{NS,\I}_{\star,m,n}$\,.

As with the stationary case, we have
\begin{equation*}
	\Big(\sum_{j\in\B_{\star,m,n}^\I} |M_{j,m,n}(f,g)(x)|^2\Big)^{\frac12} =\sup_{\substack{\{\e_j(x)\}\in\ell^2\\ \|\e_j(x)\|_{\ell^2}\le 1}}\, \sum_{j\in \B_{\star,m,n}^\I} M_{j,m,n}(f,g)(x)\e_j(x)\,.
\end{equation*}

Then, there exist $\{\e_{\star,j}(x)\}\in\ell^2$ with $\big(\sum\limits_{j\in\B_{\star,m}} |\e_{\star,j}(x)|^2\big)^{\frac12}\le 1$, such that\footnote{Strictly speaking, the choice of $\e_j$ depends of the given values $m, n$. }
\begin{equation}
	M^{NS}_{m,n}(f,g)(x)\,\le \sum_{\substack{\I\,\in\,\{I,II,II\}\\\{\star\}\in\{NL,TR,H\}}}\,\sum_{j\in\B_{\star,m,n}^\I}M_{j,m,n}(f,g)(x)\e_{\star,j}(x)\,.
\end{equation}

With this done, we aim to employ the techniques developed in Section \ref{HNL-2}. Thus, similarly to \eqref{uLjm}, we let\footnote{We omit the term ${*}$ in $\e_{\star,j}$ since we will treat each case separately.}
\begin{align}
	\uL_{j,m,n}(f,g,h)&:=\int_\R M_{j,m,n}(f,g)(x)\e_j(x)h(x)dx\nonumber\\
	&=\int_\R\int_\R \hat{f}(\xi)\,\hat{g}(\eta)\,\widehat{\e_j h}(\xi+\eta) \,\um_{j,m,n}(\xi,\eta)\,d\xi\,d\eta\,,
\end{align}
and define the trilinear forms $\uL^{NS,\I}_{\star,m,n}$ and $\uL^{NS,\I}_{\star}$ with the obvious correspondences.

\subsubsection{The treatment of the term $\underline\Lambda^{NS,I}$}

In this setting, after an integration by parts, we obtain
\[
\um_j^{NS,I}=\sum_{n\in\N} \um_{j,0,n}\,,
\]
where (wlog we assume that $a=b=1$)
\begin{equation}\label{umjnNS1}
	\um_{j,0,n}(\xi,\eta)\approx \frac1{2^n} \Big(\int_\R e^{-i\frac{\xi-\eta}{2^j}t}e^{i\frac{t^\a}{2^{\a j}}}\urho(t)dt\Big)\psi\Big(\frac{\xi-\eta}{2^j}\Big)\phi\Big(\frac{\eta}{2^{\a j+n}}\Big)\,.
\end{equation}

\paragraph{The $\uL^{NS,I}_{H}$ term.}
In this case, we consider in \eqref{umjnNS1} the values $j\in\Z$ such that $(\a-1)j\ge -n$. We, then have
\begin{align}
	\um_{j,0,n}(\xi,\eta)\approx \frac{1}{2^n}\hspace{-.2cm}  \sum_{l\sim 2^{(\a-1)j+n}}\hspace{-.3cm}
	\Big(\int_\R e^{-i\frac{\xi-\eta}{2^j}t}e^{i\frac{\eta}{2^{\a j}}t^\a}\!\!\urho(t)dt\Big)	\phi\Big(\frac{\xi}{2^{j}}{-}l\Big) \phi\Big(\frac{\eta}{2^{j}}{-}l\Big)\phi\Big(\frac{\xi{+}\eta}{2^{j}}{-}2l\Big)\,,\raisetag{.3\baselineskip}
\end{align}
which invites us to consider
\begin{align*}
	&\qquad\qquad\qquad\qquad\qquad\qquad\qquad\qquad\uL_{j,0,n}(f,g,h):=\nonumber\\
	&\frac1{2^n}\sum_{l\sim 2^{(\a-1)j+n}}\int_{\R^2}(f*\check\phi_j^\ell)\Big(x-\frac{t}{2^j}\Big)(g*\check{\phi}_j^l)\Big(x+\frac{t}{2^j}+\frac{t^\a}{2^{\a j}}\Big)
	(\e_j h*\check{\phi}_j^{2l})(x)\,\urho(t)\,dt\,dx\,.
\end{align*}

Applying the methods from Section \ref{lns1h}, we get
\begin{align}\label{h0n}
	\big|\uL_{H,0,n}^{NS,I}(f,g,h)\big|\lesssim \frac1{2^n}\hspace{-.3cm} \sum_{\substack{j\ge0\\l\sim 2^{(\a-1)j+n}\\ z\in\Z}}\!\!\!\!
	\frac{1}{|I_j|^{\frac12}}\,|\langle f,  \phi_{P_j(z,l)}\rangle|\,
	|\langle g,  \phi_{P_j(z,l)}\rangle|\,|\langle \e_j h,  \phi_{P_j(z,2l)}\rangle|\,.\raisetag{\baselineskip}
\end{align}
But now \eqref{h0n} is the analogue of \eqref{uLSHm} with $m=0$ and the extra decaying factor of $\frac1{2^n}$. Thus, following the reasoning in Section \ref{uLSHm-range}, we conclude
\begin{equation}
	\big|\uL_{H,0,n}^{NS,I}(f,g,h)\big|\lesssim \frac{n}{2^n}\|f\|_p\|g\|_q\|h\|_{r'}\,,
\end{equation}
for any $\frac1p+\frac1q+\frac1r=1$ with $1\le p,\, q\le \infty$ and $r>\frac23$.

\paragraph{The $\uL^{NS,I}_{NL}$ term.}
In this case we take advantage of the signum of $j$ and instead of bounding $\uL_{NL,0,n}^{NS,I}$ we show that $M_{j,0,n}$ is bounded uniformly on $j$.

Since $(\a-1)j\le -n$, one can reduce the multiplier to
\begin{equation}
	\um_{j,0,n}(\xi,\eta)\approx\frac{1}{2^n}\,\Big(\int_\R e^{-i\frac{\xi}{2^j}t+i\frac{\eta}{2^{\a j}}t^\a}\urho(t)\,dt\Big)\psi\Big(\frac{\xi}{2^j}\Big)\phi\Big(\frac{\eta}{2^{\a j+n}}\Big)\,,
\end{equation}
which after some standard Taylor series arguments, up to suitable error terms, gives
\begin{equation}
	\um_{j,0,n}(\xi,\eta)\approx\frac{1}{2^n}\,\Big(\int_\R e^{i\frac{\eta}{2^{\a j}}t^\a}\urho(t)dt\Big) \psi\Big(\frac{\xi}{2^j}\Big)\,\phi\Big(\frac{\eta}{2^{\a j+n}}\Big)\,.
\end{equation}
With this we have
\begin{align}
	M_{j,0,n}(f,g)(x)&:=\int_\R\int_\R \hat{f}(\xi)\hat{g}(\eta)\um_{j,0,n}(\xi,\eta)e^{i(\xi+\eta)x}d\xi\, d\eta\\
	&\approx\frac1{2^n}
	\int_\R (f*\check{\psi}_j)(x)\,(g*\check{\phi}_{\a j+n})\Big(x+\frac{t^\a}{2^{\a j}}\Big)\,\urho(t)\,dt\,\nonumber\,.
\end{align}

Then
\begin{align}
	|M_{j,0,n}(f,g)(x)|&\lesssim\frac1{2^n}	(\M f)(x)\,
	\int_\R \Big|(g*\check{\phi}_{\a j+n})\Big(x+\frac{t^\a}{2^{\a j}}\Big)\,\urho(t)\Big|\,dt\,.
\end{align}
Next, via a change of variable, we have
\begin{align}
	\Big|\int_\R &(g*\check{\phi}_{\a j+n})(x+\frac{t^\a}{2^{\a j}}) \,\urho(t)\,dt\Big|\lesssim_{\urho} 2^{\a j}
	\int_{|t|\sim 2^{-\a j}} \big|(g*\check{\phi}_{\a j+n}) (x+t)\big|\,dt\nonumber\\[1ex]
	&\lesssim_\phi \frac1{2^{-\a j}} \int_{|t|\sim 2^{-\a j}}
	(\M g)(x+t)\,dt\lesssim \M(\M g)(x)\,.
\end{align}

Hence
\begin{align*}
	|M^{NS,I}_{NL,0,n}(f,g)(x)|=\sup_{{|j|\geq C_0}\atop{(\a-1)j<-n}} |M_{j,0,n}(f,g)(x)| \lesssim\frac1{2^n}	(\M f)(x)\,\M(\M g)(x)\,.
\end{align*}

By H\"older inequity, the boundedness of $\M$, and the decay in $n$, we conclude that $M^{NS,I}_{NL}$ is bounded from $L^p\times L^q$ into $L^r$ within the desired range.

\paragraph{The $\uL^{NS,I}_{TR}$ term.}

As a consequence of the reasonings in Section \ref{NSIterm}, for any $n\in\N$ and \emph{uniformly for $j\in \Z$}, we have that
\begin{equation}
	|\uL_{j,0,n}(f,g,h)|\lesssim_{\a, p,q} \frac{1}{2^n}\,\|f\|_p\,\|g\|_q\,\|h\|_{r'}
\end{equation}
for any $\frac{1}{p}+\frac{1}{q}+\frac{1}{r'}=1$ with $1\leq p,\,q\leq \infty$, we conclude that
\begin{equation}
	|\uL^{NS,I}_{TR,n}(f,g,h)|\lesssim_{\a, p,q} \frac{n}{2^n}\,\|f\|_p\,\|g\|_q\,\|h\|_{r'}\,.
\end{equation}

\subsubsection{The treatment of the term $\uL^{NS,II}$}
In this situation we have $0\le m<n-C_0$. Then, via integration by parts, the multiplier can be reduced to
\begin{equation}\label{umjmnlNSII}
	\um_{j,m,n}\approx\frac1{2^n}\Big(\int_\R e^{-i\frac{\xi-\eta}{2^j}t}e^{i\frac\eta{2^{\a j}}t^\a}\urho(t)dt\Big)\phi\Big(\frac{\xi-\eta}{2^{j+m}}\Big) \phi\Big(\frac{\eta}{2^{\a j+n}}\Big)\,.
\end{equation}

\paragraph{The $\uL^{NS,II}_H$ term.}
In this setting, following the steps at Section \ref{uLSHm-range}, we have
\begin{align}
	&\qquad\qquad\qquad\qquad\uL^{NS,II}_{H,m,n}(f,g,h) = \sum_{\substack{|j|\geq C_0\\(\a-1)j\ge m}} \hspace{-.2cm} \uL_{j,m,n}(f,g,h)\lesssim\nonumber\\
	&\frac1{2^{n+m}}\sum_{s\sim 2^m}\sum_{\substack{(\a-1)j\ge m\\l\sim 2^{(\a-1)j+n-m}\\ z\in\Z}}
	\frac{1}{|I_{j+m}|^{\frac12}}\,|\langle f,  \phi_{P_{j+m}(z-s,l)}\rangle|\,
	|\langle g,  \phi_{P_{j+m}(z+s,l)}\rangle|\,|\langle \e_j h,  \phi_{P_{j+m}(z,2l)}\rangle|\,.\raisetag{\baselineskip}
\end{align}

This represents precisely the analogue of \eqref{uLSHm} with the extra-decaying factor $\frac1{2^n}$ and the sum over $l\sim 2^{(\a-1)j}$ replaced by $l\sim 2^{(\a-1)j+n-m}$. Once, at this point, repeating the reasonings from Section \ref{uLSHm-range}, we conclude
\begin{equation}
	\big|\uL_{H,m,n}^{NS,II}(f,g,h)\big|\lesssim_{\a} \frac{n}{2^n}\|f\|_p\|g\|_q\|h\|_{r'}\,,
\end{equation}
for any $\frac1p+\frac1q+\frac1r=1$ with $1\le p,\, q\le \infty$ and $r>\frac23$.

\paragraph{The $\uL^{NS,II}_{NL}$ term.}
In this setting, \eqref{umjmnlNSII} reduces to
\begin{equation*}
	\um_{j,m,n}(\xi,\eta)\approx \frac{1}{2^n} \left(\int_\R e^{-i\frac{\xi-\eta}{2^j}  t}e^{i\frac{\eta}{2^{\a j}}t^\a}\rho(t)dt\right)\,\phi\Big(\frac{\xi}{2^{j+m}}\Big) \phi\Big(\frac{\eta}{2^{\a j+n}}\Big) \phi\Big(\frac{\xi+\eta}{2^{j+m}}\Big)\,,
\end{equation*}
which, given that we are in the regime $(\a-1)j\le -n$, implies
\begin{align}
	&\qquad\qquad\qquad\qquad\qquad\qquad\uL_{j,m,n}(f,g,h)\approx \\
	&\frac{1}{2^n}\int_{\R^2}
	(f*\check{\phi}_{j+m})\Big(x-\frac{t}{2^j}\Big)
	(g*\check{\phi}_{\a j+n})\Big(x+\frac{t^{\a}}{2^{\a j}}\Big)
	(\e_j h*\check{\phi}_{j+m})(x) \urho(t)\,dt\,dx\,.\nonumber \raisetag{.5\baselineskip}
\end{align}

From this point on, one can follow the same reasonings as the ones from Section \ref{SNLm-range}, in order to obtain (with the same notations)
\begin{equation}\label{maxshift2AB}
\uL^{NS,\b}_{NL,m}(f,g,h)\lesssim_{\a} \frac{1}{2^n}\,\max\left\{1,\,2^{\b}\,\frac{|G|}{|H|}\right\}\,\int_{\R} S(f)(x)\,S_{\I_{\b},s}^{max}(\e,h)(x)\,dx\,,
\end{equation}
which implies the desired conclusion.

\paragraph{The $\uL^{NS,II}_{TR}$ term.} In this case we have, uniformly in $0\leq m\leq n-C_0$ and $j\in \Z$, that
\begin{equation}
	|\uL_{j,m,n}(f,g,h)|\lesssim_{\a,p,q} \frac{1}{2^n}\,\|f\|_p\,\|g\|_q\,\|h\|_{r'}
\end{equation}
for any $\frac{1}{p}+\frac{1}{q}+\frac{1}{r'}=1$ with $1\leq p,\,q\leq \infty$. Hence
\begin{equation}
	|\uL^{NS,II}_{TR,n}(f,g,h)|\lesssim_{\a,p,q} \frac{n}{2^n}\,\|f\|_p\,\|g\|_q\,\|h\|_{r'}\;.
\end{equation}

\subsubsection{The treatment of the term $\uL^{NS,III}$}

In this case $0\le n<m-C(a,b,\a)$. Then by an integration by parts we have
\begin{equation}\label{umjmnlNSIII}
	\um_{j,m,n}\approx\frac1{2^m}\Big(\int_\R e^{-i\frac{\xi-\eta}{2^j}t}e^{i\frac\eta{2^{\a j}}t^\a}\urho(t)dt\Big)\phi\Big(\frac{\xi-\eta}{2^{j+m}}\Big) \phi\Big(\frac{\eta}{2^{\a j+n}}\Big)\,.
\end{equation}

From this point on the treatment of $\L^{NS,III}_\star$ is similar to the one of $\L^{NS,II}_\star$ for any $\{\star\}\in\{NL,TR,H\}$\,.

\subsection{The analysis of the low frequency component $M^{LO}$}\label{mlo}
This represents the operator whose multiplier, after the linearization process, is given by
\begin{equation}
	\um_j^{LO}=\sum_{(m,n)\in \Z\times\Z_{-}} \um_j(\xi,\eta)\phi\Big(\frac{\xi-a\eta}{2^{j+m}}\Big)\phi\Big(\frac{\a\,b\,\eta}{2^{\a j+n}}\Big)=
	\um_j(\xi,\eta)\,\psi\Big(\frac{\a\,b\,\eta}{2^{\a j}}\Big) \,.
\end{equation}

The upper bound for the linear dominant component may be reduced to
\begin{equation}\label{defmcomponentsupbd}
	M^{LO}(f,g)(x)\lesssim \sup_{j\in\Z}\left| \int_{\R} \int_{\R} \hat{f}(\xi)\,\hat{g}(\eta)\,
	\hat{\tilde{\rho}}\Big(\frac{\xi-a\eta}{2^{j}}\Big)\,\psi\Big(\frac{\a\,b\,\eta}{2^{\a j}}\Big)\,e^{i\xi x}\,e^{i\eta x}\,d\xi\,d\eta\right|\,,
\end{equation}
from which we further deduce
\begin{align}\label{defmcomponentsupbd1}
	M^{LO}(f,g)(x)&\lesssim \sup_{j\in\Z}\left| \int_{\R} f\big(x-\frac{t}{2^j}\big)\,(g*\check{\psi}_{j \a}) \big(x+\frac{a\,t}{2^j}\big)\,\tilde{\rho}(t)\,dt\right|\nonumber\\
	&\lesssim \sup_{j\in\Z}\int_{\R} |f|(x-t)\,\M g(x+ a t)\,2^{j}\,|\tilde{\rho}(2^{j}\,t)|\,dt\:.
\end{align}
This implies that
\begin{equation}\label{defmcomponentsupbd2}
	M^{LO}(f,g)\lesssim \underline{M}_{a} (|f|, \M g)\,,
\end{equation}
where here $\underline{M}_{a}$ stands for the maximal version of the Bilinear Hilbert transform. Based on the work in \cite{Lac00}, we conclude the $L^p\times L^q\,\rightarrow\,L^r$ boundedness of $M^{LO}$ as long as $\frac1p+\frac1q=\frac1r$ with $1<p,q\leq \infty$ and $r>\frac23$.

\subsection{The analysis of the non-singular component $M^{NSI}$}
In this case, we make use of Theorem \ref{mainsinglsecale}, in order to deduce
\begin{equation*}
	\|M^{NSI}(f,g)\|_{L^r}\leq \Big\|\sup_{|j|<C_0} M_j(f,g)\Big\|_{L^r}\leq j_0\, \|f\|_{L^p}\,\|g\|_{L^q}\lesssim \|f\|_{L^p}\,\|g\|_{L^q}\,,
\end{equation*}
whenever $\frac1p+\frac1q=\frac1r$ with $1<p,q\leq \infty$ and $r>\frac12$.

\section{Connections with the Bilinear-Hilbert Carleson operator $BC^{\a}$}\label{BHCalfa}

In this section we are going to discuss some interesting connections between
\begin{itemize}
\item the family of Bilinear Hilbert transform along curves $\{H_{a,b}^{\a}\}_{\a>0}$ given by \eqref{bht} and analyzed in the current paper, and

\item the class of Bilinear Hilbert-Carleson operators $\{BC^{\a}\}_{\a>0}$ introduced in \cite{BBLV21}  and defined by

\begin{equation}\label{bhc}
	BC^{\a}(f,g)(x):= \sup_{\l>0}\,\left|\int_{\R} f(x-t)\, g(x+t)\,e^{i\,\l\,t^{\a}}\,\frac{dt}{t}\right|\:.
\end{equation}

\end{itemize}

\subsection{The boundedness of the non-resonant Bilinear Hilbert-Carleson operator $BC^{\a}$, $\a\in (0,\infty)\setminus\{1,2\}$: an outline}

Building on the deep connections with historical, celebrated themes such as Carleson's Theorem (\cite{Car66}, \cite{Feff73}) and the Bilinear Hilbert transform (\cite{LT97}, \cite{LT99}), part of the relevance of the (non-resonant) Bilinear Hilbert-Carleson operator treated in \cite{BBLV21} relies on the fact that $BC^{\a}$, $\a\in (0,\infty)\setminus\{1,2\}$, is the first multi-linear operator in the literature that exhibits a \emph{hybrid} nature, encompassing both
\begin{itemize}
\item zero-curvature features: $BC^{\a}$ is invariant under the action of modulation:
$$BC^{\a}(M_cf,M_cg)=BC^{\a}(f,g)\,;$$
\item nonzero-curvature features: $BC^{\a}$ is a maximal singular integral operator whose kernel representation involves a maximal (generalized) modulation in the presence of curvature.
\end{itemize}

With these being said, we recall the main result in \cite{BBLV21}:

\begin{theorem} \label{main}  Let $\a\in(0, \infty)\setminus\{1, 2\}$ and assume $p, q, r$ are H\"older indices, \textit{i.e.} $\frac{1}{p}+\frac{1}{q}=\frac{1}{r}$, with $1<p, q \leq \infty$ and $\frac{2}{3}<r<\infty$. Then the non-resonant bilinear Hilbert--Carleson operator $BC^{\a}$ extends continuously from $L^p(\R)\times L^q(\R)$ into $L^r(\R)$ with
\beq\label{mainr}
\|BC^{\a}(f,g)\|_{L^r}\lesssim_{\a,p,q} \|f\|_{L^p} \|g\|_{L^q}.
\eeq
\end{theorem}

In what follows, we present a brief outline of the methods employed for proving \eqref{mainr}: firstly, one applies a standard linearization procedure and splits
\beq\label{bcdec}
BC^{\a}=BC^{\a}_0\,+\,BC^{\a}_{osc}\,,
\eeq
where
\begin{itemize}

\item $BC^{\a}_0$ stands for the \emph{low oscillatory} component and is given by
\beq\label{bc0}
BC^{\a}_0(f,g)(x):=\int_{\R}  f(x-t)\, g(x+t)\,e^{i\,\l(x)\,t^{\a}}\,\chi(\l(x),t)\,\frac{dt}{t} \,,
\eeq
\item $BC^{\a}_{osc}$ stands for the \emph{high oscillatory} component given by
\beq\label{bcosc}
BC^{\a}_{osc}(f,g)(x):=\int_{\R}  f(x-t)\, g(x+t)\,e^{i\,\l(x)\,t^{\a}}\,(1-\chi(\l(x),t))\,\frac{dt}{t} \,,
\eeq
\end{itemize}
with $\l(x)$ an arbitrary measurable positive function and $\chi(\l(x),t)$ a sui\-table smooth version of the characteristic function of the set $|\l(x)\,t^{\a}|\gtrsim 1$.

Next, the strategy is as follows:
\begin{itemize}
\item the low oscillatory component $BC^{\a}_0$ behaves essentially as the maximal truncated Bilinear Hilbert transform, and thus, relying on Lacey's result in \cite{Lac00}, is under control.

\item the high oscillatory component $BC^{\a}_{osc}$ is firstly organized depending on the height of the multiplier's phase, that is
\beq\label{bcoscdecm}
BC^{\a}_{osc}=\sum_{m\in\N} BC^{\a}_{m}\,,
\eeq
with
\beq\label{bcoscdecmdef}
BC^{\a}_{m}:=\sum_{k\in\Z} BC^{\a}_{m,k}\,.
\eeq
This is further decomposed as
\beq\label{bcoscdec}
BC^{\a}_{osc}=\sum_{{m\in\N}\atop{k\in\Z}} BC^{\a}_{m,k}\,,
\eeq
with\footnote{Here $\tilde{\rho}\in C_0^{\infty}(\R)$ with $0\notin\textrm{supp}\,\tilde{\rho}$ and
$\sum_{k\in\Z} \tilde{\rho}(2^{-\a k}\,t)=1$ for all $t\in \R\setminus\{0\}$.}
\beq\label{bcoscdecmk}
 BC^{\a}_{m,k}(f,g)(x):=\tilde{\rho}(2^{-\a(m-k)}\l(x))\,\int_{\R}  f(x-t)\, g(x+t)\,e^{i\,\l(x)\,t^{\a}}\,2^{-k}\,\rho(2^{-k}\,t)\,dt\,.
\eeq
\end{itemize}

The key step is to obtain the \emph{global $m$-decaying} estimate
\beq\label{bcoscdecmdefkey}
\|BC^{\a}_{m}(f,g)\|_1\lesssim_{\a} 2^{-\d \a m}\,\|f_2\|\,\|g\|_2\,,
\eeq
for some small, absolute $\d>0$. This is approached in two stages, according to the above mentioned hybrid nature of $BC^{\a}$:

\begin{itemize}
\item the first stage consists of the \emph{(local) single scale $m$-decaying} estimate
\beq\label{bcoscdecmkdefkey}
\|BC^{\a}_{m,k}(f,g)\|_1\lesssim_{\a} 2^{-\d \a m}\,\|f_2\|\,\|g\|_2\,,
\eeq
which is essentially the content of Proposition 5.1., Section 5 in \cite{BBLV21}. The proof of \eqref{bcoscdecmkdefkey} exploits the \emph{non-zero curvature feature} of the oscillatory component within $BC^{\a}$ and relies on the LGC-methodology employed in \cite{Lie19} combined with elements of number theory (Weyl sums). Moreover, one has that \eqref{bcoscdecmkdefkey} holds iff $\a\notin\{1,2\}$.

\item the second stage seeks to propagate the information in \eqref{bcoscdecmkdefkey} from local to global level in order to deduce the overall $m$-decaying estimate \eqref{bcoscdecmdefkey}. This involves the \emph{zero curvature feature} characterizing the input structure of $BC^{\a}$, and thus, it requires a multi-scale analysis relying on a refinement of the time-frequency analysis developed in \cite{LT97} and \cite{LT99}.
\end{itemize}

In what follows we will only appeal to the \emph{non-zero curvature} feature discussed at the first item above, focusing our attention on re-interpreting \eqref{bcoscdecmkdefkey} in relation to the main topic of our present paper.

\subsection{$H_{a,b}^{2}-$a novel manifestation: quadratic \emph{quasi}-resonance}\label{H2discussion}

\indent In this section we discuss a novel type of behavior modeled by the $H_{a,b}^{2}$ operator: that of a mathematical object obeying a so-called ``\emph{higher order modulation quasi-invariance property}".\footnote{The meaning of this concept is clarified in Remark \ref{Quasi} at the end of this section, after we describe a certain parallelism between $H_{a,b}^{2}$ and $BC^2$.} As it turns out, this newly stated feature that will be explained momentarily constitutes the deep reason for which our Theorem \ref{main1} does not cover the case $\a=2$:
\medskip

\noindent\textbf{Claim.} \emph{No decay in the parameter $m$ is possible in Proposition \ref{L2mdecay} if $\a=2$.}

\begin{proof}

We start by presenting some circumstantial evidence for our claim that supports the special/singular role played by the quadratic case $\a=2$ within the family $\{H_{a,b}^{\a}\}_{\a}$. Indeed, we will first address the more elementary theme of \emph{where does the LGC-methodology implementation in Section \ref{LGCimpl} fail to provide the desired decay in \eqref{LSm}?}

The answer to the above stays in the lack of validity for the reasonings provided in Section \ref{LGCplus}, Step \ref{TFCsize}, more specifically in the failure of the key estimate \eqref{keyrel}. Indeed, inspecting closer the arguments therein, one notices that the original estimate \eqref{f5} becomes in the new context $|\vartheta'(v)|\approx \frac{1}{2^{m+2j}}$ which in turn transforms \eqref{key10} into
\begin{equation*}
|u_1-u_2|\,|v(r_1)-v(r_2)|\lesssim_{\a} 2^{\frac{m}{2}+j}\,,
\end{equation*}
thus preventing any contradiction from being reached within assumption \eqref{keyrel1} as long as $j>0$ is large enough relative to $m$. 

We end this introductory discussion by noticing that the contrast between the cases $\a=2$ and $\a\not=2$ manifests only when studying the hybrid stationary component $\L^{S}_{H,m}$ and not when addressing the non-linear stationary component $\L^{S}_{NL,m}$; this dichotomy is consistent with the asymptotic/dominant behavior within the expression of $\g(t)=t+ t^{\a}$ as the spatial parameter $|t|$ approaches $0$ or $\infty$, respectively.

With these done, we are now ready to provide the actual proof of our claim. Interestingly enough, our approach exploits the relationship between $\L^{S}_{H,m}$ and the Bilinear Hilbert-Carleson operator $BC^{\a}$, $\a=2$.
	
For expository reasons, as before, we take $a=b=1$ in \eqref{maintwo} and set
	\begin{equation}\label{bht20}
		H_{2}(f,g)(x):= \int_{\R} f(x-t)\, g(x+t+t^{2})\,\frac{dt}{t}\:.
	\end{equation}
	
Once here, it is not hard to see that the analogue of \eqref{sjm} for $\a=2$ is given by
	\begin{equation}\label{s2jm}
		H^{2,S}_{j,m}(f,g):=\int_\R\int_\R \hat{f}(\xi)\hat{g}(\eta)m_{j,m}^{2,S}(\xi,\eta)e^{i\xi x}e^{i\eta x}d\xi d\eta\,,
	\end{equation}
	with
	\begin{equation}\label{s2jmm}
		m_{j,m}^{2,S}(\xi,\eta):=\left(\int_\R e^{-i\frac{\xi-\eta}{2^j}  t}e^{i\frac{\eta}{2^{2j}}t^2}\rho(t)dt\right)\,
		\phi\Big(\frac{\xi-\eta}{2^{j+m}}\Big)\phi\Big(\frac{\eta}{2^{2j+m}}\Big)\,.
	\end{equation}
	Now following \eqref{saljmm0}--\eqref{saljm21} and assuming $j\geq 0$, simple considerations give that
	\begin{equation}\label{s2jmm0}
		m_{j,m}^{2,S}(\xi,\eta)\approx\sum_{l\sim 2^j} m_{j,m,l}^{2,S}(\xi,\eta)
	\end{equation}
	with
	\begin{equation}\label{s2jmm1}
		m_{j,m,l}^{2,S}(\xi,\eta):=\left(\int_\R e^{-i\frac{\xi-\eta}{2^j}  t}e^{i\frac{\eta}{2^{2j}}t^2}\rho(t)dt\right)\,
		\phi_1\Big(\frac{\xi}{2^{j+m}}-l\Big)\phi_2\Big(\frac{\eta}{2^{j+m}}-l\Big)\,,
	\end{equation}
	where here $\phi_1$, $\phi_2$ are functions having similar properties with $\phi$.
	
	Thus, with the obvious correspondences, we deduce that
	\begin{equation}\label{s2jm2}
		H^{2,S}_{j,m,l}(f,g)=\int_{\R}  (f*\check{\phi}_{1,j+m}^l)\Big(x-\sdfrac{t}{2^j}\Big)\,(g*\check{\phi}_{2,j+m}^l)\Big(x+\sdfrac{t}{2^j}+\frac{t^2}{2^{2j}}\Big)\,\rho(t)\,dt\,.
	\end{equation}
	where for $\iota\in\{1,2\}$ we have
	\begin{equation}\label{s2jm3-1}
		\check{\phi}_{\iota,j+m}^l(x):=\int_{\R} \phi_{\iota}\Big(\frac{\xi}{2^{j+m}}-l\Big)\,e^{i\,\xi\,x}\,d\xi=
		2^{j+m}\,e^{i\,2^{j+m}\,l\,x}\,\check{\phi}_{\iota}(2^{j+m}\,x)\,.
	\end{equation}
	Deduce based on \eqref{s2jm3-1} that, if $j\geq m$, then
	\begin{equation}\label{s2jm4ss}
		(g*\check{\phi}_{2,j+m}^l)\Big(x+\sdfrac{t}{2^j}+\sdfrac{t^2}{2^{2j}}\Big)\approx e^{i\,\frac{l}{2^{j}}\,2^m\,t^2}\,(g*\check{\phi}_{2,j+m}^l)\Big(x+\sdfrac{t}{2^j}\Big)
	\end{equation}

Thus, for $l\sim 2^j$ and $j\geq m$, we deduce that \eqref{s2jm2} can be rewritten as
\begin{equation}\label{h2S}
		H^{2,S}_{j,m,l}(f,g)(x)\approx\int_{\R}  (f*\check{\phi}_{1,j+m}^l)\Big(x-\frac{t}{2^j}\Big)
		(g*\check{\phi}_{2,j+m}^l)\Big(x+\frac{t}{2^j}\Big)\,e^{i\,\frac{l}{2^{j}}\,2^m\,t^2}\,\rho(t)\,dt\,.
\end{equation}

The above will provide us with the key insight into our proof, since from \cite{BBLV21}, we know that \eqref{bcoscdecmkdefkey} is false for $\a=2$. Indeed, taking first $\l(x)=l\, 2^j\, 2^m$, we set
\begin{equation}\label{Tmj}
T_{m,j}^{l,2}:= BC^{2}_{\frac{m}{2},-j}\,,
\end{equation}	
and notice by comparing \eqref{h2S} with \eqref{bcoscdecmk}, that, in the regime $j\geq m$, we have 	
\begin{equation}\label{Tmjeq}
H^{2,S}_{j,m,l}(f,g)\approx T_{m,j}^{l,2}(f*\check{\phi}_{1,j+m}^l,g*\check{\phi}_{2,j+m}^l)\,.
\end{equation}
		
Once at this point, let us \emph{assume by contradiction} that Proposition \ref{L2mdecay} holds for $\a=2$. Then, based on \eqref{Tmjeq}, there exists some $\d>0$ such that
\begin{equation}\label{Tmjeq1}
\|T_{m,j}^{l,2}(f*\check{\phi}_{1,j+m}^l,g*\check{\phi}_{2,j+m}^l)\|_1\lesssim 2^{-\d m}\,\|f\|_2\,\|g\|_2
\end{equation}	
uniformly in $m\in\N$, $j\geq m$, $l\sim 2^j$ and $f,\,g\in L^2(\R)$.	
	
For simplicity we take now $l=2^{j}$, and, for arbitrary $n>0$, let\footnote{Basic Fourier analysis provides the existence of $c\in\C$ such that
$\hat{e}_n(\xi)=\frac{c}{\sqrt{2^{n-1}}}\, e^{i\,\pi^2\frac{\xi^2}{2^{n-1}}}\,.$}
\begin{equation}\label{ek}
e_n(x):=e^{-i 2^{n-1} x^2}\,.
\end{equation}	

If \eqref{Tmjeq1} holds then, for $j\geq m$, the following should also hold:
\begin{equation}\label{Tmjtrue}
\|T_{m,j}^{2^j,2}((f\,e_{m+2j})*\check{\phi}_{1,j+m}^{2^j},(g\,e_{m+2j})*\check{\phi}_{2,j+m}^{2^j})\|_{L^1[0,2^{-j}]}\lesssim 2^{-\d m}\,\|f\|_2\,\|g\|_2\:.
\end{equation}	
Now for
\begin{equation}\label{xt}
x\in [0,2^{-j}]\:\:\textrm{and}\:\:t\in \textrm{supp}\,\rho\:\:(\textrm{hence}\:\:|t|\lesssim 1)
\end{equation}
we analyze the term:	
$$\big((f\,e_{m+2j})*\check{\phi}_{1,j+m}^{2^j}\big)\Big(x-\sdfrac{t}{2^j}\Big)= $$
$$e_{m+2j}\Big(x-\sdfrac{t}{2^j}\Big)
\int_{\R} f\Big(x-\frac{t}{2^j}-s\Big)\,e^{i 2^{2j+m} s}\,e^{i 2^{2j+m} \big(x-\frac{t}{2^j}\big) s }\,e_{m+2j}(s)\,2^{j+m}\,
\check{\phi}_{1}(2^{j+m} s)\,ds\,.$$
Using now \eqref{xt} and the fact that due to the time-frequency localization of $\phi_1$ the region of integration for $s$ is essentially reduced to $|s|\lesssim 2^{-j-m}$, one gets first that
\begin{equation}\label{xts}
|2^{m+2j}\,s^2|\lesssim 1\,,
\end{equation}
from which one further deduces
\begin{equation}\label{appr}
	\big((f\,e_{m+2j})*\check{\phi}_{1,j+m}^{2^j}\big)\Big(x-\sdfrac{t}{2^j}\Big)\approx
\end{equation}
$$e_{m+2j}\Big(x-\sdfrac{t}{2^j}\Big)\,
\int_{\R} f\Big(x-\sdfrac{t}{2^j}-s\Big)\,e^{i (2^j\,(x+1)-t)\, 2^{j+m}\,s }\,2^{j+m}\,\check{\nu}_{1}(2^{j+m} s)\,ds\,,$$
where $\nu_1\in C_{0}^{\infty}(\R)$ has the property $\check{\nu}_1(s)=\check{\phi}_{1}(s)$ on $|s|\leq C$ and $\textrm{supp}\,\nu_1\subseteq \{s\in\R\,|\,|s|\leq  2\,C\}$ for some suitable large $C>0$.

In a similar fashion, with the obvious correspondences, one gets
\begin{equation}\label{appr0}
	\big((g\,e_{m+2j})*\check{\phi}_{2,j+m}^{2^j}\big)\Big(x+\sdfrac{t}{2^j}\Big)\approx
\end{equation}
$$ e_{m+2j}\Big(x+\frac{t}{2^j}\Big)\,
\int_{\R} g\Big(x+\frac{t}{2^j}-u\Big)\,e^{i (2^j\,(x+1)+t)\, 2^{j+m}\,u}\,2^{j+m}\,\check{\nu}_{2}(2^{j+m} u)\,du\,.$$

For $x\in [0,2^{-j}]$ set now
\begin{equation}\label{Idef}
I(x):=\Big|T_{m,j}^{2^j,2}((f\,e_{m+2j})*\check{\phi}_{1,j+m}^{2^j},(g\,e_{m+2j})*\check{\phi}_{2,j+m}^{2^j})(x)\Big|\,.
\end{equation}

Combining now  \eqref{appr} and \eqref{appr0}, we deduce
\begin{align*}
	I(x)&\approx\left|\int_{\R}  ((f\,e_{m+2j})*\check{\phi}_{1,j+m}^{2^j})(x-\frac{t}{2^j})\,
	((g\,e_{m+2j})*\check{\phi}_{2,j+m}^{2^j})(x+\frac{t}{2^j})\,e^{i\,2^m\,t^2}\,\rho(t)\,dt\right|\\[1ex]
	&\approx \bigg|\int_{\R} \int_{\R^2} f\Big(x-\frac{t}{2^j}-s\Big)\,g\Big(x+\frac{t}{2^j}-u\Big)\,e^{i (2^j\,(x+1)+t)\, 2^{j+m}\,u}\times\\
	&\hspace{1.5cm}\times  e^{i (2^j\,(x+1)-t)\, 2^{j+m}\,s }\,2^{j+m}\,
	\check{\nu}_{1}(2^{j+m} s)\,\,2^{j+m}\,
	\check{\nu}_{2}(2^{j+m} u)\,ds\,du\,\rho(t)\,dt\bigg|\\[1ex]
	&=\bigg|\int_{\R^2} \hat{f}(\xi)\,\hat{g}(\eta)\int_{\R^2} e^{i\,\xi\,(x-s)}\,e^{i\,\eta\,(x-u)}\,e^{i (2^j(x+1))\, 2^{j+m}\,u}\,e^{i(2^j(x+1))\, 2^{j+m}\,s }\times\\
	&\hspace{1cm}\times 2^{j+m}\,\check{\nu}_{1}(2^{j+m} s)\,2^{j+m}\,\check{\nu}_{2}(2^{j+m} u)\, \hat{\rho}\Big(\sdfrac{\eta-\xi}{2^j}+2^{j+m}(u-s)\Big)\,ds\,du\,d\,\xi\,d\eta\,\bigg|\,.
\end{align*}
Choosing now $\hat{g}=\chi_{[2^{2j+m}+2^{j},\,2^{2j+m}+2^{j+1}]}$ and $\hat{f}=\chi_{[2^{2j+m},\,2^{2j+m}+2^{j}]}$ and making the change of variable $s\,\mapsto\,\frac{s}{2^{j+m}}$ and $u\,\mapsto\,\frac{u}{2^{j+m}}$ we have that the previous expression further satisfies
\begin{align}
	I(x)&\approx\left|\int_{\R^2} \hat{f}(\xi)\,\hat{g}(\eta)\,e^{i\,\xi\,x}\,e^{i\,\eta\,x}\,\nu_1\Big(2^j(x+1)-\sdfrac{\xi}{2^{j+m}}\Big)\,
	\nu_2\Big(2^j(x+1)-\sdfrac{\eta}{2^{j+m}}\Big)\,d\xi\,d\eta\right|\nonumber\\[1ex]
	&\approx  2^{2j}\,\big|\nu_1\big(2^j x\big)\, \nu_2\big(2^j x\big)\Big|\,.	\label{appr01}
\end{align}
Now, from \eqref{Idef} and \eqref{appr01}, we conclude that
$$\|T_{m,j}^{2^j,2}((f\,e_{m+2j})*\check{\phi}_{1,j+m}^{2^j},(g\,e_{m+2j})*\check{\phi}_{2,j+m}^{2^j})\|_{L^1[0,2^{-j}]}\approx 2^{j}\approx\|f\|_{L^2}\,\|g\|_{L^2}\,,$$
thus disproving \eqref{Tmjeq1} and finishing the proof of our claim.
\end{proof}

\begin{remark}\label{Quasi}[\textsf{Generalized modulation quasi-invariance}]  The operator $H_2$ defined by \eqref{bht20} does not satisfy any \emph{global modulation invariance per se}. However, at the \emph{local} level, the interaction between the physical scale $j$ (for large values relative to $m$) and the frequency scale $m$, produces within the structure of the hybrid stationary component \eqref{s2jm} a resonance in the form of a \emph{quadratic modulation quasi-invariance}: that is, asymptotically, $H^{2,S}_{j,m,l}$ approaches---in the sense described by \eqref{Tmj}--\eqref{Tmjeq}---an operator, $BC^{2}$, that does have genuine quadratic modulation invariance.
\end{remark}

\subsection{An alternative proof of the $L^2\times L^2\times L^\infty$ $m$-decaying bound for $\L^{S}_{H,m}$ in the case $\a\in (0,\infty)\setminus\{1,2\}$}

In this section we provide a different approach to the part of estimate \eqref{LSm} in Proposition \ref{L2mdecay} that concerns $\L^{S}_{H,m}$. To this end, following the general line of thought from Section \ref{H2discussion}, we rely on reducing the matters to the behavior of the operator $BC^{\a}$.

Throughout this section we fix $m\in\N$ and assume that $(\a-1)j>m$; as customary by now, for simplicity, we further assume $a=b=1$, $\a>1$ (hence $j\in\N$).

As before, we start from
	\begin{equation}\label{bht20a}
		H_{\a}(f,g)(x):= \int_{\R} f(x-t)\, g(x+t+t^{\a})\,\frac{dt}{t}\:,
	\end{equation}
then isolate its $(j,m)$-stationary component
\begin{equation}\label{saljm}
		H^{\a,S}_{j,m}(f,g):=\int_\R\int_\R \hat{f}(\xi)\hat{g}(\eta)m_{j,m}^{\a,S}(\xi,\eta)e^{i\xi x}e^{i\eta x}d\xi d\eta\,,
\end{equation}
with
\begin{equation}
		m_{j,m}^{\a,S}(\xi,\eta):=\left(\int_\R e^{-i\frac{\xi-\eta}{2^j}  t}e^{i\frac{\eta}{2^{\a j}}t^\a}\rho(t)dt\right)\,
		\phi\Big(\frac{\xi-\eta}{2^{j+m}}\Big)\phi\Big(\frac{\eta}{2^{\a j+m}}\Big)\,.
\end{equation}

Recalling now \eqref{HS}--\eqref{saljm21}, we deduce that
\begin{equation}\label{saljmsum}
H^{\a,S}_{j,m}(f,g)(x)=\sum_{l\sim 2^{(\a-1)j}}H^{\a,l,S}_{j,m}(f,g)(x)\,,
\end{equation}
where, preserving the notation from the previous section, we have
\begin{equation}\label{saljm2-3}
H^{\a,S}_{j,m,l}(f,g)=\int_{\R}  (f*\check{\phi}_{1,j+m}^l)\Big(x-\frac{t}{2^j}\Big)\,(g*\check{\phi}_{2,j+m}^l)\Big(x+\frac{t}{2^j}+\frac{t^\a}{2^{\a j}}\Big)\, \rho(t)\,dt\,,\quad
\end{equation}
which, in the usual fashion, can be further approximated by the expression
\begin{equation}\label{saljm4}
\int_{\R}  (f*\check{\phi}_{1,j+m}^l)\Big(x-\sdfrac{t}{2^j}\Big)\,
		(g*\check{\phi}_{2,j+m}^l)\Big(x+\sdfrac{t}{2^j}\Big)\,e^{i\,\frac{l}{2^{j (\a-1)}}\,2^m\,t^\a}\rho(t)\,dt\,.
\end{equation}
Thus, similarly to \eqref{Tmj}, we remark that
\begin{equation}\label{Tmjeqal}
H^{\a,S}_{j,m,l}(f,g)\approx T_{m,j}^{l,\a}(f*\check{\phi}_{1,j+m}^l,\,g*\check{\phi}_{2,j+m}^l)\,.
\end{equation}

Applying now \eqref{bcoscdecmkdefkey} we get
\begin{equation}\label{Tmjeq1al}
\big\|T_{m,j}^{l,\a}(f*\check{\phi}_{1,j+m}^l,\,g*\check{\phi}_{2,j+m}^l)\big\|_1\lesssim 2^{-\d m}\,\|f*\check{\phi}_{1,j+m}^l\|_2\,\|g*\check{\phi}_{2,j+m}^l\|_2\,,
\end{equation}	
which combined with \eqref{saljmsum}, \eqref{Tmjeqal} and a standard Cauchy-Schwartz gives
\begin{align}
	\|H^{\a,S}_{j,m}(f,g)\|_1&\leq \sum_{l\sim 2^{(\a-1)j}}\|H^{\a,l,S}_{j,m}(f,g)\|_1\nonumber\\
	&\leq \sum_{l\sim 2^{(\a-1)j}}2^{-\d m}\,\|f*\check{\phi}_{1,j+m}^l\|_2\,\|g*\check{\phi}_{2,j+m}^l\|_2\lesssim
	2^{-\d m}\,\|f\|_2\,\|g\|_2\,, \raisetag{.5\baselineskip}\label{comclal}
\end{align}
thus completing the desired alternative proof.

\section{The linear Hilbert transform $\H_{\g}$ and the $\g$-Carleson operator $C_{\g}$ along hybrid curves with no quadratic resonances}\label{Hywithout}

In this section we present an outline of the proof of Theorem \ref{main4}. As well known, see also \cite{Lie19}, the $L^2$  boundedness of $\H_{\g}$ is in fact equivalent to the $L^2$ boundedness of $C_{\g}$, and thus, it is enough to focus on the treatment of the latter. Recalling now the discussion and notation in Section \ref{Unif}, in the context offered by \eqref{C-gamma}, we define for $j\in\Z$ and $m\in\N$ the operator
\begin{equation}\label{C-gammajm}
		C_{j,m}(f)(x):=\left(\int_{\R} f(x-t)\,e^{i\,\g(x,t)}\,\rho_j(t)\,dt\right)\,\phi\left(\frac{\|\g^{NL}(x,2^{-j})\|}{2^m}\right)\,,
\end{equation}
where here $\|\g^{NL}(x,2^{-j})\|:=\sum_{k=1}^d |a_k(x)|\,2^{-j\,\a_k}$ measures the contribution of the nonlinear component of $\g$.

Letting now
\begin{equation}\label{C-gammadec}
	C_{\g}=:\sum_{m\geq 0}C_{m}=:C_0\,+\,\sum_{m\in\N}\sum_{j\in\Z} C_{j,m}\,,
\end{equation}
simple considerations reduce the behavior of the linear dominant component $C_0$ to the behavior of the standard Carleson operator. Thus, using the essential almost orthogonal character of the family of operators $\{C_{j,m}\}_{j\in\Z}$ for each given $m\in\N$, we deduce that Theorem \ref{main4} follows once we are able to prove the following

\begin{theorem}\label{cjmbound} Let $\g(x,t)$ be as in the statement of Theorem \ref{main4}. Then, there exists an absolute $\ep>0$ such that, given $j\in\Z$ and $m\in\N$, the operator $C_{j,m}$ defined by \eqref{C-gammajm} obeys
\begin{equation}\label{cjmdec}
		\|C_{j,m}f\|_{2}\lesssim 2^{-\ep\,m}\,\|f\|_2\,.
\end{equation}
\end{theorem}

\begin{proof}

In what follows, for the simplicity of the exposition, we assume without loss of generality that
\begin{equation}\label{tsrict}
 \a_0=1\quad\textrm{and}\quad \{\a_k\}_{k=1}^d\subset(0,\infty)\setminus\{1,2\}\:\:\textrm{strictly increasing}\,.
\end{equation}

Then, after a change of variable $t\,\longrightarrow\,x-t$ and up to a sign difference that for notational convenience we choose to ignore it, we have
\begin{equation}\label{C-gammajmc}
		C_{j,m}(f)(x)=\left(\int_{\R} f(t)\,e^{i\,\g(x,x-t)}\,\rho_j(x-t)\,dt\right)\,\phi\left(\frac{\|\g^{NL}(x,2^{-j})\|}{2^m}\right)\,.
\end{equation}
Once here, we follow the steps presented in Section \ref{Unif} and implement the (Rank I) LGC-methodology as follows:

$\newline$
\noindent\textsf{I. Phase linearization}
$\newline$

We decompose the time-frequency plane in boxes of area one of time-frequency size  $2^{-j-\frac{m}{2}}\times 2^{j+\frac{m}{2}}$ dictated by the desire to have a linearized behavior of the phase of the kernel/multiplier within each of these boxes. This is achieved via the space discretization (with the corresponding frequency discretization derived from the Heisenberg localization principle)
\begin{equation}\label{spatoial}
\rho_j(x-t)=2^j\,\rho(2^j(x-t))\approx 2^j\sum_{p\sim 2^{\frac{m}{2}}} \rho(2^{j+\frac{m}{2}}(x-t)-p)\,.
\end{equation}
Now on the support of $\rho(2^{j+\frac{m}{2}}(x-t)-p)$ and under the assumption $\phi\left(\frac{\|\g^{NL}(x,2^{-j})\|}{2^m}\right)\not=0$ we have indeed the linearized behavior:
\begin{equation}\label{SpphaseS}
\g(x,x-t)=\sum_{k=0}^d a_k(x)\,\a_k\,\left(\frac{p}{2^{j+\frac{m}{2}}}\right)^{\a_k-1}\,\left(x-t-\frac{p}{2^{j+\frac{m}{2}}}\right)\,
+\,\g_{j,m,p}(x)\,+\,O(1)\,,
\end{equation}
where here $\g_{j,m,p}(x)$ stands for an expression depending only on $\g$, $j,\,m,\, p$ and $x$.

$\newline$
\noindent\textsf{II. Adapted Gabor frame discretization}
$\newline$

Guided by the $2^{-j-\frac{m}{2}}$-discretization of the space, we apply an adapted linear wave-packet decomposition of the input function
\begin{equation}\label{GaborS}
f(t)=\sum_{l,w\in\Z} \langle f,  \check{\phi}_{l}^w\rangle\, \check{\phi}_{l}^w(t)\,,
\end{equation}
where here $\check{\phi}_{l}^w(t):=2^{\frac{j}{2}+\frac{m}{4}}\,\phi(2^{j+\frac{m}{2}}\,t\,-\,l)\,	e^{i\,2^{j+\frac{m}{2}}\,w\,t}$.

Putting together \eqref{C-gammajm}--\eqref{GaborS} we deduce
\begin{equation}\label{CarlPolynestS}
\end{equation}
$$|C_{j,m} f(x)|\lesssim 2^{\frac{3j}{2}+\frac{m}{4}}\,\sum_{{l,w\in\Z}\atop{p\sim 2^{\frac{m}{2}}}} | \langle f,  \check{\phi}_{l}^w\rangle|\,\phi\left(\frac{\|\g^{NL}(x,2^{-j})\|}{2^m}\right)\times$$
$$\left|\int_{\R} \check{\phi}(2^{j+\frac{m}{2}}\,t-l)\,e^{i\,2^{j+\frac{m}{2}}\,w\,t}\,e^{i\,\sum_{k=0}^d \a_k\,\left(\frac{p}{2^{j+\frac{m}{2}}}\right)^{\a_k-1}\,a_k(x)\,(x-t)}\,
  \rho(2^{j+\frac{m}{2}}(x-t)-p)\,dt\right|\,.$$
Finally, after a change of variable and integration by parts, we obtain the expression:

\begin{equation}\label{CarlPolynestKeyS}
  |C_{j,m} f(x)|\lesssim \sum_{{l,w\in\Z}\atop{p\sim 2^{\frac{m}{2}}}} \frac{2^{\frac{j}{2}-\frac{m}{4}}\,|\langle f,  \check{\phi}_{l}^w\rangle|\,\rho(2^{j+\frac{m}{2}}x-l-p)}
  {\left\lfloor\frac{\sum_{k=0}^d \a_k\,\left(\frac{p}{2^{j+\frac{m}{2}}}\right)^{\a_k-1}\,a_k(x)\,-\,w\,2^{j+\frac{m}{2}}}{2^{j+\frac{m}{2}}}\right\rfloor^2}\,
  \phi\left(\frac{\|\g^{NL}(x,2^{-j})\|}{2^m}\right)\:.
\end{equation}

Once at this point we massage a bit the RHS of \eqref{CarlPolynestKeyS} in order to put it in a form that is very close to the general setting appearing in \cite{Lie19}. Indeed, we notice that the RHS in \eqref{CarlPolynestKeyS} can be rewritten as

\begin{equation}\label{CarlPolynestKeyS1}
 \sum_{l,w\in\Z} \frac{2^{\frac{j}{2}-\frac{m}{4}}\,|\langle f,  \check{\phi}_{l}^w\rangle|\,\rho(2^{j}x-\frac{l}{2^{\frac{m}{2}}})}
  {\left\lfloor 2^{\frac{m}{2}}\,\left(\sum_{k=0}^d \a_k\,\left(2^j\,x-\frac{l}{2^{\frac{m}{2}}}\right)^{\a_k-1}\,\frac{a_k(x)}{2^{\a_k j\,+m}}\,-\,\frac{w}{2^{\frac{m}{2}}}\right)\right\rfloor^2}\,
  \phi\left(\frac{\|\g^{NL}(x,2^{-j})\|}{2^m}\right)\:.
\end{equation}
Taking now the square in \eqref{CarlPolynestKeyS1}, integrating in the $x$ variable, and, making the change of variable $x\,\rightarrow\,2^{-j}\,x$, we obtain the expression
\begin{equation}\label{CarlPolynestKeyS2}
\frac{1}{2^{\frac{m}{2}}}\,\int_{\R} \left(\sum_{l,w\in\Z} \frac{|\langle f,  \check{\phi}_{l}^w\rangle|\,\rho(x-\frac{l}{2^{\frac{m}{2}}})\,\phi\left(\frac{\|\g^{NL}(2^{-j}\,x,2^{-j})\|}{2^m}\right)}
{\left\lfloor 2^{\frac{m}{2}}\,\left(\sum_{k=0}^d \a_k\,\left(x-\frac{l}{2^{\frac{m}{2}}}\right)^{\a_k-1}\,\frac{a_k(2^{-j}\,x)}{2^{\a_k j\,+m}}\,-\,\frac{w}{2^{\frac{m}{2}}}\right)\right\rfloor^2}\right)^2\,dx\:.
\end{equation}
Since $j\in\Z$ and $m\in\N$ are fixed throughout our present reasonings we take $\tilde{a}_k(x):=\frac{a_k(2^{-j}\,x)}{2^{\a_k j\,+m}}$ and notice that $2^{-m}\,\|\g^{NL}(2^{-j}x,2^{-j})\|=\sum_{k=1}^d |\tilde{a}_k(x)|$.

Let now $\r_m$ represent the $2^{-\frac{m}{2}}$ equipartition of $[1,2]$, that is  $\r_m=\{\frac{u}{2^{\frac{m}{2}}}\}_{u=2^{\frac{m}{2}}}^{2^{\frac{m}{2}+1}}$ and $\textbf{c}=\{c_l^w\}_{l,w}$ with $c_l^{w}:=\langle f,  \check{\phi}_{l}^w\rangle$. Further on set
\begin{equation}\label{defq}
q(x,\,y):=\sum_{k=0}^d \a_k\,y^{\a_k-1}\,\tilde{a}_k(x)\qquad\textrm{and}\qquad \|q^{NL}(x)\|:=\sum_{k=1}^d |\tilde{a}_k(x)|\,.
\end{equation}

Given the form of \eqref{CarlPolynestKeyS2}, without loss of generality, we may assume from now on
\begin{equation}\label{CarlPolynestKeyS4}
\|q^{NL}(x)\|\approx 1\qquad\forall\:x\in \left[-\frac{1}{2},\,\frac{1}{2}\right]\:.
\end{equation}

Then, via some standard almost orthogonality arguments, we deduce that for proving \eqref{cjmdec} is enough to show that

\begin{equation}\label{CarlPolynestKeyS3}
\L(m, q, \textbf{c}):=\frac{1}{2^{\frac{m}{2}}}\,\int_{-\frac{1}{2}}^{\frac{1}{2}} \left(\sum_{l,w\in\r_m} \frac{|c_l^w|\,\phi(\|q^{NL}(x)\|)}
{\left\lfloor 2^{\frac{m}{2}}\,\left(q(x,\,x-l)\,-\,w\right)\right\rfloor^2}\right)^2\,dx\lesssim 2^{-2\,\ep\,m}\,\sum_{l,w\in\r_m}\,|c_l^w|^2\:.
\end{equation}

We now notice that \eqref{CarlPolynestKeyS3} represents the analogue of the estimate
\begin{equation}\label{previous}
\L(m, \tilde{q}, \textbf{c})\lesssim 2^{-2\ep m}\,\sum_{l,w\in\r_m} |c_l^w|^2,
\end{equation}
which is the main ingredient in the proof of Theorem 36---see Section 6 in \cite{Lie19}.

Once at this point, we briefly stop in order to point the following contrast:
\begin{itemize}
\item in the original setting of \cite{Lie19}, within \eqref{CarlPolynestKeyS3}--\eqref{previous} one only allows expressions of the form
    $$\tilde{q}(x,y)=\sum_{k=1}^d b_k(x)\,y^{\b_k}\quad\textrm{with}\quad \b_1<\ldots<\b_d\quad\textrm{and}\quad \{\b_k\}_{k=1}^d\subset\R\setminus\{0\}\,;$$

\item in the current setting, we allow in \eqref{previous} expressions that include linear monomials within the phase of multiplier, and hence
    $$q(x,\,y)=\tilde{a}_0(x)\,+\,\sum_{k=1}^d \a_k\,y^{\a_k-1}\,\tilde{a}_k(x)\,.$$
\end{itemize}

$\newline$
\noindent\textsf{III. Cancelation via time-frequency correlations}
$\newline$

Following the ideas in \cite{Lie19}--see the proof of Lemma 38 therein--for a suitable $\d=\d(\ep)>0$ chosen later and for each $l,w\in\r_m$, we define the set
\begin{equation}\label{Aset}
\A_{q,\d}(l,w):=\left\{x\in \left[-\frac{1}{2},\,\frac{1}{2}\right]\,\Big|\,|q(x,x-l)-w|\leq 2^{-(\frac{1}{2}-2\ep)\,m}\right\}\,,
\end{equation}
and further introduce
\begin{itemize}

\item the set of \emph{light} pairs as
\begin{equation}\label{light}
\mathcal{L}_{q,\d}:=\{(l,w)\in\r_m^2\,|\,|\A_{q,\d}(l,w)|\leq 2^{- 2\,\d\,m}\}\,;
\end{equation}

\item  the set of \emph{heavy} pairs as
\begin{equation}\label{heavy}
\mathcal{H}_{q,\d}:= \r_m^2\setminus \mathcal{L}_{q,\d}\,.
\end{equation}
\end{itemize}
In the same spirit with the key relation \eqref{keyrel} part of Step III.3.2. in Section \ref{LGCplus}, the crux of our present argument relies on obtaining a \emph{good control over the size of the time-frequency correlation set} represented here by the set of heavy pairs, \emph{i.e.}
\beq\label{keyrelHere}
\boxed{\exists\:\mu=\mu(\d)>0\quad\textrm{s.t.}\quad\#\mathcal{H}_{q,\d}\lesssim 2^{(\frac{1}{2}-\mu)\,m}\:.}
\eeq
If we assume for the moment that \eqref{keyrelHere} holds, then one can apply line by line the same reasonings as in the proof Lemma 38 in \cite{Lie19} addressing \eqref{previous} above in order to conclude that \eqref{CarlPolynestKeyS3} holds.

Thus, the only remaining task for us here is to show the validity of \eqref{keyrelHere}. Its proof follows in fact very closely the argumentation offered in the proof of Proposition 49 in \cite{Lie19}. Below we provide an outline of this:
\begin{itemize}
\item We assume by contradiction that
\beq\label{contrad}
\exists\:\nu>0\:\textrm{suitable small s.t.}\quad\#\mathcal{H}_{q,\d}\geq 2^{(\frac{1}{2}-\nu)\,m}\,.
\eeq
\item Choosing $\nu$ such that $\frac{1}{2}-\nu>2^{3d+10}\,\d$ we can find $\mathcal{H}_{q,\d}^{L,S}$ a \emph{maximal separated subset} of  $\{l\,|\,(l,w)\in \mathcal{H}_{q,\d}\}$ having the properties
\beq\label{distribcontr}
\#\mathcal{H}_{q,\d}^{L,S}=2^{2^{3d+5}\,\d\,m}\:\:\:\&\:\:\:\|\mathcal{H}_{q,\d}^{L,S}\|:=
\min_{{l_1,\,l_2\in \mathcal{H}_{q,\d}^{L,S}}\atop{l_1\not=l_2}}|l_1-l_2|\geq 2^{-(2^{3d+5}\,\d\,+\,\nu)\,m}\,.
\eeq
\item Introducing now $\A_l:=\A_{q,\d}(l,w_l)$ where $w_l$ is defined as $|\A_{q,\d}(l,w_l)|=\max_{w:\,(l,w)\in \mathcal{H}_{q,\d}} |\A_{q,\d}(l,w)|$ we apply Lemma 50 in \cite{Lie19} for $n=2^{2^{d+1}}$, $I_l=\A_l$, $M=2^{2\d m}$ and $N=\#\mathcal{H}_{q,\d}^{L,S}=2^{2^{3d+5}\,\d\,m}$ in order to deduce that there exist $\bar{\mathcal{H}}_q\subseteq \mathcal{H}_{q,\d}^{L,S}$ with $\bar{\mathcal{H}}_q=n$ and a set $X=\bigcap_{l\in \bar{\mathcal{H}}_q} A_{l}$ such that
\beq\label{KEY1}
|q(x,x-l)-w_l|\leq 2^{-(\frac{1}{2}-2\d)\,m}\:\:\textrm{for any}\:\:l\in \bar{\mathcal{H}}_q\:\:\textrm{and}\:\:x\in X\,,
\eeq
 and
 \beq\label{KEY2}
|X|\geq 2^{-2^{2d+3}\,\d\,m}\,.
\eeq
\item Applying Lemma 51 in \cite{Lie19} for $q$ as defined by \eqref{defq} above we get\footnote{Here is important to say that Lemma 51 holds for any $\tilde{q}(x,y)=\sum_{k=0}^d b_k(x)\,y^{\b_k}$ with $\{\b_k\}_{k=0}^d\subset\R$ thus covering the situation $\b_0=1$ represented by our $q$ here.}
\beq\label{coefcontrol}
\max_{0\leq k\leq d} |\a_k\,\tilde{a}_k(x)|\leq \frac{4^{2d\max_{k}|\a_k-1|}}{\|\bar{\mathcal{H}}_q\|^d\,\inf_{k}\,\prod_{0\leq i\not=k\leq d}\,|\a_k-\a_i|}\,.
\eeq
\item from \eqref{KEY1}, for any $x,\,x_0\in X$ and any $l\in \bar{\mathcal{H}}_q$, we have that
\beq\label{diff}
|\sum_{k=0}^d \a_k\,(x-l)^{\a_k-1}\,\tilde{a}_k(x)\,-\,\sum_{k=0}^d \a_k\,(x_0-l)^{\a_k-1}\,\tilde{a}_k(x_0)|\leq 2^{-(\frac{1}{2}-2\d)\,m}\:\,.
\eeq

\item Assuming for the moment that $l\in \bar{\mathcal{H}}_q$ is fixed, for $1\leq k\leq d$ we apply Taylor's formula around the point $x$ and deduce
 \beq\label{Tay}
(x_0-l)^{\a_k-1}= (x-l)^{\a_k-1}\,+(\a_k-1)\,(x-l)^{\a_k-2}\,+\,O(|x-x_0|^2)\,,
\eeq
 where in the above--exploiting \eqref{CarlPolynestKeyS3}--we make use of the fact that $|x-l|\approx 1$ for any $x\in X$ and $l\in \bar{\mathcal{H}}_q$.
\item Putting together \eqref{coefcontrol}--\eqref{Tay} we have that for any $l\in \bar{\mathcal{H}}_q$ and $x,\,x_0\in X$
 \beq\label{Tay}
 \eeq
$$\left|\left(\sum_{k=1}^d \left(\a_k\,(x-l)^{\a_k-1}\,(\tilde{a}_k(x)-\tilde{a}_k(x_0))\,-\,(\a_k-1)\,(x_0-x)\,(x-l)^{\a_k-2}\,\tilde{a}_k(x_0)\right)\right)\right.$$
$$+\,\Big(\tilde{a}_0(x)-\tilde{a}_0(x_0)\Big)\Big|\leq 2^{-(\frac{1}{2}-2\d)\,m}\,+\,O\left(\frac{|x-x_0|^2}{\|\bar{\mathcal{H}}_q\|^d}\right)\,.$$

\item At this point we use of the key hypothesis \eqref{tsrict} in order to deduce that
\begin{itemize}
\item the coefficient of the ``lowest nonzero degree" term $(x-l)^{\a_1-2}$ is precisely $-\,(\a_1-1)\,(x_0-x)\,\tilde{a}_1(x_0)$;

\item the coefficient of the zero degree term $(x-l)^{0}$ is precisely $\tilde{a}_0(x)-\tilde{a}_0(x_0)$.
\end{itemize}

As a consequence, applying Lemma 51 in \cite{Lie19} we get
\beq\label{a1}
|\tilde{a}_1(x)|,\,|\tilde{a}_1(x_0)|\lesssim_{d,\vec{\a}} \frac{2^{-(\frac{1}{2}-2\d)\,m}}{|x-x_0|\,\|\bar{\mathcal{H}}_q\|^{2d}}\,+\,
  \,O\left(\frac{|x-x_0|}{\|\bar{\mathcal{H}}_q\|^{3d}}\right)\,,
\eeq
and
\beq\label{0degree}
|\tilde{a}_0(x)-\tilde{a}_0(x_0)|\lesssim_{d,\vec{\a}} \frac{2^{-(\frac{1}{2}-2\d)\,m}}{\|\bar{\mathcal{H}}_q\|^{2d}}\,+\,
  \,O\left(\frac{|x-x_0|^2}{\|\bar{\mathcal{H}}_q\|^{3d}}\right)\,.
\eeq

\item Set $\I$ the collection of dyadic intervals having the length $2^{-\frac{m}{4}}$ that partition the interval $[-\frac{1}{2},\,\frac{1}{2}]$. Let now $\I_{1}$ be the collection of intervals $J\in \I_{1}$ for which
\beq\label{a2}
|X\cap J|\geq 2^{-5}\,2^{-2^{2d+3}\,\d\,m}\,|J|\,,
\eeq
and set
\beq\label{a3}
X_{1}:=\bigcup_{J\in \I_{1}} X\cap J\;.
\eeq
From \eqref{KEY2}, \eqref{a2} and \eqref{a3} one deduces immediately that
\beq\label{a4}
|X_{1}|\geq \frac{|X|}{2}\geq 2^{-2}\,2^{-2^{2d+3}\,\d\,m}\;.
\eeq
\item Applying now \eqref{a1} within each set $X\cap J$ with $J\in \I_{1}$ we notice for any $x\in X\cap J$ we can choose $x_0\in X\cap J$ such that $|x-x_0|\geq 2^{-5}\,2^{-2^{2d+3}\,\d\,m}\,|J|$ and hence
\beq\label{a5}
|a_1(x)|\lesssim_{d,\vec{\a}}   \frac{2^{-(\frac{1}{4}-2\d)\,m}\,2^{2^{2d+3}\,\d\,m}}
{\|\bar{\H}_{q}\|^{3d}}\:\;\:\:\textrm{for any}\:x\in X_{1}\:.
\eeq
\item Returning to \eqref{diff} and making use of \eqref{a5} we deduce that for any $l\in\bar{\H}_{q}$ and $x,\,x_0 \in X_1$ one has
\beq\label{diff1}
\eeq
$$\left|\left(\tilde{a}_0(x)-\tilde{a}_0(x_0)\right)\,+\,\left(\sum_{k=2}^d \a_k\,(x-l)^{\a_k-1}\,\tilde{a}_k(x)\,-\,\sum_{k=0}^d \a_k\,(x_0-l)^{\a_k-1}\,\tilde{a}_k(x_0)\right)\right|$$
$$\lesssim_{d,\vec{\a}} \frac{2^{-(\frac{1}{4}-2\d)\,m}\,2^{2^{2d+3}\,\d\,m}}
{\|\bar{\H}_{q}\|^{3d}}\,.$$

\item One can now repeat the above algorithm to deduce an upper bound for the size of $a_2(x)$ where here $x\in X_2$ with $X_2\subset X_1$ such that $|X_2|\geq \frac{1}{2}\,|X_1|$. Iterating this argument $d$ times we conclude that for any $1\leq k\leq d$ and $x\in X_d$ with $X_d\subset X$ and $|X_d|\geq \frac{1}{2^{d+1}}\,|X|$ the following holds:
\beq\label{a6}
|a_k(x)|\lesssim_{d,\vec{\a}}   \frac{2^{-(\frac{1}{2^{k+1}}-2\d)\,m}\,2^{2^{2d+3}\,k\,\d\,m}}{\|\bar{\H}_{q}\|^{3d k}}\:.
\eeq
\item Since from \eqref{distribcontr} we deduce that $\|\bar{\H}_{q}\|\geq \|\mathcal{H}_{q,\d}^{L,S}\|\geq 2^{-(2^{3d+5}\,\d\,+\,\nu)\,m}$ by choosing $\d,\,\nu>0$ small enough such that
\beq\label{a8}
\frac{1}{2^{d+1}}-2\d-2^{2d+3}\,d\,\d-3d^2(2^{3d+5}\,\d\,+\,\nu)>0\,,
\eeq
we deduce that for any $x\in X_d$ one has
\beq\label{a9}
\|q^{NL}(x)\|<<1\,,
\eeq
which contradicts the assumption \eqref{CarlPolynestKeyS4} thus concluding the proof of \eqref{keyrelHere} and hence of our Theorem \ref{cjmbound}.
\end{itemize}
\end{proof}

\begin{remark}
It is not hard to see now that by combining the results in \cite{Car66}, \cite{hu67} with the proofs of the Main Theorem in \cite{Lie19} and Theorem \ref{cjmbound} above, one immediately gets that $C_{\g}$ is $L^p$ to $L^p$ bounded for any $1<p<\infty$.
\end{remark}

\begin{observation}[\textsf{Break of the LGC-methodology in the presence of quadratic resonance}]\label{quadratict} Notice that if requirement \eqref{tsrict} is modified such that one allows the existence of some $k_0\in\{1,\ldots, d\}$ such that $\a_{k_0}=2$ then, instead of \eqref{a1} and \eqref{0degree}, relation \eqref{Tay} gives us just
\beq\label{0degreeMod}
\Big|\tilde{a}_0(x)\,-\,\tilde{a}_0(x_0)\,+\,(x-x_0)\,\tilde{a}_{k_0}(x_0)\Big|\lesssim_{d,\vec{\a}} \frac{2^{-(\frac{1}{2}-2\d)\,m}}{\|\bar{\mathcal{H}}_q\|^{2d}}\,+\,
  \,O\left(\frac{|x-x_0|^2}{\|\bar{\mathcal{H}}_q\|^{3d}}\right)\,.
\eeq
Relation \eqref{0degreeMod} reveals an interdependence between the behaviors of $\tilde{a}_0$ and $\tilde{a}_{k_0}$ that prevents us from refuting \eqref{contrad} via \eqref{a9}. Reinforcing the geometric argument described in Section \ref{Unif}, this explains why in the presence of quadratic resonances the analogue of Theorem \ref{cjmbound} is no longer true.
\end{observation}

\section{Epilogue: A curved model for the triangular Hilbert transform revisited}\label{TrHC}

In this last section of our paper, inspired by the philosophy behind the LGC approach---see below---we provide a short, self-contained proof of the following smoothing inequality that is the crux of the result obtained in \cite{CDR20}:
\smallskip

\begin{proposition}\label{MSTrHC}
	Fix $m\in\N$ and define\footnote{Throughout this section, if $f$ is a (locally) integrable function of two variables and $\phi$ is a single variable bump function then $(f*_{1}\check{\phi})(x,y):=\int_\R f(x-s,y)\,\check{\phi}(s)\,ds$ with $\check{\phi}$ the inverse Fourier transform of $\phi$.}
	\begin{equation}\label{defTm}
		T_m(f,g)(x,y):=\int_{\R} (f*_{1}\check{\phi}_m)(x+t,y)\,(g*_2\check{\phi}_m)(x,y+t^2)\,\rho(t)\,dt,
	\end{equation}
	where  $\rho,\,\phi\in C_0^{\infty}(\R)$ with $\textrm{supp}\,\rho,\,\textrm{supp}\,\phi\subseteq\{\xi:\frac{1}{2}<|\xi|<2\}$ and $\phi_m(\xi):=\phi(\frac{\xi}{2^m})$.

Then, there exists  $\ep_0>0$---one may take $\ep_0=\frac{1}{75}$---such that
	\begin{equation}\label{decmtr}
		\|T_m(f,g)\|_{L^1(\R^2)}\lesssim 2^{-\ep_0 m}\,\|f\|_{L^2(\R^2)}\,\|g\|_{L^2(\R^2)}\;.
	\end{equation}
\end{proposition}
\smallskip

\begin{proof}

Our proof relies on two key ingredients that, at the conceptual level, are shared with the LGC-method:
\begin{itemize}
\item a time-frequency/wave-packet discretization optimizing the interaction between the phase of the multiplier and the input function;

\item a sparse-uniform dichotomy.
\end{itemize}

$\newline$
\noindent\textbf{Step 1.} \textbf{Time-frequency discretization}
$\newline$

\noindent\underline{\textsf{Stage 1.1. Preliminary spatial discretization.}} Based on the observation that the $t$ integrand spans an interval of length $\approx1$, standard almost orthogonality arguments together with Cauchy-Schwarz show that \eqref{decmtr} is an immediate consequence of
\begin{equation}\label{decmtrl2}
	I_m:=\int_{[1,2]^2}|T_m(f,g)(x,y)|\, dx\,dy\lesssim 2^{-\,\ep_0\,m}\,\|f\|_{L^2(\R^2)}\,\|g\|_{L^2(\R^2)}\;,
\end{equation}
and thus we can reduce matters to the situation when essentially\footnote{Here, for expository reasons, we are slightly abusing the reference to a precise support location; in reality, one deals with a ``moral" support at which one adds suitable error terms involving tail-behavior. However, the treatment of the latter terms is straightforward relying on standard arguments.}
\begin{equation}\label{supp}
\textrm{supp}f(\cdot,\cdot),\,\textrm{supp}\,g(\cdot,\cdot)\subset[1,2]^2\,.
\end{equation}

\noindent\underline{\textsf{Stage 1.2. Mixed time-frequency behavior and spatial constancy propagation.}} \emph{The content of this subsection is merely a heuristic, but it provides a key insight for the discretization implemented at the next stage.} Our intuition is guided by the following observations:
\begin{itemize}
\item (I) the presence of the $t-$integral in \eqref{defTm} mixing the second variable information carried by $f(\cdot, \cdot)$ with the first variable information carried by $g(\cdot,\cdot)$, signals the relevance of considering the \emph{mixed space-frequency} behavior of $\widehat{f}^1(\xi\,,y)$ and $\widehat{g}^2(x\,,\eta)$.
\item (II) further inspecting \eqref{defTm}, we notice that $(f*_{1}\check{\phi}_m)(x,y)$ encapsulates a spatial averaging in the first variable at the level of the $2^{-m}$-scale  with a similar behavior for $(g*_2\check{\phi}_m)(x,y)$ in the second variable;
\end{itemize}
As a result, derived from (I), it becomes natural to consider
\begin{equation}\label{not}
	\F(\xi,y):=\widehat{f}^1(\xi\,,y)\,\phi_m(\xi)\,\chi_{[1,2]}(y)\qquad\textrm{and}\qquad
	\G(x,\eta):=\widehat{g}^2(x\,,\eta)\,\phi_m(\eta)\,\chi_{[1,2]}(x)\;.
\end{equation}
Indeed, notice that via \eqref{not}, one can rephrase\footnote{At the level of main term, \textit{i.e.}, up to fast decaying error term.} \eqref{defTm} in the form
\begin{equation}\label{reduct}
	T_m(f,g)(x,y)\equiv T(\F,\G)(x,y)\approx \frac{1}{2^{\frac{m}{2}}}\,\int_{\R^2} \F(\xi,y)\,\G(x,\eta)\,e^{i\,\xi\,x}\,e^{i\,\eta\,y}\,e^{-i\,\frac{\xi^2}{4 \eta}}\,d\xi\,d\eta\:,
\end{equation}
where here we used some standard (non)stationary phase arguments for computing the dominant term of the multiplier: for
$\xi,\eta\in \textrm{supp}\,\phi_m$, the stationary phase method gives (up to error terms)
\begin{equation}\label{statphase}
	\int_{\R} e^{i \xi t}\,e^{i \eta t^2}\rho(t)\,dt\approx \frac{1}{2^{\frac{m}{2}}}\,e^{-i\,\frac{\xi^2}{4 \eta}}\,\rho\Big(\frac{\xi}{2 \eta}\Big)\,.
\end{equation}

Combining now the effects of both (I) and (II) and joining the information obtained from \eqref{defTm} and \eqref{reduct}, we are naturally brought to the following
\bigskip

\noindent \textsf{Heuristic}:  \emph{the $x-y$ integration of $T_m(f,g)(x,y)$ should essentially be sensitive only to the spatial $2^{-m}$-average behavior of $\F(\xi,\cdot)$ and $\G(\cdot,\eta)$}.
\bigskip

The above heuristic, invites us to consider the following preliminary model:
\bigskip

\noindent\textbf{Model Problem/Assumption.} [\textsf{ $2^{-m}$ spatial blurring effect: local constancy}] \emph{Assume in what follows that
the following holds:
\begin{equation}\label{decmtrl22prel}
	\F(\xi,\cdot)\: {\text{ and }}\:\G(\cdot,\eta)\:\textrm{ are \underline{essentially} piecewise constant on }\: \big[\sdfrac{k}{2^m},\,\sdfrac{k+1}{2^m}\big]\quad\forall\:k\sim 2^m\,.
\end{equation}
\indent In effect, we can relax the above, and only require the weaker condition:
$\newline$
for any $\xi,\,\eta\in \textrm{supp}\,\phi_m$, $k\sim 2^m$ and $x,\,y\in \Big[\sdfrac{k}{2^m},\,\sdfrac{k+1}{2^m}\Big]$ the following holds\footnote{This is essentially equivalent to the existence of a spatially local constant majorant.}:}
\begin{equation}\label{decmtrl22}
	|\F(\xi,x)|\lesssim 2^m\,\int_{\frac{k}{2^m}}^{\frac{k+1}{2^m}} |\F(\xi,\tilde{x})|\,d\tilde{x}\: {\text{ and }}\:|\G(y,\eta)|\lesssim 2^m\,\int_{\frac{k}{2^m}}^{\frac{k+1}{2^m}} |\G(\tilde{y},\eta)|\,d\tilde{y}\,.
\end{equation}
\medskip

In what follows, our plan is to provide the proof of our Proposition \ref{MSTrHC} under the assumption\footnote{It is worth saying that we will use this assumption only in Step 3 of our proof.} \eqref{decmtrl22} and save only for the very end--see Step 4--the treatment of the general case, which, as it turns out, can be essentially reduced to the initial model problem.

$\smallskip$

\noindent\underline{\textsf{Stage 1.3. Zero-order wave-packet discretization.}} Relying on the intuition acquired at the second stage, in this third stage we will achieve the time-frequency discretization of our input function:

Firstly, in view of (I), (II) and our assumption \eqref{decmtrl22}, it is natural to split the $x$-spatial location in (dyadic) intervals of size $2^{-m}$. Secondly, guided by the expression $-\,\frac{\xi^2}{4 \eta}$ representing the phase of our multiplier, and by the already settled $x$-spatial grid, we are dividing the $\xi-\eta$ frequency plane in cubes of unit size. \emph{It is worth noticing here that, unlike the LGC-method that focuses on the phase linearization corresponding to first order Taylor approximations of the phase, our present approach highlights the zero-order Taylor approximation of the phase.}

With these being said, for $j,\,k\sim 2^m$ we let now
\begin{equation}\label{Ff}
	\mathcal{F}_{j,k}(\xi,y):=\mathcal{F}(\xi,y)\,\chi_{[\frac{k}{2^m},\,\frac{k+1}{2^m}]}(y)\,\chi_{[j,\,j+1]}(\xi)\,,
\end{equation}
and similarly
\begin{equation}\label{Gg}
	\mathcal{G}_{j,k}(x,\eta):=\mathcal{G}(x,\eta)\,\chi_{[\frac{k}{2^m},\,\frac{k+1}{2^m}]}(x)\,\chi_{[j,\,j+1]}(\eta)\,,
\end{equation}
and conclude our time-frequency discretization process with the final product
\begin{equation}\label{FGdec}
\mathcal{F}(\xi,y)=\sum_{j,k\sim 2^m} \mathcal{F}_{j,k}(\xi,y)\qquad\textrm{and}\qquad	\mathcal{G}(x,\eta)=\sum_{j,k\sim 2^m} \mathcal{G}_{j,k}(x,\eta)\,.
\end{equation}

$\newline$
\noindent\textbf{Step 2.} \textbf{A sparse-uniform dichotomy}
$\newline$

Inspired by the related reasoning within the time-frequency correlation analysis stage of the LGC-method--see Step III.1 in Section \ref{LGCplus}, we design a sparse-uniform dichotomy that, of course, must be properly tailored to our present context. Given the discretization performed in Step 1, we introduce the following quantities:
\begin{equation}\label{fg}
	f_{j,k}:=\int_{\R^2}|\mathcal{F}_{j,k}(\xi,y)|^2\,d\xi\,dy\quad\textrm{ and }\quad g_{j,k}:=\int_{\R^2}|\mathcal{G}_{j,k}(x,\eta)|^2\,dx\,d\eta\,.
\end{equation}
Normalizing now to one the $L^2(\R^2)$ norms of $\F$ and $\G$ we deduce that
\begin{equation}\label{fg1}
	\sum_{j,k\sim 2^m}f_{j,k}\lesssim 1\quad\textrm{and}\quad \sum_{j,k\sim 2^m}g_{j,k}\lesssim 1\,.
\end{equation}

Fix now $\d,\v\in (0,1)$ small to be chosen later with $0<2\v<\d$. In what follows, for each of the functions $\F$ and $\G$, we partition the set of indices
\begin{equation}\label{A}
	\A:=\{(j,k)\,|\,j,k\sim 2^m\}\,,
\end{equation}
in three components corresponding to the amount of information carried by the localized masses $f_{j,k}$ and $g_{j,k}$, respectively. More precisely, we partition
\begin{equation}\label{decA}
	\A:=\A_{\F}^{1}\cup \A_{\F}^{2}\cup \A_{\F}^{3} \,,
\end{equation}
with
\begin{itemize}
	\item the \emph{set of heavy $y$-fibers}
	\begin{equation}\label{A1}
		\A_{\F}^{1}:=\Big\{(j,k)\in\A\,|\,\sum_{l\sim 2^m} f_{l,k}\gtrsim 2^{-m+m\v},\:\:j\sim 2^m\Big\}\,;
	\end{equation}
	\item  the \emph{set of heavy masses}
	\begin{equation}\label{A2}
		\A_{\F}^{2}:=\Big\{(j,k)\in \A\setminus \A_{\F}^{1}\,|\,2^{-m+m\v-m\d}\lesssim f_{j,k}\lesssim 2^{-m+m\v}\Big\}\,.
	\end{equation}
	\item the \emph{set of light masses}
	\begin{equation}\label{A3}
		\A_{\F}^{3}:=\Big\{(j,k)\in \A\setminus \A_{\F}^{1}\,|\,f_{j,k}\lesssim 2^{-m+m\v-m\d}\Big\}\,.
	\end{equation}
\end{itemize}

With the obvious correspondences we also have
\begin{equation}\label{decAG}
	\A:=\A_{\G}^{1}\cup \A_{\G}^{2}\cup \A_{\G}^{3} \,.
\end{equation}
Finally, for $i\in\{1,2,3\}$, we set
\begin{equation}\label{FGi}
	\F^{i}:=\sum_{(j,k)\in \A_{\F}^{i}} \F_{j,k}\qquad\textrm{and}\qquad \G^{i}:=\sum_{(j,k)\in \A_{\G}^{i}} \G_{j,k}\,,
\end{equation}
and notice that
\begin{equation}\label{FGsum}
	\F=\F^{1}+\F^{2}+\F^{3}\qquad\textrm{and}\qquad\G=\G^{1}+\G^{2}+\G^{3}\,.
\end{equation}

Split now $I_m$ accordingly:
\begin{equation}\label{Im}
	I_m\leq  I_m^{1}\,+\,I_m^2\,+\,I_m^3\,,
\end{equation}
with
\begin{itemize}
	\item the \emph{heavy (spatial) fiber component} defined as
	\begin{align}
		I_m^{1}&:= \int|T_m(\F^1,\G)(x,y)|\,dx\,dy\,+\,\int|T_m(\F^2,\G^1)(x,y)|\, dx\,dy\label{I1}\\
		&\hspace{2cm}\,+\,\int|T_m(\F^3,\G^1)(x,y)|\, dx\,dy\,;\nonumber
	\end{align}
	\item the \emph{heavy mass component} given by
	\begin{align}
		I_m^{2}&:= \int|T_m(\F^2,\G^2)(x,y)|\, dx\,dy\,+\,\int|T_m(\F^2,\G^3)(x,y)|\,dx\,dy\label{I2}\\
		&\hspace{2cm}\,\,+\,\int|T_m(\F^3,\G^2)(x,y)| dx\,dy\,;\nonumber
	\end{align}
	\item the \emph{light (uniform) mass component} defined as
	\begin{equation}\label{I3}
		I_m^{3}:=\int|T_m(\F^3,\G^3)(x,y)|\, dx\,dy\,.
	\end{equation}
\end{itemize}

$\newline$
\noindent\textbf{Step 3.} \textbf{Control of the resulting three components}
$\newline$

In this section we will obtain the desired estimates on each of the three terms $I_m^1$, $I_m^2$ and $I_m^3$ by exploiting the properties of the decomposition \eqref{A1}--\eqref{A3} in conjunction with the oscillatory features of the expression in \eqref{reduct}.
$\medskip$

\noindent\underline{\textsf{Stage 3.1. Treatment of the heavy fiber component.}} The key fact employed in the treatment of $I_m^1$ will be the smallness of the physical support of $\F^1$ and $\G^1$.

Indeed, we first notice that due to symmetry it is enough to only treat the first term in \eqref{I1}, \textit{i.e.}, wlog we may assume, abusing the notation, that
\begin{equation}\label{I10}
	I_m^{1}\approx\int|T_m(\F^1,\G)(x,y)|\,dx\,dy\,.
\end{equation}
Now, using definition \eqref{A1} and Chebyshev's inequality, we notice that $\A_{\F}^{1}$ has a cross product structure with $\A_{\F}^{1}=\mathcal{J}\times \mathcal{K}$ such that $\mathcal{J}=[2^m, 2^{m+1}]\cap \N$, and, more importantly, $\# \mathcal{K}\lesssim 2^{m-m\v}$. Deduce from here that
\begin{equation}\label{supp}
	\textrm{supp}\,\F^1(\xi,y)\subseteq [2^m, 2^{m+1}]\times \mathcal{Y}\quad{with}\quad \mathcal{Y}\subseteq[0,1]\:\:\&\:\:|\mathcal{Y}|\lesssim 2^{-\v m}\,.
\end{equation}

Consequently, from Cauchy--Schwarz and \eqref{supp}, we have that
\begin{align*}
	(I_m^{1})^2&=\bigg|\int_{[1,2]\times \mathcal{Y}}|T(\F^1,\G)(x,y)|\, dx\,dy\bigg|^2\lesssim |\mathcal{Y}|\,\int_{[1,2]^2}|T(\F^1,\G)(x,y)|^2\,dx\,dy\\
	&\lesssim 2^{- \v m}\, \int_{[1,2]^2}\left|\int_{\R} {\F^1}^{\widecheck{}_1}(x+t,y)\,\G^{\widecheck{}_2}(x,y+t^2)\,\rho(t)\,dt\right|^2 dx\,dy\\
	&\lesssim 2^{-\v m}\, \int_{[1,2]^2}\bigg(\int_{[1,2]} |{\F^1}^{\widecheck{}_1}(t,y)|^2\,dt\bigg)
	\left(\int_{[1,2]}|\G^{\widecheck{}_2}(x,t)|^2\,dt\right)dx\,dy\,,
\end{align*}
from which we conclude that
\begin{equation}\label{cmestim}
	I_m^{1} \lesssim 2^{-\frac{\v}{2} m}\,\|\F\|_{L^2(\R^2)}\, \|\G\|_{L^2(\R^2)}\:.
\end{equation}
$\smallskip$

\noindent\underline{\textsf{Stage 3.2. Treatment of the heavy mass component.}} Proceeding as before, due to symmetry reasons, we can wlog pretend that
\begin{equation}\label{I20}
	I_m^{2}= \int|T(\F^2,\G)(x,y)|\, dx\,dy\,,
\end{equation}
where here $\G$ is a generic function as defined by \eqref{not} that may be thought of being any of $\G^1$, $\G^2$ or $\G^3$ or for that matter $\G^1+\G^2+\G^3=\G$.

In this setting, the two main ingredients in estimating $I_m^{2}$ are:
\begin{itemize}
\item the uniform control on the $k$-section set $\A_{\F}^{2}(k):=\left\{j\,|\,(j,k)\in \A_{\F}^{2}\right\}$:
\begin{equation}\label{keystage2}
\# \A_{\F}^{2}(k)\lesssim 2^{\d m}\,.
\end{equation}
\item a stationary phase analysis capturing the $(x,\eta)$ space-phase interaction (part of the mixed behavior of $\G$); this phase analysis though, can only be performed after a preliminary decoupling of the parameters $\xi$ and $\eta$ derived from the properties of the time-frequency discretization \eqref{FGdec}.
\end{itemize}

Indeed, from \eqref{reduct} and \eqref{Ff}, we first deduce that
\begin{equation}\label{reduct20}
	I_m^2\lesssim \frac{1}{2^{\frac{m}{2}}}\,\sum_{(j,k)\in \A_{\F}^{2}}\int_{[1,2]^2}\left|\int_{\R^2} \F_{j,k}(\xi,y)\,\G(x,\eta)\,e^{i\,\xi\,x}\,e^{i\,\eta\,y}\,e^{-i\,\frac{\xi^2}{4 \eta}}\,d\xi\,d\eta\right| dx\,dy\:.
\end{equation}

Using now the time-frequency localization properties of $\F_{j,k}$, we can decouple the $(\xi, \eta)$ frequency parameters, as follows\footnote{Here we ignore the error terms arising from the zero-order approximation.}:
\begin{equation}\label{reduct21}
	I_m^2\lesssim \frac{1}{2^{\frac{m}{2}}}\,\sum_{(j,k)\in \A_{\F}^{2}}\left(\int_{\R^2} \left|\F_{j,k}(\xi,y)\right| \,d\xi\,dy\right)\,\left(\int_{\R}\left|\int_{\R}\G(x,\eta)\,e^{i\,\eta\,\frac{k}{2^m}}\,e^{-i\,\frac{j^2}{4 \eta}}\,d\eta\right| dx\right)\:.
\end{equation}

Next, applying an $l^1-l^{\infty}$ H\"older argument in the $j$-parameter followed by a Cauchy-Schwarz argument in the $k$-parameter, we deduce

\begin{equation}\label{reduct211}
	I_m^2\lesssim \frac{1}{2^{\frac{m}{2}}}\,A_2(\F)^{\frac{1}{2}}\,B_2(\G)^{\frac{1}{2}}\,,
\end{equation}
with
\begin{equation}\label{reduct2a}
A_2(\F):=\sum_k \left(\sum_{j\in \A_{\F}^{2}(k)}\int_{\R^2} \left|\F_{j,k}(\xi,y)\right| \,d\xi\,dy\right)^2\,
\end{equation}
and
\begin{equation}\label{reduct2b}
B_2(\G):=\sum_k \left(\int_{\R}\left|\int_{\R}\G(x,\eta)\,e^{i\,\eta\,\frac{k}{2^m}}\,e^{-i\,\frac{j^2(k)}{4 \eta}}\,d\eta\right| dx \right)^2\:,
\end{equation}
where here $k\,\mapsto\,j(k)$ is a measurable function taking values from $[2^m,2^{m+1}]\cap \N$ into $[2^m,2^{m+1}]\cap \N$.

Now, for the first term, we use \eqref{A2}, \eqref{keystage2}, Cauchy-Schwarz and the time-frequency localization of $\F_{j,k}$ in order to deduce
\begin{equation}\label{reduct2a1}
A_2(\F)\lesssim \sum_k \left(\# \A_{\F}^{2}(k)\right) \sum_{j\in \A_{\F}^{2}(k)}2^{-m}\,\int_{\R^2} \left|\F_{j,k}(\xi,y)\right|^2 \,d\xi\,dy\lesssim 2^{-m}\,2^{m \d}\,.
\end{equation}

For the second term, we develop the square and apply successively Cauchy--Schwarz:
\begin{align}
	& B_2(\G)\lesssim \sum_{k\sim 2^m}\int_{\R}\int_{\R^2}\G(x,\eta)\,\overline{\G(x,\eta+s)}\,e^{-i\,s\,\frac{k}{2^m}}\,e^{-i\,j^2(k)\,(\frac{1}{4 \eta}-\frac{1}{4 (\eta+s)})}\,d\eta\,ds\,dx\nonumber \\[1ex]
	&\lesssim\int_{\R^2} |\G(x,\eta)| \bigg(\int_{\R}|\G(x,s)|^2\,ds\bigg)^{\frac{1}{2}}\times\nonumber\\
	&\hspace{3cm}\times \bigg(\int_{\R}\bigg|\sum_{k\sim 2^m}e^{-i\,s\,\frac{k}{2^m}}e^{-i\,j^2(k)\big(\frac{1}{4 \eta}-\frac{1}{4(\eta+s)}\big)}\bigg|^2 \phi\big(\sdfrac{s}{2^m}\big)ds\bigg)^{\frac{1}{2}}d\eta\,dx\nonumber\\[1ex]
	&\lesssim\bigg(\int_{\R}|\G(x,s)|^2\,ds\,dx\bigg)\times\nonumber\\
	&\hspace{2cm}\times
	\bigg(\int_{\R^2}\bigg|\sum_{k\sim 2^m} e^{-i\,s\,\frac{k}{2^m}} e^{-i\,j^2(k) \big(\frac{1}{4 \eta}-\frac{1}{4 (\eta+s)}\big)}\bigg|^2\phi\Big(\frac{s}{2^m}\Big)\phi\Big(\frac{\eta}{2^m}\Big)\,ds\,d\eta\bigg)^{\frac{1}{2}}\,.\nonumber
\end{align}
Thus, we deduce that
\begin{equation}\label{reduct22}
B_2(\G)\lesssim 2^{m}\,V_m^{\frac{1}{2}}\,,
\end{equation}
where here
\begin{equation}\label{V}
	V_m:=\int_{\R^2}\left|\sum_{k\sim 2^m}\,e^{-i\,s\,k}\,e^{-i\,\frac{j^2(k)}{2^m}\,\left(\frac{1}{4 \eta}-\frac{1}{4 (\eta+s)}\right)}\right|^2\phi(s)\,\phi(\eta)\,ds\,d\eta\:.
\end{equation}
Then
$$V_m=\sum_{k,k_1\sim 2^m}\,\int_{\R}\left(\int_{\R}\,e^{-i\,s\,(k-k_1)}\,e^{-i\,\frac{j^2(k)-j^2(k_1)}{2^m}\,\left(\sdfrac{1}{4 \eta}-\sdfrac{1}{4 (\eta+s)}\right)}\,\phi(s)\,ds\right)\phi(\eta)\,d\eta\:.$$
Taking now the phase
\begin{equation}\label{Ph}
	\Phi_{k,k_1,\eta}(s)=\Phi(s):=s\,(k-k_1)+ \frac{j^2(k)-j^2(k_1)}{2^m}\Big(\frac{1}{4 \eta}-\frac{1}{4 (\eta+s)}\Big)\:,
\end{equation}
we notice that
\begin{equation}\label{Ph1}
	\Phi'(s)=(k-k_1)+ \frac{j^2(k)-j^2(k_1)}{2^m}\,\frac{1}{4 (\eta+s)^2}\:,
\end{equation}
and
\begin{equation}\label{Ph2}
	\Phi''(s)=-\frac{j^2(k)-j^2(k_1)}{2^m}\,\frac{1}{2 (\eta+s)^3}\:.
\end{equation}
Focussing on the worst scenario (weakest decay) represented by the presence of stationary points in \eqref{Ph1}, we must have
\begin{equation}\label{Ph3}
	|k-k_1|\approx |j(k)-j(k_1)|\,,
\end{equation}
which implies that
\begin{equation}\label{Ph4}
	|\Phi''(s)|\approx |k-k_1|\:.
\end{equation}
As a consequence
\begin{equation}\label{VM}
	V_m\lesssim \sum_{k,k_1\sim 2^m}\,\frac{1}{1+\sqrt{|k-k_1|}}\lesssim 2^{\frac{3m}{2}}\:.
\end{equation}
Putting together \eqref{reduct211}, \eqref{reduct2a1}, \eqref{reduct22} and \eqref{VM} we conclude that
\begin{equation}\label{Concl2}
	I_m^{2}\lesssim 2^{-\frac{m}{2}\,(\frac{1}{4}-\d)}\:.
\end{equation}

$\smallskip$

\noindent\underline{\textsf{Stage 3.3. Treatment of the light (uniform) mass component.}}  In this last setting, the two main ingredients are:
\begin{itemize}
\item the smallness of $f_{j,k}$ for $(j,k)\in \A_{\F}^{3}$;

\item a stationary phase analysis capturing the \emph{global} $(\xi,\eta)$ phase interaction that is the manifestation of the twisted behavior of the integrand in \eqref{reduct}. This is the only stage in our proof in which we make use of assumption \eqref{decmtrl22}.
\end{itemize}

With these being said, we start by defining
\begin{equation}\label{F}
	F(\xi,\xi_1,\mu):=\int_{[1,2]} \F^3(\xi,y)\,\overline{\F^3(\xi_1,y)}\,e^{i \mu y}\,dy \;,
\end{equation}
and
\begin{equation}\label{G}
	G(\eta,\eta_1,\tau):=\int_{[1,2]} \G^3(x,\eta)\,\overline{\G^3(x,\eta_1)}\,e^{i \tau x}\,dx \;.
\end{equation}

Then, developing the square below and using \eqref{reduct} and Fubini we have
\begin{equation}\label{explsq}
	(I_m^3)^2\lesssim \int_{[1,2]^2}|T(\F^3,\G^3)(x,y)|^2\, dx\,dy=
\end{equation}
$$\frac{1}{2^m}\,\int_{\R^4} F(\xi,\xi_1,\eta-\eta_1)\,G(\eta,\eta_1,\xi-\xi_1)\,
e^{-i\,\frac{\xi^2}{4 \eta}}\,e^{i\,\frac{\xi_1^2}{4 \eta_1}}\,d\xi\,d\xi_1\,d\eta\,d\eta_1\,.$$

Making now the change of variable $\xi-\xi_1=\tau$ and $\eta-\eta_1=\mu$ and letting
\begin{equation}\label{FG}
	F_{\tau,\mu}(\xi):=F(\xi,\xi-\tau,\mu)\qquad\textrm{and}\qquad G_{\tau,\mu}(\eta):=G(\eta,\eta-\mu,\tau)\,,
\end{equation}
we can rewrite \eqref{explsq} as
\begin{equation}\label{explsq1}
	(I_m^3)^2\lesssim
	\frac{1}{2^m}\,\int_{\R^4} F_{\tau,\mu}(\xi)\,G_{\tau,\mu}(\eta)\,
	e^{-i\,\frac{\xi^2}{4 \eta}}\,e^{i\,\frac{(\xi+\tau)^2}{4(\eta-\mu)}}\,d\xi\,d\eta\,d\tau\,d\mu\,,
\end{equation}
which, via a Cauchy-Schwarz argument, becomes
\begin{equation}\label{IM3}
	(I_m^3)^2\lesssim A_3(\F)^{\frac{1}{2}}\,B_3(\G)^{\frac{1}{2}}\,,
\end{equation}
with
\begin{equation}\label{reduct3a}
A_3(\F):=\int_{\R^3} \left| F_{\tau,\mu}(\xi)\right|^2 d\xi\,d\tau\,d\mu\,
\end{equation}
and
\begin{equation}\label{reduct3b}
B_3(\G):=\frac{1}{2^{2m}}\,\int_{\R^3}\phi_m(\xi)\,\left|\int_{\R} G_{\tau,\mu}(\eta)\,
	e^{-i\,\frac{\xi^2}{4 \eta}}\,e^{i\,\frac{(\xi+\tau)^2}{4(\eta-\mu)}}\,d\eta\right|^2\,d\xi\,d\tau\,d\mu\,.
\end{equation}

For the first term, we claim that
\begin{equation}\label{l2F}
A_3(\F)\lesssim 2^{\v m}\,\left(\int_{[1,2]\times\R} |\F^3(\xi,y)|^2\,d\xi\,dy\right)^{2}\lesssim 2^{\v m}\,.
\end{equation}
Indeed, applying now Fubini and Parseval and using assumption \eqref{decmtrl22} together with \eqref{A3}, we have
\begin{align}\label{l2cont}
	\int_{\R^3} |F_{\tau,\mu}(\xi)|^2\,d\xi\,d\tau\,d\mu&= \int_{\R^3}\left|\int_{[1,2]} F^3(\xi,y)\,\overline{F^3(\xi-\tau,y)}\,e^{i \mu y}\,dy\right|^2\,d\mu\,d\xi\,d\tau\nonumber\\
	&= \int_{\R^2}\int_{[1,2]} \left|\F^{3}(\xi,y)\,\F^3(\xi-\tau,y)\right|^2 \,dy\,d\xi\,d\tau\nonumber\\[1ex]
	&\approx 2^{m}\,\sum_{(j,k),(j_1,k)\in\A_{\F}^{3}} f_{j,k}\,f_{j_1,k}\lesssim 2^{\v m}\,.
\end{align}

For the second term, opening up the square and using Fubini, we have that
\begin{equation}\label{VM1}
	B_3(\G){=}\sdfrac{1}{2^{2m}}\!\!\int_{\R^4}\!\! G_{\tau,\mu}(\eta)\,\overline{G_{\tau,\mu}(\eta_1)}\!
	\left(\int_{\R}\!e^{i\big(-\frac{\xi^2}{4 \eta}+ \frac{\xi^2}{4 \eta_1}+\frac{(\xi+\tau)^2}{4(\eta-\mu)}- \frac{(\xi+\tau)^2}{4(\eta_1-\mu)}\big)}\! \phi_m(\xi)\, d\xi\!\right)\!d\eta\, d\eta_1\, d\tau\, d\mu\,.
\end{equation}
Define now the phase function
\begin{equation}\label{Pht}
	\Psi_{\eta,\eta_1,\tau,\mu}(\xi)=\Psi(\xi):=-\frac{\xi^2}{4 \eta}+\frac{\xi^2}{4 \eta_1}+\frac{(\xi+\tau)^2}{4(\eta-\mu)}- \frac{(\xi+\tau)^2}{4(\eta_1-\mu)}\,.
\end{equation}
Applying now in \eqref{VM1} the change of variable $\xi\,\mapsto 2^{m}\,\xi$, $\eta\,\mapsto\,2^{m}\,\eta$, $\eta_1\,\mapsto\,2^{m}\,\eta_1$, $\tau\,\mapsto\,2^{m}\,\tau$, $\mu\,\mapsto\,2^{m}\,\mu$ and setting
$G^{m}_{\tau,\mu}(\eta):=G_{2^{m}\tau,\, 2^{m}\mu}(2^{m}\eta)$ we rewrite \eqref{VM1} in the form
\begin{equation}\label{VM2}
	B_3(\G)=2^{3m}\,\int_{\R^2}\,\left(\int_{\R^2} G^{m}_{\tau,\mu}(\eta)\,\overline{G^{m}_{\tau,\mu}(\eta_1)}\,
	\left(\int_{\R}e^{i\,2^m\,\Psi(\xi)}\,\phi(\xi)\,d\xi\right)\,d\eta,d\eta_1\right)\,d\tau\,d\mu\,.
\end{equation}
Define now
\begin{equation}\label{Pht0}
	J_m:=\int_{\R}e^{i\,2^m\,\Psi(\xi)}\,\phi(\xi)\,d\xi\:.
\end{equation}
Since
\begin{equation}\label{Pht1}
	2\,\Psi''(\xi)= -\frac{1}{\eta}+\frac{1}{\eta_1}+\frac{1}{\eta-\mu}-\frac{1}{\eta_1-\mu}\,,
\end{equation}
a simple computation shows that
\begin{equation}\label{Pht2}
	|2\,\Psi''(\xi)|= \left| \frac{\mu(\eta-\eta_1)(\eta+\eta_1-\mu)}{\eta\,\eta_1\,(\eta-\mu)\,(\eta_1-\mu)} \right|\gtrsim |\mu\,(\eta-\eta_1)| \,.
\end{equation}
Applying now the stationary phase principle we further deduce that
\begin{equation}\label{SPH}
	|J_m|\lesssim \frac{1}{(1+2^m\,|\mu\,(\eta-\eta_1)|)^{\frac{1}{2}}}\:.
\end{equation}
Thus inserting \eqref{SPH} in \eqref{VM2} and then applying the change of variable $\eta_1=\eta+s$ we deduce that
\begin{equation}\label{VM3}
	B_3(\G)\lesssim 2^{3m}\,\int_{\R^2}\,\left(\int_{\R^2} |G^{m}_{\tau,\mu}(\eta)|\,|G^{m}_{\tau,\mu}(\eta+s)|\,
	\frac{1}{(1+2^m\,|\mu\,s|)^{\frac{1}{2}}}\,d\eta\,ds\right)\,d\tau\,d\mu\,.
\end{equation}
Now a Cauchy-Schwarz argument gives that
\begin{equation}\label{VM4}
	B_3(\G)\lesssim 2^{3m}\,\int_{\R^3} |G^{m}_{\tau,\mu}(\eta)|^2\,\left(\int_{[0,1]}\,
	\frac{1}{(1+2^m\,|\mu\,s|)^{\frac{1}{2}}}\,ds\right)\,d\eta\,d\tau\,d\mu\,,
\end{equation}
which, after reversing the change of variable in $\eta, \mu$ and $\tau$, reduces to
\begin{equation}\label{VM5}
	B_3(\G)\lesssim \int_{\R^3} \frac{|G_{\tau,\mu}(\eta)|^2}{(1+|\mu|)^{\frac{1}{2}}}\,d\eta\,d\tau\,d\mu\,.
\end{equation}
Using now \eqref{G}, \eqref{FG} and \eqref{VM5} and appealing again to assumption \eqref{decmtrl22} together with \eqref{A3} and a Cauchy-Schwarz argument we deduce
\begin{align}
	B_3(\G)&\lesssim \int_{\R^3} \frac{|\int_{[1,2]} \G^3(x,\eta)\,\overline{\G^3(x,\eta-\mu)}\,e^{i \tau x}\,dx |^2}{(1+|\mu|)^{\frac{1}{2}}}\,d\eta\,d\tau\,d\mu\nonumber\\[1ex]
	&\lesssim \int_{[1,2]} \int_{\R}  |\G^3(x,\eta)|^2\,\left(\int_{\R}\frac{|\G^3(x,\mu)|^2}{(1+|\eta-\mu|)^{\frac{1}{2}}}\,d\mu\right) \,d\eta\,d x  \nonumber\\[1ex]	
	&\lesssim 2^{m}\,\sum_{(j,k),(j_1,k)\in\A_{\G}^{3}} \frac{g_{j,k}\,g_{j_1,k}}{1+|j-j_1|^{\frac{1}{2}}}  \nonumber\\
	&\lesssim 2^{m}\,m^{\frac{1}{2}}\,\sum_{(j,k)\in\A_{\G}^{3}} g_{j,k}\,
	\Bigg(\sum_{\substack{j_1\sim 2^m\\(j_1,k)\in\A_{\G}^{3}}} |g_{j_1,k}|^2\Bigg)^{\frac{1}{2}}\lesssim
m^{\frac{1}{2}}\,2^{\v m}\,2^{-\frac{\d}{2}m}\,. \label{VM6}
\end{align}

Putting now together \eqref{IM3}, \eqref{l2F} and  \eqref{VM6} we deduce that
\begin{equation}\label{IM3F}
	I_m^3\lesssim m\,2^{-\frac{m}{2}(\frac{\d}{4}-\ep)}\,.
\end{equation}
Finally, from \eqref{cmestim}, \eqref{Concl2} and \eqref{IM3F}, and making the choice $\d=8\v=\frac{2}{9}$, we conclude that under assumption \eqref{decmtrl22} the following relation holds:
\begin{equation}\label{Conc}
	I_m\lesssim m\,2^{-\frac{m}{72}}\,.
\end{equation}

$\newline$
\noindent\textbf{Step 4.} \textbf{The general case: getting around the local constancy/majorant assumption \eqref{decmtrl22}}
$\newline$

We start this section with several straightforward observations:
\begin{itemize}
\item a simple inspection of the proof--one can focus entirely on the arguments \eqref{l2cont} and \eqref{VM6} at Step 3--gives that our reasonings remain valid if one replaces assumption \eqref{decmtrl22} by its slightly weaker quadratic analog
\begin{equation}\label{decmtrl22q}
	|\F_{j,k}(\xi,x)|^2\lesssim 2^m\,\int_{\frac{k}{2^m}}^{\frac{k+1}{2^m}} |\F_{j,k}(\xi,\tilde{x})|^2\,d\tilde{x}\: {\text{ and }}\:|\G_{j,k}(y,\eta)|^2\lesssim 2^m\,\int_{\frac{k}{2^m}}^{\frac{k+1}{2^m}} |\G_{j,k}(\tilde{y},\eta)|^2\,d\tilde{y}\,,
\end{equation}
where the above is assumed to hold for any $j,\,k\sim 2^m$.

\item moreover, our proof continues to hold if one replaces the frequency-\emph{pointwise} condition \eqref{decmtrl22q} by its frequency-\emph{mean} analogue given by
\begin{align}\label{decmtrl22qf}
	&\int_{j}^{j+1}|\F_{j,k}(\xi,x)|^2\,d\xi\lesssim 2^m\,\int_{\frac{k}{2^m}}^{\frac{k+1}{2^m}}\int_{j}^{j+1} |\F_{j,k}(\xi,\tilde{x})|^2\,d\tilde{x}\,d\xi\quad {\text{ and }}\nonumber\\
&\int_{j}^{j+1} |\G_{j,k}(y,\eta)|^2\,d\eta\lesssim 2^m\,\int_{\frac{k}{2^m}}^{\frac{k+1}{2^m}}\int_{j}^{j+1}  |\G_{j,k}(\tilde{y},\eta)|^2\,d\tilde{y}\,d\eta\qquad\forall\:j\sim 2^m\,.
\end{align}
\end{itemize}
Combining now the intuition provided by \eqref{fg} with that offered by \eqref{decmtrl22qf} it is only natural to introduce the \emph{spatial-pointwise/frequency-mean} quantities:
\begin{equation}\label{fgmodif}
	f_{j,k}(y):=2^{-m}\,\int_{\R}|\mathcal{F}_{j,k}(\xi,y)|^2\,d\xi\quad\textrm{ and }\quad g_{j,k}(x):=2^{-m}\,\int_{\R}|\mathcal{G}_{j,k}(x,\eta)|^2\,d\eta\,.
\end{equation}
Notice now that the original coefficients defined in \eqref{fg} represent precisely the  mean values of the quantities introduced in \eqref{fgmodif}, that is
\begin{equation}\label{fgmodiff}
	f_{j,k}=\frac{1}{2^{-m}}\,\int_{\frac{k}{2^m}}^{\frac{k+1}{2^m}}f_{j,k}(y)\,dy\quad\textrm{ and }\quad g_{j,k}=\frac{1}{2^{-m}}\,\int_{\frac{k}{2^m}}^{\frac{k+1}{2^m}}g_{j,k}(x)\,dx.
\end{equation}

Once at this point the are at least two possible options\footnote{The investigation of these options served as an inspiration for the continuous approach developed in \cite{HsL24}.}:
\begin{itemize}
\item the first one, more intuitive but more pedestrian, is---at informal level---to simply cut off the spikes of $|f_{j,k}(y)|$ by collecting the set of exceptional values (\emph{i.e.} the values that are large relative to the mean $f_{j,k}$) into a set that has a controlled (small) size and apply an argument resembling the one at Stage 3.1. However this reasoning requires some extra-care since one has to remove the $j-$dependence of the $y-$exceptional set via a construction of a suitable mixed maximal-type function that is acting on the $y-$fibers of $f$. 

\item the second  one, is more direct and somehow surprising through its simplicity: in order to complete our proof, it is enough to adapt the earlier definitions and estimates to a \emph{spatial-pointwise} behavior instead of the previous \emph{spatial-mean} behavior!
\end{itemize}

For brevity and concreteness we provide a brief outline of the latter approach: 

The initial definitions in \eqref{decA}--\eqref{A3} have now to be modified to allow a spatial dependence; thus
\begin{equation}\label{decAmod}
	\A:=\A_{\F}^{1}[y]\cup \A_{\F}^{2}[y]\cup \A_{\F}^{3}[y] \,,
\end{equation}
with
\begin{itemize}
	\item the \emph{set of heavy $y$-fibers}
	\begin{equation}\label{A1mod}
		\A_{\F}^{1}[y]:=\Big\{(j,k)\in\A\,|\,\sum_{l\sim 2^m} f_{l,k}(y)\gtrsim 2^{-m+m\v},\:\:j\sim 2^m\Big\}\,;
	\end{equation}
	\item  the \emph{$y$-variable set of heavy masses}
	\begin{equation}\label{A2mod}
		\A_{\F}^{2}[y]:=\Big\{(j,k)\in \A\setminus \A_{\F}^{1}\,|\,2^{-m+m\v-m\d}\lesssim f_{j,k}(y)\lesssim 2^{-m+m\v}\Big\}\,.
	\end{equation}
	\item the \emph{$y$-variable set of light masses}
	\begin{equation}\label{A3mod}
		\A_{\F}^{3}[y]:=\Big\{(j,k)\in \A\setminus \A_{\F}^{1}\,|\,f_{j,k}(y)\lesssim 2^{-m+m\v-m\d}\Big\}\,.
	\end{equation}
\end{itemize}
The adaptations of the remaining definitions \eqref{decAG}--\eqref{I3} are now straightforward. With these, we are left with just a few remarks:
\begin{itemize}
\item the treatment of $I_m^1$ part of Stage 3.1 does not involve any significant changes;

\item the treatment of $I_m^2$ performed within Stage 3.2 requires just a bit more care: indeed, relation \eqref{reduct211} has now to be reshaped as
\begin{equation}\label{reduct211mod}
	I_m^2\lesssim \frac{1}{2^{\frac{m}{2}}} \int_{[1,2]}\,A_2(\F)^{\frac{1}{2}}(y)\,B_2(\G)^{\frac{1}{2}}(y)\,dy\,,
\end{equation}
with\footnote{Notice here that, for a fixed $y$, the expression in \eqref{reduct2amod} has only one $k-$term that is nonzero.}
\begin{equation}\label{reduct2amod}
A_2(\F)(y):=\sum_k \left(\sum_{j\in \A_{\F}^{2}[y]}\int_{\R} \left|\F_{j,k}(\xi,y)\right| \,d\xi\right)^2\,
\end{equation}
and
\begin{equation}\label{reduct2bmod}
B_2(\G)(y):=\sum_k \left(\int_{\R}\left|\int_{\R}\G(x,\eta)\,e^{i\,\eta\,\frac{k}{2^m}}\,e^{-i\,\frac{j(y)^2}{4 \eta}}\,d\eta\right| dx \right)^2\:,
\end{equation}
where now the measurable function $j(\cdot)$ depends on $y$ instead of $k$.

Once at this point, it is enough to notice that: (i) the key estimate \eqref{keystage2} remains valid for its $y-$dependent analogue $\A_{\F}^{2}[y]$, and, (ii) the reasonings involved in bounding $B_2(\G)$ can be transferred line by line to $B_2(\G)(y)$ since the $y-$dependence of the function $j(\cdot)$ is irrelevant.

\item finally, the treatment of $I_m^3$ at Stage 3.3, involves only two spatial-pointwise dependent relations: the first one--\eqref{l2cont}--becomes now
\begin{equation}\label{l2contmodif}
	\int_{\R^3} |F_{\tau,\mu}(\xi)|^2\,d\xi\,d\tau\,d\mu\lesssim 2^{m}\,\int_{[1,2]}\sum_{(j,k),(j_1,k)\in\A_{\F}^{3}[y]} f_{j,k}(y)\,f_{j_1,k}(y)\,dy\lesssim 2^{\v m}\,,
\end{equation}
where for the last inequality we used \eqref{A3mod}, while the second one given by \eqref{VM6}, becomes
\begin{align}\label{VM6mod}
	B_3(\G)\lesssim2^{m}\,\int_{[1,2]}\sum_{{(j,k)\in\A_{\G}^{3}[x]}\atop{(j_1,k)\in\A_{\G}^{3}[x]}} \frac{g_{j,k}(x)\,g_{j_1,k}(x)}{1+|j-j_1|^{\frac{1}{2}}}\,dx \qquad\qquad\qquad\nonumber\\
\:\:\:\:\:\lesssim 2^{m}\,m^{\frac{1}{2}}\,\int_{[1,2]}\sum_{(j,k)\in\A_{\G}^{3}[x]} g_{j,k}(x)\,
	\Bigg(\sum_{j_1\in\A_{\G}^{3}[x](k)} |g_{j_1,k}(x)|^2\Bigg)^{\frac{1}{2}}dx\lesssim
m^{\frac{1}{2}}\,2^{\v m}\,2^{-\frac{\d}{2}m}\,.
\end{align}
\end{itemize}

This completes our proof.

\end{proof}


\end{document}